\documentclass[headings=optiontoheadandtoc,toc=idx,leqno]{scrbook}
\usepackage{amsmath,amssymb,amsthm}

\usepackage{enumerate}

\usepackage[T1]{fontenc}
\usepackage{mathpazo,courier}
\usepackage[scaled]{helvet}
\usepackage{calrsfs} 
\usepackage{cancel}

\usepackage{imakeidx} 
\makeindex

\usepackage[long]{datetime}
\usepackage[headsepline]{scrlayer-scrpage} 
\setkomafont{pageheadfoot}{\normalfont}
\setkomafont{pagefoot}{\slshape}
\defpagestyle{scrheadings}{%
{\pagemark\hfill\headmark}{\headmark\hfill\pagemark}{}
}{
{Lecture Notes\hfill}{\hfill Version of \today, \currenttime}{}
}

\usepackage[usenames,dvipsnames]{xcolor}
\definecolor{darkblue}{rgb}{0.0, 0.0, 0.55}
\usepackage[pagebackref,colorlinks,linkcolor=BrickRed,citecolor=OliveGreen,urlcolor=darkblue,hypertexnames=true]{hyperref}

\usepackage{tikz}
\usetikzlibrary{matrix}
\usetikzlibrary{shapes.geometric}
\newcommand{\warningsign}{\tikz[baseline=-.75ex] \node[scale=0.8,shape=regular polygon, regular polygon sides=3, inner sep=0pt, draw, thick] {\textbf{!}};}

\theoremstyle{plain}
\newtheorem{pro}{Proposition}[section]
\newtheorem{pront}[pro]{Proposition and Notation}
\newtheorem{proterm}[pro]{Proposition and Terminology}
\newtheorem{thm}[pro]{Theorem}
\newtheorem{dfpro}[pro]{Definition and Proposition}
\newtheorem{thmdef}[pro]{Theorem and Definition}
\newtheorem{cordef}[pro]{Corollary and Definition}
\newtheorem{cor}[pro]{Corollary}
\theoremstyle{definition}
\newtheorem{lem}[pro]{Lemma}
\newtheorem{df}[pro]{Definition}
\newtheorem{dflem}[pro]{Definition and Lemma}
\newtheorem{term}[pro]{Terminology}
\newtheorem{convention}[pro]{Convention}
\newtheorem{notation}[pro]{Notation}
\newtheorem{dfrem}[pro]{Definition and Remark}
\newtheorem{notrem}[pro]{Notation and Remark}
\newtheorem{notterm}[pro]{Notation and Terminology}
\newtheorem{termnot}[pro]{Terminology and Notation}
\newtheorem{remterm}[pro]{Remark and Terminology}
\newtheorem{rem}[pro]{Remark}

\newtheorem{motivation}[pro]{Motivation}
\newtheorem{ex}[pro]{Example}
\newtheorem{exo}[pro]{Exercise}
\newtheorem{nt}[pro]{Notation}
\newtheorem{reminder}[pro]{Reminder}

\usepackage{bbm} 
\newcommand\ii{\mathbbm i}

\usepackage{stmaryrd} 

\newcommand\bigdotcup{\mathop{\dot\bigcup}}

\newcommand\et{\mathbin{\&}}
\newcommand\biget{\mathop{\text{\LARGE\&}}}

\newcommand\N{\mathbb N}
\newcommand\Z{\mathbb Z}
\newcommand\Q{\mathbb Q}
\newcommand\R{\mathbb R}
\newcommand\C{\mathbb C}
\newcommand\F{\mathbb F}

\newcommand\pow{\mathcal P}

\usepackage{ushort}
\newcommand\x{\ushort X}
\newcommand\y{\ushort Y}
\newcommand\g{\ushort g}

\newcommand\al\alpha
\newcommand\be\beta
\newcommand\la\lambda
\newcommand\ep\varepsilon
\newcommand\La\Lambda
\newcommand\De\Delta
\newcommand\de\delta
\newcommand\ga\gamma
\newcommand\Ph\Phi
\newcommand\Ps\Psi
\newcommand\ph\varphi
\newcommand\ps\psi
\newcommand\si\sigma
\newcommand\ze\zeta
\newcommand\rh\varrho
\newcommand\io\iota
\newcommand\ch\chi
\newcommand\Om\Omega
\newcommand\Ga\Gamma

\renewcommand\O{\mathcal O}
\newcommand\p{\mathfrak p}
\newcommand\q{\mathfrak q}
\newcommand\m{\mathfrak m}

\newcommand\cc[2]{\overline{#1}^{\substack{\vspace*{-0.5em}{\scriptscriptstyle#2}}}}

\newcommand\mypm[2]{\mathbin{\smash{\raisebox{0.5ex}{$\underset{\smash #2}{\smash #1}$}}}} 

\DeclareMathOperator\relint{relint}

\DeclareMathOperator\irr{irr}
\DeclareMathOperator\supp{supp}
\DeclareMathOperator\id{id}
\DeclareMathOperator\sgn{sgn}
\DeclareMathOperator\conv{conv}
\DeclareMathOperator\cone{cone}
\DeclareMathOperator\chara{char}
\DeclareMathOperator\Aut{Aut}
\DeclareMathOperator\End{End}
\DeclareMathOperator\GL{GL}
\DeclareMathOperator\rk{rk}
\DeclareMathOperator\sg{sg}
\DeclareMathOperator\tr{tr}
\DeclareMathOperator\ev{ev}
\DeclareMathOperator\set{Set}
\DeclareMathOperator\class{Class}
\DeclareMathOperator\transfer{Transfer}
\DeclareMathOperator\slim{Slim}
\DeclareMathOperator\fatten{Fatten}
\DeclareMathOperator\lf{LF}
\DeclareMathOperator\spec{spec}
\DeclareMathOperator\sper{sper}
\DeclareMathOperator\qf{qf}
\DeclareMathOperator\rrad{rrad}
\DeclareMathOperator\rspec{rspec}
\DeclareMathOperator\rnil{rnil}
\DeclareMathOperator\st{st}
\DeclareMathOperator\extr{extr}
\DeclareMathOperator\aff{aff}
\DeclareMathOperator\hess{Hess}
\DeclareMathOperator\convbd{convbd}

\newcommand{\operator}[1]{\mathop{\vphantom{\sum}\mathchoice
{\vcenter{\hbox{\huge $#1$}}}
{\vcenter{\hbox{\Large $#1$}}}{#1}{#1}}\displaylimits}
\newcommand{\Et}{\operator{\mathrm{\&}}}

\newcommand\alal[2]{$\begin{Bmatrix}\text{#1}\\\text{#2}\end{Bmatrix}$}
\newcommand\alalal[3]{$\begin{Bmatrix}\text{#1}\\\text{#2}\\\text{#3}\end{Bmatrix}$}
\newcommand\malal[2]{\begin{Bmatrix}#1\\#2\end{Bmatrix}}
\newcommand\malalal[3]{\begin{Bmatrix}#1\\#2\\#3\end{Bmatrix}}

\begin{document}
\titlehead{\url{http://www.math.uni-konstanz.de/~schweigh/}}
\subject{Lecture notes}
\title{Real Algebraic Geometry,\\Positivity and Convexity}
\subtitle{2020--2022}
\publishers{Universität Konstanz, Germany}
\author{Markus Schweighofer}
\uppertitleback{\textbf{Preface.}
Chapters 1--4 are the lecture notes of my course ``Real Algebraic Geometry I''
from the winter term 2020/2021.
Chapters 5--8 are the lecture notes of its continuation ``Real Algebraic Geometry II''
from the summer term 2021.
Chapters 9--10 are the lecture notes of its further continuation ``Geometry of Linear Matrix Inequalities''
from the winter term 2021/2022.

\bigskip
Because of the COVID-19 pandemic, I produced a series of accompanying screencasts of varying length where go in detail through the material.
These can be found within a playlist of
my Youtube Channel \begin{center}\url{https://www.youtube.com/channel/UCcsp7yAQjJQHx3SN3aL_r1A}\end{center}
Other material like
the exercise sheets are available somewhere on my website:
\begin{center}
\url{http://www.math.uni-konstanz.de/~schweigh/}
\end{center}

\bigskip
Whoever and wherever you are, please report any ambiguities and errors (including typos) to:
\begin{center}
\texttt{markus.schweighofer@uni-konstanz.de}
\end{center}
}
\lowertitleback{This document is to a large extent based on the work of other people.
For the relevant scientific sources, we refer to the literature referenced at the end of this 
document as well as the bibliographies of the books \cite{abr,bcr,bpr,ks,mar,pd}.
I would like to thank the numerous people that helped to improve these lecture notes:
First of all, I thank my former doctoral student Tom-Lukas Kriel, especially for coauthoring Chapters 9 and 10. Thanks go also
to Alexander Taveira Blomenhofer, Sebastian Gruler and María López Quijorna for leading several accompanying exercise classes.
Last but not least numerous participants of my lectures pointed out errors and typos. Among them I mention especially
Alexander Taveira Blomenhofer, Johannes Buchwald, Nicolas Daans, Carl Eggen, Jakob Everling, Rüdiger Grunwald, Pirmin Klink, Arne Lien, Leonhard Nenno, Emre Öztürk, Jonas Riehle, David Sawall, Joschka Schmidt, Alison Surey in alphabetical order.}
\date{Version of \today, \currenttime}
\frontmatter
\pagestyle{empty}
\maketitle
\pagestyle{scrheadings}
\manualmark

\tableofcontents

\chapter{Introduction}

The study of polynomial equations is a canonical subject in mathematics education, as is illustrated by the following examples: Quadratic equations in one variable (high school), systems of linear equations (linear algebra), polynomial equations in one variable and their symmetries (algebra, Galois theory), diophantine equations (number theory) and systems of polynomial equations (algebraic geometry, commutative algebra).

\bigskip\noindent
In contrast to this, the study of polynomial inequalities (in the sense of ``greater than'' or ``greater or equal than'') is mostly neglected even though it is much more important for applications: Indeed, in applications one often searches for a real solution rather than a complex one (as in classical algebraic geometry) and this solution must not necessarily be exact but only approximate.

\bigskip\noindent
In a course about linear algebra there is frequently no time for linear optimization. An introductory course about algebra usually treats groups, rings and fields but disregards ordered and real closed fields as well as preorders or prime cones of rings. In a first course on algebraic geometry there is often no special attention paid to the real part of a variety and in commutative algebra quadratic modules are practically never treated.

\bigskip\noindent
Most algebraists do not even know the notion of a preorder although it is as important for the study of systems of polynomial inequalities as the notion of an ideal is for the study of systems of polynomial equations. People from more applied areas such as numerical analysis, mathematical optimization or functional analysis know often more about real algebraic geometry than some algebraists, but often do not even recognize that polynomials play a decisive role in what they are doing. There are for example countless articles from functional analysis which are full of equations with binomial coefficients which turn out to be just disguised simple polynomial identities.

\bigskip\noindent
In the same way as the study of polynomial systems of equations leads to the study of rings and their generalizations (such as modules), the study of systems of polynomial inequalities leads to the study of rings which are endowed with something that resembles an order. This additional structure raises many new questions that have to be clarified. These questions arise already at a very basic level so that we need as prerequisites only basic linear algebra, algebra and analysis. In particular, at least the first half of this
course is really extremely well suited to students heading for students enrolled in programs for mathematics education. It includes several topics which are directly relevant for high school teaching.

\bigskip\noindent
To arouse the reader's curiosity, we present the following table.
It contains on the left column notions we assume the reader is familiar with. On the right column we name what could be seen
more or less as their real counterparts mostly introduced in this course.

\begin{tabular}{r|l}
Algebra&Real Algebra\\
Algebraic Geometry&Real Algebraic Geometry\\
systems of polynomial equations&systems of polynomial inequalities\\
``$=$''&``$\ge$''\\
complex solutions&real solutions\\
$\C$&$\R$\\
algebraically closed fields&real closed fields\\
fields&ordered fields\\
ideals&preorders\\
prime ideals&prime cones\\
spectrum&real spectrum\\
Noetherian&quasi-compact\\
radical&real radical\\
fundamental theorem of algebra&fundamental theorem of algebra\\
Aachen, Aalborg, Aarhus, \dots&Dortmund, Dresden, Dublin, Innsbruck, \dots\\
\dots, Zagreb, Zürich&\qquad\quad \dots, Konstanz, Leipzig, Ljubljana, Rennes
\end{tabular}

\bigskip\noindent
It is intended that the fundamental theorem of algebra appears on both sides of the table. In its usual form, it says
that each non-constant univariate complex polynomial has a complex root. In Section \ref{sec:rcf}, we will formulate it in a
``real'' way. The difficulties one has to deal with in the ``real world'' become already apparent when one asks the corresponding
``real question'': When does a univariate complex polynomial have a real root? The answer to this will be given in Section
\ref{sec:hermite} and requires already quite a bit of thoughts.

\bigskip\noindent
Traditionally, Real Algebraic Geometry has many ties with fields like Model Theory, Valuation Theory, Quadratic Form Theory
and Algebraic Topology. In this lecture, we mainly emphasize however connections to fields like Optimization, Functional Analysis and Convexity
that came up during the recent years and are now fully established.

\bigskip\noindent
Throughout the lecture, $\N:=\{1,2,3,\dots\}$ and $\N_0:=\{0\}\cup\N$ denote the set of positive and nonnegative integers, respectively.

\mainmatter

\chapter{Ordered fields}

\section{Orders of fields}

\begin{reminder}\label{ordered-set}
Let $M$ be a set. An \emph{order} on $M$ is a relation $\le$ on $M$ such that for all $a,b,c\in M$:
\[
\begin{array}{rcl}
&a\le a & \text{(reflexivity)}\\
&(a\le b \et b\le c)\implies a\le c & \text{(transitivity)}\\
&(a\le b \et b\le a)\implies a=b & \text{(antisymmetry)}\\
\text{and}&a\le b\text{ or }b\le a & \text{(linearity)}
\end{array}
\]
In this case, $(M,\le)$ (or simply $M$ if $\le$ is clear from the context) is called an \emph{ordered set}. For $a,b\in M$, one defines
\begin{align*}
a<b&:\iff a\le b\et a\ne b,\\
a\ge b&:\iff b\le a
\end{align*}
and so on.
\end{reminder}

\begin{df}\label{ordhom}
Let $(M,\le_1)$ and $(N,\le_2)$ be ordered sets and $\ph\colon M\to N$ be a map.
Then $\ph$ is called a \emph{homomorphism} (of ordered sets)
or \emph{monotonic} if \[a\le_1 b\implies\ph(a)\le_2\ph(b)\] for all $a,b\in M$. If $\ph$ is
\alal{injective}{bijective} and if
\[a\le_1 b\iff\ph(a)\le_2\ph(b)\]
for all $a,b\in M$, then $\ph$ is called an \alal{\emph{embedding}}{\emph{isomorphism}} (of ordered sets).
\end{df}

\begin{pro}\label{automono}
Let $(M,\le_1)$ and $(N,\le_2)$ be ordered sets and $\ph\colon M\to N$ a homomorphism. Then the following are equivalent:
\begin{enumerate}[\normalfont(a)]
\item $\ph$ is an embedding
\item $\ph$ is injective
\item $\forall a,b\in M:(\ph(a)\le_2\ph(b)\implies a\le_1b)$
\end{enumerate}
\end{pro}

\begin{proof}
\underline{(c)$\implies$(b)} \quad Suppose (c) holds and let $a,b\in M$ such that $\ph(a)=\ph(b)$. Then
$\ph(a)\le_2\ph(b)$ and $\ph(a)\ge_2\ph(b)$. Now (c) implies $a\le_1b$ and $a\ge_1b$. Hence $a=b$.

\smallskip
\underline{(b)$\implies$(c)} \quad Suppose (b) holds and let $a,b\in M$ with $a\not\le_1b$. To show: $\ph(a)\not\le_2\ph(b)$. We have
$a>_1b$ and it suffices to show $\ph(a)>_2\ph(b)$. From $a\ge_1b$ it follows by the monotonicity of $\ph$ that $\ph(a)\ge_2\ph(b)$.
From $a\ne b$ and the injectivity of $\ph$ we get $\ph(a)\ne\ph(b)$.

\smallskip
From (b)$\iff$(c) and (a)$\iff$((b)$\et$(c)) [$\to$ \ref{ordhom}] the claim now follows.
\end{proof}

\begin{df}\label{def-ordered-field}
Let $K$ be a field. An \emph{order} of $K$ is an order $\le$ on $K$ such that for all $a,b,c\in K$ we have:
\[
\begin{array}{rcl}
&a\le b\implies a+c\le b+c & \text{(monotonicity of addition)}\\
\text{and}&(a\le b \et c\ge0)\implies ac\le bc & \text{(monotonicity of multiplication).}\\
\end{array}
\]
In this case, $(K,\le)$ (or simply $K$ when $\le$ is clear from the context) is called an \emph{ordered field}.
\end{df}

\begin{df}\label{ordfieldhom}
Let $(K,\le_1)$ and $(L,\le_2)$ be ordered fields.

A field homomorphism (or equivalently, field embedding!) $\ph\colon K\to L$
is called a \emph{homomorphism} or \emph{embedding} of ordered fields if $\ph$ is monotonic
(pay attention to \ref{automono} together with the fact that field homomorphisms are injective).
If $\ph$ is moreover surjective, then $\ph$ is called an \emph{isomorphism} of ordered fields. 

If there exists an embedding of ordered fields
from $(K,\le_1)$ into $(L,\le_2)$, then $(K,\le_1)$ is called \emph{embeddable} in $(L,\le_2)$ and one denotes
$(K,\le_1)\hookrightarrow(L,\le_2)$.
If there is an isomorphism of ordered fields from $(K,\le_1)$ to $(L,\le_2)$, then $(K,\le_1)$ and $(L,\le_2)$ are called
\emph{isomorphic}. This is denoted by $(K,\le_1)\cong(L,\le_2)$.

$(K,\le_1)$ is called an \emph{ordered subfield} of $(L,\le_2)$, or equivalently $(L,\le_2)$ an \emph{ordered extension field} of $(K,\le_1)$,
if $(K,\le_1)\to(L,\le_2),\ a\mapsto a$ is an embedding, that is if $K$ is a subfield of $L$ and $(\le_1)=(\le_2)\cap(K\times K)$. For every
subfield of $L$ there is obviously a unique order making it into an ordered subfield of $(L,\le_2)$. This order is called the order \emph{induced}
by $(L,\le_2)$.
\end{df}

\begin{pro}\label{squares}
Let $(K,\le)$ be an ordered field. Then $a^2\ge0$ for all $a\in K$.
\end{pro}

\begin{proof}
Let $a\in K$.
When $a\ge0$ this follows immediately from the monotonicity of multiplication [$\to$ \ref{def-ordered-field}]. When $a\le0$ the monotonicity of
addition [$\to$ \ref{def-ordered-field}] yields $0=a-a\le-a$, whence $-a\ge0$ and therefore $a^2=(-a)^2\ge0$.
\end{proof}

\begin{pro}\label{qembeds}
Let $(K,\le)$ be an ordered field. Then $K$ is of characteristic $0$ and the uniquely determined field homomorphism
$\Q\to K$ is an embedding of ordered fields ${(\Q,\le_\Q)}\hookrightarrow(K,\le)$. Hence $(K,\le)$ can be seen as an ordered
extension field of $(\Q,\le_{\Q})$. In particular, for $K=\Q$ it follows that $(\le_\Q)=(\le)$, i.e., $\Q$ can only be ordered in the familiar way.
\end{pro}

\begin{proof} From \ref{squares} we have $0\le1^2=1$ in $(K,\le)$. Using the monotonicity of the addition, we deduce
\begin{equation}
\tag{$*$}
0\le1\le1+1\le1+1+1\le\dots
\end{equation}
If we had $\chara K\ne0$, then $(*)$ would give $0\le1\le0$ by the transitivity of $\le$ which would imply
$0=1$ in $K$ by the antisymmetry of $\le$, contradicting the definition of a field. Let $\ph$ denote the field homomorphism
$\Q\to K$ and let $a,b\in\Q$ with $a\le_{\Q}b$. To show: $\ph(a)\le\ph(b)$. Write $a=\frac kn$ and $b=\frac\ell n$ with
$k,\ell\in\Z$ and $n\in\N$. Then
\[\ph(n)=\underbrace{1+\dots+1}_{\text{$n$ times}}\underset{\chara K=0}{\overset{(*)}>}0\] and, by the monotonicity of multiplication
and Proposition \ref{squares}, also
\[\frac1{\ph(n)}=\left(\frac1{\ph(n)}\right)^2\ph(n)\ge0.\]
Hence it suffices to show that $\ph(a)\ph(n)\le\ph(b)\ph(n)$.
This reduces to $\ph(an)\le\ph(bn)$, that is $\ph(k)\le\ph(\ell)$, or equivalently $\ph(\ell-k)\ge0$.
But due to $\ell-k\ge_\Q0$ this follows from $(*)$.
\end{proof}

\begin{pront}\label{introabssgn}
Let $(K,\le)$ be an ordered field. Then for every $a\in K^\times$ there are uniquely determined
$\sgn a\in\{-1,1\}$ (``sign'' of $a$) and $|a|\in K_{\ge0}:=\{x\in K\mid x\ge0\}$ (``absolute value'' of $a$) such that
\[a=(\sgn a)|a|.\]
One declares $\sgn0:=|0|:=0$. It follows that $|ab|=|a||b|$, $\sgn(ab)=(\sgn a)(\sgn b)$ and
$|a+b|\le|a|+|b|$ for all $a,b\in K$.
\end{pront}

\begin{proof}
The first part is very easy. Let now $a,b\in K$. Then $ab=(\sgn a)(\sgn b)|a||b|$, implying $|ab|=|a||b|$ as well as
$\sgn(ab)=(\sgn a)(\sgn b)$. For the claimed triangle inequality, we can suppose $a+b\ge0$ (otherwise replace $a$ by $-a$ and $b$ by $-b$).
Then $|a+b|=a+b\le a+|b|\le|a|+|b|$.
\end{proof}

\begin{df}\label{archetcdef}
Let $(K,\le)$ be an ordered field.
\begin{enumerate}[(a)]
\item $(K,\le)$ is called \emph{Archimedean} if $\forall a\in K:\exists N\in\N:a\le N$
(or equivalently, $\forall a\in K:\exists N\in\N:-N\le a$).
\item A sequence $(a_n)_{n\in\N}$ in $K$ is called
\begin{itemize}
\item a \emph{Cauchy sequence} if $\forall\ep\in K_{>0}:\exists N\in\N:\forall m,n\ge N:|a_m-a_n|<\ep$,
\item \emph{convergent to $a\in K$} if $\forall\ep\in K_{>0}:\exists N\in\N:\forall n\ge N:|a_n-a|<\ep$
(one easily shows that $a$ is then uniquely determined and writes $\lim_{n\to\infty}a_n=a$),
\item \emph{convergent} if there is some $a\in K$ such that $\lim_{n\to\infty}a_n=a$.
\end{itemize}
We call $(K,\le)$ \emph{Cauchy complete} if every Cauchy sequence converges in $K$ (by the way it is immediate that every convergent
sequence is a Cauchy sequence).
\item We call a subset $A\subseteq K$ \emph{bounded from above} if $K$ contains an upper bound for $A$ (meaning some
$b\in K$ such
that $\forall a\in A:a\le b$). We call $(K,\le)$ \emph{complete} if every nonempty subset of $K$ bounded from above possesses a least
upper bound, i.e., a supremum.
\end{enumerate}
\end{df}

\begin{pro}\label{qdense}
Let $(K,\le)$ be an ordered field. Then the following are equivalent:
\begin{enumerate}[\normalfont(a)]
\item $(K,\le)$ is Archimedean
\item $\forall a,b\in K:(a<b\implies\exists c\in\Q:a<c<b)$
\end{enumerate}
\end{pro}

\begin{proof}
\underline{(b)$\implies$(a)} \quad Suppose (b) holds and let $a\in K$. To show: $\exists N\in \N:a\le N$. WLOG $a>0$. To show: $\exists N\in\N:
\frac1N\le\frac 1a$. Choose $c\in\Q$ such that $0<c<\frac 1a$. Write $c=\frac mN$ for certain $m,N\in\N$. Then
$\frac 1N\le\frac mN=c<\frac 1a$.

\smallskip
\underline{(a)$\implies$(b)} \quad Suppose (a) holds and let $a,b\in K$ such that $a<b$. Choose $N\in\N$ such that
$\frac1{b-a}<N$. Then $\frac 1N<b-a$ and hence $a+\frac 1N<b$. Now choose the smallest $m\in\Z$ such that
$a<\frac mN$. If we had $\frac mN\ge b$, then $a+\frac 1N<\frac mN$ and therefore $a<\frac{m-1}N$, contradicting our choice of $m$.
Therefore $a<\frac mN<b$.
\end{proof}

\begin{lem} Let $(K,\le)$ be an Archimedean ordered field. Then
\[K=\left\{\lim_{n\to\infty}a_n\mid
\text{$(a_n)_{n\in\N}$ sequence in $\Q$ that converges in $K$}
\right\}.\]
\end{lem}

\begin{proof}
Let $a\in K$. We have to show that there is a sequence $(a_n)_{n\in\N}$ in $\Q$ that converges in $K$ to $a$. Choose
for every $n\in\N$ according to \ref{qdense} some $a_n\in\Q$ such that $a\le a_n<a+\frac 1n$. Let $\ep\in K_{>0}$. Choose
$N\in\N$ such that $\frac 1\ep<N$. For $n\ge N$ we now have
$|a_n-a|=a_n-a<\frac 1n\le\frac 1N<\ep$.
\end{proof}

\begin{lem} Suppose $(K,\le)$ is an Archimedean ordered field and $(a_n)_{n\in\N}$ is a sequence in $\Q$. Then the following are
equivalent:
\begin{enumerate}[(a)]
\item $(a_n)_{n\in\N}$ is a Cauchy sequence in $(\Q,\le_\Q)$
\item $(a_n)_{n\in\N}$ is a Cauchy sequence in $(K,\le)$
\end{enumerate}
\end{lem}

\begin{proof} This follows easily from \ref{qdense}.
\end{proof}

\begin{exo}
Suppose $(K,\le)$ is an ordered field and $(a_n)_{n\in\N}$, $(b_n)_{n\in\N}$ are convergent sequences in $K$. Then
\[\lim_{n\to\infty}(a_n+b_n)=\left(\lim_{n\to\infty}a_n\right)+\left(\lim_{n\to\infty}b_n\right)\qquad\text{and}\qquad
\lim_{n\to\infty}(a_nb_n)=\left(\lim_{n\to\infty}a_n\right)\left(\lim_{n\to\infty}b_n\right).\]
\end{exo}

\begin{thm}\label{realschar}
Let $(K,\le)$ be an ordered field. Then the following are equivalent:
\begin{enumerate}[\normalfont(a)]
\item $(K,\le)$ is Archimedean and Cauchy complete
\item $(K,\le)$ is complete
\end{enumerate}
\end{thm}

\begin{proof}
\underline{(a)$\implies$(b)} \quad Suppose (a) holds and let $A\subseteq K$ be a nonempty subset bounded from above. Choose for  
every $n\in\N$ the smallest $k_n\in\Z$ such that $\forall a\in A:a\le\frac{k_n}n$ and set $a_n:=\frac{k_n}n\in\Q$ (use the Archimedean property!). Using again the Archimedean property, one can show easily
that $(a_n)_{n\in\N}$ is a Cauchy sequence and therefore convergent by hypothesis. We leave it as an exercise to the reader to show that
$a:=\lim_{n\to\infty}a_n$ is the least upper bound of $A$ in $(K,\le)$.

\underline{(b)$\implies$(a)} \quad We prove the contraposition.

First, suppose that $(K,\le)$ is not Archimedean, i.e., the set \[A:=\{a\in K\mid\forall N\in\N:a\le-N\}\] is not empty.
We claim  that $A$ does not have a least upper bound: Indeed, if $a\in K$ is an upper bound of $A$, then so is $a-1<a$ since
$A=\{a\in K\mid\forall N\in\Z:a\le N\}={\{a\in K\mid\forall N\in\Z:a+1\le N\}}=\{a-1\mid a\in K,\forall N\in\N:a\le N\}=A-1$.

Finally, suppose that  $(K,\le)$ is not Cauchy complete, say $(a_n)_{n\in\N}$ is a Cauchy sequence in $K$ that does not converge.
We claim that \[A:=\{a\in K\mid\exists N\in\N:\forall n\ge N:a\le a_n\}\] is nonempty and bounded from above but does not possess a
least upper bound. We leave this as an exercise to the reader.
\end{proof}

\begin{lem}\label{archemb}
Suppose $(K,\le)$ is an Archimedean ordered field and $(R,\le_R)$ a complete ordered field. Then there is exactly one
embedding $(K,\le)\hookrightarrow(R,\le_R)$. This embedding is an isomorphism if and only if $(K,\le)$ is complete.
\end{lem}

\begin{proof}
Exercise.
\end{proof}

\begin{thm}\label{introduce-the-reals}
There is a complete ordered field $(\R,\le)$. It is essentially unique, for if ${(K,\le_K)}$ is another complete
ordered field, then there is exactly one
isomorphism from $(K,\le_K)$ to $(\R,\le)$.
\end{thm}

\begin{proof}
The uniqueness is clear from \ref{archemb} together with \ref{realschar}. We only sketch the proof of existence and leave the details
as an exercise to the reader: Show that the Cauchy sequences in $\Q$ form a subring $C$ of $\Q^\N$ and that
\[I:=\left\{(a_n)_{n\in\N}\in C\mid\lim_{n\to\infty}a_n=0\right\}\]
is a maximal ideal of $C$. Set $\R:=C/I$. Show that
\[a\le b:\iff\exists (a_n)_{n\in\N},(b_n)_{n\in\N}\text{ in }C:(a=\cc{(a_n)_{n\in\N}}I ~\&~b=\cc{(b_n)_{n\in\N}}I
~\&~\forall n\in\N:a_n\le b_n)\]
($a,b\in\R$) defines an order $\le$ on $\R$. It is clear that $(\R,\le)$ is Archimedean. By Theorem \ref{realschar} it suffices to show that
$(\R,\le)$ is Cauchy complete. To this end, let $(a_n)_{n\in\N}$ be a Cauchy sequence in $(\R,\le)$. By \ref{qdense},
there exists a sequence $(q_n)_{n\in\N}$ in $\Q$ such that
$|a_n-q_n|<\frac{1}{n}$ for $n\in\N$. Now deduce from the fact that $(a_n)_{n\in\N}$ is a Cauchy sequence in $(\R,\le)$
that $(q_n)_{n\in\N}$ is such in $(\R,\le)$ and hence also in $(\Q,\le)$. Now $(q_n)_{n\in\N}\in C$. Set
$a:=\cc{(q_n)_{n\in\N}}{\scriptstyle I}$. It is enough to show $\lim_{n\to\infty}a_n=a$. Finally show that this is equivalent to
$\lim_{n\to\infty}q_n=a$ in $(K,\le)$ and prove the latter.
\end{proof}

\begin{cor}\label{archsubfieldreals}
$(\R,\le)$ is an Archimedean ordered field into which every Archimedean ordered field can be embedded. Up to isomorphy it is the only such ordered field.
\end{cor}

\begin{proof}
The first statement is clear from \ref{realschar}, \ref{archemb} and \ref{introduce-the-reals}. Uniqueness: Let $(K,\le_K)$ be another such
ordered field. Then \[(\R,\le)\overset\ph\hookrightarrow(K,\le_K)\overset\ps\hookrightarrow(\R,\le)\] and $\ps\circ\ph$ is the by \ref{archemb}
unique embedding $(\R,\le)\hookrightarrow(\R,\le)$, i.e., $\ps\circ\ph=\id$. This implies that $\ps$ is surjective. Hence $(K,\le_K)\cong(\R,\le)$.
\end{proof}

\begin{nt}\label{divnot}
Let $A$ be a ring. Then we often use suggestive notation to describe certain subsets of $A$ such as the following:
\begin{itemize}
\item $A^2=\{a^2\mid a\in A\}$\qquad(``squares'')
\item $\sum A^2=\{\sum_{i=1}^\ell a_i^2\mid\ell\in\N_0,a_i\in A\}$\qquad(``sums of squares'')
\item $\sum A^2T=\left\{\sum_{i=1}^\ell a_i^2t_i\mid\ell\in\N_0,a_i\in A,t_i\in T\right\}\qquad(T\subseteq A)$\\
\hspace*{15em}(``sums of elements of $T$ weighted by squares'')
\item $T+T=\{s+t\mid s,t\in T\}\qquad(T\subseteq A)$
\item $TT=\{st\mid s,t\in T\}\qquad(T\subseteq A)$
\item $-T=\{-t\mid t\in T\}\qquad(T\subseteq A)$
\item $T+aT=\{s+at\mid s,t\in T\}\qquad(T\subseteq A,a\in A)$
\end{itemize}
\end{nt}

\begin{pro}\label{unary-order}
Let $K$ be a field.
\begin{enumerate}[\normalfont(a)]
\item If $\le$ is an order of $K$ \emph{[$\to$ \ref{def-ordered-field}]}, then $P:=K_{\ge0}=\{a\in K\mid a\ge0\}$ has the following properties:
\begin{equation}
\tag{$*$} P+P\subseteq P,\quad PP\subseteq P,\quad P\cup-P=K\quad\text{ and }\quad P\cap-P=\{0\}.
\end{equation}
\item If $P$ is a subset of $K$ satisfying $(*)$, then the relation $\le_P$ on $K$ defined by
\[a\le_P b:\iff b-a\in P\qquad(a,b\in K)\]
is an order of $K$.
\item The correspondence
\begin{align*}
(\le)&\mapsto K_{\ge0}\\
(\le_P)&\mapsfrom P
\end{align*}
defines a bijection between the set of all orders on $K$ and the set of all subsets of $K$ satisfying $(*)$.
\end{enumerate}
\end{pro}

\begin{proof} (a) We get $P\mypm+\cdot P\subseteq P$ from the monotonicity of \alal{addition}{multiplication} [$\to$ \ref{def-ordered-field}],
$P\cup-P=K$ from the linearity [$\to$ \ref{ordered-set}] and $P\cap-P=\{0\}$ from the antisymmetry [$\to$ \ref{ordered-set}].

\smallskip
(b) We get reflexivity from $0\in P$, transitivity from $P+P\subseteq P$, antisymmetry from $P\cap-P=\{0\}$, linearity from
$P\cup-P=K$, monotonicity of addition from the definition of $\le_P$ and monotonicity of multiplication
$PP\subseteq P$.

\smallskip
(c) Suppose first that $\le$ is an order of $K$ and set $P:=K_{\ge0}$. Then $(\le)=(\le_P)$ since
$a\le b\iff b-a\ge0\iff b-a\in P\iff a\le_P b$ for all $a,b\in K$.
Conversely, let $P\subseteq K$ be given such that $P$ satisfies $(*)$. We show $K_{\ge_P\,0}=P$.
Indeed, \[K_{\ge_P\,0}=\{a\in K\mid 0\le_Pa\}=\{a\in K\mid a\in P\}=P.\]
\end{proof}

\begin{rem}\label{unaryrem}
\ref{unary-order}(c) allows us to view orders of fields $K$ as subsets of $K$. We reformulate some of the preceding notions and results in this new language:
\begin{enumerate}[(a)]
\item
Definition \ref{def-ordered-field}: Let $K$ be a field. An order of $K$ is a subset $P$ of $K$ satisfying
\[P+P\subseteq P,\quad PP\subseteq P,\quad P\cup-P=K\quad\text{ and }\quad P\cap-P=\{0\}.\]
\item Definition \ref{ordfieldhom}: Let $(K,P)$ and $(L,Q)$ be ordered fields. A field homomorphism $\ph\colon K\to L$ is called a
homomorphism or an embedding of ordered fields if $\ph(P)\subseteq Q$. One calls $(K,P)$ an ordered subfield of $(L,Q)$ if
$K$ is a subfield of $L$ and $P=Q\cap K$ (or equivalently $P\subseteq Q$).
\item Proposition \ref{squares}: Let $(K,P)$ be an ordered field. Then $K^2\subseteq P$.
\item Definition \ref{archetcdef}: An ordered field $(K,P)$ is called \emph{Archimedean} if
\[\forall a\in K:\exists N\in\N:N+a\in P,\]
($\iff P-\N=K\iff P+\Z=K\iff P+\Q=K$).
\end{enumerate}
\end{rem}

\section{Preorders}

\begin{df}\label{defpreorder}
Let $A$ be a commutative ring and $T\subseteq A$. Then $T$ is called a \emph{preorder} of $A$ if
$A^2\subseteq T$, $T+T\subseteq T$ and $TT\subseteq T$. If moreover $-1\notin T$, then $T$ is called a \emph{proper} preorder of $A$.
\end{df}

\begin{ex}\label{sqsm}
\begin{enumerate}[(a)]
\item If $A$ is a commutative ring, then $\sum A^2$ is the smallest preorder of $A$.
\item Every order of a field is a proper preorder.
\end{enumerate}
\end{ex}

\begin{pro}\label{diffsquare}
Let $A$ be a commutative ring with $\frac12\in A$ (i.e., $2\in A^\times$). Then
\[a=\left(\frac{a+1}2\right)^2-\left(\frac{a-1}2\right)^2\] for all $a\in A$. In particular, $A=A^2-A^2$.
\end{pro}

\begin{dfpro}\label{supportideal}
Let $A$ be a commutative ring with $\frac12\in A$ and $T\subseteq A$ a preorder. Then the \emph{support} $T\cap-T$ of
$T$ is an ideal of $A$.
\end{dfpro}

\begin{proof}
$T\cap-T$ is obviously a subgroup of (the additive group of) $A$ and we have
\[
\begin{array}{rcl}
A(T\cap-T)&\overset{\ref{diffsquare}}=&(A^2-A^2)(T\cap-T)\\
&\subseteq&(T-T)(T\cap-T)\\
&\subseteq&(T(T\cap-T))-(T(T\cap-T))\\
&\subseteq&((TT)\cap(-TT))+((-TT)\cap TT)\\
&\subseteq&(T\cap-T)+((-T)\cap T)\subseteq T\cap -T.
\end{array}
\]
\end{proof}

\begin{cor}\label{preproper}
Suppose $A$ is a commutative ring with $\frac12\in A$ and $T\subseteq A$ is a preorder. Then
\[\text{T is proper}\iff T\ne A.\]
\end{cor}

\begin{proof}
``$\Longrightarrow$'' trivial

\smallskip
``$\Longleftarrow$'' Suppose $T\ne A$. Then of course also $T\cap-T\ne A$. Since $T\cap-T$ is an ideal, we have $1\notin T\cap -T$.
Since $1=1^2\in T$, it follows that $1\notin-T$, i.e., $-1\notin T$.
\end{proof}

\begin{ex} In \ref{diffsquare}, \ref{supportideal} and \ref{preproper}, it is essential to require $\frac12\in A$. Take for example
$A=\F_2(X)$. Then $A^2=\F_2(X^2)$ since $\F_2(X)\to\F_2(X),\ p\mapsto p^2$ is a homomorphism (Frobenius). Therefore $A^2-A^2
=\F_2(X^2)\ne\F_2(X)$. Moreover $T:=\F_2(X^2)=\sum A^2$ is a preorder of $A$ but $T\cap-T=\F_2(X^2)$ is not an ideal of $A$
(since $1\in T\cap-T\ne\F_2(X)$). Also $T$ is not proper although $T\ne A$. The same is true for $\F_2[X]$ instead of $\F_2(X)$ and
from this one can get similar examples in the ring $\Z[X]$ (exercise).
\end{ex}

\begin{pro}\label{propfield}
Let $K$ be a field and $T\subseteq K$ a preorder. Then
\[\text{T is proper}\iff T\cap-T=\{0\}.\]
\end{pro}

\begin{proof}
If $\chara K=2$, then $-1=1\in T\cap-T$. Therefore suppose now $\chara K\ne2$. Then
\[-1\notin T\overset{1\in T}\iff1\notin T\cap-T\overset{\ref{supportideal}}{\underset{\substack{\text{$K$ field}\\\chara K\ne2}}\iff}T\cap-T=\{0\}.\]
\end{proof}

\begin{lem}\label{againpreorder}
Suppose $A$ is a commutative ring, $T\subseteq A$ a preorder and $a\in A$. Then $T+aT$ is again a preorder.
\end{lem}

\begin{proof}
$A^2\subseteq T\subseteq T+aT$, $(T+aT)+(T+aT)\subseteq(T+T)+a(T+T)\subseteq T+aT$ and
$(T+aT)(T+aT)\subseteq TT+aTT+aTT+a^2TT\subseteq T+aT+aT+A^2T\subseteq T+a(T+T)+TT\subseteq T+aT+T\subseteq T+aT$
\end{proof}

\begin{thm}\label{orderpreorder}
Let $K$ be a field and $P\subseteq K$. Then the following are equivalent:
\begin{enumerate}[\normalfont(a)]
\item $P$ is an order of $K$ \emph{[$\to$ \ref{unaryrem}]}.
\item $P$ is a proper preorder of $K$ \emph{[$\to$ \ref{defpreorder}]} such that $P\cup-P=K$.
\item $P$ is a maximal proper preorder of $K$.
\end{enumerate}
\end{thm}

\begin{proof}
\underline{(a)$\implies$(b)}\quad\ref{sqsm}(b)

\smallskip
\underline{(b)$\implies$(c)}\quad Suppose (b) holds and let $T$ be a proper preorder of $K$ with $P\subseteq T$. To show: $T\subseteq P$.
To this end, let $a\in T$. If $a$ was not in $P$, then $-a\in P\subseteq T$ (since $P\cup-P=K$) and therefore
$a\in T\cap-T\overset{\ref{propfield}}=\{0\}$ in contradiction to $0=0^2\in P$.

\smallskip
\underline{(c)$\implies$(a)}\quad Suppose (c) holds. Because of \ref{propfield}, we have to show only $P\cup-P=K$. Assume $P\cup-P\ne K$.
Choose then $a\in K$ such that $a\notin P$ and $-a\notin P$. Then $P+aP$ and $P-aP$ are preorders according to Lemma \ref{againpreorder}
and both contain $P$ properly (note that $0,1\in P$). Because of the maximality of $P$ none of $P+aP$ and $P-aP$ is proper, i.e.,
$-1\in P+aP$ and $-1\in P-aP$. Write $-1=s+as'$ and $-1=t-at'$ for certain $s,s',t,t'\in P$. Then
$-as'=1+s$ and $at'=1+t$. It follows that $-a^2s't'=1+s+t+st$ and therefore
$-1=s+t+st+a^2s't'\in P+P+PP+A^2PP\subseteq P\ \lightning$.
\end{proof}

\begin{thm}\label{artin-schreier}
Let $K$ be a field and $T\subseteq K$ a proper preorder. Then there is an order $P$ of $K$ such that $T\subseteq P$ and we have
$T=\bigcap\{P\mid\text{$P$ order of $K$},T\subseteq P\}$.
\end{thm}

\begin{proof} Consider the partially ordered set of all proper preorders of $K$ containing $T$. In this partially ordered set, every chain
has an upper bound (the empty chain has $T$ as an upper bound and every nonempty chain possesses its union as an upper bound).
By Zorn's lemma, the partially ordered set has a maximal element. Every such element is obviously a maximal proper preorder and therefore by
\ref{orderpreorder} an order. Now we turn to the second statement:

``$\subseteq$'' is clear.

``$\supseteq$'' Let $a\in K\setminus T$. To show: There is an order $P$ of $K$ with $T\subseteq P$ and $a\notin P$.
By \ref{againpreorder}, $T-aT$ is a preorder. It is proper for otherwise there would be $s,t\in T$ with $-1=s-at$ and it would follow that
$t\ne0$, $at=1+s$ and $a=\left(\frac 1t\right)^2t(1+s)\in K^2TT\subseteq T$. By the already proved, there is an order $P$ of $K$
with $T-aT\subseteq P$. If $a$ lied in $P$, then $a\in P\cap-P=\{0\}$ in contradiction to $a\notin T$.
\end{proof}

\begin{df}\label{dfreal}
A field is called \emph{real} (in older literature mostly \emph{formally real}) if it admits an order.
\end{df}

\begin{thm}\label{realchar}
Let $K$ be a field. Then the following are equivalent:
\begin{enumerate}[\normalfont(a)]
\item $K$ is real.
\item $-1\notin\sum K^2$
\item $\forall n\in\N:\forall a_1,\dots,a_n\in K:(a_1^2+\dots+a_n^2=0\implies a_1=0)$
\end{enumerate}
\end{thm}

\begin{proof}
\underline{(a)$\implies$(b)}\quad follows from \ref{squares}.

\smallskip
\underline{(b)$\implies$(a)}\quad By \ref{sqsm}, $\sum K^2$ is a preorder. If it is proper, then it is contained
in an order by \ref{artin-schreier}.

\smallskip
\underline{(b)$\iff$(c)}\quad
\begin{align*}-1\in\sum K^2&\iff\exists n\in\N:\exists a_2,\dots,a_n\in K:-1=a_2^2+\dots+a_n^2\\
&\iff\exists n\in\N:\exists a_2,\dots,a_n\in K:1^2+a_2^2+\dots+a_n^2=0\\
&\iff\exists n\in\N:\exists a_1\in K^\times:\exists a_2,\dots,a_n\in K:a_1^2+a_2^2+\dots+a_n^2=0\\
\end{align*}
\end{proof}

\begin{ex} Because of $-1=\ii^2\in\sum\C^2$, the field $\C:=\R(\ii)$ does not admit an order.
\end{ex}

\section{Extensions of orders}

\begin{df}\label{oexf}
Let $(K,P)$ be an ordered field and $L$ an extension field of $K$ (or in other words: let $L|K$ be
a field extension and $P$ be an order of $K$). We call $Q$ an \emph{extension} of the order $P$ to $L$ if
the following equivalent conditions are fulfilled [$\to$ \ref{unaryrem}(b)]:
\begin{enumerate}[(a)]
\item $(L,Q)$ is an ordered extension field of $(K,P)$.
\item $Q$ is an order of $L$ such that $P\subseteq Q$.
\item $Q$ is an order of $L$ such that $Q\cap K=P$.
\end{enumerate}
\end{df}

\begin{thm}\label{extend-ordering}
Let $(K,P)$ be an ordered field and $L$ an extension field of $K$. Then the order $P$ of $K$ can be extended
to $L$ if and only if $-1\notin\sum L^2P$.
\end{thm}

\begin{proof}
Since every order is a preorder [$\to$ \ref{sqsm}], an order of $L$ contains $P$ if and only if it contains
the preorder generated in $L$ by $P$ (i.e., the smallest preorder of $L$ containing $P$, or in other words, the
intersection of all preorders of $L$ containing $P$), namely $\sum L^2P$. If $\sum L^2P$ is not proper,
then there is of course no order of $L$ containing it. On the contrary, if $\sum L^2P$ is proper, then there is
such an order by Theorem \ref{artin-schreier}.
\end{proof}

\begin{reminder}\label{quadratic-remind}
Let $L|K$ be a field extension with $\chara K\ne2$. Then
\[[L:K]\le2\iff\exists d\in K:L=K(\sqrt d)\] since for $x\in L$ and $a,b,c\in K$ with $a\ne0$ and
$ax^2+bx+c=0$ we have $(2ax+b)^2=4a(ax^2+bx)+b^2=b^2-4ac=:d$ and therefore
$K(x)=K(2ax+b)=K(\sqrt d)$.
\end{reminder}

\begin{thm}\label{extend2}
Let $(K,P)$ be an ordered field and $d\in K$. The order $P$ can be extended to $K(\sqrt d)$ if and only if
$d\in P$.
\end{thm}

\begin{proof}
If $\sqrt d\in K$, then $d=(\sqrt d)^2\in P$. Suppose now that $\sqrt d\notin K$. Set $L:=K(\sqrt d)$.
Because of $P+dP\subseteq\sum L^2P\subseteq P+dP+K\sqrt d$, we have
$-1\notin\sum L^2P\iff-1\notin P+dP$. Since $P$ is a maximal proper preorder by
\ref{orderpreorder} and $P+dP$ is a preorder by \ref{againpreorder}, we obtain
$-1\notin P+dP\iff P=P+dP\iff d\in P$. Combining, we get $-1\notin\sum L^2P\iff d\in P$ and we conclude
by Theorem \ref{extend-ordering}.
\end{proof}

\begin{ex}\label{sqrt2}
In \ref{extend2}, the extension is in general not unique: $\Q(\sqrt 2)$ admits exactly two
orders, namely the ones induced by the two field embeddings $\Q(\sqrt 2)\hookrightarrow\R$
(in the one $\sqrt 2$ is positive, in the other negative). That it does not admit a third one, follows from the
fact that for every order $P$ of $\Q(\sqrt2)$ we have by \ref{archsubfieldreals}
$(\Q(\sqrt 2),P)\hookrightarrow(\R,\R_{\ge0})$ because $P$ is Archimedean since
$\Q(\sqrt 2)=\Q+\Q\sqrt2$ and
\[|\sqrt2|_P-1\overset{\ref{diffsquare}}=
\left(\frac{|\sqrt2|_P}2\right)^2-\left(\frac{|\sqrt2|_P-2}2\right)^2\overset{\ref{squares}}\le_P
\left(\frac{|\sqrt2|_P}2\right)^2=\frac12\]
[$\to$ \ref{archetcdef}(a)].
\end{ex}

\begin{thm}\label{extendodd}
If $L|K$ is a field extension of odd degree, then each order of $K$ can be extended to $L$.
\end{thm}

\begin{proof}
Assume the claim does not hold. Then there is a counterexample $L|K$ with $[L:K]=2n+1$ for some
$n\in\N$. We choose the counterexample in a way such that $n$ is as small as possible.
We will now produce another counterexample $L'|K$ with $[L':K]\le2n-1$ which will contradict
the minimality of $n$. Due to $\chara K=0$, we have that $L|K$ is separable. By the primitive element theorem,
there is some $a\in L$ with $L=K(a)=K[a]$. The condition $-1\in\sum L^2P$ which is satisfied by
\ref{extend-ordering} translates via the isomorphism
$K[X]/(f)\to L,\ \overline g\mapsto g(a)$ in
\begin{equation}
\tag{$*$}
1+\sum_{i=1}^\ell a_ig_i^2=hf
\end{equation}
 with $\ell\in\N$, $a_i\in P$, $g_i,h\in K[X]$, where $f$ denotes the minimal
polynomial of $a$ over $K$ (in particular $\deg f=[K(a):K]=[L:K]=2n+1$) and the $g_i$ are chosen in such a way
that $\deg g_i\le 2n$. Each of the $\ell+1$ terms in the sum on the left hand side of $(*)$ has an \emph{even}
degree $\le4n$ and a leading coefficient from $PK^2\subseteq P$ (except those terms that are zero of course).
Since $P$ is an order, the monomials of highest degree appearing on the left hand side of $(*)$ cannot cancel
out. So the left hand side and therefore also the right hand side of $(*)$ has an \emph{even} degree $\le4n$.
It follows that $h$ has an \emph{odd} degree $\le2n-1$. Choose an irreducible divisor $h_1$ of $h$ in
$K[X]$ of \emph{odd} degree and a root $b$ of $h_1$ in an extension field of $K$ (e.g., in the
splitting field of $h_1$ over $K$ or in the algebraic closure of $K$). Set $L':=K(b)$.
Substituting $b$ in $(*)$ yields $-1=\sum_{i=1}^\ell a_ig_i(b)^2\in\sum PL'^2$.
By \ref{extend-ordering}, $P$ cannot be extended to $L'$. Since
$[L':K]=[K(b):K]=\deg\irr_Kb=\deg h_1\le 2n-1$ is \emph{odd}, we gain the desired still smaller counterexample.
 \end{proof}
 
 \begin{thm}\label{rtlfct}
 Let $K$ be a field. Then every order of $K$ can be extended to $K(X)$.
 \end{thm}
 
 \begin{proof}
Let $P$ be an order of $K$. Assume that $P$ cannot be extended to $K(X)$. By \ref{extend-ordering}
we then have $-1\in\sum K(X)^2P$. Because of $\#K=\infty$ [$\to$ \ref{qembeds}] we can plug in a suitable
point from $K$ (``avoid finitely many poles'') and get $-1\in\sum K^2P=P\ \lightning$.
 \end{proof}
 
 \begin{ex}\label{ordersrx}
 Due to \ref{rtlfct} there is an order on $\R(X)$. If $P$ is such an order, then by the completeness
 of $(\R,\le)$ [$\to$ \ref{introduce-the-reals}], the set $\R_{\le_PX}=\{a\in\R\mid a\le_PX\}$ is either empty or
 not bounded from above (in which case it is $\R$) or it has a supremum $t$ in $\R$ (in which case it equals
 $(-\infty,t)$ if $t>_PX$ or $(-\infty,t]$ if $t<_PX$). Hence
 \[\R_{\le_PX}=\{a\in\R\mid a\le_PX\}\in\{\emptyset\}\cup\{(-\infty,t)\mid t\in\R\}\cup\{(-\infty,t]\mid t\in\R\}\cup\{\R\}
 =:C.\]
 We claim now that the map
\begin{align*}
\Ph\colon\{P\mid P\text{ order of $\R(X)$}\}&\to C\\P&\mapsto\R_{\le_PX}
\end{align*}
is a bijection. It is easy to see that for all $I,J\in C$ there is a ring automorphism $\ph_{I,J}$ of $\R(X)$ such that
for all orders $P$ of $\R(X)$, we have
\[\Ph(P)=I\iff\Ph(\ph_{I,J}(P))=J:\]
\begin{itemize}
\item $I=\R \et J=(-\infty,0] \leadsto \ph_{I,J}\colon X\mapsto\frac1X$
\item $I=\emptyset \et J=(-\infty,0) \leadsto \ph_{I,J}\colon X\mapsto\frac1X$
\item $I=(-\infty,t) \et J=(-\infty,0) \leadsto \ph_{I,J}\colon X\mapsto X+t$
\item $I=(-\infty,t] \et J=(-\infty,0] \leadsto \ph_{I,J}\colon X\mapsto X+t$
\item $I=(-\infty,0) \et J=(-\infty,0] \leadsto \ph_{I,J}\colon X\mapsto-X$
\item other $I$ and $J$ $\leadsto$ composition of the above automorphisms
\end{itemize}
From this we get the \emph{surjectivity} of $\Ph$, since as already mentioned there is an order $P$ of $\R(X)$
and if we set $I:=\Ph(P)$, then $\Ph(\ph_{I,J}(P))=J$ for all $J\in C$. For the \emph{injectivity} of $\Ph$, it suffices
to show that \emph{there is} some $I\in C$ having only one preimage under $\Ph$ since then
\begin{align*}
\#\{P\mid\Ph(P)=J\}&=\#\{P\mid\Ph(\ph_{J,I}(P))=I\}\\
&=\#\{\ph_{J,I}(P)\mid\Ph(\ph_{J,I}(P))=I\}=\#\{P\mid\Ph(P)=I\}=1
\end{align*}
 for all $J\in C$. We therefore fix $I:=\R\in C$ and show that there at most (and therefore exactly) one order
 $P$ of $\R(X)$ such that $\Ph(P)=I$. To this end, suppose $\Ph(P)=I$. If $f,g\in\R[X]\setminus\{0\}$, then
 one easily verifies that
 \[\frac fg\in P\overset{\R(X)^2\subseteq P}\iff fg\in P\iff\text{the leading coefficient of $fg$ is positive.}\]
 This uniquely determines $P$. Consequently, $\Ph$ is a bijection. We fix the following notation:
\begin{align*}
P_{-\infty}&:=\Ph^{-1}(\emptyset)\\
P_{t-}&:=\Ph^{-1}((-\infty,t))\text{ for $t\in\R$}\\
P_{t+}&:=\Ph^{-1}((-\infty,t])\text{ for $t\in\R$}\\
P_\infty&:=\Ph^{-1}(\R)
\end{align*}
Now $\{P\mid P\text{ order of }\R(X)\}=\{P_{-\infty}\}\cup\{P_{t-},P_{t+}\mid t\in\R\}\cup\{P_\infty\}$.
By easy considerations one obtains,
\begin{align*}
P_{-\infty}&=\{r\in\R(X)\mid\exists c\in\R:\forall x\in(-\infty,c):r(x)\ge0\},\\
P_{t-}&=\{r\in\R(X)\mid\exists\ep\in\R_{>0}:\forall x\in(t-\ep,t):r(x)\ge0\}\qquad(t\in\R),\\
P_{t+}&=\{r\in\R(X)\mid\exists\ep\in\R_{>0}:\forall x\in(t,t+\ep):r(x)\ge0\}\qquad(t\in\R),\\
P_\infty&=\{r\in\R(X)\mid\exists c\in\R:\forall x\in(c,\infty):r(x)\ge0\}.
\end{align*}
None of these orders is Archimedean.
\end{ex}

\section{Real closed fields}\label{sec:rcf}

\begin{pro} Let $K$ be a field. Then the following are equivalent:
\begin{enumerate}[\normalfont(a)]
\item $K$ admits exactly one order.
\item $\sum K^2$ is an order of $K$.
\item $(\sum K^2)\cup(-\sum K^2)=K$ and $-1\notin\sum K^2$
\end{enumerate}
\end{pro}

\begin{proof}
\underline{(a)$\implies$(b)}\quad Suppose $P$ is the unique order of $K$. By \ref{sqsm} and \ref{artin-schreier}, we then get $\sum K^2=P$.

\smallskip
\underline{(b)$\implies$(c)} is trivial.

\smallskip
\underline{(c)$\implies$(a)}\quad Suppose (c) holds. Using \ref{unaryrem}(a) and \ref{propfield}, we see that $\sum K^2$ is an order of $K$, and it is the only one by
\ref{sqsm} and \ref{orderpreorder}(b).
\end{proof}

\begin{ex}
$\Q$ and $\R$ possess exactly one order.
\end{ex}

\begin{convention}\label{convention}
If $K$ is a field admitting exactly one order, then we will often understand $K$ as an ordered field, that is we speak of $K$ but mean $(K,\sum K^2)$.
\end{convention}

\begin{df}\label{dfeuclid}
A field $K$ is called \emph{Euclidean} if $K^2$ is an order of $K$.
\end{df}

\begin{rem}\label{euclidunique}
If $K$ is Euclidean, then $K^2$ is the unique order of $K$.
\end{rem}

\begin{ex}
$\R$ is Euclidean but not $\Q$.
\end{ex}

\begin{notrem}\label{notremsqrt}
\begin{enumerate}[(a)]
\item Let $(K,\le)$ be an ordered field. If $a,b\in K$ such that $a=b^2$, then we write
$\sqrt a:=|b|\in K_{\ge0}$ [$\to$ \ref{introabssgn}] (this is obviously well-defined). If $a\in K\setminus K^2$, we continue to denote by $\sqrt a\in\overline K\setminus K$ an
arbitrary but fixed square root of $a$ in the algebraic closure $\overline K$ of $K$. One shows easily that $a\le b\iff\sqrt a\le\sqrt b$ for all $a,b\in K^2$.
\item If $K$ is a Euclidean field (with order $\le$ [$\to$ \ref{convention}, \ref{euclidunique}]), then in particular $\sqrt a\in K_{\ge0}$ and $(\sqrt a)^2=a$ for
\emph{all} $a\in K_{\ge0}=K^2=\sum K^2$.
\item We write $\ii:=\sqrt{-1}$. If $K$ is a real field, then $K(\ii)=K\oplus K\ii$ as a $K$-vector space
\end{enumerate}
\end{notrem}

\begin{pro}\label{cheuclid}
Let $K$ be a real field. Then the following are equivalent:
\begin{enumerate}[\normalfont(a)]
\item $K$ is Euclidean.
\item $K=-K^2\cup K^2$
\item $K(\ii)=K(\ii)^2$
\item Every polynomial of degree $2$ in $K(\ii)[X]$ has a root in $K(\ii)$.
\end{enumerate}
\end{pro}

\begin{proof}
\underline{(d)$\implies$(c)} is trivial.

\smallskip
\underline{(c)$\implies$(b)}\quad Suppose (c) holds and let $a\in K$. Write $a=(b+\ii c)^2$ for some $b,c\in K$. Then $a=b^2-c^2$ and $bc=0$
[$\to$ \ref{notremsqrt}(c)]. Therefore $a=b^2$ or $a=-c^2$.

\smallskip
\underline{(b)$\implies$(a)}\quad Suppose (b) holds. It suffices to show $K^2+K^2\subseteq K^2$. For this purpose, let $a,b\in K$. To show: $a^2+b^2\in K^2$.
If we had $a^2+b^2\notin K^2$, then $a^2+b^2\in-K^2$, say $a^2+b^2+c^2=0$ for some $c\in K$ and \ref{realchar}(c) would imply $c=0\ \lightning$.

\smallskip
\underline{(a)$\implies$(c)}\quad Suppose (a) holds and let $a,b\in K$. By \ref{notremsqrt}(c), we have to show $a+b\ii\in K(\ii)^2$. Set $r:=\sqrt{a^2+b^2}\in K_{\ge0}$
according to \ref{notremsqrt}(b). Then $r^2=a^2+b^2\ge a^2=|a|^2$ and therefore $r\ge|a|$ by \ref{notremsqrt}(a), i.e., $r\pm a\ge0$. It follows that
$\sqrt{\frac{r+a}2},\sqrt{\frac{r-a}2}\in K_{\ge0}$ and the calculation
\[\left(\sqrt{\frac{r+a}2}\pm\sqrt{\frac{r-a}2}\ii\right)^2=\frac{r+a}2\pm 2\sqrt{\frac{r^2-a^2}4}\ii-\frac{r-a}2=a\pm 2\left|\frac b2\right|\ii=a\pm|b|\ii\]
shows $a+b\ii\in K(\ii)^2$.

\smallskip
\underline{(c)$\implies$(d)} follows from $X^2+bX+c=(X+\frac b2)^2+(c-\frac{b^2}4)$ for $b,c\in K(\ii)$.
\end{proof}

\begin{df}\label{dfrealclosed}
Let $R$ be a field. Then $R$ is called \emph{real closed} if $R$ is Euclidean [$\to$~\ref{dfeuclid}, \ref{cheuclid}] and every polynomial of \emph{odd} degree from $R[X]$
has a root in $R$.
\end{df}

\begin{ex}
$\R$ is real closed by the intermediate value theorem from calculus and by \ref{dfeuclid}.
\end{ex}

\begin{rem} We now generalize the fundamental theorem of algebra from $\C=\R(\ii)$ to $R(\ii)$ for any real closed field $R$. The usual Galois theoretic proof
goes through as we will
see immediately.
\end{rem}

\begin{thm}[``generalized fundamental theorem of algebra'']\label{fund}
Let $R$ be a real closed field. Then $C:=R(\ii)$ is algebraically closed.
\end{thm}

\begin{proof}
Let $z\in\overline C$. To show: $z\in C$. Choose an intermediate field $L$ of $\overline C|C$ with $z\in L$ such that
$L|R$ is a finite Galois extension (e.g., the splitting field of $(X^2+1)\irr_Rz$ over $R$). We show $L=C$.
Choose a $2$-Sylow subgroup $H$ of the Galois group $G:=\Aut(L|R)$. From Galois theory we know that $[L^H:R]=[G:H]$ is odd. Hence $L^H=R$ since
every element of $L^H$ has over $R$ a minimal polynomial of odd degree which has a root in $R$ and therefore must be linear. Galois theory then implies $G=H$.
Hence $G$ is a $2$-group. Therefore the subgroup $I:=\Aut(L|C)$ of $G$ is also a $2$-group. By Galois theory, it is enough to show $I=\{1\}$. If we had $I\ne\{1\}$, then
there would exist, as one knows from algebra, a subgroup $J$ of $I$ with $[I:J]=2$. From this we get with Galois theory
$[L^J:C]=[L^J:L^I]=[I:J]=2$, contradicting \ref{cheuclid}(d).
\end{proof}

\begin{thm}\label{rcchar}
Let $R$ be a field. Then the following are equivalent:
\begin{enumerate}[\normalfont(a)]
\item $R$ is real closed.
\item $R\ne R(\ii)$ and $R(\ii)$ is algebraically closed.
\item $R$ is real but there is no real extension field $L\ne R$ of $R$ such that $L|R$ is algebraic.
\end{enumerate}
\end{thm}

\begin{proof}
\underline{(a)$\implies$(b)} follows from \ref{fund}.

\smallskip
\underline{(b)$\implies$(c)}\quad Suppose (b) holds. In order to show that $R$ is real, it is enough to show by Theorem \ref{realchar} that $\sum R^2=R^2$ since
$-1\notin R^2$ because $R\ne R(\ii)$. To this end, let $a,b\in R$. To show: $a^2+b^2\in R^2$. Since $R(\ii)$ is algebraically closed, we have $a+b\ii\in R(\ii)^2$, that is
there are $c,d\in R$ such that $a+b\ii=(c+d\ii)^2$ and it follows that
$a^2+b^2=(a+b\ii)(a-b\ii)=(c+d\ii)^2(c-d\ii)^2=((c+d\ii)(c-d\ii))^2=(c^2+d^2)^2\in R^2$. Now let $L|R$ be an algebraic field extension and suppose $L$ is real.
To show: $L=R$. Since $L(\ii)|R(\ii)$ is again algebraic and $R(\ii)$ is algebraically closed, we obtain $L(\ii)=R(\ii)$. For this reason $L$ is a real intermediate field of
$R(\ii)|R$ and it follows that $L=R$.

\smallskip
\underline{(c)$\implies$(a)}\quad Suppose (c) holds. Choose an order $P$ of $R$ according to Definition \ref{dfreal}. For all $d\in P$, $R(\sqrt d)$ is real
by \ref{extend2} and therefore $R(\sqrt d)=R$. It follows that $P\subseteq R^2\subseteq P$ and hence $P=R^2$, i.e., $R$ is Euclidean. According to Definition
\ref{dfrealclosed} it remains to show that each polynomial $f\in R[X]$ of odd degree has a root in $R$. Let $f\in R[X]$ be of odd degree. Choose an irreducible divisor
$g$ of $f$ in $R[X]$ of odd degree. Choose a root $a$ of $g$ in an extension field of $R$. Since $[R(a):R]=\deg g$ is odd, $R(a)$ is real by \ref{extendodd} and therefore
$R(a)=R$. Thus $a\in R$ satisfies $g(a)=0$ and hence $f(a)=0$.
\end{proof}

\begin{thm}[``real version of the generalized fundamental theorem of algebra'']\label{realfund}
Let $R$ be a field. Then the following are equivalent:
\begin{enumerate}[\normalfont(a)]
\item $R$ is real closed.
\item
$\{f\in R[X]\mid\text{$f$ is irreducible and monic}\}=\\
\{X-a\mid a\in R\}\cup\{(X-a)^2+b^2\mid a,b\in R,b\ne 0\}
$
\end{enumerate}
\end{thm}

\begin{proof}
\underline{(a)$\implies$(b)}\quad Suppose (a) holds.

``$\supseteq$'' is clear since $R$ is real.

``$\subseteq$'' Let $f\in R[X]$ be irreducible and monic of degree $\ge2$. Since $R(\ii)$ is algebraically closed by \ref{fund}, there are $a,b\in R$ such that $f(a+b\ii)=0$.
Due to $R\ne R(\ii)$ we can apply the automorphism of the field extension $R(\ii)|R$ given by $\ii\mapsto-\ii$ in order to obtain $f(a-b\ii)=0$.
By observing $a+b\ii\ne a-b\ii$ (since $b\ne0$ because $f\in R[X]$ is irreducible of degree $\ge2$), we get
\[f=\underbrace{(X-(a+b\ii))(X-(a-b\ii))}_{(X-a)^2+b^2\in R[X]}g\] for some $g\in R(\ii)[X]$. But then $g\in R[X]$ and hence even $g=1$.

\smallskip
\underline{(b)$\implies$(a)}\quad Suppose (b) holds. We will show \ref{rcchar}(b), i.e., that $R\ne R(\ii)$ and $R(\ii)$ is algebraically closed. The first claim $R\ne R(\ii)$ follows
from the irreducibility of $X^2+1=(X-0)^2+1^2\in R[X]$. Now suppose $f\in R(\ii)[X]$ is of degree $\ge1$.
Consider the ring automorphism
\[R(\ii)[X]\to R(\ii)[X],\ p\mapsto p^*\]
given by $a^*=a$ for $a\in R$, $\ii^*=-\ii$ and $X^*=X$. Then $f^*f\in R[X]$. If $f^*f$ has a root $a\in R$,
then $f(a)=0$ or $f^*(a)=0$ and then again $f(a)=0$. Suppose therefore that $f^*f$ has no root in $R$. Then there exist $a,b\in R$ with $b\ne0$
such that $(X-a)^2+b^2$ divides $f^*f$ in $R[X]$. Because of $(X-a)^2+b^2=(X-(a+b\ii))(X-(a-b\ii))$, $a+b\ii$ is a root of $f$ or of $f^*$.
If $f^*(a+b\ii)=0$, then $f(a-\ii b)=f^{**}((a+\ii b)^*)=(f^*(a+\ii b))^*=0^*=0$. Therefore $a+\ii b$ or $a-\ii b$ is a root of $f$ in $R(\ii)$.
\end{proof}

\begin{notterm}\label{intervals}
Let $(K,\le)$ be an ordered field.
\begin{enumerate}[(a)]
\item We extend the order $\le$ in the obvious way to the set $\{-\infty\}\cup K\cup\{\infty\}$ by declaring $-\infty<a<\infty$ for all $a\in K$.
\item We adopt the usual notation for intervals
\begin{align*}
(a,b)&:=(a,b)_K:=\{x\in K\cup\{\pm\infty\}\mid a<x<b\}\qquad(a,b\in K\cup\{\pm\infty\})\\
&\text{(``interval from $a$ to $b$ without endpoints'')}\\
[a,b)&:=[a,b)_K:=\{x\in K\cup\{\pm\infty\}\mid a\le x<b\}\qquad(a,b\in K\cup\{\pm\infty\})\\
&\text{(``interval from $a$ to $b$ with $a$ and without $b$'')}
\end{align*}
and so forth.
\item We use terminology like
\begin{align*}
\text{$f\ge0$ on $S$}&:\iff\forall x\in S:f(x)\ge0\qquad(f\in K[X_1,\dots,X_n],S\subseteq K^n)\\
&\text{(``$f$ is nonnegative on $S$'')}\\
\text{$f>0$ on $S$}&:\iff\forall x\in S:f(x)>0\qquad(f\in K[X_1,\dots,X_n],S\subseteq K^n)\\
&\text{(``$f$ is positive on $S$'').}
\end{align*}
\end{enumerate}
\end{notterm}

\begin{cor}[``intermediate value theorem for polynomials'']\label{intermediate}
Let $R$ be a real closed field, $f\in R[X]$ and $a,b\in R$ such that $a\le b$ and $\sgn(f(a))\ne\sgn(f(b))$. Then there is $c\in[a,b]_R$ with $f(c)=0$.
\end{cor}

\begin{proof}
WLOG $f$ is monic. By \ref{realfund}, all nonlinear monic irreducible polynomials in $R[X]$ are positive on $R$. Hence
$f=g\prod_{i=1}^k(X-a_i)^{\al_i}$ with $k\in\N_0$, $a_i\in R$, $\al_i\in\N$, $a_1<\dots<a_k$ and some $g\in R[X]$ that is positive on $R$. On the sets
$(-\infty,a_1)$, $(a_1,a_2)$, \dots, $(a_{k-1},a_k)$ and $(a_k,\infty)$ each $X-a_i$ and therefore $f$ has constant sign. Hence $a$ and $b$ cannot lie
both in the same such set. Consequently, there is an $i\in\{1,\dots,k\}$ such that $a_i\in[a,b]$. Set $c:=a_i$.
\end{proof}

\begin{cor}[``Rolle's theorem for polynomials'']\label{rolle}
Suppose $R$ is a real closed field, $f\in R[X]$ and $a,b\in R$ with $a<b$ and $f(a)=f(b)$.
Then there exists a $c\in(a,b)_R$ such that $f'(c)=0$.
\end{cor}

\begin{proof} WLOG $f\ne0$, $f(a)=0=f(b)$ and $\nexists x\in(a,b):f(x)=0$. Write \[f=(X-a)^\al(X-b)^\be g\] for some $\al,\be\in\N$ and $g\in R[X]$ with
$\forall x\in[a,b]:g(x)\ne0$. We find
\begin{align*}
f'&=(X-a)^\al\be(X-b)^{\be-1}g+\al(X-a)^{\al-1}(X-b)^\be g+(X-a)^\al(X-b)^\be g'\\
&=(X-a)^{\al-1}(X-b)^{\be-1}h
\end{align*}
where $h:=\be(X-a)g+\al(X-b)g+(X-a)(X-b)g'$. Hence it suffices to find $c\in(a,b)$ such that $h(c)=0$. We can apply the intermediate value theorem \ref{intermediate} because
\[h(a)=\al(a-b)g(a)\quad\text{and}\quad h(b)=\be(b-a)g(b)\] and again by \ref{intermediate} we have $\sgn(g(a))=\sgn(g(b))$.
\end{proof}

\begin{cor}[``mean value theorem for polynomials'']\label{mean}
Let $R$ be a real closed field, $f\in R[X]$ and $a,b\in R$ with $a<b$. Then there is some $c\in(a,b)_R$ satisfying $f'(c)=\frac{f(b)-f(a)}{b-a}$.
\end{cor}

\begin{proof}
Setting $g:=(X-a)(f(b)-f(a))-(b-a)(f-f(a))$, we get $g(a)=0=g(b)$ and $g'=f(b)-f(a)-(b-a)f'$.
Rolle's theorem \ref{rolle} yields $c\in(a,b)$ such that $g'(c)=0$.
\end{proof}

\begin{df}\label{monodef}
\begin{enumerate}[(a)]
\item Let $(M,\le_1)$ and $(N,\le_2)$ be ordered sets. A map $\ph\colon M\to N$ is called \emph{anti-monotonic} [$\to$ \ref{ordhom}] if
\[a\le_1b\implies\ph(a)\ge_2\ph(b)\]
for all $a,b\in M$.
\item If $(K,\le)$ is an ordered field, $f\in K[X]$ and $I\subseteq K$, then we say that
$f$ is \alalal{monotonic}{injective}{anti-monotonic} on $I$ if $I\to K,\ x\mapsto f(x)$ is \alalal{monotonic}{injective}{anti-monotonic}.
\end{enumerate}
\end{df}

\begin{cor}\label{dermono}
Let $R$ be a real closed field, $f\in R[X]$ and $a,b\in R$.
If $\malal{f'\ge0}{f'\le0}$ on $(a,b)$ \emph{[$\to$ \ref{intervals}(c)]}, then $f$ is \alal{}{anti-} monotonic on $[a,b]$.
If $f'$ has no root on $(a,b)$, then $f$ is injective on $[a,b]$.
\end{cor}

\begin{proof}
The statement is empty in case $a>b$, trivial in the case $a=b$ and it follows from the mean value theorem \ref{mean} in case $a<b$.
\end{proof}

\section{Descartes' rule of signs}\label{sec:descartes}

\begin{notation}\label{degnot}
Let $A$ be a commutative ring with $0\ne1$ and $d\in\R$. We denote
\[A[X_1,\dots,X_n]_d:=\{f\in A[X_1,\dots,X_n]\mid\deg f\le d\}\]
(where $\deg0:=-\infty$).
\end{notation}

\begin{pro}[``Taylor formula for polynomials'']
Suppose $K$ is a field of characteristic $0$, $d\in\N_0$, $f\in K[X]_d$ and $a\in K$. Then
\[f=\sum_{k=0}^d\frac{f^{(k)}(a)}{k!}(X-a)^k.\]
\end{pro}

\begin{proof}
Since $K[X]\to K[X],\ p\mapsto p'$ commutes with the ring automorphism $K[X]\to K[X],\ p\mapsto p(X+a)$, we can WLOG suppose
$a=0$. But then the claim follows from the definition of the (formal) derivative.
\end{proof}

\begin{lem}\label{sgnbounds}
Suppose $(K,\le)$ is an ordered field, $k\in\N$, $c_1,\dots,c_k\in K^\times$, $\al_1,\dots,\al_k\in\N_0$,
$\al_1<\ldots<\al_k$ and $f=\sum_{i=1}^kc_iX^{\al_i}$.
\begin{enumerate}[(a)]
\item $\sgn(f(x))=(\sgn x)^{\al_k}\sgn(c_k)$ for all $x\in K$ satisfying $|x|>\max\left\{1,\frac{|c_1|+\ldots+|c_{k-1}|}{|c_k|}\right\}$
\item $\sgn(f(x))=(\sgn x)^{\al_1}\sgn(c_1)$ for all $x\in K^\times$ satisfying $\frac1{|x|}>\max\left\{1,\frac{|c_2|+\ldots+|c_{k}|}{|c_1|}\right\}$
\end{enumerate}
\end{lem}

\begin{proof}
\begin{enumerate}[(a)]
\item For all $x\in K$ with $|x|>\max\left\{1,\frac{|c_1|+\ldots+|c_{k-1}|}{|c_k|}\right\}$, we have
\[\left|\sum_{i=1}^{k-1}c_ix^{\al_i}\right|\overset{\ref{introabssgn}}\le
\sum_{i=1}^{k-1}|c_i||x|^{\al_i}\overset{1\le|x|}\le
\sum_{i=1}^{k-1}|c_i||x|^{\al_{k}-1}=
|c_k|\left(\frac{\sum_{i=1}^{k-1}|c_i|}{|c_k|}\right)|x|^{\al_k-1}<|c_kx^{\al_k}|.\]
\item For all $x\in K^\times$ with $\frac1{|x|}>\max\left\{1,\frac{|c_2|+\ldots+|c_{k}|}{|c_1|}\right\}$, we have
\[\left|\sum_{i=2}^kc_ix^{\al_i}\right|\overset{\ref{introabssgn}}\le
\sum_{i=2}^k|c_i||x|^{\al_i}\overset{|x|\le1}\le
\sum_{i=2}^k|c_i||x|^{\al_1+1}=
|c_1|\left(\frac{\sum_{i=2}^k|c_i|}{|c_1|}\right)|x|^{\al_1+1}<|c_1x^{\al_1}|.\]
\end{enumerate}
\end{proof}

\begin{reminder}\label{multiplicity}
Let $K$ be a field, $f\in K[X]$ and $a\in K$. Then
\[
\mu(a,f):=\sup\{k\in\N_0\mid\text{$(X-a)^k$ divides $f$ in $K[X]$}\}\in\N_0\cup\{\infty\}
\]
is called the \emph{multiplicity} of $a$ in $f$. We have
\[\mu(a,f)=\infty\iff f=0\]
and
\[\mu(a,f)\ge1\iff f(a)=0.\]
We call $a$ a \emph{multiple} root of $f$ if $\mu(a,f)\ge2$ and we call it a $k$-fold root of $f$ ($k\in\N$) if
$\mu(a,f)=k$. In case $\chara K=0$, one has
\[\mu(a,f)=\sup\{k\in\N_0\mid f^{(0)}(a)=\ldots=f^{(k-1)}(a)=0\}\]
as one can see easily.
\end{reminder}

\begin{df}\label{defst}
Let $(K,\le)$ be an ordered field and $0\ne f\in K[X]$.
\begin{enumerate}[(a)]
\item The \emph{number of positive roots counted with multiplicity} of $f$ is
\[\mu(f):=\sum_{a\in K_{>0}}\mu(a,f)\in\N_0.\]
Writing $f=g\prod_{i=1}^m(X-a_i)$ with $a_1,\dots,a_m\in K_{>0}$ and $g\in K[X]$ with $g(x)\ne0$ for all $x\in K_{>0}$, we
therefore have $\mu(f)=m$.
\item Writing $f=\sum_{i=1}^kc_iX^{\al_i}$ with $c_1,\dots,c_k\in K^\times$ and $\al_1,\dots,\al_k\in\N_0$ such that
\[\al_1<\ldots<\al_k,\]
we define the \emph{number of sign changes in the coefficients} of $f$
\[\si(f):=\#\{i\in\{1,\dots,k-1\}\mid \sgn(c_i)\ne\sgn(c_{i+1})\}\in\N_0.\]  
\end{enumerate}
\end{df}

\begin{pro}\label{parity}
Let $R$ be a real closed field and $f\in R[X]\setminus\{0\}$. Then $\mu(f)$ and $\si(f)$ have the same parity.
\end{pro}

\begin{proof}
Write
$f=\sum_{i=1}^kc_iX^{\al_i}=g\prod_{i=1}^m(X-a_i)$ with $c_1,\dots,c_k\in R^\times$, $\al_1,\dots,\al_k\in\N_0$,
$a_1,\dots,a_m\in R_{>0}$ and $g\in R[X]$ such that $\al_1<\ldots<\al_m$ and $g(x)\ne0$ for all $x\in R_{>0}$.
Since $R$ is real closed, WLOG $g(x)>0$ for all $x\in R_{>0}$ by the intermediate value theorem \ref{intermediate}.
But then by Lemma \ref{sgnbounds}, both the lowest and highest coefficient of $g$ is positive.
Now the claim follows from
$\mu(f)=m$, $\sgn(c_1)=(-1)^m$ and $\sgn(c_k)=1$.
\end{proof}

\begin{lem}\label{loose1}
Let $R$ be a real closed field and $f\in R[X]\setminus R$. Then
$\mu(f)\le\mu(f')+1$ and $\si(f)\le\si(f')+1$.
\end{lem}

\begin{proof}
The second statement is easy to prove. For the first statement, suppose $a_1,\dots,a_m\in R$ are the positive roots of $f$ and
$a_1<\ldots<a_m$. Since $R$ is real closed, there exist roots $b_1,\dots,b_{m-1}\in R$ of $f'$ such that
$a_1<b_1<a_2<\ldots<b_{m-1}<a_m$ by Rolle's Theorem \ref{rolle}. Now $\mu(f')=\sum_{a\in K_{>0}}\mu(a,f')\ge
\sum_{i=1}^m\mu(a_i,f')+\sum_{i=1}^{m-1}\mu(b_i,f')\ge\sum_{i=1}^m\mu(a_i,f')+m-1=
\sum_{i=1}^m(\mu(a_i,f)-1)+m-1=\sum_{i=1}^m\mu(a_i,f)-1=\mu(f)-1$.
\end{proof}

\begin{rem}\label{loo}
In the situation of Lemma \ref{loose1}, $\si(f')\le\si(f)$ holds trivially but
$\mu(f')\le\mu(f)$ fails in general as the example $f=(X-1)^2+1$ shows.
\end{rem}

\begin{thm}[Descartes' rule of signs]\label{descartes}
Let $R$ be a real closed field. Then $\mu(f)\le\si(f)$ for all $f\in R[X]\setminus\{0\}$.
\end{thm}

\begin{proof}
Induction on $d:=\deg f\in\N_0$.

\smallskip
\underline{$d=0$}\qquad $\mu(f)=0=\si(f)$

\smallskip
\underline{$d-1\to d$\quad$(d\in\N)$}\qquad
$\mu(f)\overset{\ref{loose1}}\le\mu(f')+1\overset{\text{induction}}{\underset{\text{hypothesis}}\le}\si(f')+1\overset{\ref{loo}}\le\si(f)+1$
and therefore $\mu(f)\le\si(f)$ by Proposition \ref{parity}.
\end{proof}

\begin{ex}
Let $R$ be a real closed field and $f:=X^4-5X^3-21X^2+115X-150\in R[X]$. Then $\si(f)=3$ and therefore
$\mu(f)\in\{1,3\}$ by \ref{descartes} and \ref{parity}. For $f(-X)=X^4+5X^3-21X^2-115X-150$, we have
$\si(f(-X))=1$ and therefore $\mu(f(-X))=1$. One can verify that $\si((1+X)^{22}f)=1$ from which we get
$\mu(f)=\mu((1+X)^{22}f)=1$. Hence $f$ has exactly two roots in $R$, namely two simple (i.e., $1$-fold
[$\to$ \ref{multiplicity}]) ones,
one positive and one negative.
\end{ex}

\begin{df}\label{defrr}
Let $R$ be a real closed field. We call a polynomial $f\in R[X]$ \emph{real-rooted} if it has no root in $R(\ii)\setminus R$
[$\to$ \ref{fund}].
\end{df}

\begin{pro}\label{rrr}
Let $R$ be real closed field and $f\in R[X]$. Then the following are equivalent:
\begin{enumerate}[\normalfont(a)]
\item $f$ is real-rooted.
\item There are $d\in\N_0$, $c\in R^\times$ and $a_1,\dots,a_d\in R$ such that $f=c\prod_{i=1}^d(X-a_i)$.
\end{enumerate}
\end{pro}

\begin{proof}
For (a)$\implies$(b) use the fundamental theorem \ref{fund} or \ref{realfund}.
\end{proof}

\begin{thm}\emph{[$\to$ \ref{loo}]}\label{derrr}
Suppose $R$ is a real closed field and $f\in R[X]\setminus R$ is real-rooted. Then $f'$ is real-rooted and $\mu(f')\le\mu(f)$.
\end{thm}

\begin{proof}
Using \ref{rrr}, write $f=c\prod_{i=1}^n(X-a_i)^{\al_i}$ with $c,a_1,\dots,a_n\in R$ and $\al_1,\dots,\al_n\in\N$ such that $c\ne0$ and
\[a_1<\ldots<a_n.\]
Since $R$ is real closed, there exist roots $b_1,\dots,b_{n-1}\in R$ of $f'$ such that
\[a_1<b_1<a_2<\ldots<b_{n-1}<a_n\] by Rolle's Theorem \ref{rolle}. We have $\mu(a_i,f)=\al_i$ and therefore
\[\mu(a_i,f')=\al_i-1\] for all $i\in\{1,\dots,n\}$. It follows that
\begin{align*}
\deg(f')&\ge\sum_{i=1}^n\mu(a_i,f')+\sum_{i=1}^{n-1}\mu(b_i,f')\ge\sum_{i=1}^n\mu(a_i,f')+n-1\\
&=\sum_{i=1}^n(\al_i-1)+n-1=\deg(f)-1=\deg(f'),
\end{align*}
whence
\[\deg(f')=\sum_{i=1}^n\mu(a_i,f')+\sum_{i=1}^{n-1}\mu(b_i,f')\]
and
\[\mu(b_i,f')=1\]
for all $i\in\{1,\dots,n-1\}$.
It follows that
\[\{x\in R(\ii)\mid f'(x)=0\}\subseteq\{a_1,b_1,a_2,\dots,b_{n-1},a_n\}\subseteq R,\]
in particular $f'$ is real-rooted.
Choose $k\in\{1,\dots,n+1\}$ such that $a_k,\dots,a_n$ are the positive roots of $f$. Then
\[\{x\in R\mid f'(x)=0,x>0\}\subseteq\begin{cases}\{b_{k-1},a_k,\dots,b_{n-1},a_n\}&\text{if }k\ge2\\\{a_1,b_1,\dots,b_{n-1},a_n\}
&\text{if }k=1\end{cases}.\]
If $k\ge2$, then
\[\mu(f')\le\sum_{i=k}^n(\underbrace{\mu(b_{i-1},f')}_{=1}+\underbrace{\mu(a_i,f')}_{=\mu(a_i,f)-1})=\mu(f).\]
If $k=1$, then one sees similarly that $\mu(f')=\mu(f)-1\le\mu(f)$.
\end{proof}

\begin{thm}[Descartes' rule of signs for real-rooted polynomials]\label{descartesrr}
Let $R$ be a real closed field. Then $\mu(f)=\si(f)$ for all real-rooted $f\in R[X]$.
\end{thm}

\begin{proof}
By Theorem \ref{descartes}, it is enough to show $\mu(f)\ge\si(f)$ for all real-rooted $f\in R[X]$ by
induction on $d:=\deg f\in\N_0$.

\smallskip
\underline{$d=0$}\qquad $\mu(f)=0=\si(f)$

\smallskip
\underline{$d-1\to d$\quad$(d\in\N)$}\qquad
$\mu(f)\overset{\ref{derrr}}\ge\mu(f')
\overset{
\substack{\text{induction}\\{\text{hypothesis}}}}{\underset{\ref{derrr}}
\ge}\si(f')\overset{\ref{loose1}}\ge\si(f)-1$
and therefore $\mu(f)\ge\si(f)$ by Proposition \ref{parity}.
\end{proof}

\begin{ex}\label{symmex}
\begin{multline*}
\det\begin{pmatrix}1-X&0&1\\0&-2-X&1\\1&1&-X\end{pmatrix}=(1-X)(2+X)X+2+X+X-1\\
=(2+X-2X-X^2)X+2X+1=-X^3-X^2+4X+1\in\R[X]
\end{multline*}
is real-rooted since it is the characteristic polynomial
of a symmetric matrix. By Descartes' rule \ref{descartesrr}, it has exactly one positive root.
\end{ex}

\section{Counting real zeros with Hermite's method}\label{sec:hermite}

\begin{reminder}\label{longremi}
\begin{enumerate}[(a)]
\item Let $A$ be a commutative ring with $0\ne1$
and $f\in A[X_1,\dots,X_n]$. Then $f$ is called \emph{homogeneous} if $f$ is a an
$A$-linear combination of
monomials of the same degree. Moreover, $f$ is called a \emph{$k$-form} ($k\in\N_0$) if $f$ is an $A$-linear combination
of monomials of degree
$k$ (i.e., if $f=0$ or $f$ is homogeneous of degree $k$). One often says \emph{linear form} instead of $1$-form and
\emph{quadratic form} instead of $2$-form.
\item If $K$ is a field, one can identify the $K$-vector subspace of $K[X_1,\dots,X_n]$ consisting of the
\alal{linear}{quadratic} forms introduced in (a) via the isomorphism $f\mapsto(x\mapsto f(x))$ with the $K$-vector space
$\malal{(K^n)^*}{Q(K^n)}$ introduced in linear algebra. Hence the notion of a linear or quadratic form introduced in (a) differs only
insignificantly from the corresponding notion from linear algebra.
\item Let $A$ be a set and $M=(a_{ij})_{\substack{1\le i\le m\\1\le j\le n}}\in A^{m\times n}$ a matrix. Then
$M^T:=(a_{ij})_{\substack{1\le j\le n\\1\le i\le m}}\in A^{n\times m}$ is called the \emph{transpose} of $M$. The elements of
$SA^{n\times n}:=\{M\in A^{n\times n}\mid M=M^T\}$ are called \emph{symmetric} matrices.
\item Let $K$ be a field. Then $(a_1,\dots,a_n)\mapsto a_1X_1+\ldots+a_nX_n$ ($a_i\in K$) defines an isomorphism between
$K^{1\times n}\cong K^n$ and the $K$-vector space of the linear forms in $K[X_1,\dots,X_n]$. If $\chara K\ne2$, then
$(a_{ij})_{1\le i,j\le n}\mapsto\sum_{i,j=1}^na_{ij}X_iX_j$ ($a_{ij}\in K$) defines an isomorphism between $SK^{n\times n}$
and the $K$-vector space of the quadratic forms in $K[X_1,\dots,X_n]$. If $f\in K[X_1,\dots,X_n]$ is a linear or quadratic form,
then we call the preimage $M(f)$ of $f$ under the respective isomorphism the \emph{representing matrix} of $f$. This is the
representing matrix of $f$ in the sense of linear algebra with respect to the canonical bases.
\item Suppose $K$ is a field satisfying $\chara K\ne2$, $q\in K[X_1,\dots,X_n]$ a quadratic form,
$\ell_1,\dots,\ell_m\in K[X_1,\dots,X_n]$ linear forms and $\la_1,\dots,\la_m\in K$. Then
\[q=\sum_{k=1}^m\la_k\ell_k^2\iff M(q)=P^T
\begin{pmatrix}~
\begin{tikzpicture}[inner sep=0]
\node (a1) {$\la_1$};
\node (an) at (1.5,-1.5) [anchor=north west] {$\la_m$};
\node[scale=3.2] at (1.3,-0.4) {$0$};
\node[scale=3.2] at (0.2,-1.2) {$0$};
\draw[loosely dotted,very thick,dash phase=3pt] (a1)--(an);
\end{tikzpicture}
\end{pmatrix}
P
\]
where
\[P:=\begin{pmatrix}M(\ell_1)\\\vdots\\M(\ell_m)\end{pmatrix}\in K^{m\times n}.\]
Here $P$ is invertible if and only if $m=n$ and $\ell_1,\dots,\ell_m$ are linearly independent.
\item Let $K$ be a field satisfying $\chara K\ne2$ and $q\in K[X_1,\dots,X_n]$ a quadratic form. One can \emph{easily} calculate
linearly independent linear forms $\ell_1,\dots,\ell_m\in K[X_1,\dots,X_n]$ and $\la_1,\dots,\la_m\in K$ such that
$q=\sum_{k=1}^m\la_k\ell_k^2$. Indeed, one can write
\[X_1^2+a_2X_1X_2+\dots+a_nX_1X_n\qquad(a_i\in K)\]
as
$\Big(\underbrace{X_1+\frac{a_2}2X_2+\dots+\frac{a_n}2X_n}_{\ell_1}\Big)^2-
\underbrace{\left(\frac{a_2}2X_2+\dots+\frac{a_n}2X_n\right)^2}_{\in K[X_2,\dots,X_n]}$ and
\[X_1X_2+a_3X_1X_3+\ldots+a_nX_1X_n+b_3X_2X_3+\ldots+b_nX_2X_n\]
as
\begin{multline*}
\big(\underbrace{X_1+b_3X_3+\ldots+b_nX_n}_{h_1}\big)\big(\underbrace{X_2+a_3X_3+\ldots+a_nX_n}_{h_2}\big)\\
-
\underbrace{\left(a_3X_3+\ldots+a_nX_n\right)\left(b_3X_3+\ldots+b_nX_n\right)}_{\in K[X_3,\dots,X_n]}
\end{multline*}
where
$h_1h_2=\Big(\underbrace{\frac{h_1+h_2}2}_{\ell_1}\Big)^2-\Big(\underbrace{\frac{h_1-h_2}2}_{\ell_2}\Big)^2$
[$\to$ \ref{diffsquare}]. In this way one can in each step place one or two variables in one or two squares and the arising linear forms
are obviously linearly independent.  Consider $q:=2X_1X_2+2X_1X_3+2X_2X_3+2X_3X_4$ as an example:
\begin{align*}
q&:=2(X_1X_2+X_1X_3+X_2X_3)+2X_3X_4\\
&=2((\underbrace{X_1+X_3}_{h_1})(\underbrace{X_2+X_3}_{h_2}))-2X_3^2+2X_3X_4\\
&=\underbrace{\frac12}_{\la_1=-\la_2}((\underbrace{h_1+h_2}_{\ell_1})^2-(\underbrace{h_1-h_2}_{\ell_2})^2)
\underbrace{-2}_{\la_3}\Big(\underbrace{X_3-\frac12X_4}_{\ell_3}\Big)^2+\underbrace{\frac24}_{\la_4}{\underbrace{X_4}_{\ell_4}}^2.
\end{align*}
Hence $q=\sum_{k=1}^4\la_k\ell_k^2=\frac12(X_1+X_2+2X_3)^2-\frac12(X_1-X_2)^2-2(X_3-\frac12X_4)^2+\frac12X_4^2$
and by (e)
\[\begin{pmatrix}0&1&1&0\\1&0&1&0\\1&1&0&1\\0&0&1&0\end{pmatrix}=P^TDP\]
where
\[P:=\begin{pmatrix}1&1&2&0\\1&-1&0&0\\0&0&1&-\frac12\\0&0&0&1\end{pmatrix}\qquad\text{and}\qquad
D:=\begin{pmatrix}\frac12&0&0&0\\0&-\frac12&0&0\\0&0&-2&0\\0&0&0&\frac12\end{pmatrix}.
\]
\item Translating (f) into the language of matrices, one obtains for each field $K$ with $\chara K\ne2$ and each $M\in SK^{n\times n}$
the following:
One can \emph{easily} find a $P\in\GL_n(K)=(K^{n\times n})^\times$ and a diagonal matrix $D\in K^{n\times n}$ such that
$M=\underline{\underline{P^T}}DP$. This is the diagonalization of $M$ as a quadratic form which is much simpler than the
diagonalization of $M$ as an endomorphism where one wants to reach $M=\underline{\underline{P^{-1}}}DP$
(in case $K=\R$ perhaps even with $P^{-1}=P^T$).
\item Let $K$ be a Euclidean field [$\to$ \ref{dfeuclid}] and $q\in K[X_1,\dots,X_n]$ a quadratic form. According to (f),
one can then \emph{easily} compute \emph{linearly independent} linear forms \[\ell_1,\dots,\ell_s,\ell_{s+1},\dots,\ell_{s+t}\in
K[X_1,\dots,X_n]\] satisfying $q=\sum_{i=1}^s\ell_i^2-\sum_{j=1}^t\ell_{s+j}^2$. By completing $\ell_1,\dots,\ell_{s+t}$ to a
basis $\ell_1,\dots,\ell_n$ of the vector space of all linear forms in $K[X_1,\dots,X_n]$ and by writing
$q=1\cdot\sum_{i=1}^s\ell_i^2+(-1)\sum_{j=1}^t\ell_{s+j}^2+0\cdot\sum_{k=t+1}^n\ell_k^2$, one sees for the \emph{rank}
$\rk(q):=\rk M(q)$ of $q$ that
$\rk(q)\overset{(e)}=s+t$.
We define the \emph{signature} of $q$ as $\sg(q):=s-t$. This is well-defined by \emph{Sylvester's law of inertia}:
If $\ell_1',\dots,\ell_{s'}',\ell_{s'+1}',\dots,\ell_{s'+t'}'\in K[X_1,\dots,X_n]$ are other linearly independent linear forms satisfying
$q=\sum_{i=1}^{s'}\ell_i'^2-\sum_{j=1}^{t'}\ell_{s'+j}'^2$, then $s'+t'=\rk(q)=s+t$ and one sees again by completing to a basis and (e)
that there are subspaces $U$, $W$, $U'$, $W'$ of $K^n$ such that
$q(U)\subseteq K_{\ge0}$, $\dim U=n-t$, $q(W\setminus\{0\})\subseteq K_{<0}$, $\dim W=t$,
$q(U')\subseteq K_{\ge0}$, $\dim U'=n-t'$, $q(W'\setminus\{0\})\subseteq K_{<0}$, $\dim W'=t'$.
One deduces $U\cap W'=\{0\}$ and $U'\cap W=\{0\}$, whence
$(n-t)+t'\le n$ and $(n-t')+t\le n$. Therefore $t=t'$ and thus $s=s'$.
\item Let $K$ be a field and $f=X^d+a_{d-1}X^{d-1}+\ldots+a_0\in K[X]$ with $d\in\N_0$ and $a_i\in K$. The
\emph{companion matrix} $C_f$ of $f$ is the representing matrix of the $K$-vector space endomorphism
\[K[X]/(f)\to K[X]/(f),\ \overline p\mapsto\overline{Xp}\qquad(p\in K[X])\] with respect to the basis $\overline 1,\dots,\overline{X^{d-1}}$,
i.e.,
\[C_f=
\begin{tikzpicture}[loosely dotted,thick,baseline]
\matrix (m) [matrix of math nodes,nodes in empty cells,right delimiter=),left delimiter=(,column sep=1em]{
0&0&&&&0&-a_0\\
1&0&&&&  &-a_1\\
0&1&&&&  &-a_2\\
0&0\\
\\
 &   &&&&0\\
0&0&&&0&1&-a_{d-1}\\
} ;
\draw (m-1-2)-- (m-1-6);
\draw (m-2-2)-- (m-6-6);
\draw (m-3-2)-- (m-7-6);
\draw (m-4-2)-- (m-7-5);
\draw (m-1-6)-- (m-6-6);
\draw (m-3-7)-- (m-7-7);
\draw (m-4-1)-- (m-7-1);
\draw (m-4-2)-- (m-7-2) -- (m-7-5);
\end{tikzpicture}\in K^{d\times d}.
\]
One sees easily that $f$ is the minimal polynomial and therefore for degree reasons, up to a sign, also the characteristic polynomial of $C_f$.
Now suppose furthermore that $f$ splits into linear factors, i.e.,
\[f=\prod_{k=1}^m(X-x_k)^{\al_k}\]
for some $m\in\N_0$, $\al_1,\ldots,\al_m\in\N$ and $x_1,\dots,x_m\in K$
(here the $x_i$ do not yet have to be pairwise distinct so that one could take $\al_k=1$ but to avoid a confusing change of notation in view of the proof of
Theorem \ref{hermite1}, we allow here and in Proposition \ref{herm} that $\al_k\in\N$).
Then $C_f$ is similar to a triangular matrix with diagonal entries
\[\underbrace{x_1,\dots,x_1}_{\al_1},\quad\underbrace{x_2,\dots,x_2}_{\al_2},\quad\dots\quad,\quad\underbrace{x_m,\dots,x_m}_{\al_m}.\]
Then $C_f^i$ is for every $i\in\N_0$ similar to a triangular matrix whose diagonal entries are
\[\underbrace{x_1^i,\dots,x_1^i}_{\al_1},\quad\underbrace{x_2^i,\dots,x_2^i}_{\al_2},\quad\dots\quad,\quad\underbrace{x_m^i,\dots,x_m^i}_{\al_m}.\]
In particular, we have $\tr(C_f^i)=\sum_{k=1}^m\al_kx_k^i$ for all $i\in\N_0$ and consequently
\[\tr(g(C_f))=\sum_{k=1}^m\al_kg(x_k)\] for all $g\in K[X]$.
\item If $K$ is a field and $x_1,\dots,x_m\in K$ are pairwise distinct, then the \emph{Vandermonde} matrix
\[
\begin{pmatrix}
1&x_1&\dots&x_1^{m-1}\\
\vdots&\vdots&&\vdots\\
1&x_m&\dots&x_m^{m-1}
\end{pmatrix}\in K^{m\times m}
\]
is invertible since it is the representing matrix of the injective and therefore bijective linear map
\[K[X]_{m-1}\to K^m,\ p\mapsto\begin{pmatrix}p(x_1)\\\vdots\\p(x_m)\end{pmatrix}\]
[$\to$ \ref{degnot}] with respect to the canonical bases.
\item Let $K$ be a field and let $x_1,\dots,x_m\in K$ be pairwise distinct. Furthermore, let $d\in\N_0$ with $m\le d$.
Consider for $k\in\{1,\dots,m\}$ the linear forms $\ell_k:=\sum_{i=1}^dx_k^{i-1}T_i\in K[T_1,\dots,T_d]$. Then
$\ell_1,\dots,\ell_m$ are linearly independent. Indeed, because of (d) this is equivalent to the linear independence of
the vectors $(x_k^0,\dots,x_k^{d-1})$ ($k\in\{1,\dots,m\}$) in $K^d$. But already the truncated vectors $(x_k^0,\dots,x_k^{m-1})$
($k\in\{1,\dots,m\}$) are linearly independent by (j).
\end{enumerate}
\end{reminder}

\begin{df}\label{hermite}
Let $K$ be a field and $f,g\in K[X]$ where $f$ is monic of degree $d$. Then the quadratic form
\[H(f,g):=\sum_{i,j=1}^d\tr(g(C_f)C_f^{i+j-2})T_iT_j\in K[T_1,\dots,T_d]\] is called
the \emph{Hermite form} of $f$ with respect to $g$.
The quadratic form $H(f):=H(f,1)$ is simply called the Hermite form of $f$.
\end{df}

\begin{rem}\label{hankel}
Let $K$ be a field with $\chara K\ne2$ and let $f,g\in K[X]$ where $f$ is monic of degree $d$. Then $M(H(f,g))$
[$\to$ \ref{longremi}(d)] is called the \emph{Hermite matrix of $f$ with respect to $g$}. This is a Hankel matrix, i.e., of the form
\[\begin{pmatrix}
\begin{tikzpicture}[scale=0.2,thick]
{
\draw(6+0.05,1+0.05)--(6-0.05,1-0.05);
\draw(6+0.05,2+0.05)--(5-0.05,1-0.05);
\draw(6+0.05,3+0.05)--(4-0.05,1-0.05);
\draw(6+0.05,4+0.05)--(3-0.05,1-0.05);
\draw(6+0.05,5+0.05)--(2-0.05,1-0.05);
\draw(4+0.05,7+0.05)--(0-0.05,3-0.05);
\draw(3+0.05,7+0.05)--(0-0.05,4-0.05);
\draw(2+0.05,7+0.05)--(0-0.05,5-0.05);
\draw(1+0.05,7+0.05)--(0-0.05,6-0.05);
\draw(0+0.05,7+0.05)--(0-0.05,7-0.05);
\draw[thick,dotted] (2+0.2,5-0.2)--(4-0.2,3+0.2);
}
\end{tikzpicture}
\end{pmatrix}.
\]
Furthermore, $M(H(f))$ is called the \emph{Hermite matrix of $f$}.
\end{rem}

\begin{pro}\label{herm}
Let $K$ be a field and $f,g\in K[X]$. Suppose $x_1,\dots,x_m\in K$ and $\al_1,\dots,\al_m\in\N_0$ such that
$f=\prod_{k=1}^m(X-x_k)^{\al_k}$ and $d:=\deg f$. Then
\[H(f,g)=\sum_{i,j=1}^d\left(\sum_{k=1}^m\al_kg(x_k)x_k^{i+j-2}\right)T_iT_j=\sum_{k=1}^m\al_kg(x_k)\left(\sum_{i=1}^dx_k^{i-1}T_i\right)^2.\]
\end{pro}

\begin{proof}
\ref{hermite} and \ref{longremi}(i).
\end{proof}

\begin{thm}[Counting roots with one side condition]\label{hermite1}
Let $R$ be a real closed field, $C:=R(\ii)$,
$f,g\in R[X]$ and $f$ monic. Then
\begin{align*}
\rk H(f,g)&=\#\{x\in C\mid f(x)=0,g(x)\ne0\}\qquad\text{and}\\
\sg H(f,g)&=\#\{x\in R\mid f(x)=0,g(x)>0\}\\
&\,-\#\{x\in R\mid f(x)=0,g(x)<0\}.
\end{align*}
\end{thm}

\begin{proof} Denote by $p\mapsto p^*$ again the ring automorphism of $C[X]$ with $x^*=x$ for all $x\in R$,
$\ii^*=-\ii$ and $X^*=X$.
Using the fundamental theorem of algebra \ref{realfund} and this automorphism, we can write
\[f=\prod_{k=1}^m(X-x_k)^{\al_k}\prod_{t=1}^n(X-z_t)^{\be_t}\prod_{t=1}^n(X-z_t^*)^{\be_t}\]
for some $m,n\in\N_0$ $\al_k,\be_t\in\N$, $x_k\in R$, $z_t\in C\setminus R$ and
$x_1,\dots,x_m,z_1,\dots,z_n,z_1^*,\dots,z_n^*$ pairwise distinct. By renumbering the $z_t$, we can find $r\in\{0,\dots,n\}$ such
that $g(z_1)\ne0,\dots,g(z_r)\ne0$ and $g(z_{r+1})=0,\dots,g(z_n)=0$.
By \ref{herm},  \ref{longremi}(k) and \ref{cheuclid}(c), we obtain linear forms
$\ell_1,\dots,\ell_m,g_1,\dots,g_r,h_1,\dots h_r\in R[T_1,\dots,T_d]$ where $d:=\deg f$ such that
\begin{align*}
H(f,g)&=\sum_{k=1}^m\al_kg(x_k)\ell_k^2+\sum_{t=1}^r(g_t+\ii h_t)^2+\sum_{t=1}^r(g_t-\ii h_t)^2\\
&=\sum_{k=1}^m\al_kg(x_k)\ell_k^2+2\sum_{t=1}^rg_t^2-2\sum_{t=1}^rh_t^2
\end{align*}
where $\ell_1,\dots,\ell_m,g_1+\ii h_1,g_1-\ii h_1,\dots,g_r+\ii h_r,g_r-\ii h_r\in C[T_1,\dots,T_d]$ are linearly independent.
Due to $C(g_i+\ii h_i)+C(g_i-\ii h_i)=Cg_i+Ch_i$, we have that \[\ell_1,\dots,\ell_m,g_1,\dots,g_r,h_1,\dots,h_r\] are also
linearly independent in $C[T_1,\dots,T_d]$ and therefore also in $R[T_1,\dots,T_d]$. It follows that
\begin{align*}
\rk H(f,g)&=\#\{k\in\{1,\dots,m\}\mid g(x_k)\ne0\}+2r\\
&=\#\{k\in\{1,\dots,m\}\mid g(x_k)\ne0\}+2\#\{t\in\{1,\dots,n\}\mid g(z_t)\ne0\}\\
&=\#\{x\in C\mid f(x)=0,g(x)\ne0\}\qquad\text{and}\\
\sg H(f,g)&=\#\{k\in\{1,\dots,m\}\mid g(x_k)>0\}-\#\{k\in\{1,\dots,m\}\mid g(x_k)<0\}+r-r\\
&=\#\{x\in R\mid f(x)=0,g(x)>0\}-\#\{x\in R\mid f(x)=0,g(x)<0\}.
\end{align*}
\end{proof}

\begin{cor}[Counting roots without side conditions]\label{crwsc}
Let $R$ be a real closed field, $C:=R(\ii)$ and suppose
$f\in R[X]$ is monic. Then
\begin{align*}
\rk H(f)&=\#\{x\in C\mid f(x)=0\}\qquad\text{and}\\
\sg H(f)&=\#\{x\in R\mid f(x)=0\}.
\end{align*}
\end{cor}

\begin{cor}[Counting roots with several side conditions]\label{several}
Let $R$ be a real closed field, $m\in\N_0$, $f,g_1,\dots,g_m\in R[X]$ and $f$ monic. Then
\[
\frac1{2^m}\sum_{\al\in\{1,2\}^m}\sg H(f,g_1^{\al_1}\dots g_m^{\al_m})=\#\{x\in R\mid f(x)=0,g_1(x)>0,\dots,g_m(x)>0\}
\]
\end{cor}

\begin{proof} The left hand side equals
\[
\frac1{2^m}\sum_{\al\in\{1,2\}^m}
\sum_{\substack{x\in R\\f(x)=0}}\sgn((g_1^{\al_1}\dots g_m^{\al_m})(x))\\
=\frac1{2^m}\sum_{\substack{x\in R\\f(x)=0}}\prod_{k=1}^m
(\underbrace{\sgn(g_k(x))+(\sgn(g_k(x)))^2}_{\textstyle=\begin{cases}0&\text{if }g_k(x)\le0\\2&\text{if }g_k(x)>0\end{cases}}).
\]
\end{proof}

\section{The real closure}

\begin{df}\label{dfrealclosure}
Let $(K,P)$ be an ordered field. An extension field $R$ of $K$ is called a \emph{real closure} of $(K,P)$ if
$R$ is real closed, $R|K$ is algebraic and the order of $R$ [$\to$ \ref{convention}, \ref{dfeuclid}] is an extension of $P$
[$\to$ \ref{oexf}].
\end{df}

\begin{pro}\label{maxordered}
Let $(R,P)$ be an ordered field. Then $R$ is real closed if and only if there is no ordered extension field $(L,Q)$ of $(R,P)$ such that
$L\ne R$ and $L|R$ is algebraic.
\end{pro}

\begin{proof}
One direction follows from \ref{rcchar}(c). Conversely, suppose that every ordered extension field $(L,Q)$ of $(R,P)$
with $L|R$ algebraic satisfies $L=R$. To show:
\begin{enumerate}[(a)]
\item $R$ is Euclidean.
\item Every polynomial of odd degree from $R[X]$ has a root in $R$.
\end{enumerate}
For (a), we show $P=R^2$. To this end, let $a\in P$. By \ref{extend2}, we can extend $P$ to $R(\sqrt a)$. Due to the hypothesis, this
implies $R(\sqrt a)=R$ and therefore $a=(\sqrt a)^2\in R^2$.

To show (b), let $f\in R[X]$ be of odd degree. Choose in $R[X]$ an irreducible divisor $g$ of $f$ of odd degree. Choose a root $x$ of
$g$ in some extension field of $R$. Then $R(x)$ is an extension field of $R$ with odd $[R(x):R]$ so that
$P$ can be extended to $R(x)$
by \ref{extendodd}. By hypothesis, this gives $R(x)=R$. In particular, $g$ and therefore $f$ has a root in $R$.
\end{proof}

\begin{thm}\label{existsrc}
Every ordered field has a real closure.
\end{thm}

\begin{proof} Let $(K,P)$ be an ordered field. Consider the algebraic closure $\overline K$ of $K$ and the set
\[M:=\{(L,Q)\mid L\text{ subfield of }\overline K, Q\text{ order of }L,(K,P)\text{ is an ordered subfield of }(L,Q)\}\]
which is partially ordered by declaring
\begin{align*}
(L,Q)\preceq(L',Q')&:\iff(L,Q)\text{ is an ordered subfield of }(L',Q')\\
&\overset{\text{\ref{unaryrem}(b)}}\iff(L\subseteq L'\et Q\subseteq Q')
\end{align*}
for all $(L,Q),(L',Q')\in M$. In $M$ every chain possesses an upper bound: The empty chain has $(K,P)$ as an upper bound.
A nonempty chain $C\subseteq M$ has \[\left(\bigcup\{L\mid(L,Q)\in C\},\bigcup\{Q\mid(L,Q)\in C\}\right)\in M\]
as an upper bound. By Zorn's lemma, $M$ possesses a maximal element $(R,Q)$. Of course, $\overline K$ is also the algebraic
closure of $R$ and therefore each algebraic extension of $R$ is (up to $R$-isomorphy) an intermediate field of $\overline K|R$.
The maximality of $(R,Q)$ in $M$ signifies by \ref{maxordered} just that $R$ is real closed. Because of $(R,Q)\in M$, the field
extension $R|K$ is algebraic and the order $Q$ is an extension of $P$.
\end{proof}

\begin{lem}\label{sameno}
Let $(K,P)$ be an \emph{ordered} subfield of the real closed fields $R$ and $R'$ [$\to$~\ref{convention}] and
$f\in K[X]$. Then $f$ has the same number of roots in both $R$ and $R'$.
\end{lem}

\begin{proof}
WLOG $f$ is monic. The number in question is by \ref{crwsc} equal to the signature of $H(f)$ that can be calculated already over
$(K,P)$ [$\to$ \ref{longremi}(f)(h)].
\end{proof}

\begin{thm}\label{unic}
Let $(K,P)$ be an ordered subfield of $(L,Q)$ such that $L|K$ is algebraic. Let $\ph$ be a homomorphism of ordered fields
from $(K,P)$ into a real closed field $R$.
Then there is exactly one homomorphism $\ps$ of ordered fields from $(L,Q)$ to $R$ with $\ps|_K=\ph$.
\end{thm}

\begin{proof}
Choose a real closure $R'$ of $(L,Q)$ according to \ref{existsrc}.

Existence: Using Zorn's lemma, one reduces easily to the case where $L|K$ is finite.
Since $\ph\colon K\to\ph(K)\subseteq R$ is an isomorphism of ordered fields,
we can suppose WLOG that $(K,P)$ is an ordered subfield of $R$ and $\ph=\id_K$.
We denote the different $K$-homo\-morphisms
from $L$ to $R$ by $\ps_1,\dots,\ps_m$ ($m\in\N_0$). Assume that none of these is a homomorphism of ordered
fields from $(L,Q)$ to $R$ (for example if $m=0$). Then there are $b_1,\dots,b_m\in Q$ such that $\ps_1(b_1)\notin R^2,\dots,
\ps_m(b_m)\notin R^2$. By the primitive element theorem there exists
\[a\in L':=L(\sqrt{b_1},\dots,\sqrt{b_m})\overset{b_i\in Q\subseteq R'^2}\subseteq R'\]
such that $L'=K(a)$. The minimal polynomial of $a$ over $K$ has by \ref{sameno} the same number of roots in $R'$ and $R$
and therefore in particular a root in $R$. Hence there is a $K$-homomorphism $\ps\colon L'\to R$.
Choose $i\in\{1,\dots,m\}$ with $\ps|_L=\ps_i$ (in particular $m>0$). Then $\ps_i(b_i)=\ps(b_i)=(\ps(\sqrt{b_i}))^2\in R^2$
$\lightning$.

Unicity: Let $a\in L$. Choose $f\in K[X]\setminus\{0\}$ with $f(a)=0$. Choose $a_1,\dots,a_m\in R'$ with $a_1<\ldots<a_m$
such that $\{x\in R'\mid f(x)=0\}=\{a_1,\dots,a_m\}$. Again WLOG $\ph|_K=\id_K$
and hence $(K,P)$ is an ordered subfield of $R$. By \ref{sameno} there are $b_1,\dots,b_m\in R$
such that $b_1<\ldots<b_m$ and
$\{x\in R\mid f(x)=0\}=\{b_1,\dots,b_m\}$. Choose now $i\in\{1,\dots,m\}$ such that $a=a_i$. We show that each
homomorphism $\ps$ of ordered fields from $(L,Q)$ to $R$ with $\ps|_K=\id$ satisfies $\ps(a)=b_i$. To this end, fix such a $\ps$.
By the already proved existence statement, there is a homomorphism of ordered fields $\rh\colon R'\to R$ such that
$\rh|_L=\ps$. Since $\rh$ is an embedding, we have $\{\rh(a_1),\dots,\rh(a_m)\}=\{b_1,\dots,b_m\}$ and by the monotonicity
we even get $\rh(a_j)=b_j$ for all $j\in\{1,\dots,m\}$. We deduce $\ps(a)=\ps(a_i)=\rh(a_i)=b_i$.
\end{proof}

\begin{cor}\label{genauein}
Let $R$ and $R'$ be real closures of the ordered field $(K,P)$. Then there is exactly one $K$-isomorphism from $R$ to $R'$.
\end{cor}

\begin{proof}
The $K$-isomorphisms from $R$ to $R'$ are obviously exactly the isomorphisms of ordered fields from $R$ to $R'$ whose
restriction to $K$ is the identity. For this reason, the claim follows easily from \ref{unic} (for the surjectivity in the existence part
use either \ref{rcchar}(c) or the unicity of $K$-automorphisms of $R$ and of $R'$ [$\to$ \ref{unic}]).
\end{proof}

\begin{notterm}\label{theclosure}
Because of \ref{genauein}, we speak of \emph{the} real closure $\overline{(K,P)}$ of
$(K,P)$. It contains by \ref{unic} (up to $K$-isomorphy)
every ordered field extension $(L,Q)$ of $(K,P)$ with $L|K$ algebraic.
\end{notterm}

\begin{thm}\label{bijext}
Suppose $(K,P)$ is an ordered field, $L|K$ an algebraic extension, $R$ a real closed field and $\ph$ a homomorphism of ordered
fields from $(K,P)$ to $R$. Then
\begin{align*}
\{\ps\mid\ps\colon L\to R\text{ homomorphism},\ps|_K=\ph\}&\to\{Q\mid\text{$Q$ is an extension of $P$ to $L$}\}\\
\ps&\mapsto\ps^{-1}(R^2)
\end{align*}
is a bijection.
\end{thm}

\begin{proof}
The well-definedness is easy to see. To verify the bijectivity, let $Q$ be an extension of $P$ to $L$. We have to show that
there is exactly one homomorphism $\ps\colon L\to R$ with $\ps|_K=\ph$ fulfilling the condition $\ps^{-1}(R^2)=Q$ that is equivalent
to $\ps$ being a homomorphism of ordered fields from $(L,Q)$ to $R$ since
\begin{align*}
\ps^{-1}(R^2)=Q&\iff\ps^{-1}(R^2\cap\ps(L))=Q\overset{\ps\colon L\to\ps(L)}{\underset{\text{bijective}}\iff} R^2\cap\ps(L)=\ps(Q)\\
&\overset{R^2\cap\ps(L)}{\underset{\text{order of $\ps(L)$}}\iff}\ps(Q)\subseteq R^2\cap\ps(L)\iff\ps(Q)\subseteq R^2.
\end{align*}
Hence we get the unicity and existence of $\ps$ from \ref{unic}.
\end{proof}

\begin{cor} Suppose $(K,P)$ is an ordered field, $R:=\overline{(K,P)}$ and $L|K$ a finite extension.
Let $a\in L$ with $L=K(a)$ and $f$ be the minimal polynomial of $a$ over $K$. Then
\begin{align*}
\{x\in R\mid f(x)=0\}&\to\{Q\mid\text{$Q$ is an extension of $P$ to $L$}\}\\
x&\mapsto\{g(a)\mid g\in K[X],g(x)\in R^2\}
\end{align*}
is a bijection.
\end{cor}

\begin{proof}
By \ref{bijext} it is enough to see that
\begin{align*}
\{x\in R\mid f(x)=0\}&\to\{\ps\mid\text{$\ps\colon L\to R$ is a $K$-homomorphism}\}\\
x&\mapsto(g(a)\mapsto g(x))\qquad(g\in K[X])
\end{align*}
is a bijection. This is easy to see.
\end{proof}

\begin{ex} Let $(K,P)$ be an ordered field with $2\notin K^2$. Denote by $\sqrt 2$ one of the two
square roots of $2$ in the algebraic closure $\overline K$ of $K$ [$\to$ \ref{notremsqrt}(a)].
Then there are exactly $2$ orders of
$K(\sqrt2)$ that extend $P$, namely the two induced by the field embeddings $K(\sqrt2)\hookrightarrow\overline{(K,P)}$
(in one of which $\sqrt2$ is positive and in one of which it is negative). In particular, this is true if $(K,P)$ is
not Archimedean [$\to$ \ref{unaryrem}(d)] and in this case we cannot argue with $\R$ instead of $\overline{(K,P)}$
as we did in \ref{sqrt2}.
\end{ex}

\begin{pro}\label{relcl}
Let $R$ be a real closed field and $K$ a subfield of $R$ that is (relatively) algebraically closed in $R$ (i.e., no element of
$R\setminus K$ is algebraic over $K$). Then $K$ is real closed.
\end{pro}

\begin{proof}
Apply the criterion from \ref{maxordered}: Every ordered extension field $(L,Q)$ of $(K,R^2\cap K)$
such that $L|K$ is algebraic is contained
in $R$ up to $K$-isomorphy [$\to$ \ref{unic}, \ref{theclosure}] and therefore equal to $K$.
\end{proof}

\begin{ex}\label{ralg}
The field $\R_{\text{alg}}:=\{x\in\R\mid\text{$x$ algebraic over $\Q$}\}$ of \emph{real algebraic numbers}
is the algebraic closure of $\Q$ in $\R$. By \ref{relcl}, $\R_{\text{alg}}$ is real closed and therefore the real closure
of $\Q$ [$\to$ \ref{convention}]. Hence $\R_{\text{alg}}$ is uniquely embeddable in every real closed field by \ref{unic}.
In this sense, $\R_{\text{alg}}$ is the smallest real closed field.
\end{ex}

\section{Real quantifier elimination}\label{sec:qe}

\begin{rem}\label{unionsection}
Let $M$, $I$ and $J_i$ for each $i\in I$ be sets and suppose $A_{ij}\subseteq M$ for all $i\in I$ and $j\in J_i$. Defining
the empty intersection as $M$ (that is $\bigcap_{i\in\emptyset}\ldots:=\bigcap\emptyset:=M$), one has
\begin{align*}
\bigcup_{i\in I}\bigcap_{j\in J_i}A_{ij}&=\bigcap_{(j_i)_{i\in I}\in\prod_{i\in I}J_i}\bigcup_{i\in I}A_{ij_i},\\
\bigcap_{i\in I}\bigcup_{j\in J_i}A_{ij}&=\bigcup_{(j_i)_{i\in I}\in\prod_{i\in I}J_i}\bigcap_{i\in I}A_{ij_i},\\
\complement\bigcup_{i\in I}\bigcap_{j\in J_i}A_{ij}&=\bigcap_{i\in I}\bigcup_{j\in J_i}\complement A_{ij}\qquad\text{and}\\
\complement\bigcap_{i\in I}\bigcup_{j\in J_i}A_{ij}&=\bigcup_{i\in I}\bigcap_{j\in J_i}\complement A_{ij}
\end{align*}
where the \emph{complement} of $A\subseteq M$ is given by
$\complement A:=\complement_MA:=M\setminus A$.
\end{rem}

\begin{dfpro}\label{booleanalgebra}
Let $M$ be a set and $\pow(M)$ its power set.
\begin{enumerate}[\normalfont(a)]
\item
We call $\mathcal S\subseteq\pow(M)$ a \emph{Boolean algebra} on $M$ if
\begin{itemize}
\item $\emptyset\in\mathcal S$,
\item $\forall S\in\mathcal S:\complement S\in\mathcal S$,
\item $\forall S_1,S_2\in\mathcal S:S_1\cap S_2\in\mathcal S$ and
\item $\forall S_1,S_2\in\mathcal S:S_1\cup S_2\in\mathcal S$.
\end{itemize}
\item
Let $\mathcal G\subseteq\pow(M)$. Then the set of all finite \alal{unions}{intersections} of finite \alal{intersections}{unions} of
elements of $\mathcal G$ and their complements (with $\bigcap\emptyset:=M$) is obviously the smallest Boolean algebra $\mathcal S$
on $M$ with $\mathcal G\subseteq\mathcal S$. It is called the Boolean algebra \emph{generated} by $\mathcal G$ (on $M$). Its
elements are called the \emph{Boolean combinations} of elements of $\mathcal G$.
\end{enumerate}
\end{dfpro}

\begin{dfrem}\label{introsemialg}
In the sequel, we let $(K,P)$ always be an ordered field, for example $(K,P)=(\Q,\Q_{\ge0})$ unless otherwise stated.
Moreover, we let $\mathcal R$ be a set of real closed fields containing $(K,P)$ as an ordered subfield. For $n\in\N_0$,
we set \[\mathcal R_n:=\{(R,x)\mid R\in\mathcal R,x\in R^n\}.\]  Thereby we have $R^0=\{\emptyset\}=\{0\}$ and
we identify $\mathcal R_0$ with $\mathcal R$. A Boolean combination of sets of the form
\[\{(R,x)\in\mathcal R_n\mid p(x)\ge0\text{ (in $R$)}\}\qquad(p\in K[X_1,\dots,X_n])\]
is called a
\begin{itemize}
\item \emph{$K$-semialgebraic set in $R^n$} if $\mathcal R=\{R\}$, and
\item an \emph{$n$-ary $(K,P)$-semialgebraic class}
if $\mathcal R$ is ``potentially very big'' (in any case big enough to contain all real closed ordered extension fields of $(K,P)$ that
are currently in the game).
\end{itemize}
We identify $K$-semialgebraic sets in $R^n$ with subsets of $R^n$. Thus these are simply the subsets
of $R^n$ that can be defined by combining finitely many polynomial inequalities with coefficients in $K$ by the logical connectives
``not'', ``and'' and ``or''.
A \emph{semialgebraic set} in $R^n$ is an $R$-semialgebraic set in $R^n$. A \emph{semialgebraic class} is a
$\Q$-semialgebraic class.
\end{dfrem}

\begin{rem}\label{rcfclass}
\begin{enumerate}[(a)]
\item On the first reading, the reader might want to think of $\mathcal R=\{R\}$ or even of $\mathcal R=\{\R\}$ in order to have
a good geometric perception. Initially one can therefore think of $(K,P)$-semialgebraic classes as $K$-semialgebraic sets.
\item One can conceive $\mathcal R$ as the ``set'' of all real closed ordered extension fields of $(K,P)$. Unfortunately, this is not a set
(otherwise Zorn's lemma would yield real closed fields having no proper real closed extension field in contradiction to
\ref{rtlfct} combined with \ref{existsrc}) but a proper \emph{class}. But we do not want to get into the formal notion of a class
and instead adopt a naïve point of view from which sets and classes are synonymous where ``big'' sets often
tend to be called classes.
\item Whoever gets vertiginous from (b), has several ways out: Our resort here is that $\mathcal R$ is a honest set that is
at any one time  sufficiently big (often $\#\mathcal R=1$ is enough and almost always  $\#\mathcal R=2$ is enough).
Alternatively, one could learn the subtle non-naïve handling of sets and classes. As a third option, one could work,
instead of with $(K,P)$-semialgebraic classes, with formulas of first-order logic in the language of ordered fields with
additional constants for the elements of $K$. The last two options are technically very involved.
\end{enumerate}
\end{rem}

\begin{rem}\label{nothingorall}
Obviously, $\emptyset$ and $\mathcal R$ are the only $0$-ary $(K,P)$-semialgebraic classes. Note that this uses heavily that $(K,P)$ is an ordered subfield of every $R\in\mathcal R$.
\end{rem}

\begin{pro}\label{sanf}
Every $(K,P)$-semialgebraic class is of the form
\[\bigcup_{i=1}^k\left\{(R,x)\in\mathcal R_n\mid f_i(x)=0,g_{i1}(x)>0,\dots,g_{im}(x)>0\right\}\]
for some $n,k,m\in\N_0$, $f_i,g_{ij}\in K[X_1,\dots,X_n]$.
\end{pro}

\begin{proof} By \ref{introsemialg} and \ref{booleanalgebra}(b) such a class is a finite union of classes of the form
\begin{multline*}
\left\{(R,x)\in\mathcal R_n\mid h_1(x)\ge0,\dots,h_s(x)\ge0,h_{s+1}(x)<0,\dots,h_{s+t}(x)<0\right\}\\
=\bigcup_{\de\in\{0,1\}^s}\left\{(R,x)\in\mathcal R_n\mid
\begin{aligned}
\sgn(h_1(x))=\de_1,\dots,\sgn(h_s(x))=\de_s,\\
-h_{s+1}(x)>0,\dots,-h_{s+t}(x)>0
\end{aligned}
\right\}\\
=\bigcup_{\de\in\{0,1\}^s}\left\{(R,x)\in\mathcal R_n\mid
\begin{aligned}
\left(\sum_{\substack{i=1\\\de_i=0}}^sh_i^2\right)(x)=0,
\Et_{\substack{i=1\\\de_i=1}}^sh_i(x)>0,\\
-h_{s+1}(x)>0,\dots,-h_{s+t}(x)>0
\end{aligned}
\right\}
\end{multline*}
for some $s,t\in\N_0$ and $h_i\in K[X_1,\dots,X_n]$
\end{proof}

\begin{pro}\label{sapre}
Let $m,n\in\N_0$, $h_1,\dots,h_m\in K[X_1,\dots,X_n]$ and $S\subseteq\mathcal R_m$ a $(K,P)$-semialgebraic
class. Then
$\{(R,x)\in\mathcal R_n\mid(R,(h_1(x),\dots,h_m(x)))\in S\}$ is a $(K,P)$-semialgebraic class.
\end{pro}

\begin{proof}
If $S=\bigcup_{i=1}^k\{(R,y)\in\mathcal R_m\mid f_i(y)=0,g_{i1}(y)>0,\dots,g_{i\ell}(y)>0\}$ with $m,k,\ell\in\N_0$,
$f_i,g_{ij}\in K[Y_1,\dots,Y_m]$ so that
\begin{multline*}
\{(R,x)\in\mathcal R_n\mid(h_1(x),\dots,h_m(x))\in S\}\\
=\bigcup_{i=1}^k\left\{(R,x)\in\mathcal R_n\mid
(f_i(h_1,\dots,h_m))(x)=0,\right.\qquad\qquad\qquad\qquad\qquad\ \;\\
\left.(g_{i1}(h_1,\dots,h_m))(x)>0,\dots,
(g_{i\ell}(h_1,\dots,h_m))(x)>0\right\}.
\end{multline*}
\end{proof}

\begin{cor}\label{preimagesa}
Let $R$ be a real closed field. Preimages of semialgebraic subsets of $R^m$ under polynomial maps $R^n\to R^m$ are again
semialgebraic in $R^n$. 
\end{cor}

\begin{lem}\label{sigsa}
For every $s\in\N_0$, \[\left\{(R,x)\in\mathcal R_{d+1}\mid\si\left(\sum_{i=0}^d x_iT^i\right)=s\text{ with respect to $R[T]$}\right\}\]
is a semialgebraic class.
\end{lem}

\begin{proof}
The class in question equals
\[\bigcup_{\substack{\de\in\{-1,0,1\}^{d+1}\\\si\left(\sum_{i=0}^d\de_iT^i\right)=s\text{ with respect to $\R[T]$}}}
\left\{(R,x)\in\mathcal R_{d+1}\mid\sgn_R(x_0)=\de_0,\dots,\sgn_R(x_{d})=\de_d\right\}.\]
\end{proof}

\begin{rem}\label{advertisetarski}
We will now need the \emph{simultaneous} diagonalization of a symmetric matrix as a quadratic form and as an endomorphism
[$\to$ \ref{longremi}(g)]. The reader should know this over $\R$ from linear algebra but we will now
need it more generally over an arbitrary real closed field. Later in this chapter, we will provide methods from which it becomes
immediately clear that, for each fixed matrix size, the class of all fields $R\in\mathcal R$, over which the corresponding
statement is true, is a $0$-ary semialgebraic class. Since the statement is true over $\R$, it must then by \ref{nothingorall} also
hold true over every real closed field. In a similar way, we will soon be able to carry over
a great many statements from $\R$ to all real closed
fields. Unfortunately, we are not that far yet and therefore we have to check if the proof from linear algebra goes through over
an arbitrary real closed field. Some of the proofs of the diagonalization in question use however proper analysis instead of
just the fundamental theorem of algebra. Since the whole analysis is built on the completeness of $\R$
[$\to$ \ref{introduce-the-reals}], those proofs do not generalize without further ado. Thus we give a compact ad-hoc-proof.
\end{rem}

\begin{thm}\label{eucliddiag}
Let $R$ be a real closed field and $M\in SR^{n\times n}$. Then there is some $P\in\GL_n(R)$ satisfying $P^TP=I_n$ such that
$P^TMP$ is a diagonal matrix.
\end{thm}

\begin{proof}
Call a symmetric bilinear form $V\times V\to R,\ (v,w)\mapsto\langle v,w\rangle$ on an $R$-vector space $V$
positive definite if $\langle v,v\rangle>0$ for all $v\in V\setminus\{0\}$. Call an $R$-vector space together with a positive definite
symmetric bilinear form a Euclidean $R$-vector space. Call an endomorphism $f$ of a Euclidean $R$-vector space $V$
self-adjoint if $\langle f(v),w\rangle=\langle v,f(w)\rangle$ for all $v,w\in V$.

\medskip
\textbf{Claim 1:} Let $V$ be a Euclidean $R$-vector space, $f\in\End(V)$ self-adjoint and $v$ an eigenvector of $f$. Then
$U:=\{u\in V\mid\langle u,v\rangle=0\}$ is a subspace of $V$ with $v\notin U$ and $f(U)\subseteq U$.

\smallskip
\emph{Explanation.} Choose $\la\in R$ with $f(v)=\la v$ and let $u\in U$. Then
$\langle f(u),v\rangle=\langle u,f(v)\rangle=\langle u,\la v\rangle=\la\langle u,v\rangle=\la0=0$.

\medskip
\textbf{Claim 2:}
Let $V\ne\{0\}$ be a finite-dimensional Euclidean $R$-vector space and $f\in\End(V)$ self-adjoint. Then $f$ possesses an eigenvalue
in $R$.

\smallskip
\emph{Explanation.} Assume $f$ has no eigenvalue. By Cayley-Hamilton and the fundamental theorem \ref{realfund}, it is easy to show that there are
$a,b\in R$ with $b\ne0$ such that \[(f-a\id_V)^2+b^2\id_V\] has a non-trivial kernel. Since $f$ is self-adjoint, $g:=f-a\id_V$ is so.
Choose $v\in V$ with $g^2(v)=-b^2v$. Then $0\le\langle g(v),g(v)\rangle=\langle g^2(v),v\rangle=\langle-b^2v,v\rangle
=-b^2\langle v,v\rangle<0$. $\lightning$

\medskip
\textbf{Claim 3:}
Let $V$ be a finite-dimensional Euclidean $R$-vector space and $f\in\End(V)$ self-adjoint. Then there is an eigenbasis
$v_1,\dots,v_n$ for $f$ with $(\langle v_i,v_j\rangle)_{1\le i,j\le n}=I_n$.

\smallskip
\emph{Explanation.} Use Claim 1, Claim 2 and induction over the dimension $V$.

\smallskip\noindent
In virtue of $\langle x,y\rangle:=\sum_{i=1}^nx_iy_i$ ($x,y\in R^n$), $R^n$ is a Euclidean $R$-vector space and
$f\colon R^n\to R^n,\ x\mapsto Mx$ is self-adjoint. By Claim 3, there is an eigenbasis $v_1,\dots,v_n$ for $f$ such that
$(\langle v_i,v_j\rangle)_{1\le i,j\le n}=I_n$. Set $P:=(v_1\ldots v_n)\in\GL_n(R)$. Then
\[P^TP=\begin{pmatrix}v_1^T\\\vdots\\v_n^T\end{pmatrix}\begin{pmatrix}v_1&\ldots&v_n\end{pmatrix}=I_n\]
and $P$ is the change-of-basis matrix from $(v_1,\dots,v_n)$ to the standard basis. It follows that
$P^TMP=P^{-1}MP$ is the representing matrix of $f$ with respect to $(v_1,\dots,v_n)$.
\end{proof}

\begin{cor}[Determination of the signature using Descartes' rule of signs]\label{combining}
Let $R$ be a real closed field, $q\in R[T_1,\dots,T_d]$ a quadratic form and $h:=\det(M(q)-XI_d)\in R[X]$ the characteristic
polynomial of the representing matrix \emph{[$\to$ \ref{longremi}(d)]} of $q$. Then we have:
\begin{enumerate}[\normalfont(a)]
\item $h$ is real-rooted \emph{[$\to$ \ref{defrr}]}
\item $\sg q=\mu(h)-\mu(h(-X))$ \emph{[$\to$ \ref{longremi}(h), \ref{defst}(a)]}
\item $\sg q=\si(h)-\si(h(-X))$ \emph{[$\to$ \ref{defst}(b)]}
\end{enumerate}
\end{cor}

\begin{proof}
Using \ref{eucliddiag}, choose $P\in\GL_d(R)$ such that $P^TP=I_d$ and $P^TM(q)P$ is diagonal, say
\[P^TM(q)P=
\begin{pmatrix}~
\begin{tikzpicture}[inner sep=0]
\node (a1) {$\la_1$};
\node (an) at (1.5,-1.5) [anchor=north west] {$\la_d$};
\node[scale=3.2] at (1.3,-0.4) {$0$};
\node[scale=3.2] at (0.2,-1.2) {$0$};
\draw[loosely dotted,very thick,dash phase=3pt] (a1)--(an);
\end{tikzpicture}
\end{pmatrix}
\]
with $\la_i\in R$. We have
\begin{align*}
h&=h\det(P^TP)\\
&=(\det(P^T))(\det(M(q)-XI_d))(\det P)\\
&=\det(P^TM(q)P-XP^TP)=\det
\begin{pmatrix}~
\begin{tikzpicture}[inner sep=0]
\node (a1) {$\la_1-X$};
\node (an) at (1.5,-1.5) [anchor=north west] {$\la_d-X$};
\node[scale=3.2] at (1.8,-0.4) {$0$};
\node[scale=3.2] at (0.2,-1.2) {$0$};
\draw[loosely dotted,very thick,dash phase=3pt] (a1)--(an);
\end{tikzpicture}
\end{pmatrix}=\prod_{i=1}^d(\la_i-X),
\end{align*}
from which (a) follows immediately. Because of
\[M(q)=(P^T)^T
\begin{pmatrix}~
\begin{tikzpicture}[inner sep=0]
\node (a1) {$\la_1$};
\node (an) at (1.5,-1.5) [anchor=north west] {$\la_d$};
\node[scale=3.2] at (1.3,-0.4) {$0$};
\node[scale=3.2] at (0.2,-1.2) {$0$};
\draw[loosely dotted,very thick,dash phase=3pt] (a1)--(an);
\end{tikzpicture}
\end{pmatrix}
P^T\]
and $P^T\in\GL_d(R)$, it follows from \ref{longremi}(e) that
\[
\sg q=\#\{i\in\{1,\dots,d\}\mid\la_i>0\}-\#\{i\in\{1,\dots,d\}\mid\la_i<0\}=\mu(h)-\mu(h(-X)),
\]
which proves (b). Finally, (c) follows from (a) and (b) due to the exactness of Descartes' rule of signs for
real rooted polynomials [$\to$ \ref{descartesrr}]. 
\end{proof}

\begin{rem} Combining \ref{combining} with \ref{several}, one can reduce the count of real roots of polynomials without multiplicity
with side conditions by means of the Hermite method from §\ref{sec:hermite} to the count of roots of real-rooted polynomials with
multiplicity by means of Descartes' rule from §\ref{sec:descartes}.
\end{rem}

\begin{lem}\label{elim1}
Let $m,n,d\in\N_0$ and $f,g_1,\dots,g_m\in K[X_1,\dots,X_{n+1}]$. Then
\begin{align*}
\{(R,x)\in\mathcal R_n\mid&\deg f(x,X_{n+1})=d\quad\et\\
&\exists x_{n+1}\in R\colon(f(x,x_{n+1})=0\et g_1(x,x_{n+1})>0\et\ldots\et g_m(x,x_{n+1})>0)\}
\end{align*}
is a $(K,P)$-semialgebraic class.
\end{lem}

\begin{proof}
Write $f=\sum_{i=0}^Dh_iX_{n+1}^i$ for some $D\in\N_0$, $D\ge d$ and $h_i\in K[X_1,\dots,X_n]$. WLOG $h_d\ne0$.
Then \[f_0:=\sum_{i=0}^d\frac{h_i}{h_d}X_{n+1}^i\in K(X_1,\dots,X_n)[X_{n+1}]\] is monic of degree $d$. For every $\al\in\{1,2\}^m$,
we consider also $g_1^{\al_1}\dotsm g_m^{\al_m}$ as a polynomial in $X_{n+1}$ with coefficients from the field
$K(X_1,\dots,X_n)$ and set
\[h_\al:=\det(M(H(f_0,g_1^{\al_1}\dotsm g_m^{\al_m}))-XI_d)\in K(X_1,\dots,X_n)[X].\]
By construction [$\to$ \ref{longremi}(i), \ref{hermite}, \ref{hankel}], there is some $N\in\N$ such that
\[h_d^Nh_\al\in K[X_1,\dots,X_n,X]\] for all $\al\in\{1,2\}^m$. Now the class from the claim can be written by
\ref{several} as
\begin{align*}
\Bigg\{(R,x)\in\mathcal R_n\mid&h_D(x)=\ldots=h_{d+1}(x)=0\ne h_d(x)\quad\et\\
&\sum_{\al\in\{1,2\}^m}\sg H(f_0(x,X_{n+1}),(g_1^{\al_1}\dotsm g_m^{\al_m})(x,X_{n+1}))>0\Bigg\}.
\end{align*}
But
\begin{multline*}
\left\{(R,x)\in\mathcal R_n\mid h_d(x)\ne0\et
\sum_{\al\in\{1,2\}^m}\sg H(f_0(x,X_{n+1}),(g_1^{\al_1}\dotsm g_m^{\al_m})(x,X_{n+1}))>0\right\}\\
\overset{\ref{combining}}{\underset{\warningsign}=}\left\{(R,x)\in\mathcal R_n\mid h_d(x)\ne0\et
\sum_{\al\in\{1,2\}^m}(\si(h_\al(x,X))-\si(h_\al(x,-X)))>0\right\}\\
=\bigcup_{\substack{(s_{\al})_{\al\in\{1,2\}^m},(t_{\al})_{\al\in\{1,2\}^m}\in\{0,\dots,d\}^{\{1,2\}^m}\\\sum_{\al\in\{1,2\}^m}(s_\al-t_\al)>0}}
\bigcap_{\al\in\{1,2\}^m}\left\{(R,x)\in\mathcal R_n\mid
\begin{aligned}
&h_d(x)\ne0,\\
&\si((h_d^Nh_\al)(x,X))=s_\al,\\
&\si((h_d^Nh_\al)(x,-X))=t_\al
\end{aligned}\right\}
\end{multline*}
is $(K,P)$-semialgebraic by \ref{sigsa} and \ref{sapre}.
Here the warning sign $\warningsign$ indicates where
an important argument flows in:
\[h_\al(x,X)=\det(M(H(f_0(x,X_{n+1}),(g_1^{\al_1}\dotsm g_m^{\al_m})(x,X_{n+1})))-XI_d)\]
since evaluating in $x$ commutes with building companion matrices, Hermite forms and with taking determinants
[$\to$ \ref{hermite}, \ref{longremi}(i)].
\end{proof}

\begin{lem}\label{adjust}
Let $R$ be a real closed field, $m\in\N_0$ and $g_1,\dots,g_m\in R[X]$. Setting $g:=g_1\dotsm g_m$ and $f:=(1-g^2)g'$, we have
\begin{enumerate}[(a)]
\item There is an $x\in R$ satisfying $g_1(x)>0,\dots,g_m(x)>0$ if and only if there is such an $x\in R$ satisfying in addition
$f(x)=0$.
\item If $f=0$ and $g_1\ne0,\dots,g_m\ne0$, then $g_1,\dots,g_m\in R$.
\end{enumerate}
\end{lem}

\begin{proof}
(b) Suppose $f=0$. Then $g^2=1$ or $g'=0$. In both cases it follows $g\in R$ and thus $g_1,\dots,g_m\in R$ provided that
$g_1\ne0,\dots,g_m\ne0$.

\smallskip
(a) Let $x\in R$ such that $g_1(x)>0,\dots,g_m(x)>0$. Denote by $a_1,\dots,a_r$ where $r\in\N_0$ and $a_1<\ldots<a_r$ the
roots of $g$ in $R$.

First consider the case where $r=0$. By the intermediate value theorem \ref{intermediate} each of the $g_i$ is positive on $R$. It suffices therefore to show that $f$ has a root in $R$. By Definition \ref{dfrealclosed}, $g$ has even degree.
If $g$ has degree $0$, then $g'=0$ and we are done. So suppose now $\deg g\ge 2$. Then the degree of $g'$ is odd so that $g'$
and in particular $f$ has a root in $R$ by Definition \ref{dfrealclosed}.

From now on suppose that $r>0$.
By the intermediate value theorem \ref{intermediate} each of the $g_i$ has constant sign on each of the
intervals $(-\infty,a_1),(a_1,a_2),\dots,(a_{r-1},a_r),(a_r,\infty)$. It is therefore enough to show that $f$ possesses
in each of these sets a root.
By Rolle's theorem \ref{rolle}, $g'$ and therefore $f$ has on each of the sets $(a_i,a_{i+1})$
($1\le i\le r-1$) a root. WLOG $f\ne0$. Then $g'\ne0$ and $g$ has degree $\ge1$. Consequently, $1-g^2$ has a leading
monomial of even degree with a negative leading coefficient. By Lemma \ref{sgnbounds}(a), $(1-g^2)(y)<0$ for all $y\in R$
with $|y|$ sufficiently big. On the other hand, $(1-g^2)(a_1)=1=(1-g^2)(a_r)$. By the intermediate value theorem \ref{intermediate},
$1-g^2$ and therefore $f$ has a root on each of the sets $(-\infty,a_1)$ and $(a_r,\infty)$. 
\end{proof}

\begin{lem}\label{elim2}
Let $m,n\in\N_0$ and $g_1,\dots,g_m\in K[X_1,\dots,X_{n+1}]$. Then
\[\{(R,x)\in\mathcal R_n\mid\exists x_{n+1}\in R:(g_1(x,x_{n+1})>0\et\dots\et g_m(x,x_{n+1})>0)\}\] is a $(K,P)$-semialgebraic
class.
\end{lem}

\begin{proof}
Set $g:=g_1\dotsm g_m$ and $f:=(1-g^2)\frac{\partial g}{\partial X_{n+1}}$. Denote by $D:=\deg_{X_{n+1}}f\in\{-\infty\}\cup\N_0$
the degree of $f$ considered as a polynomial in $X_{n+1}$ with coefficients from $K[X_1,\dots,X_n]$. The class in question equals
because of \ref{adjust}
\begin{multline*}
\bigcup_{d=0}^D\left\{(R,x)\in\mathcal R_n\mid\deg f(x,X_{n+1})=d\et\exists x_{n+1}\in R:\left(
\begin{aligned}
f(x,x_{n+1})&=0\ \et\\
g_1(x,x_{n+1})&>0\ \et\\
&\ \ \vdots\\
g_m(x,x_{n+1})&>0
\end{aligned}
\right)
\right\}\\
\cup\{(R,x)\in\mathcal R_n\mid f(x,X_{n+1})=0\et g_1(x,0)>0\et\dots\et g_m(x,0)>0\}
\end{multline*}
and therefore is $(K,P)$-semialgebraic by \ref{elim1}.
\end{proof}

\begin{thm}[Real quantifier elimination]\label{elim}
Suppose $n\in\N_0$ and $S$ is an $(n+1)$-ary $(K,P)$-semialgebraic class. Then
$\{(R,x)\in\mathcal R_n\mid\exists x_{n+1}\in R:(R,(x,x_{n+1}))\in S\}$ and
$\{(R,x)\in\mathcal R_n\mid\forall x_{n+1}\in R:(R,(x,x_{n+1}))\in S\}$ are $n$-ary $(K,P)$-semialgebraic classes.
\end{thm}

\begin{proof}
Because the second class is the complement of
\[\{(R,x)\in\mathcal R_n\mid\exists x_{n+1}\in R:(R,(x,x_{n+1}))\in\complement S\},\]
it is enough to consider the first class.
By means of \ref{sanf}, one can assume WLOG that $S$ is of the form
\[S=\{(R,(x,x_{n+1})\in\mathcal R_{n+1}\mid f(x,x_{n+1})=0,g_1(x,x_{n+1})>0,\dots,g_m(x,x_{n+1})>0\}\]
for some $f,g_i\in K[X_1,\dots,X_{n+1}]$. Setting $D:=\deg_{X_{n+1}}f$, we obtain
\begin{multline*}
\{(R,x)\in\mathcal R_n\mid\exists x_{n+1}\in R:(R,(x,x_{n+1}))\in S\}\\
=\bigcup_{d=0}^D\left\{(R,x)\in\mathcal R_n\mid\deg f(x,X_{n+1})=d\et\exists x_{n+1}\in R:\left(
\begin{aligned}
f(x,x_{n+1})&=0\ \et\\
g_1(x,x_{n+1})&>0\ \et\\
&\ \ \vdots\\
g_m(x,x_{n+1})&>0
\end{aligned}
\right)\right\}\\
\cup\left(
\begin{aligned}
&\{(R,x)\in\mathcal R_n\mid f(x,X_{n+1})=0\}\cap\\
&\{(R,x)\in\mathcal R_n\mid\exists x_{n+1}\in R:(g_1(x,x_{n+1})>0\et\dots\et g_m(x,x_{n+1})>0)\}
\end{aligned}
\right)
\end{multline*}
which is $(K,P)$-semialgebraic by \ref{elim1} and \ref{elim2}.
\end{proof}

\begin{thm}{}\emph{[$\to$ \ref{preimagesa}]}\label{imagesa}
Let $R$ be a real closed field. Images of semialgebraic subsets of $R^n$ under polynomial maps $R^n\to R^m$ are again
semialgebraic in $R^m$.
\end{thm}

\begin{proof}
Let $S\subseteq R^n$ be semialgebraic and let $h_1,\dots,h_m\in R[X_1,\dots,X_n]$.
We have to show that $\{y\in\R^m\mid\exists x\in R^n:(x\in S \et y_1=h_1(x)\et\ldots\et y_m=h_m(x))\}$
is again semialgebraic. But this follows by applying $n$ times the quantifier elimination \ref{elim}.
\end{proof}

\begin{ex}[Tarski principle]\label{tprinciple}
The real quantifier elimination \ref{elim} can be used together with \ref{nothingorall} to generalize many statements from $\R$ to
other real closed fields. This has already been advertised in \ref{advertisetarski}. To give the reader a sense of the type of statements
admitting such a generalization, we give several examples.
\begin{enumerate}[(a)]
\item(``intermediate value theorem for rational functions'') [$\to$ \ref{intermediate}] From analysis, we know for $R=\R$:
If $f,g\in R[X]$, $a,b\in R$ with $a\le b$, $g(c)\ne0$ for all $c\in[a,b]$ and
$\sgn\left(\frac{f(a)}{g(a)}\right)\ne\sgn\left(\frac{f(b)}{g(b)}\right)$, then there is a $c\in[a,b]$ with $f(c)=0$.
We claim that this is valid even for all real closed fields $R$. To this end, it is enough to show that for each $d\in\N$
\[
S_d:=
\left\{R\in\mathcal R\mid
\begin{aligned}
&\forall x_0,\dots,x_d,y_0,\dots,y_d,a,b\in R:\\
&
\left.
\left(\begin{aligned}
&\left.\left(\begin{aligned}
&\scriptstyle(a\le b~\et~\left(\forall c\in[a,b]:\sum_{i=0}^dy_ic^i\ne0\right)~\et\\
&\scriptstyle\sgn\left(\left(\sum_{i=0}^dx_ia^i\right)\left(\sum_{i=0}^dy_ib^i\right)\right)\ne\sgn\left(\left(\sum_{i=0}^dx_ib^i\right)\left(\sum_{i=0}^dy_ia^i\right)\right)
\end{aligned}\right)\right\}\scriptstyle(*)\\
&\implies\left.\exists c\in[a,b]:\sum_{i=0}^dx_ic^i=0\right\}\scriptstyle(**)
\end{aligned}\right)
\right\}\scriptstyle(***)
\end{aligned}
\right\}
\]
is a semialgebraic class because then $\R\in S_d$ implies by \ref{nothingorall} $S_d=\mathcal R$. Fix $d\in\N$.
Applying the quantifier elimination \ref{elim} $2d+4$ times, it is enough to show that the following class is
semialgebraic:
\begin{multline*}
\{(R,(x_0,\dots,x_d,y_0,\dots,y_d,a,b))\in\mathcal R^{2d+4}\mid(***)\}=\\
\complement
\underbrace{\{(R,(x_0,\dots,x_d,y_0,\dots,y_d,a,b))\in\mathcal R^{2d+4}\mid(*)\}}_{S'}\\
\cup\underbrace{\{(R,(x_0,\dots,x_d,y_0,\dots,y_d,a,b))\in\mathcal R^{2d+4}\mid(**)\}}_{S''}
\end{multline*}
It is thus enough to show that $S'$ and $S''$ are semialgebraic. We accomplish this in each case by applying the
quantifier elimination \ref{elim}. We explicate this only for $S'$ since it is analogous and even simpler for $S''$:
\begin{multline*}
S'=\{(R,(x_0,\dots,x_d,y_0,\dots,y_d,a,b))\mid b-a\ge0\}\cap\\
\{(R,(x_0,\dots,x_d,y_0,\dots,y_d,a,b))\mid\forall c\in R:(\overbrace{c\in[a,b]\implies\sum_{i=0}^dy_ic^i\ne0}^{(****)})\}\cap\\
\bigcup_{\substack{\de,\ep\in\{-1,0,1\}\\\de\ne\ep}}
\left\{(R,(x_0,\dots,x_d,y_0,\dots,y_d,a,b))\mid
\begin{aligned}
\scriptstyle\sgn\left(\left(\sum_{i=0}^dx_ia^i\right)\left(\sum_{i=0}^dy_ib^i\right)\right)=\de,\\
\scriptstyle\sgn\left(\left(\sum_{i=0}^dx_ib^i\right)\left(\sum_{i=0}^dy_ia^i\right)\right)=\ep
\end{aligned}
\right\}.
\end{multline*}
By quantifier elimination it is enough to show that
\[\{(R,(x_0,\dots,x_d,y_0,\dots,y_d,a,b))\mid({*}{*}{*}{*})\}\]
is semialgebraic. But this class equals
\begin{multline*}
\{(R,(x_0,\dots,x_d,y_0,\dots,y_d,a,b))\mid c<a\text{ or }b<c\}~\cup\\
\left\{(R,(x_0,\dots,x_d,y_0,\dots,y_d,a,b,c))\mid \sum_{i=0}^dy_ic^i\ne0\right\}
\end{multline*}
\item Let $R$ be a real closed field and $f\in R[X]$ with $f\ge0$ on $R$. We claim that the sum $g:=f+f'+f''+\dots$ of all
derivatives of $f$ satisfies again $g\ge0$ on $R$. We show this first for $R=\R$: In this case, we have for all $x\in\R$
\[\frac{dg(x)e^{-x}}{dx}=g'(x)e^{-x}-g(x)e^{-x}=(g'(x)-g(x))e^{-x}=-f(x)e^{-x}\le0,\]
from which it follows that $h\colon\R\to\R,\ x\mapsto g(x)e^{-x}$ is
anti-monotonic [$\to$~\ref{monodef}].
From this and the fact that $\lim_{x\to\infty}h(x)=\lim_{x\to\infty}(g(x)e^{-x})=0$, we deduce that $h(x)\ge0$ and therefore
$g(x)\ge0$ for all $x\in\R$. Thus the claim is proved for $R=\R$. To show it for all real closed fields $R$, it is now enough to
show that for all $d\in\N$
\begin{align*}
S_d:=\Bigg\{R\in\mathcal R\mid~&\forall a_0,\dots,a_d\in R:\Bigg(\left(\forall x\in R:\sum_{i=0}^da_ix^i\ge0\right)\implies\\
&\forall x\in R:\sum_{k=0}^d\sum_{i=k}^di(i-1)\dotsm(i-k+1)a_ix^{i-k}\ge0\Bigg)\Bigg\}
\end{align*}
is semialgebraic since then by \ref{nothingorall} $\R\in S_d$ implies $S_d=\mathcal R$. This can be shown for each
$d\in\N$ by applying the quantifier elimination $d+3$ times.
\item We can reprove \ref{eucliddiag} since for $R=\R$ it is already known from linear algebra and it suffices to show for
fixed $n\in\N$ that
\[
S_n:=\left\{R\in\mathcal R\mid
\begin{aligned}
&\forall a_{11},a_{12},\dots,a_{nn}\in R:\\
&\left(
\begin{aligned}
&(\forall i,j\in\{1,\dots,n\}:a_{ij}=a_{ji})\implies\\
&\left(\begin{aligned}
&\exists b_{11},b_{12},\dots,b_{nn}\in R:\\
&\left(\begin{aligned}
\begin{aligned}
&\left(
\forall i,k\in\{1,\dots,n\}:
\sum_{j=1}^nb_{ji}b_{jk}=\de_{ik}
\right)
\et\\
&\forall i,\ell\in\{1,\dots,n\}:\left(i\ne\ell\implies\sum_{j,k=1}^nb_{ji}a_{jk}b_{k\ell}=0\right)
\end{aligned}
\end{aligned}\right)
\end{aligned}\right)
\end{aligned}
\right)
\end{aligned}
\right\}
\]
is semialgebraic. We manage to do so by implementing the quantifications over $i,j,k,\ell$ as finite intersections of semialgebraic
classes and by eliminating the quantification over $a_{11},\dots,b_{nn}$ by applying $2n^2$ times \ref{elim}.
\item By \ref{nothingorall}, $\{R\in\mathcal R\mid R\text{ archimedean}\}$ [$\to$ \ref{archetcdef}(a)] is not a semialgebraic class
(if $\mathcal R$ is big enough) since it contains $\R$ but not $\overline{(\R(X),P)}$ where $P$ is an arbitrary order of $\R(X)$.
\end{enumerate}
\end{ex}

\section{Canonical isomorphisms of Boolean algebras of semialgebraic sets and classes}

In this section, we fix again an ordered field $(K,P)$ and a set $\mathcal R$ of real closed extensions of $(K,P)$
[$\to$ \ref{introsemialg}]. 

\begin{df}\label{bahom}
Let $M_1$ and $M_2$ be sets, $\mathcal S_1$ a Boolean algebra on $M_1$ and $\mathcal S_2$ a Boolean algebra on
$M_2$. Then $\Ph\colon\mathcal S_1\to\mathcal S_2$ is called a \emph{homomorphism of Boolean algebras} if
$\Ph(\emptyset)=\emptyset$, $\Ph\left(\complement S\right)=\complement\Ph(S)$,
$\Ph(S\cap T)=\Ph(S)\cap\Ph(T)$ and $\Ph(S\cup T)=\Ph(S)\cup\Ph(T)$ for all $S,T\in\mathcal S_1$. If $\Ph$ is in addition
\alalal{injective}{surjective}{bijective}, then $\Ph$ is called an \alalal{embedding}{epimorphismus}{isomophism} of Boolean
algebras.
\end{df}

\begin{lem}\label{charemb}
Suppose $\mathcal S_1$ and $\mathcal S_2$ are Boolean algebras and
$\Ph\colon\mathcal S_1\to\mathcal S_2$ is a homomorphism. Then the following are equivalent:
\begin{enumerate}[(a)]
\item $\Ph$ is an embedding.
\item $\forall S\in\mathcal S_1:(\Ph(S)=\emptyset\implies S=\emptyset)$
\end{enumerate}
\end{lem}

\begin{proof}
\underline{(a)$\implies$(b)}\quad Suppose (a) holds and consider $S\in\mathcal S_1$ such that $\Ph(S)=\emptyset$. Then 
$\Ph(S)=\emptyset=\Ph(\emptyset)$ and hence $S=\emptyset$ by the injectivity of $\Ph$.

\smallskip
\underline{(b)$\implies$(a)}\quad Suppose (b) holds and let $S,T\in\mathcal S_1$ such that $\Ph(S)=\Ph(T)$. Then
$\Ph(S\setminus T)=\Ph\left(S\cap\complement T\right)=\Ph(S)\cap\complement\Ph(T)=\emptyset$ and
therefore $S\setminus T=\emptyset$. Analogously, we obtain $T\setminus S=\emptyset$. Then $S=T$.
\end{proof}

\begin{notation}\label{introsn}
Let $n\in\N_0$. From now on, we denote by $\mathcal S_n$ the Boolean algebra of all $n$-ary
$(K,P)$-semialgebraic classes.
For every $R\in\mathcal R$, we let furthermore $\mathcal S_{n,R}$ denote the Boolean algebra of all
$K$-semialgebraic subsets of $R^n$
(i.e., $\mathcal S_{n,R}=\mathcal S_n$ for $\mathcal R=\{R\}$). We call the
map $\set_R\colon\mathcal S_n\to\mathcal S_{n,R},\ S\mapsto\{x\in R^n\mid(R,x)\in S\}$ the \emph{setification} to $R$ for every $R\in\mathcal R$.
\end{notation}

\begin{thmdef}\label{setification}
Let $n\in\N_0$ and $R\in\mathcal R$. The setification \[\set_R\colon\mathcal S_n\to\mathcal S_{n,R}\] is an isomorphism of
Boolean algebras. We call its inverse map \[\class_R:=\set_R^{-1}\colon\mathcal S_{n,R}\to\mathcal S_n\] the
\emph{classification}.
\end{thmdef}

\begin{proof}
It is clear that $\set_R$ is an epimorphism. Suppose $\emptyset\ne S\in\mathcal S_n$. By Lemma \ref{charemb}, it suffices to show
$\set_RS\ne\emptyset$. By the quantifier elimination \ref{elim},
\[T:=\{R'\in\mathcal R\mid\exists x\in R'^n:(R',x)\in S\}\] is
$(K,P)$-semialgebraic and hence by  \ref{nothingorall} either empty or $\mathcal R$. From $S\ne\emptyset$, we have of course
$T\ne\emptyset$. Therefore $R\in\mathcal R=T$, i.e., there is some $x\in R^n$ with $(R,x)\in S$. Then $x\in\set_RS$ and thus
$\set_RS\ne\emptyset$.
\end{proof}

\begin{cordef}\label{transfer}
Let $n\in\N_0$ and $R,R'\in\mathcal R$. Then there is exactly one isomorphism of Boolean algebras
$\transfer_{R,R'}\colon\mathcal S_{n,R}\to\mathcal S_{n,R'}$ satisfying
\[\transfer_{R,R'}(\{x\in R^n\mid p(x)\ge0\})=\{x\in R'^n\mid p(x)\ge0\}\]
for all $p\in K[X_1,\dots,X_n]$. We call $\transfer_{R,R'}$ the \emph{transfer} from $R$ to $R'$.
\end{cordef}

\begin{proof}
The uniqueness is clear since $\mathcal S_{n,R}$ is generated by \[\{\{x\in R^n\mid p(x)\ge0\}\mid p\in K[X_1,\dots,X_n]\}\]
[$\to$ \ref{booleanalgebra}(b)]. Existence is established by setting $\transfer_{R,R'}:=\set_{R'}\circ\class_R$.
Indeed, let $p\in K[X_1,\dots,X_n]$ and set $S:=\{(R'',x)\in\mathcal R_n\mid p(x)\ge0\text{ in $R''$}\}$. Then the claim is
that $\transfer_{R,R'}(\set_R S)=\set_{R'}(S)$ which is clear since
$\transfer_{R,R'}(\set_R S)=(\set_{R'}\circ\class_R)(\set_R S)=\set_{R'}((\underbrace{\class_R\circ\set_R}_{\id_{\mathcal S_n}})(S))$.
\end{proof}

\chapter{Hilbert's 17th problem}

\section{Nonnegative polynomials in one variable}

\begin{thm}\label{so2s}
Suppose $R$ is a real closed field and $f\in R[X]$. Then the following are equivalent:
\begin{enumerate}[\normalfont(a)]
\item $f\ge0$ on $R$ \qquad\emph{[$\to$ \ref{intervals}]}
\item $f$ is a sum of two squares in $R[X]$.
\item $f\in\sum R[X]^2$ \qquad\emph{[$\to$ \ref{divnot}]}
\end{enumerate}
\end{thm}

\begin{proof}(b)$\implies$(c)$\implies$(a) is trivial. In order to show (a)$\implies$(b), we
set $C:=R(\ii)$ and consider the ring automorphism
\[C[X]\to C[X],\ p\mapsto p^*\]
given by $a^*=a$ for $a\in R$, $\ii^*=-\ii$ and $X^*=X$. WLOG $f\ne0$. By
the fundamental theorem of algebra \ref{realfund}, there exist
$k,\ell\in\N_0$, $c\in R^\times$, $a_1,\dots,a_k\in R$, $b_1,\dots,b_k\in R^\times$, $\al_1,\dots,\al_\ell\in\N$ and
pairwise different $d_1,\dots,d_\ell\in R$ such that
\begin{align*}
f&=c\left(\prod_{i=1}^k((X-a_i)^2+b_i^2)\right)\prod_{j=1}^\ell(X-d_j)^{\al_j}\\
&=c\left(\prod_{i=1}^k(X-(a_i+b_i\ii))\right)\left(\prod_{i=1}^k(X-(a_i-b_i\ii))\right)\prod_{j=1}^\ell(X-d_j)^{\al_j}.
\end{align*}
Suppose now $f\ge0$ on $R$. Then we have $0\le\sgn(f(x))=(\sgn c)\prod_{j=1}^\ell(\sgn(x-d_j))^{\al_j}$ for all $x\in R$.
From this, we deduce easily $\al_j\in2\N$ and $c\in R^2$. Setting
\[g:=\sqrt c\left(\prod_{i=1}^k(X-(a_i+b_i\ii))\right)\prod_{j=1}^\ell(X-d_j)^{\frac{\al_j}2}\in C[X],\]
we have now $f=g^*g$. Writing $g=p+\ii q$ with $p,q\in R[X]$, this amounts to
$f=(p-\ii q)(p+\ii q)=p^2+q^2$.
\end{proof}

\begin{thm}[Cassels]\label{cassels}
Let $(K,\le)$ be an ordered field. Suppose $\ell\in\N_0$, $f_1,\dots,f_\ell\in K[X]$, $g_1,\dots,g_\ell\in K[X]\setminus\{0\}$ and
$a_1,\dots,a_\ell\in K_{\ge0}$ with $\sum_{i=1}^\ell a_i\left(\frac{f_i}{g_i}\right)^2\in K[X]$. Then there are $p_1,\dots,p_\ell\in K[X]$ such
that \[\sum_{i=1}^\ell a_i\left(\frac{f_i}{g_i}\right)^2=\sum_{i=1}^\ell a_ip_i^2.\]
\end{thm}

\begin{proof}
WLOG $a_i>0$ for all $i\in\{1,\dots,\ell\}$ and $g_1=\ldots=g_\ell$. It suffices to show:
Let $h\in K[X]$ a polynomial for which there exists some $g\in K[X]$ of degree $\ge1$ and
$f_1,\dots,f_\ell\in K[X]$ satisfying $hg^2=\sum_{i=1}^\ell a_if_i^2$. Then there is some $G\in K[X]\setminus\{0\}$
with a degree that is smaller than that of $g$ and $F_1,\dots,F_\ell\in K[X]$ satisfying
$hG^2=\sum_{i=1}^\ell a_iF_i^2$.
We prove this: Write $f_i=q_ig+r_i$ with $q_i,r_i\in K[X]$ and $\deg r_i<\deg g$ 
for all $i\in\{1,\dots,\ell\}$. If $r_i=0$ for all $i\in\{1,\dots,\ell\}$, then we set $G:=1$ and $F_i:=q_i$ for all
$i\in\{1,\dots,\ell\}$ and have
\[hG^2=h=\frac1{g^2}(hg^2)=\frac1{g^2}\sum_{i=1}^\ell a_if_i^2=
\sum_{i=1}^\ell a_i\left(\frac{f_i}g\right)^2=\sum_{i=1}^\ell a_iq_i^2=\sum_{i=1}^\ell a_iF_i^2.\]
In the sequel, we suppose that the set $I:=\{i\in\{1,\dots,\ell\}\mid r_i\ne0\}$ is nonempty.
Now we set $s:=\sum_{i=1}^\ell a_iq_i^2-h$, $t:=\sum_{i=1}^\ell a_if_iq_i-gh$,
$F_i:=sf_i-2tq_i$ for $i\in\{1,\dots,\ell\}$ and $G:=sg-2t$. Then we obtain
\begin{align*}
hG^2&=s^2hg^2-4stgh+4t^2h\\
&=s^2\sum_{i=1}^\ell a_if_i^2-4st(t+gh)+4t^2(s+h)\\
&=s^2\sum_{i=1}^\ell a_if_i^2-4st\sum_{i=1}^\ell a_if_iq_i+4t^2\sum_{i=1}^\ell a_iq_i^2\\
&=\sum_{i=1}^\ell a_i(sf_i-2tq_i)^2=\sum_{i=1}^\ell a_iF_i^2.
\end{align*}
It remains to show that $G\ne0$ and $\deg G<\deg g$. To this end, we calculate
\begin{align*}
G&=g\sum_{i=1}^\ell a_iq_i^2-gh-2\sum_{i=1}^\ell a_if_iq_i+2gh\\
&=\frac1g\left(g^2\sum_{i=1}^\ell a_iq_i^2+g^2h-2g\sum_{i=1}^\ell a_if_iq_i\right)\\
&=\frac1g\left(g^2\sum_{i=1}^\ell a_iq_i^2+\sum_{i=1}^\ell a_if_i^2-2g\sum_{i=1}^\ell a_if_iq_i\right)\\
&=\frac1g\sum_{i=1}^\ell a_i(g^2q_i^2-2(gq_i)f_i+f_i^2)\\
&=\frac1g\sum_{i=1}^\ell a_i(gq_i-f_i)^2=\frac1g\sum_{i=1}^\ell a_ir_i^2=\frac1g\sum_{i\in I}a_ir_i^2.
\end{align*}
If we had $G=0$, then this would mean $\sum_{i\in I}a_ir_i^2=0$. Since the leading coefficient of $a_ir_i^2$
is positive for all $i\in I\ne\emptyset$, this is impossible. Hence $G\ne0$. Because of
$\deg r_i<\deg g$ for all $i\in I$, we have $\deg G<2\deg g-\deg g=\deg g$.
\end{proof}

\section{Homogenization and dehomogenization}

\begin{df}\label{introhom}
Let $A$ be commutative ring with $0\ne1$.
\begin{enumerate}[(a)]
\item If $k\in\N_0$ and $f\in A[X_1,\dots,X_n]$, then the sum of all terms (i.e.,
monomials with their coefficients)
of degree $k$ of $f$ is called the \emph{$k$-th homogeneous part} of $f$. This is a $k$-form
[$\to$ \ref{longremi}(a)].
\item If $f\in A[X_1,\dots,X_n]\setminus\{0\}$ and $d:=\deg f$, then the $d$-th homogeneous part of $f$
is called the \emph{leading form} $\lf(f)$ of $f$. We set $\lf(0):=0$.
\item If $f\in A[X_1,\dots,X_n]$, $d:=\deg f\in\N_0$ and $f=\sum_{k=0}^df_k$ with a $k$-form $f_k$ for
all $k\in\{0,\dots,d\}$, then the \emph{homogenization} $f^*\in A[X_0,\dots,X_n]$ of $f$ (with respect to
$X_0$) is
given by \[f^*:=\sum_{k=0}^dX_0^{d-k}f_k\]
which equals $X_0^df\left(\frac{X_1}{X_0},\dots,\frac{X_n}{X_0}\right)$ in case $A$ is a field
(since then the field of rational functions $A(X_0,X_1,\ldots,X_n)$ exists).
We set $0^*:=0$.
\item For homogeneous $f\in A[X_0,\dots,X_n]$, we call $\widetilde f:=f(1,X_1,\dots,X_n)$
the \emph{dehomogenization} of $f$ (with respect to $X_0$).
\end{enumerate}
\end{df}

\begin{rem}\label{homdehom}
Let $A$ be a commutative ring with $0\ne1$.
\begin{enumerate}[(a)]
\item $\lf(f)=f^*(0,X_1,\dots,X_n)$ for all $f\in A[X_1,\dots,X_n]$.
\item For $f,g\in A[X_1,\dots,X_n]$, we have \[(f+g)^*=f^*+g^*\] in case $\deg f=\deg g=\deg(f+g)$ and \[(fg)^*=f^*g^*\]
if $A$ is an integral domain.
\item $A[X_0,\dots,X_n]\to A[X_1,\dots,X_n],\ f\mapsto\widetilde f$ is a ring homomorphism.
\item For all $f,g\in A[X_1,\dots,X_n]$, we have
\[\lf(f+g)=\lf(f)+\lf(g)\] in case $\deg f=\deg g=\deg(f+g)$ and \[\lf(fg)=\lf(f)\lf(g)\]
if $A$ is an integral domain.
\item For all $f\in A[X_1,\dots,X_n]$, we have $\widetilde{f^*\,}=f$.
\item If $f\in A[X_0,\dots,X_n]\setminus\{0\}$ is homogeneous and $m:=\max\{k\in\N_0\mid X_0^k\mid f\}$,
then $X_0^m{\widetilde f\;}^*=f$.
\end{enumerate}
\end{rem}

\begin{lem}\label{pol0}
Suppose $K$ is a field, $n,d\in\N_0$, $f\in K[X_1,\dots,X_n]_d$ and let $I_1,\dots,I_n\subseteq K$
be sets of cardinality at least $d+1$ each
such that $f(x)=0$ for all $x\in I_1\times\ldots\times I_n$. Then $f=0$.
\end{lem}

\begin{proof}
Induction by $n$.

\smallskip
\underline{$n=0$}\quad\checkmark

\smallskip
\underline{$n-1\to n\quad(n\in\N)$}\quad Write $f=\sum_{k=0}^df_kX_n^k$ with $f_k\in K[X_1,\dots,X_{n-1}]_d$.
For all \[(x_1,\dots,x_{n-1})\in I_1\times\ldots\times I_{n-1},\] the polynomial
$f(x_1,\dots,x_{n-1},X_n)=\sum_{k=0}^df_k(x_1,\dots,x_{n-1})X_n^k\in K[X_n]_d$ is
a polynomial with at $d+1$ roots. Thus $f_k(x_1,\dots,x_{n-1})=0$ for all $k\in\{0,\dots,d\}$ and
$(x_1,\dots,x_{n-1})\in I_1\times\ldots\times I_{n-1}$. By induction hypothesis, $f_k=0$ for all
$k\in\{0,\dots,d\}$.
\end{proof}

\begin{rem}\label{soslongrem}
Let $K$ be a real field, $\ell,n\in\N_0$, $p_1,\dots,p_\ell\in K[X_1,\dots,X_n]$ and
\[f:=\sum_{i=1}^\ell p_i^2.\]
\begin{enumerate}[(a)]
\item If $f=0$, then $p_1=\ldots=p_\ell=0$. This follows from \ref{pol0} together with \ref{realchar}(c).
Instead of \ref{pol0}, one can alternatively employ the fact that $K(X_1,\dots,X_n)$ is real which is clear
by applying \ref{rtlfct} $n$ times.
\item If $f\ne0$, then $\deg f=2d$ with $d:=\max\{\deg(p_i)\mid i\in\{1,\dots,\ell\}\}$ since otherwise
$\sum_{i=1,\deg(p_i)=d}^\ell\lf(p_i)^2=0$, contradicting (a).
\item If $d\in\N_0$ and $f$ is a $2d$-form, then every $p_i$ is a $d$-form. This can be seen similarly to
(b) by considering the homogeneous parts of the $p_i$ of smallest (instead of largest) degree.
\item We have $f^*\in\sum K[X_0,\dots,X_n]^2$. More precisely, $f^*$ is a $2d$-form for some $d\in\N_0$
that is a sum of $\ell$ squares of $d$-forms since
\[f^*=X_0^{2d}f\left(\frac{X_1}{X_0},\dots,\frac{X_n}{X_0}\right)=\sum_{i=1}^\ell\left(X_0^dp_i\left(
\frac{X_1}{X_0},\dots,\frac{X_n}{X_0}\right)\right)^2\]
and $X_0^dp_i\left(\frac{X_1}{X_0},\dots,\frac{X_n}{X_0}\right)=X_0^{d-\deg p_i}p_i^*\in K[X_0,\dots,X_n]$
for all $i\in\{1,\dots,\ell\}$ with $p_i\ne0$ (note that $\deg p_i\le d$ by (b)).
\end{enumerate}
\end{rem}

\begin{pro}\label{lf2}
Let $(K,\le)$ be an ordered field and $f\in K[X_1,\dots,X_n]$ with $f\ge0$ on $K^n$. Then $f$ has
an even degree except if $f=0$, and we have $\lf(f)\ge0$ on $K^n$.
\end{pro}

\begin{proof}
WLOG $f\ne0$. Then $g:=\lf(f)\ne0$. Set $d:=\deg g$. For all $x\in K^n$, $f_x:=f(Tx)\in K[T]$ is a polynomial
in one variable with $f_x\ge0$ on $K$ whose leading coefficient is $g(x)$ in case that $g(x)\ne0$.
Choose $x_0\in K^n$ with $g(x_0)\ne0$ [$\to$ \ref{pol0}]. Then $f_{x_0}$ has degree $d$ and because
of $f_{x_0}\ge0$ on $K$, it follows that $d\in2\N_0$ by \ref{sgnbounds}(a).
Now let $x\in K^n$ be arbitrary such that
$g(x)\ne0$. Again by \ref{sgnbounds}(a), it follows from $f_x\ge0$ on $K$ that $g(x)\ge0$.
\end{proof}

\begin{pro}\label{psdpsdhom}
Let $(K,\le)$ be an ordered field and $f\in K[X_1,\dots,X_n]$.
\begin{enumerate}[\normalfont(a)]
\item $f\ge 0$\text{ on }$K^n\iff f^*\ge0\text{ on }K^{n+1}$
\item $f\in\sum K[X_1,\dots,X_n]^2\iff f^*\in\sum K[X_0,\dots,X_n]^2$
\end{enumerate}
\end{pro}

\begin{proof}
(a) ``$\Longleftarrow$ '' If $f^*$ is nonnegative on $K^{n+1}$, then also on $\{1\}\times K^n$.

\smallskip
``$\Longrightarrow$'' Suppose $f\ge0$ on $K^n$. WLOG $f\ne0$. By \ref{lf2}, we can write $\deg f=2d$
with $d\in\N_0$.
Due to $f^*\overset{\ref{introhom}(c)}=X_0^{2d}f\left(\frac{X_1}{X_0},\dots,\frac{X_n}{X_0}\right)$,
we deduce $f^*\ge0$ on $K^\times\times K^n$. It remains to show $f^*\ge0$ on $\{0\}\times K^n$
which is equivalent by \ref{homdehom}(a) to $\lf(f)\ge0$ on $K^n$. The latter holds by \ref{lf2}.

\smallskip
(b) ``$\Longrightarrow$'' has been shown in \ref{soslongrem}(d).

\smallskip
``$\Longleftarrow$ '' follows from \ref{homdehom}(c).
\end{proof}

\section{Nonnegative quadratic polynomials}

\begin{df}\label{psdpd}
Let $(K,\le)$ be an ordered field.
\begin{enumerate}[(a)]
\item If $f\in K[X_1,\dots,X_n]$ is homogeneous [$\to$ \ref{longremi}(a)], then $f$ is called\\
\alal{\emph{positive semidefinite (psd)}}{\emph{positive definite (pd)}} (over $K$) if
$f\malal{\ge0\text{ on }K^n}{>0\text{ on }K^n\setminus\{0\}}$.
\item If $M\in SK^{n\times n}$, then $M$ is called \alal{psd}{pd} (over K) if the quadratic form
represented by $M$  [$\to$ \ref{longremi}(d)] is \alal{psd}{pd}, i.e., $x^TMx\malal{\ge0\text{ for all }x\in K^n}{>0\text{ for all }x\in K^n\setminus\{0\}}$.
\end{enumerate}
\end{df}

\begin{pro}\label{sospsd2}
Let $K$ be a Euclidean field and $q\in K[X_1,\dots,X_n]$ a quadratic form. Then the
following are equivalent:
\begin{enumerate}[\normalfont(a)]
\item $q$ is psd \emph{[$\to$ \ref{psdpd}(a)]}
\item $q\in\sum K[X_1,\dots,X_n]^2$ \emph{[$\to$ \ref{divnot}]}
\item $q$ is a sum of $n$ squares of linear forms \emph{[$\to$ \ref{longremi}(a)]}.
\item $\sg q=\rk q$ \emph{[$\to$ \ref{longremi}(h)]}.
\end{enumerate}
\end{pro}

\begin{proof}
(d)$\implies$(c)$\implies$(b)$\implies$(a) is trivial. Now suppose that (d) does not hold. We show that
then (a) also fails. Write $q=\sum_{i=1}^s\ell_i^2-\sum_{j=1}^t\ell_{s+j}^2$ with $s,t\in\N_0$ and
linearly independent linear forms $\ell_1,\dots,\ell_s,\ell_{s+1},\dots,\ell_{s+t}\in K[X_1,\dots,X_n]$.
Since $s-t=\sg q\ne\rk q=s+t$, we have $t\ge1$. By linear algebra,
\[\ph\colon K^n\to K^{s+t},\ x\mapsto\begin{pmatrix}\ell_1(x)\\\vdots\\\ell_{s+t}(x)\end{pmatrix}\]
is surjective. Choose $x\in K^n$ with $\ph(x)=\begin{pmatrix}0\\\vdots\\0\\1\end{pmatrix}$. Then
$q(x)=-1<0$.
\end{proof}

\begin{pro}\label{psdeq}
Let $K$ be a Euclidean field and $M\in SK^{n\times n}$. Then the following are equivalent:
\begin{enumerate}[\normalfont(a)]
\item $M$ is psd \emph{[$\to$ \ref{psdpd}(b)]}.
\item $\exists s\in\N_0:\exists A\in K^{s\times n}:M=A^TA$
\item $\exists A\in K^{n\times n}:M=A^TA$
\item All eigenvalues of $M$ in the real closure $\overline{(K,K^2)}$ are nonnegative.
\item All coefficients of $\det(M+XI_n)\in K[X]$ are nonnegative.
\item If $M=(a_{ij})_{1\le i,j\le n}$, then for all $I\subseteq\{1,\dots,n\}$, we have
$\det((a_{ij})_{(i,j)\in I\times I})\ge0$.
\end{enumerate}
\end{pro}

\begin{proof} Using \ref{longremi}(e) and \ref{soslongrem}(c), one sees that (a), (b) and (c) are nothing else than the corresponding statements in \ref{sospsd2}.

\smallskip
\underline{(a)$\implies$(f)}\quad follows from applying (a)$\implies$(c) to the submatrices of $M$ in question.

\smallskip
\underline{(f)$\implies$(e)}\quad Each coefficients of $\det(M+XI_n)$ is a sum of certain determinants
appearing in (f).

\smallskip
\underline{(e)$\implies$(d)} is trivial.

\smallskip
\underline{(d)$\implies$(a)} follows easily from \ref{eucliddiag}.
\end{proof}

\begin{term}{}[$\to$ \ref{degnot}, \ref{longremi}(a)]\label{quintic}
Let $A$ be a commutative ring with $0\ne1$.
Polynomials from $A[X_1,\dots,X_n]_d$ [$\to$ \ref{degnot}] are called
\emph{constant} for $d=0$, \emph{linear} for $d=1$, \emph{quadratic} for $d=2$, \emph{cubic} for $d=3$,
\emph{quartic} for $d=4$, \emph{quintic} for $d=5$, \dots
\end{term}

\begin{pro}\label{son1s}
Let $K$ be a Euclidean field and $q\in K[X_1,\dots,X_n]_2$.
The following are equivalent:
\begin{enumerate}[\normalfont(a)]
\item $q\ge0$ on $K^n$
\item $q\in\sum K[X_1,\dots,X_n]^2$
\item $q$ is a sum of $n+1$ squares of linear polynomials.
\end{enumerate}
\end{pro}

\begin{proof}
(a)$\overset{\text{\ref{psdpsdhom}(a)}}\implies q^*\ge0\text{ on $K^{n+1}$}\overset{\ref{sospsd2}}\implies$(c)$\implies$(b)
$\implies$(a)
\end{proof}

\section{The Newton polytope}

\begin{dfpro}\label{dfconv}
Let $(K,\le)$ be an ordered field, $V$ a $K$-vector space and $A\subseteq V$. Then $A$ is called
\emph{convex} if $\forall x,y\in A:\forall\la\in[0,1]_K:\la x+(1-\la)y\in A$.
The smallest convex superset of $A$ is obviously
\[\conv A:=\left\{\sum_{i=1}^m\la_ix_i\mid m\in\N,\la_i\in K_{\ge0},x_i\in A,\sum_{i=1}^m\la_i=1\right\},\]
called the convex set \emph{generated} by $A$ or the \emph{convex hull} of $A$.
We call finitely generated convex sets, i.e., convex hulls of finite sets, \emph{polytopes}.
A polytope is thus of the form
\[\conv\{x_1,\dots,x_m\}=\left\{\sum_{i=1}^m\la_ix_i\mid\la_i\in K_{\ge0},\sum_{i=1}^m\la_i=1\right\}\]
for some $m\in\N_0$ and $x_1,\dots,x_m\in V$. If $A$ is a convex set, then a point $x\in A$ is called
an \emph{extreme point} of $A$ if there are no $y,z\in A$ such that $y\ne z$ and $x=\frac{y+z}2$.
Extreme points of polytopes are also called \emph{vertices} of the polytope.
\end{dfpro}

\begin{exo}\label{extremeexo}
Suppose $(K,\le)$ is an ordered field, $V$ a $K$-vector space, $A\subseteq V$ convex,
$x\in A$ and $\la\in(0,1)_K$.
Then the following are equivalent:
\begin{enumerate}[\normalfont(a)]
\item $x$ is an extreme point of $A$.
\item There are no $y,z\in A$ such that $y\ne z$ and $x=\la y+(1-\la)z$.
\end{enumerate}
\end{exo}

\begin{lem}\label{nonredisvertex}
Let $(K,\le)$ be an ordered field, $V$ a $K$-vector space, $m\in\N_0$,
$x_1,\dots,x_m\in V$, $P:=\conv\{x_1,\dots,x_m\}$ and suppose
$P\ne\conv(\{x_1,\dots,x_m\}\setminus\{x_i\})$ for all $i\in\{1,\dots,m\}$. Then $P$ is a polytope
and $x_1,\dots,x_m$ are its vertices.
\end{lem}

\begin{proof} To show:
\begin{enumerate}[(a)]
\item Every vertex of $P$ equals one of the $x_i$.
\item Every $x_i$ is a vertex of $P$.
\end{enumerate}
For (a), let $x$ be a vertex of $P$. Write $x=\sum_{i=1}^m\la_ix_i$ with $\la_i\in K_{\ge0}$ and
$\sum_{i=1}^m\la_i=1$. WLOG $\la_1\ne0$. Then $\la_1=1$ for otherwise
$\mu:=\sum_{i=2}^m\la_i=1-\la_1>0$ and
$x=\la_1x_1+\mu\Bigg(\underbrace{\sum_{i=2}^m\frac{\la_i}\mu x_i}_{\rlap{$\scriptstyle\in\conv\{x_2,\dots,x_m\}$}}\Bigg)$,
contradicting \ref{extremeexo}(b).

\noindent\smallskip
To prove (b), we let $y,z\in P$ with $x_1=\frac{y+z}2$. To show: $y=z$. Write $y=\sum_{i=1}^m\la_ix_i$
and $z=\sum_{i=1}^m\mu_ix_i$ with $\la_i,\mu_i\in K_{\ge0}$ and
$\sum_{i=1}^m\la_i=1=\sum_{i=1}^m\mu_i$. We show that $\la_1=1=\mu_1$. It is enough to show
$\frac{\la_1+\mu_1}2=1$. If we had $\frac{\la_1+\mu_1}2<1$, then it would follow from
$(1-\frac{\la_1+\mu_1}2)x_1=\sum_{i=2}^m\frac{\la_i+\mu_i}2x_i$ that $x_1\in\conv\{x_2,\dots,x_m\}$
and therefore $P=\conv\{x_1,\dots,x_m\}=\conv\{x_2,\dots,x_m\}\ \lightning$.
\end{proof}

\begin{cor}\label{polyisconvexhull}
Every polytope is the convex hull of its finitely many vertices.
\end{cor}

\begin{dfpro}\label{minkowskisum}
Suppose $(K,\le)$ is an ordered field, $V$ is a $K$-vector space and let $A$ and $B$ be subsets of
$V$. Then $A+B:=\{x+y\mid x\in A,y\in B\}$ is called the \emph{Minkowski sum} of $A$ and $B$. We
have $(\conv A)+(\conv B)=\conv(A+B)$. Let now $A$ and $B$ be convex. Then $A+B$ is also convex.
If $z$ is an extreme point of $A+B$, then there are uniquely determined $x\in A$ and $y\in B$ such that
$z=x+y$, and $x$ is an extreme point of $A$ and $y$ is one of $B$.
\end{dfpro}

\begin{proof}
``$\subseteq$'' Let $x_1,\dots,x_m\in A$, $y_1,\dots,y_n\in B$, $\la_1,\dots,\la_m\in K_{\ge0}$,
$\mu_1,\dots,\mu_n\in K_{\ge0}$ and $\sum_{i=1}^m\la_i=1=\sum_{j=1}^n\mu_j$. Then
$\sum_{i=1}^m\sum_{j=1}^n\la_i\mu_j=\left(\sum_{i=1}^m\la_i\right)(\sum_{j=1}^n\mu_j)=1\cdot1=1$
and
\[
\sum_{i=1}^m\la_ix_i+\sum_{j=1}^n\mu_jy_j=\left(\sum_{j=1}^n\mu_j\right)\sum_{i=1}^m\la_ix_i
+\left(\sum_{i=1}^m\la_i\right)\sum_{j=1}^n\mu_jy_j
=\sum_{i=1}^m\sum_{j=1}^n\la_i\mu_j(x_i+y_j).
\]

\smallskip
``$\supseteq$'' is trivial.

\smallskip
Let now $A$ and $B$ be convex. Then $A+B=(\conv A)+(\conv B)=\conv(A+B)$ is convex. Finally,
let $z$ be
an extreme point of $A+B$ and let $x\in A$ and $y\in B$ with $z=x+y$. Then $x$ is an extreme point of
$A$ since if we had $x=\frac{x_1+x_2}2$ with different $x_1,x_2\in A$, then it would follow
that $z=\frac{(x_1+y)+(x_2+y)}2$ and $x_1+y\ne x_2+y\ \lightning$. In the same way, $y$ is an extreme point
of $B$. Suppose now that $x'\in A$ and $y'\in B$ such that $z=x'+y'$. Then
$z=\frac{x+x'}2+\frac{y+y'}2$ and $\frac{x+x'}2$ is also an extreme point of $A$ which is possible only for
$x=x'$. Analogously, $y=y'$. 
\end{proof}

\begin{notation}\label{monomnotation}
Suppressing $n$ in the notation, we denote by $\x:=(X_1,\dots,X_n)$ a tuple of variables and set
$A[\x]:=A[X_1,\dots,X_n]$ for every commutative ring $A$ with $0\ne1$ in $A$.
For $\al\in\N_0^n$, we write $|\al|:=\al_1+\dots+\al_n$ and $\x^\al:=X_1^{\al_1}\dotsm X_n^{\al_n}$.
\end{notation}

\begin{df}
Let $K$ be a field and $f\in K[\x]$. Write $f=\sum_{\al\in\N_0^n}a_\al\x^\al$ with $a_\al\in K$. Then the finite set $\supp(f):=\{\al\in\N_0^n\mid a_\al\ne0\}$ is called
the \emph{support} of $f$ and its convex hull $N(f):=\conv(\supp(f))\subseteq\R^n$ the
\emph{Newton polytope} of $f$.
\end{df}

\begin{df} Let $K$ be a field, $f\in K[\x]$ and $a\in K$. We say that $a$ is a
\emph{vertex coefficient} of $f$ if there is a vertex $\al$ of $N(f)$ such that
$a\x^\al$ is a term of $f$.
\end{df}

\begin{rem}\label{vcnon0}
Since every vertex of the Newton polytope of a polynomial lies by
\ref{nonredisvertex} in the support of the polynomial, vertex coefficients are always $\ne0$.
\end{rem}

\begin{thm}\label{newtontimes}
Let $K$ be a field and $f,g\in K[\x]$. Then $N(fg)=N(f)+N(g)$ and every vertex coefficient of $fg$
is the product of a vertex coefficient of $f$ with a vertex coefficient of $g$.
\end{thm}

\begin{proof}``$\subseteq$'' $\supp(fg)\subseteq\supp(f)+\supp(g)\subseteq N(f)+N(g)$ and therefore
$N(fg)=\conv(\supp(fg))\subseteq N(f)+N(g)$ since $N(f)+N(g)$ is convex by \ref{minkowskisum}.

\smallskip
``$\supseteq$'' By \ref{minkowskisum}, $N(f)+N(g)$ is a polytope. By virtue of \ref{polyisconvexhull}, it suffices
to show that its vertices lie in $N(fg)$. Consider therefore a vertex $\ga$ of $N(f)+N(g)$. We even show that
$\ga\in\supp(fg)$. By \ref{minkowskisum}, there are uniquely determined $\al\in N(f)$ and $\be\in N(g)$ such that
$\ga=\al+\be$, and $\al$ is a vertex of $N(f)$ and $\be$ a vertex of $N(g)$. By \ref{vcnon0}, we have
$\al\in\supp(f)$ and $\be\in\supp(g)$. Because of unicity of $\al$ and $\be$, the coefficient of
$\x^\ga$ in $fg$ equals the product of the respective coefficients of
$\x^\al$ and $\x^\be$ in $f$ and $g$, respectively, and hence is in particular $\ne0$.
Thus $N(fg)=N(f)+N(g)$ is shown. Also the extra claim follows from the above.
\end{proof}

\begin{pro}\label{newtoncontainment}
Let $K$ be a field and $f,g\in K[\x]$. Then $N(f+g)\subseteq\conv(N(f)\cup N(g))$.
\end{pro}

\begin{proof}
$\supp(f+g)\subseteq\supp(f)\cup\supp(g)\subseteq N(f)\cup N(g)$ implies
\[N(f+g)=\conv(\supp(f+g))\subseteq\conv(N(f)\cup N(g)).\]
\end{proof}

\begin{thm}\label{novertexcancellation}
Let $(K,\le)$ be an ordered field and $f,g\in K[\x]$ such that all vertex coefficients of $f$ and $g$
have the same sign. Then $N(f+g)=\conv(N(f)\cup N(g))$ and all vertex coefficients of $f+g$ also have this
sign.
\end{thm}

\begin{proof}
``$\subseteq$'' is \ref{newtoncontainment}

\smallskip
``$\supseteq$'' We have that $\conv(N(f)\cup N(g))=\conv(\supp(f)\cup\supp(g))$ is a polytope.
Let $\al$ be one of its vertices. By \ref{polyisconvexhull}, it is enough to show that $\al\in N(f+g)$.
We even show that $\al\in\supp(f+g)$. By \ref{nonredisvertex}, $\al$ lies in at least one of the sets
$\supp(f)$ and $\supp(g)$. If $\al$ lies only in one of these two, then the claim is clear. If on the other hand
$\al$ lies in both, then $\al$ is a vertex of both $\conv(\supp(f))=N(f)$ and $\conv(\supp(g))=N(g)$
and the coefficients of $\x^\al$ in $f$ and in $g$ and hence also in $f+g$ have the same sign, from which
it follows again that $\al\in\supp(f+g)$. Thus $N(f+g)=\conv(N(f)\cup N(g))$ is proven. The extra
claim follows from what was shown.
\end{proof}

\begin{lem}\label{aa2a}
Let $(K,\le)$ be an ordered field, $V$ a $K$-vector space and $A$ a convex subset of $V$.
Then $A+A=2A:=\{2x\mid x\in A\}$.
\end{lem}

\begin{proof}
``$\supseteq$'' trivial

\smallskip
``$\subseteq$'' Let $x,y\in A$. Then $x+y=2\frac{x+y}2\in 2A$.
\end{proof}

\begin{thm}\label{vertexsquare}
Let $(K,\le)$ be an ordered field and $f\in K[\x]$. Then $N(f^2)=2N(f)$ and all vertex
coefficients of $f^2$ are squares of vertex coefficients of $f$ and therefore positive.
\end{thm}

\begin{proof}
$N(f^2)=2N(f)$ follows from \ref{newtontimes} and \ref{aa2a}.
Suppose $\ga$ is a vertex of $N(f^2)\overset{\ref{newtontimes}}=N(f)+N(f)$.
By \ref{minkowskisum}, there are uniquely determined $\al,\be\in N(f)$ with $\ga=\al+\be$.
Due to $\ga=\be+\al$, it follows that $\al=\be$. But then the coefficient of $\x^\ga$ in $f^2$
is just the coefficient belonging to $\x^\al$ in $f$ squared.
\end{proof}

\begin{thm}\label{sosvc}
Let $(K,\le)$ be an ordered field, $\ell\in\N_0$, $p_1,\dots,p_\ell\in K[\x]$ and $f:=\sum_{i=1}^\ell p_i^2$.
Then $N(f)=2\conv(N(p_1)\cup\ldots\cup N(p_\ell))$ and all vertex coefficients of $f$ are positive.
\end{thm}

\begin{proof}
For each $i\in\{1,\dots,\ell\}$, we have by \ref{vertexsquare} that $N(p_i^2)=2N(p_i)$ and that all
vertex coefficients of $p_i^2$ are positive. By \ref{novertexcancellation},
\begin{align*}
N(f)&=\conv(N(p_1^2)\cup\ldots\cup N(p_\ell^2))=\conv(2N(p_1)\cup\ldots\cup2N(p_\ell))\\
&=2\conv(N(p_1)\cup\ldots\cup N(p_\ell))
\end{align*}
and all vertex coefficients of $f$ are positive.
\end{proof}

\begin{ex}\label{motzkin}
For the \emph{Motzkin polynomial} $f:=X^4Y^2+X^2Y^4-3X^2Y^2+1\in\R[X,Y]$, we have
$f\ge0$ on $\R^2$ but $f\notin\sum\R[X,Y]^2$. At first we show $f\ge0$ on $\R^2$ in three different ways:
\begin{enumerate}[(1)]
\item From the inequality of arithmetic and geometric means known from analysis, it follows that
$\sqrt[3]{abc}\le\frac13(a+b+c)$ for all $a,b,c\in\R_{\ge0}$. Setting here $a:=x^4y^2$, $b:=x^2y^4$
and $c:=1$ for arbitrary $x,y\in\R$, we deduce $x^2y^2\le\frac13(x^4y^2+x^2y^4+1)$.
\item
\begin{align*}
(1+X^2)f&=X^4Y^2+X^2Y^4-3X^2Y^2+1+X^6Y^2+X^4Y^4-3X^4Y^2+X^2\\
&=1-2X^2Y^2+X^4Y^4+X^2-2X^2Y^2+X^2Y^4+X^2Y^2-2X^4Y^2+X^6Y^2\\
&=(1-X^2Y^2)^2+X^2(1-Y^2)^2+X^2Y^2(1-X^2 )^2\in\sum\R[X,Y]^2
\end{align*}
\item
\begin{align*}
f(X^3,Y^3)&=X^{12}Y^6+X^6Y^{12}-3X^6Y^6+1\\
&=X^4Y^2-X^8Y^4-X^6Y^6+\frac14X^{12}Y^6+\frac12X^{10}Y^8+\frac14X^8Y^{10}\\
&\quad+X^2Y^4-X^6Y^6-X^4Y^8+\frac14X^{10}Y^8+\frac12X^8Y^{10}+\frac14X^6Y^{12}\\
&\quad+1-X^4Y^2-X^2Y^4+\frac14X^8Y^4+\frac12X^6Y^6+\frac14X^4Y^8\\
&\quad+\frac34X^8Y^4-\frac32X^6Y^6+\frac34X^4Y^8\\
&\quad+\frac34X^{10}Y^8-\frac32X^8Y^{10}+\frac34X^6Y^{12}\\
&\quad+\frac34X^{12}Y^6-\frac32X^{10}Y^8+\frac34X^8Y^{10}\\
&=\left(X^2Y-\frac12X^4Y^5-\frac12X^6Y^3\right)^2\\
&\quad+\left(XY^2-\frac12X^3Y^6-\frac12X^5Y^4\right)^2\\
&\quad+\left(1-\frac12X^2Y^4-\frac12X^4Y^2\right)^2\\
&\quad+\frac34\left(X^2Y^4-X^4Y^2\right)^2\\
&\quad+\frac34(X^3Y^6-X^5Y^4)^2\\
&\quad+\frac34(X^4Y^5-X^6Y^3)^2
\end{align*}
Now we show $f\notin\sum\R[X,Y]^2$:
\begin{align*}
N(f)&=\conv(\supp(f))=\conv\{(4,2),(2,4),(2,2),(0,0)\}\\
&=\conv\{(4,2),(2,4),(0,0)\}.
\end{align*}
Assume $f=\sum_{i=1}^\ell p_i^2$ with $\ell\in\N_0$ and $p_1,\dots,p_\ell\in\sum\R[X,Y]$. Then
\[N(p_i)\subseteq\conv(N(p_1)\cup\ldots\cup N(p_\ell))=\frac12N(f)=\conv\{(2,1),(1,2),(0,0)\}\]
by \ref{sosvc} and hence
$\supp(p_i)\subseteq\N_0^2\cap N(p_i)\subseteq\N_0^2\cap\conv\{(2,1),(1,2),(0,0)\}=
\{(0,0),(1,1),(2,1),(1,2)\}$ for all $i\in\{1,\dots,\ell\}$.
The coefficient of $X^2Y^2$ in $p_i^2$ is therefore the coefficient of $XY$ in $p_i$ squared and therefore
nonnegative. Then the coefficient of $X^2Y^2$ in $f$ is also nonnegative $\lightning$.
This shows $f\notin\sum\R[X,Y]^2$.  Thus one can neither generalize \ref{so2s}(a)$\implies$(c)
to polynomials in several variables nor \ref{son1s}(a)$\implies$(b) to polynomials of arbitrary degree.
Note also that exactly the same proof shows even $f+c\notin\sum\R[X,Y]^2$ for all $c\in\R$. By 
\ref{psdpsdhom}, the \emph{Motzkin form} $f^*:=X^4Y^2+X^2Y^4-3X^2Y^2Z^2+Z^6$ is psd
[$\to$ \ref{psdpd}] but is likewise no sum of squares of polynomials. Again by \ref{psdpsdhom},
the dehomogenizations $f^*(1,Y,Z)=Y^2+Y^4-3Y^2Z^2+Z^6$ and
$f^*(X,1,Z)=X^4+X^2-3X^2Z^2+Z^6$ are also polynomials that are $\ge0$ on $\R^2$ but that are no
sums of squares of polynomials.
\end{enumerate}
\end{ex}

\section{Artin's solution to Hilbert's 17th problem}

\begin{lem}\label{pqpos}
Let $R$ be a real closed field and 
$f,p,q\in R[\x]$.
Suppose $q\ne0$, $f=\frac pq$, $p\ge0$ on $R^n$ and $q\ge0$ on $R^n$.
Then $f\ge0$ on $R^n$.
\end{lem}

\begin{proof}
Using the Tarski principle \ref{tprinciple}, one can reduce to the case $R=\R$. 
But then the subset $\{x\in\R^n\mid f(x)<0\}$ of $\{x\in\R^n\mid q(x)=0\}$ is open in $\R^n$
and therefore empty since otherwise $q=0$ would follow from \ref{pol0}.
\end{proof}

\noindent
In the year 1900, Hilbert presented his famous list of 23 seminal problems
at the International Congress of Mathematicians in Paris.
In 1927, Artin gave a positive solution to the 17th of these problems.
This corresponds to the case $K=\R$ in the following theorem.

\begin{thm}[Artin]\label{artin}
Suppose $R$ is a real closed field and $(K,\le)$ an ordered subfield of $R$. Let $f\in K[\x]$.
Then the following are equivalent:
\begin{enumerate}[\normalfont(a)]
\item $f\ge0$ on $R^n$
\item $f\in\sum K_{\ge0}K(\x)^2$
\end{enumerate}
\end{thm}

\begin{proof}
(b)$\implies$(a) follows from Lemma \ref{pqpos}. We show (a)$\implies$(b) by contraposition.
Suppose $f\notin\sum K_{\ge0}K(\x)^2$. To show: $\exists x\in R^n:f(x)<0$.
Since $\sum K_{\ge0}K(\x)^2$ is now a proper preorder of $K(\x)$ [$\to$ \ref{defpreorder}, \ref{preproper}],
there is by \ref{artin-schreier} an order $P$ of $K(\x)$ with $f\notin P$. Set $R':=\overline{(K(\x),P)}$.
Then there is an $x\in R'^n$ with $f(x)<0$ namely $x:=(X_1,\dots,X_n)$ since
$f(x)=f<0$ in $R'$. Due to $K_{\ge0}\subseteq P\subseteq R'^2$, $(K,\le)$ is an \emph{ordered} subfield
of $R'$. Since the $K$-semialgebraic set $\{x\in R'^n\mid f(x)<0\}$ is nonempty, its transfer
$\{x\in R^n\mid f(x)<0\}$ to $R$ [$\to$ \ref{transfer}] is also nonempty.
\end{proof}

\begin{cor}{}\emph{[$\to$ \ref{cassels}]}
Suppose $R$ is a real closed field and $(K,\le)$ an ordered subfield of $R$. Let $f\in K[X]$.
Then the following are equivalent:
\begin{enumerate}[\normalfont(a)]
\item $f\ge0$ on $R$
\item $f\in\sum K_{\ge0}K[X]^2$
\end{enumerate}
\end{cor}

\begin{proof}
(b)$\implies$(a) is trivial.

\smallskip
(a)$\implies$(b) follows from \ref{artin} and \ref{cassels}.
\end{proof}

\section{The Gram matrix method}

\begin{thm}\label{gram}
Let $K$ be a Euclidean field, $f\in K[\x]$ and
$\frac12N(f)\cap\N_0^n\subseteq\{\al_1,\dots,\al_m\}\subseteq\N_0^n$
\emph{(for instance set $\{\al_1,\dots,\al_m\}$ equal to $\frac12N(f)\cap\N_0^n$ or to
$\{\al\in\N_0^n\mid 2|\al|\le\deg f\}$)}. Set
$v:=\begin{pmatrix}\x^{\al_1}\\\vdots\\\x^{\al_m}\end{pmatrix}$. Then the following are equivalent:
\begin{enumerate}[\normalfont(a)]
\item $f\in\sum K[\x]^2$
\item There is a \emph{psd} matrix \emph{[$\to$ \ref{psdpd}(b)]} $G\in SK^{m\times m}$ (``Gram matrix'')
satisfying $f=v^TGv$.
\item $f$ is a sum of $m$ squares in $K[\x]$.
\end{enumerate}
\end{thm}

\begin{proof}
\underline{(a)$\implies$(b)}\quad Let $\ell\in\N_0$ and $p_1,\dots,p_\ell\in K[\x]$ with
$f=\sum_{i=1}^\ell p_i^2$.
By \ref{sosvc}, we have $\supp(p_i)\subseteq\frac12N(f)\cap\N_0^n\subseteq\{\al_1,\dots,\al_m\}$.
Hence there is an $A\in K^{\ell\times m}$ such that
\[Av=\begin{pmatrix}p_1\\\vdots\\p_\ell\end{pmatrix}.\]
It follows that $f=\begin{pmatrix}p_1&\dots&p_\ell\end{pmatrix}
\begin{pmatrix}p_1\\\vdots\\p_\ell\end{pmatrix}=(Av)^TAv=v^TA^TAv=v^TGv$ where
$G:=A^TA\in SK^{m\times m}$. By \ref{psdeq}, $G$ is psd.

\smallskip
\underline{(b)$\implies$(c)}\quad
Let $G\in SK^{m\times m}$ be psd with $f=v^TGv$. Choose according to \ref{psdeq} an $A\in K^{m\times m}$
satisfying $G=A^TA$. Write \[Av=\begin{pmatrix}p_1\\\vdots\\p_m\end{pmatrix}.\]
Then $p_1,\dots,p_m\in K[\x]$ and \[v^TGv=v^TA^TAv=(Av)^TAv=
\begin{pmatrix}p_1&\dots&p_m\end{pmatrix}\begin{pmatrix}p_1\\\vdots\\p_m\end{pmatrix}
=\sum_{i=1}^mp_i^2.\]

\smallskip
\underline{(c)$\implies$(a)} is trivial.
\end{proof}

\begin{ex}
Let $K$ be a Euclidean field and $f:=2X_1^4+5X_2^4-X_1^2X_2^2+2X_1^3X_2\in K[X_1,X_2]$.
Then $N(f)=\conv\{(4,0),(0,4)\}$ and therefore \[\frac12N(f)\cap\N_0^2=\{(2,0),(1,1),(0,2)\}.\]
Set $v:=\begin{pmatrix}X_1^2\\X_1X_2\\X_2^2\end{pmatrix}$. From
$\{G\in SK^{3\times 3}\mid f=v^TGv\}=\left\{\begin{pmatrix}2&1&a\\1&-2a-1&0\\a&0&5\end{pmatrix}\mid
a\in K\right\}$, we obtain
\[f\in\sum K[X_1,X_2]^2\iff\exists a\in K:\begin{pmatrix}2&1&a\\1&-2a-1&0\\a&0&5\end{pmatrix}
\text{ psd.}\]
For all $a\in K$, we have
\begin{align*}
&\det\begin{pmatrix}2+T&1&a\\1&T-2a-1&0\\a&0&5+T\end{pmatrix}
=(2+T)(T-2a-1)(5+T)-a^2(T-2a-1)-5-T\\
&=(T^2-2aT+T-4a-2)(5+T)-(1+a^2)T+2a^3+a^2-5\\
&=T^3-2aT^2+T^2-4aT-2T+5T^2-10aT+5T-20a-10-(1+a^2)T+2a^3+a^2-5\\
&=T^3+(6-2a)T^2+(2-14a-a^2)T-15-20a+a^2+2a^3
\end{align*}
and by \ref{psdeq}(e), we obtain
\[\begin{pmatrix}2&1&a\\1&-2a-1&0\\a&0&5\end{pmatrix}\text{ psd}\iff
\begin{array}[c]{rl}
2a^3+a^2-20a-15\ge0\\
\et\ -a^2-14a+2\ge0\\
\et\ -2a+6\ge0&.
\end{array}
\]
Set $a:=-3$. Then $2a^3+a^2-20a-15=-2\cdot27+9+60-15=-54+9+60-15=0$,
$-a^2-14a+2=-9+42+2=35\ge0$ and $-2a+6=12\ge0$.
For this reason $f\in\sum K[X_1,X_2]^2$. The quadratic form
\[q:=\begin{pmatrix}T_1&T_2&T_3\end{pmatrix}
\begin{pmatrix}2&1&a\\1&-2a-1&0\\a&0&5\end{pmatrix}
\begin{pmatrix}T_1\\T_2\\T_3\end{pmatrix}\in K[T_1,T_2,T_3]\]
obviously satisfies
\[q(X_1^2,X_1X_2,X_2^3)=v^T\begin{pmatrix}2&1&a\\1&-2a-1&0\\a&0&5\end{pmatrix}v=f.\]
Because of
\[\sg q\overset{\text{\ref{sospsd2}(d)}}=\rk q=\rk\begin{pmatrix}2&1&-3\\1&5&0\\-3&0&5\end{pmatrix}=2,\]
$q$ is a sum of $2$ squares of
linear forms in $K[T_1,T_2,T_3]$ and thus $f$ a sum of $2$ squares of polynomials. To compute this
representation explicitely, we employ the procedure from \ref{longremi}(f):
\begin{align*}
q&=2T_1^2+2T_1T_2-6T_1T_3+5T_2^2+5T_3^2\\
&=2\Big(\underbrace{T_1+\frac12T_2-\frac32T_3}_{\ell_1}\Big)^2-2\Big(\frac12T_2-\frac32T_3\Big)^2+5T_2^2+5T_3^2\\
&=2\ell_1^2+\frac92T_2^2+3T_2T_3+\frac12T_3^2\\
&=2\ell_1^2+\frac92\Big(\underbrace{T_2+\frac13T_3}_{\ell_2}\Big)^2=2\ell_1^2+\frac92\ell_2^2\\
&=\frac12(2T_1+T_2-3T_3)^2+\frac12(3T_2+T_3)^2.
\end{align*}
Hence $f=\frac12(2X_1^2+X_1X_2-3X_2^2)^2+\frac12(3X_1X_2+X_2^2)^2$.
\end{ex}

\chapter{Prime cones and real Stellensätze}

\section{The real spectrum of a commutative ring}

In this section, we let $A$, $B$ and $C$ always be commutative rings.

\begin{reminder}\label{specfunctor}
An ideal $\p$ of $A$ is called a prime ideal of $A$ if
\[1\notin\p\quad\text{ and }\quad\forall a,b\in A:(ab\in\p\implies(a\in\p\text{ or }b\in\p)).\] We call
$\spec A=\{\p\mid\p\text{ prime ideal of $A$}\}$
the \emph{spectrum} of $A$. If $I$ is an ideal of $A$, then
\[I\in\spec A\iff A/I\text{ is an integral domain.}\]
Because every integral domain extends to a field (e.g., to its quotient field) and every field to an
algebraically closed field (e.g., to its algebraic closure), $\spec A$ consists exactly of the kernels of
ring homomorphisms of $A$ in \alalal{integral domains}{fields}{algebraically closed fields}.
Every ring homomorphism $\ph\colon A\to B$ induces a map
\[
\spec\ph\colon\spec B\to\spec A, \q\mapsto\ph^{-1}(\q),
\]
for if $\q\in\spec B$, then $\p:=\ph^{-1}(\q)\in\spec A$ since $\ph$ induces an
embedding
$A/\p\hookrightarrow B/\q$ by the homomorphism theorem. If $\ph\colon A\to B$ and
$\ps\colon B\to C$ are ring homomorphisms, then
\[\spec(\ps\circ\ph)=(\spec\ph)\circ(\spec\ps).\]
\end{reminder}

\begin{notation} If $A$ is an integral domain, then
\[\qf A:=(A\setminus\{0\})^{-1}A=\left\{\frac ab\mid a,b\in A,b\ne0\right\}\]
denotes its quotient field.
\end{notation}

\begin{df}\label{introrealspectrum}
We call $\sper A:=\{(\p,\le)\mid\p\in\spec A,\ \text{$\le$ order of }\qf(A/\p)\}$ the \emph{real spectrum} of $A$.
\end{df}

\begin{rem}\label{sperfunctor}
Every ring homomorphism $\ph\colon A\to B$ induces a map
\[\sper\ph\colon\sper B\to\sper A,\ (\q,\le)\mapsto(\ph^{-1}(\q),\le'),\]
where $\le'$ denotes the order of $\qf(A/\p)$ with $\p:=\ph^{-1}(\q)$ which makes
the canonical embedding $\qf(A/\p)\hookrightarrow\qf(B/\q)$ into an embedding
$(\qf(A/\p),\le')\hookrightarrow(\qf(B/\q),\le)$ of ordered fields.
If $\ph\colon A\to B$ and
$\ps\colon B\to C$ are ring homomorphisms, then we have again
\[\sper(\ps\circ\ph)=(\sper\ph)\circ(\sper\ps).\]
\end{rem}

\begin{ex}\label{specsperex}
Since $\R[X]$ is a principal ideal domain, the fundamental theorem \ref{realfund} implies
\[\spec\R[X]=\{(0)\}\cup\{(X-a)\mid a\in\R\}\cup
\{(\underbrace{(X-a)^2+b^2}_{\rlap{$\scriptstyle=(X-(a+b\ii))(X-(a-b\ii))$}})\mid a,b\in\R,b\ne0\}\]
where $((X-a)^2+b^2)=((X-a')^2+b'^2)\iff(a=a'\et|b|=|b'|)$ for all $a,a',b,b'\in\R$. The spectrum of
$\R[X]$ therefore can be seen as consisting of
\begin{itemize}
\item one ``generic point'',
\item the real numbers, and
\item the unordered pairs of two distinct conjugated complex numbers.
\end{itemize}
Because of $\qf(\R[X]/(0))\cong\qf(\R[X])=\R(X)$,
$\qf(\R[X]/(X-a))=\R[X]/(X-a)\cong\R$ for all $a\in\R$ and
$\qf(\R[X]/((X-a)^2+b^2))\cong\R[X]/((X-a)^2+b^2)\cong\C$ for $a,b\in\R$ with $b\ne0$, we obtain
in the notation of \ref{ordersrx} (and with the identification $\R[X]/(0)=\R[X]$)
\begin{align*}
\sper\R[X]=\{((0),P_{-\infty}),((0),P_\infty)\}&\cup\{((0),P_{a-})\mid a\in\R\}\cup\{((0),P_{a+})\mid a\in\R\}\\
&\cup\{((X-a),(\R[X]/(X-a))^2)\mid a\in\R\}.
\end{align*}
The \emph{real} spectrum of $\R[X]$ thus corresponds to an accumulation consisting of
\begin{itemize}
\item the two points at infinity,
\item for each real number two points infinitely close, and
\item the real numbers.
\end{itemize}
\end{ex}

\begin{df}\label{supportmap}
We call $\supp\colon\sper A\to\spec A,\ (\p,\le)\mapsto\p$ the \emph{support map}.
\end{df}

\begin{df}{}[$\to$ \ref{unary-order}(a), \ref{specfunctor}]\label{dfprimecone}
A subset $P$ of $A$ is called a \emph{prime cone}
of $A$ if $P+P\subseteq P$, $PP\subseteq P$, $P\cup-P=A$, $-1\notin P$ and
$\forall a,b\in A:(ab\in P\implies(a\in P\text{ or }-b\in P))$.
\end{df}

\begin{pro}\label{primeconeispreorder}
Every prime cone of $A$ is a proper preorder of $A$ \emph{[$\to$ \ref{defpreorder}]}.
\end{pro}

\begin{proof}
Suppose $P$ is a prime cone of $A$ and $a\in A$. To show: $a^2\in P$. Due to $a\in A=P\cup-P$, we have
$a\in P$ or $-a\in P$. In the first case we get $a^2=aa\in PP\subseteq P$ and in the second
$a^2=(-a)^2=(-a)(-a)\in PP\subseteq P$.
\end{proof}

\begin{pro}\label{primeconechar}
Suppose $P\subseteq A$ satisfies $P+P\subseteq P$, $PP\subseteq P$ and $P\cup-P=A$.
Then the following are equivalent:
\begin{enumerate}[\normalfont(a)]
\item $P$ is a prime cone of $A$.
\item $-1\notin P$ and $\forall a,b\in A:(ab\in P\implies(a\in P\text{ or }-b\in P))$
\item $P\cap-P$ is a prime ideal of $A$
\end{enumerate}
\end{pro}

\begin{proof}
\underline{(a)$\iff$(b)} is Definition \ref{dfprimecone}.

\smallskip
\underline{(b)$\implies$(c)}\quad Suppose (b) holds and set $\p:=P\cap-P$. Then $\p$ is obviously a
subgroup of $A$ and we have $A\p=(P\cup-P)\p=P\p\cup-P\p=P(P\cap-P)\cup-P(P\cap-P)
\subseteq(PP\cap-PP)\cup(-PP\cap PP)\subseteq(P\cap-P)\cup(-P\cap P)=P\cap-P=\p$, i.e.,
$\p$ is an ideal of $A$ (if $\frac12\in A$ this follows alternatively from \ref{primeconeispreorder} and
\ref{supportideal}). From $-1\notin P$ we get $1\notin\p$. It remains to show
$\forall a,b\in A\colon(ab\in\p\implies(a\in\p\text{ or }b\in\p))$. To this end, let $a,b\in A$ with $a\notin\p$
and $b\notin\p$. To show: $ab\notin\p$. WLOG $a\notin P$ and $-b\notin P$
(otherwise replace $a$ by $-a$ and/or $-b$ by $b$, taking into account $-\p=\p$).
By hypothesis, we obtain then $ab\notin P$ and thus $ab\notin\p$.

\smallskip
\underline{(c)$\implies$(b)}\quad Suppose (c) holds. Due to $P\cup-P=A$, we have $1\in P$ or
$-1\in P$. If $-1\in P$, then again $1=(-1)(-1)\in PP\subseteq P$. Hence $1\in P$. If we had
$-1\in P$, then $1\in\p:=P\cap-P\in\spec A$ $\lightning$. Thus $-1\notin P$. Let now $a,b\in A$ such that
$a\notin P$ and $-b\notin P$. To show: $ab\notin P$. Because of $P\cup-P=A$, we have $a\in-P$ and
$b\in P$ from which $-ab=(-a)b\in PP\subseteq P$. If we had in addition $ab\in P$, then
$ab\in\p$ and thus $a\in\p\subseteq P$ or $b\in\p\subseteq-P$ $\lightning$.
Hence $ab\notin P$.
\end{proof}

\begin{rem}\label{primeconefield}
If $K$ is a field, then \ref{primeconechar} signifies because of $\spec K=\{(0)\}$ just that the
prime cones of $K$ are exactly the orders of $K$ [$\to$ \ref{unaryrem}].
\end{rem}

\begin{lem}\label{primeconeinfield}
Let $P$ be a prime cone of $A$ and $\p:=P\cap-P$ [$\to$ \ref{primeconechar}(c)]. Then
\[P_\p:=\left\{\frac{\cc a\p}{\cc s\p}\mid a\in A,s\in A\setminus\p,as\in P\right\}\] is an order (i.e., a prime cone
[$\to$ \ref{primeconefield}]) of $\qf(A/\p)$.
\end{lem}

\begin{proof} To show [$\to$ \ref{unaryrem}(a)]:
\begin{enumerate}[(a)]
\item $P_\p+P_\p\subseteq P_\p$,
\item $P_\p P_\p\subseteq P_\p$,
\item $P_\p\cup-P_\p=\qf(A/\p)$, and
\item $P_\p\cap-P_\p=(0)$.
\end{enumerate}

(a) Suppose that $a,b\in A$ and $s,t\in A\setminus\p$ with $as,bt\in P$ define arbitrary elements
$\frac{\overline a}{\overline s},\frac{\overline b}{\overline t}\in P_\p$.
Then
\[\frac{\cc a\p}{\cc s\p}+\frac{\cc b\p}{\cc t\p}=\frac{\cc{at}\p}{\cc{st}\p}+\frac{\cc{bs}\p}{\cc{st}\p}=
\frac{\cc{at+bs}\p}{\cc{st}\p}\in P_\p,\]
since $at+bs\in A$, $st\in A\setminus\p$ and $(at+bs)st=ast^2+bts^2\in PA^2+PA^2\subseteq PP+PP
\subseteq P+P\subseteq P$.

\smallskip
(b) Let again $a,b\in A$ and $s,t\in A\setminus\p$ satisfy $as,bt\in P$. Then
\[\frac{\cc a\p}{\cc s\p}\frac{\cc b\p}{\cc t\p}=\frac{\cc{ab}\p}{\cc{st}\p}\in P_\p\]
since $ab\in A$, $st\in A\setminus\p$ and $abst=(as)(bt)\in PP\subseteq P$.

\smallskip
(c) Let $a\in A$ and $s\in A\setminus\p$ define an arbitrary element
$\frac{\overline a}{\overline s}\in\qf(A/\p)$. Because of $P\cup-P=A$, we have
$as\in P$ or $-as\in P$, i.e.,
$-\frac{\overline a}{\overline s}=\frac{\overline{-a}}{\overline s}\in P_\p$ or
$\frac{\overline a}{\overline s}\in P_\p$.

\smallskip
(d) Suppose $a,b\in A$ and $s,t\in A\setminus\p$ with $as,bt\in P$ satisfy
\[\frac{\cc a\p}{\cc s\p}=-\frac{\cc b\p}{\cc t\p}.\]
Then $at+bs\in\p$ and therefore $ast^2+bts^2=st(at+bs)\in\p\subseteq-P$, i.e.,
$-ast^2-bts^2\in P$. From $ast^2=(as)t^2\in PA^2\subseteq P$ and $bts^2=(bt)s^2\in PA^2\subseteq P$ we
deduce $-ast^2,-bts^2\in P$. Consequently, $ast^2,bts^2\in\p$ and thus $a,b\in\p$. We obtain
\[\frac{\cc a\p}{\cc s\p}=0=\frac{\cc b\p}{\cc t\p}\] as desired.
\end{proof}

\begin{lem}{}[$\to$ \ref{unary-order}]\label{makeintoprimecone}
Let $(\p,\le)\in\sper A$. Then $\{a\in A\mid\cc a\p\ge0\}$ is a prime cone
of $A$.
\end{lem}

\begin{proof}
Set $P:=\{a\in A\mid\cc a\p\ge0\}$. Then $P+P\subseteq P$, $PP\subseteq P$, $P\cup-P=A$ and
$P\cap-P=\p\in\spec A$. Now $P$ is a prime cone of $A$ by \ref{primeconechar}(c).
\end{proof}

\begin{pro}{}\emph{[$\to$ \ref{unary-order}(c)]}\label{sperprimecones}
The correspondence
\begin{align*}
(\p,\le)&\mapsto\{a\in A\mid\cc a\p\ge0\}\\
(P\cap-P,P_{P\cap-P})&\mapsfrom P
\end{align*}
defines a bijection between $\sper A$ and the set of all prime cones of $A$.
\end{pro}

\begin{proof}
The well-definedness of both maps follows from Lemmata \ref{primeconeinfield} and \ref{makeintoprimecone}.
Now first let $(\p,\le)\in\sper A$ and $P:=\{a\in A\mid\cc a\p\ge0\}$. We show
$(\p,\le)=(P\cap-P,P_{P\cap-P})$. It is clear that $\p=P\cap-P$. Finally,
\begin{align*}
P_{P\cap-P}=P_\p
&=\left\{\frac{\cc a\p}{\cc s\p}\mid a\in A,s\in A\setminus\p,as\in P\right\}\\
&=\left\{\frac{\cc a\p}{\cc s\p}\mid a\in A,s\in A\setminus\p,\cc{as}\p\ge0\right\}\\
&=\left\{\frac{\cc a\p}{\cc s\p}\mid a\in A,s\in A\setminus\p,\frac{\cc a\p}{\cc s\p}\ge0\right\}=
\{x\in\qf(A/\p)\mid x\ge0\}.
\end{align*}
Conversely, suppose that $P$ is a prime cone of $A$ and $\p:=P\cap-P$. We show
\[P=\{a\in A\mid\cc a\p\in P_\p\}.\] Here ``$\subseteq$'' is trivial. To show ``$\supseteq$'', let $a\in A$
such that $\cc a\p\in P_\p$. Then there are $b\in A$ and $s\in A\setminus\p$  such that $bs\in P$ and
$\overline a=\frac{\overline b}{\overline s}$. It follows that $\cc{as^2}\p=\cc{bs}\p$ and thus
$as^2\in bs+\p\subseteq P+\p\subseteq P+P\subseteq P$. Since $P$ is a prime cone, we deduce $a\in P$
or $-s^2\in P$. If we had $-s^2\in P$, then $s^2\in P\cap-P=\p$ (since $s^2\in A^2\subseteq P$) and therefore
$s\in\p\ \lightning$.
\end{proof}

\begin{rem}{}[$\to$ \ref{unaryrem}]\label{newlang}
As a result of \ref{sperprimecones}, we can see elements of the real spectrum as prime cones. We 
reformulate some of the above in this new language:
\begin{enumerate}[(a)]
\item Remark \ref{sperfunctor}: Let $\ph\colon A\to B$ be a ring homomorphism. Then $\ph$ induces the map
$\sper\ph\colon\sper B\to\sper A,\ Q\mapsto\ph^{-1}(Q)$. Suppose namely that $Q\in\sper B$,
$\q:=Q\cap-Q$, $P:=\ph^{-1}(Q)$ and $\p:=P\cap-P$. Then $\ph^{-1}(\q)=\ph^{-1}(Q)\cap-\ph^{-1}(Q)=
P\cap-P=\p$ and the embedding $\qf(A/\p)\hookrightarrow\qf(B/\q)$ induced by $\ph$ is an embedding of
ordered fields $(\qf(A/\p),P_\p)\hookrightarrow(\qf(B/\q),Q_\q)$ because for $a\in A$ and $s\in A\setminus\p$
with $as\in P$ we have $\ph(a)\in B$, $\ph(s)\in B\setminus\q$, $\ph(a)\ph(s)=\ph(as)\in\ph(P)\subseteq Q$.
\item Definition \ref{supportmap}: The support map is $\supp\colon\sper A\to\spec A,\ P\mapsto P\cap-P$
[$\to$ \ref{sperprimecones}]. In particular, the Definitions \ref{supportmap} and \ref{supportideal} are
compatible.
\end{enumerate}
\end{rem}

\begin{df}\label{realrep}
For every $(\p,\le)\in\sper A$, we call the real closed field \[R_{(\p,\le)}:=\overline{(\qf(A/\p),\le)}\]
the \emph{representation field} of $(\p,\le)$ and the ring homomorphism
\[\rh_{(\p,\le)}\colon A\to R_{(\p,\le)},\ a\mapsto\cc a\p\]
the \emph{representation} of $(\p,\le)$.
\end{df}

\begin{pro}\label{kernelrep}
Let $P\in\sper A$. Then $P=\rh_P^{-1}(R_P^2)$ and $\supp P=\ker\rh_P$.
\end{pro}

\begin{proof}
$\rh_P^{-1}(R_P^2)=\{a\in A\mid\rh_P(a)\ge0\text{ in }R_P\}=\{a\in A\mid\cc a{\supp P}\in P_{\supp P}\}
\overset{\ref{sperprimecones}}=P$ and therefore \[\supp P=P\cap-P=\rh_P^{-1}(R_P^2)\cap-\rh_P^{-1}(R_P^2)
=\rh_P^{-1}(R_P^2\cap-R_P^2)=\rh_P^{-1}(\{0\})=\ker\rh_P.\]
\end{proof}

\begin{pro}{}\emph{[$\to$ \ref{specfunctor}]}\label{sperviahom}
Let $P$ be a set. Then the following are equivalent:
\begin{enumerate}[\normalfont(a)]
\item $P\in\sper A$
\item There is an ordered field $(K,\le)$ and a ring homomorphism $\ph\colon A\to K$ such that
$P=\ph^{-1}(K_{\ge0})$.
\item There exists a real closed field $R$ and a ring homomorphism $\ph\colon A\to R$ such that
$P=\ph^{-1}(R^2)$.
\end{enumerate}
\end{pro}

\begin{proof}
(a)$\overset{\ref{kernelrep}}\implies$(c)$\overset{\text{trivial}}\implies$(b)$\overset{\text{\ref{newlang}(a)}}\implies$(a)
\end{proof}

\section{Preorders and maximal prime cones}

Throughout this section, let $A$ be a commutative ring.

\begin{pro}\label{preorderprimecone}
Let $T$ be a proper preorder of $A$ \emph{[$\to$ \ref{defpreorder}]}.
Then the following are equivalent:
\begin{enumerate}[\normalfont(a)]
\item $T$ is a prime cone of $A$.
\item $\forall a,b\in A:(ab\in T\implies(a\in T\text{ or }-b\in T))$
\end{enumerate}
\end{pro}

\begin{proof}
\underline{(a)$\implies$(b)} is trivial by Definition \ref{dfprimecone}.

\smallskip\underline{(b)$\implies$(a)}\quad Suppose (b) holds. By Definition \ref{dfprimecone},
it suffices to show $T\cup-T=A$. But for all $a\in A$ it follows from (b) that
$a\in T$ or $-a\in T$ because of $aa=a^2\in T$.
\end{proof}

\begin{thm}{}\emph{[$\to$ \ref{orderpreorder}]}\label{maxpreorder}
Suppose $T$ is a maximal proper preorder
\emph{[$\to$ \ref{defpreorder}]} of $A$. Then $T$ is a prime cone of $A$.
\end{thm}

\begin{proof} We show \ref{preorderprimecone}(b). For this purpose let $a,b\in A$ satisfy
$a\notin T$ and $-b\notin T$. Then $T+aT$ and $T-bT$ are preorders of $A$ [$\to$ \ref{againpreorder}] that
properly contain $T$. Due to the maximality of $T$, therefore neither $T+aT$ nor $T-bT$ is proper as a
preorder, i.e., $-1\in T+aT$ and $-1\in T-bT$. Choose $s,t\in T$ such that $-as\in1+T$ and $bt\in1+T$.
Then $-abst\in(1+T)(1+T)\subseteq1+T$ and thus $-1\in T+abst\subseteq T+abT$. Since $T$ is proper,
we conclude that $ab\notin T$ as desired.
\end{proof}

\begin{cor}\label{inmaxprimecone}
Every proper preorder of $A$ is contained in a maximal prime cone of $A$.
\end{cor}

\begin{proof} Use \ref{maxpreorder} and Zorn's lemma.
\end{proof}

\begin{pro}\label{primeconeinclusion}
Let $P,Q\in\sper A$ such that $P\subseteq Q$ and set $\q:=\supp Q$. Then $Q=P\cup\q$.
\end{pro}

\begin{proof}
``$\supseteq$'' is trivial.

``$\subseteq$'' Let $a\in Q\setminus P$. To show: $a\in\q$. From $-a\in P\subseteq Q$ we get
$a\in Q\cap-Q=\q$.
\end{proof}

\begin{proterm}\label{spear}
Let $P\in\sper A$. Then ``the spear'' \[\{Q\in\sper A\mid P\subseteq Q\}\] is a chain
in the partially ordered set $\sper A$ that possesses a largest element (``a spearhead'').
\end{proterm}

\begin{proof}
Let $Q_1,Q_2\in\sper A$ with $P\subseteq Q_1$ and $P\subseteq Q_2$. Suppose
$Q_1\not\subseteq Q_2$. To show: $Q_2\subseteq Q_1$. Choose $a\in Q_1\setminus Q_2$. Let $b\in Q_2$.
To show $b\in Q_1$. We have $a-b\notin Q_2$ (or else $a\in Q_2\ \lightning$) and thus $a-b\notin P$ because of
$P\subseteq Q_2$. Then $b-a\in P\subseteq Q_1$ and thus
$b\in Q_1$. The existence
of the ``spearhead'' follows now from \ref{inmaxprimecone}.
\end{proof}

\section{Quotients and localization}

Throughout this section, we let $A$ be a commutative ring.

\begin{pro} \alal{Preimages}{Images} of preorders \emph{[$\to$ \ref{defpreorder}]}
under \alal{homomorphisms}{epimorphisms} of 
commutative rings are again preorders.
\end{pro}

\begin{proof} Exercise.
\end{proof}

\begin{pro} Let $I$ be an ideal of $A$. The correspondence
\begin{align*}
T&\mapsto\cc TI:=\{\cc aI\mid a\in T\}\\
\{a\in A\mid\cc aI\in P\}&\mapsfrom P
\end{align*}
defines a bijection between the set of
\alal{preorders \emph{[$\to$ \ref{defpreorder}]}}
{prime cones \emph{[$\to$ \ref{dfprimecone}]}} $T$ of $A$ with
$I\subseteq T$ and the set of \alal{preorders}{prime cones} of $A/I$.
\end{pro}

\begin{proof}
Exercise.
\end{proof}

\begin{lem}\label{preorderlocalize}
Let $S\subseteq A$ be multiplicative and $T\subseteq A$ a preorder. 
Let \[\io\colon A\to S^{-1}A,\ a\mapsto\frac a1\]
denote the canonical homomorphism. Then the preorder
generated by $\io(T)$ in $S^{-1}A$ equals $S^{-2}T=\left\{\frac a{s^2}\mid a\in T,s\in S\right\}$.
This preorder is proper if and only if $T\cap -S^2=\emptyset$.
\end{lem}

\begin{proof}
Exercise.
\end{proof}

\begin{pro}\label{sperlocalize}
Let $S\subseteq A$ be multiplicative. The correspondence
\begin{align*}
P&\mapsto S^{-2}P\\
\left\{a\in A\mid\frac a1\in Q\right\}&\mapsfrom Q
\end{align*}
gives rise to a bijection between $\{P\in\sper A\mid(\supp P)\cap S=\emptyset\}$ and $\sper(S^{-1}A)$.
\end{pro}

\begin{proof}
Let $P\in\sper A$ with $(\supp P)\cap S=\emptyset$. By \ref{preorderlocalize}, $S^{-2}P$ is a proper
preorder of $S^{-1}A$ since $P\cap-S^2\subseteq(P\cap-A^2)\cap(-S)\subseteq(P\cap-P)\cap(-S)=
(\supp P)\cap-S=-((\supp P)\cap S)=-\emptyset=\emptyset$. To show that $S^{-2}P$ is a prime cone of
$S^{-1}A$, we verify the condition from \ref{preorderprimecone}(b) where we use that for any two
fractions in $S^{-1}A$, one can find a common denominator from $S^2$. Let $a,b\in A$ and $s\in S$ with
$\frac a{s^2}\cdot\frac b{s^2}\in S^{-2}P$. To show: $\frac a{s^2}\in S^{-2}P$ or $-\frac b{s^2}\in S^{-2}P$.
Choose $c\in P$ and $u\in S$ with $\frac{ab}{s^4}=\frac c{u^2}$. Then there is $v\in S$ such that
$abu^2v=cs^4v$ and therefore $(au^2)(bv^2)=abu^2v^2=cs^4v^2\in P$. Since $P$ is a prime cone,
it follows that $au^2\in P$ or $-bv^2\in P$. Hence $\frac a{s^2}=\frac{au^2}{(su)^2}\in S^{-2}P$ or
$-\frac b{s^2}=\frac{-bv^2}{(sv)^2}\in S^{-2}P$.

Conversely, let $Q\in\sper(S^{-1}A)$. For
$\io\colon A\to S^{-1}A,\ a\mapsto\frac a1$, we have [$\to$ \ref{newlang}(a)]
\[\left\{a\in A\mid\frac a1\in Q\right\}
=(\sper\io)(Q)\in\sper A.\]
If we had $s\in S$ with $\frac s1\in Q\cap-Q$, then
$1=\frac ss=\frac1s\cdot\frac s1\in S^{-1}A(\supp Q)\subseteq\supp Q\ \lightning$.

\smallskip
It remains to show that the maps are inverse to each other:
\begin{enumerate}[(a)]
\item If $P\in\sper A$ with $(\supp P)\cap S=\emptyset$, then $P=\left\{a\in A\mid\frac a1\in S^{-2}P\right\}$.
\item If $Q\in\sper(S^{-1}A)$, then $Q=\left\{\frac a{s^2}\mid a\in A,\frac a1\in Q,s\in S\right\}$.
\end{enumerate}

\smallskip
To show (a), let $P\in\sper A$ with $(\supp P)\cap S=\emptyset$.

``$\subseteq$'' is trivial.

``$\supseteq$'' Let $a\in A$ with $\frac a1\in S^{-2}P$. Choose $b\in P$ and $s\in S$ with
$\frac a1=\frac b{s^2}$.
Then there is $t\in S$ such that $as^2t=bt$ and thus $as^2t^2=bt^2\in P$. It follows that $a\in P$ or
$-s^2t^2\in P$. The latter would lead to $s^2t^2\in(\supp P)\cap S \lightning$. Hence $a\in P$.

\smallskip
To show (b), consider an arbitrary $Q\in\sper(S^{-1}A)$.

``$\supseteq$'' is trivial.

``$\subseteq$'' Let $b\in A$ and $s\in S$ with $\frac bs\in Q$. Then for $a:=sb\in A$, we have
$\frac bs=\frac{sb}{s^2}=\frac a{s^2}$ and $\frac a1=\frac{sb}1=\left(\frac s1\right)^2\frac bs\in Q$.
\end{proof}

\section{Abstract real Stellensätze}

\begin{df} Let $A$ be a commutative ring. We call the ring homomorphism
\[A\to\prod_{(\p,\le)\in\sper A}R_{(\p,\le)},\ a\mapsto\left(\widehat a\colon(\p,\le)\mapsto\cc a\p\right)\]
the \emph{real representation} of $A$. For $a\in A$, we say that $\widehat a$ is the \emph{function represented}
by $a$ on the real spectrum of $A$.
\end{df}

\begin{thm}[abstract real Stellensatz \cite{kri,ste,pre}]\label{abstractstellensatz}
Suppose $A$ is a commutative ring, $I\subseteq A$ an ideal, $S\subseteq A$ a multiplicative set and
$T\subseteq A$ a preorder. Then the following conditions are equivalent:
\begin{enumerate}[\normalfont(a)]
\item There does \emph{not} exist any $P\in\sper A$ satisfying
\begin{align*}
\forall a\in I&:\widehat a(P)=0,\\
\forall s\in S&:\widehat s(P)\ne0\quad\text{and}\\
\forall t\in T&:\widehat t(P)\ge0.
\end{align*}
\item There are $a\in I$, $s\in S$ and $t\in T$ such that $a+s^2+t=0$.
\end{enumerate}
\end{thm}

\begin{proof}
\underline{(b)$\implies$(a)} is trivial.

\smallskip
\underline{(a)$\implies$(b)}\quad Replacing $T$ by the preorder $T+I$, we can suppose WLOG $I=(0)$.
Suppose (b) does not hold. By \ref{preorderlocalize}, $S^{-2}T$ is then a proper preorder of $S^{-1}A$.
Consequently, $S^{-2}T$ is contained in a prime cone $Q$ of $S^{-1}A$ by \ref{inmaxprimecone}. Now
\ref{sperlocalize} yields $P:=\left\{a\in A\mid\frac a1\in Q\right\}\in\sper A$ and $(\supp P)\cap S=\emptyset$.
For all $s\in S$, we have $\widehat s(P)=\cc s{\supp P}\ne0$ in $R_P$
[$\to$ \ref{realrep}, \ref{sperprimecones}] since $s\notin\supp P$. For all $t\in T$, we have $\widehat t(P)\ge0$
because $t\in P$.
\end{proof}

\begin{termnot}\label{preorderedring}
\begin{enumerate}[(a)]
\item We call a pair $(A,T)$ consisting of a commutative ring $A$ and a preorder $T$ of $A$ a
\emph{preordered ring}.
\item If $(A,T)$ is a preordered ring, then we define its \emph{real spectrum}
\[\sper(A,T):=\{P\in\sper A\mid T\subseteq P\}.\]
\item{}[$\to$ \ref{intervals}(c)] If $A$ is a commutative ring, $a\in A$ and $S\subseteq\sper A$, then we write
\begin{align*}
\widehat a\ge0\text{ on $S$} &:\iff\forall P\in S:\widehat a(P)\ge0,\\
\widehat a>0\text{ on $S$} &:\iff\forall P\in S:\widehat a(P)>0,
\end{align*}
and so forth.
\end{enumerate}
\end{termnot}

\begin{cor}[abstract Positivstellensatz]\label{abstractpositivstellensatz}
Let $(A,T)$ be a preordered ring and $a\in A$. Then the following are equivalent:
\begin{enumerate}[\normalfont(a)]
\item $\widehat a>0$ on $\sper(A,T)$
\item $\exists t\in T:ta\in1+T$
\item $\exists t\in T:(1+t)a\in1+T$.
\end{enumerate}
\end{cor}

\begin{proof}
\underline{(b)$\implies$(c)}\quad If $t,t'\in T$ satisfy $ta=1+t'$, then
\[(1+t+t')a=ta+(1+t')a=1+t'+ta^2\in1+T.\]

\smallskip\underline{(c)$\implies$(a)} is trivial.

\smallskip\underline{(a)$\implies$(b)} follows by applying \ref{abstractstellensatz} on the ideal $(0)$, the
multiplicative set $\{1\}$ and the preorder $T-aT$.
\end{proof}

\begin{cor}[abstract Nichtnegativstellensatz]\label{abstractnichtnegativstellensatz}
Let $(A,T)$ be a preordered ring and $a\in A$. Then the following are equivalent:
\begin{enumerate}[\normalfont(a)]
\item $\widehat a\ge0$ on $\sper(A,T)$
\item $\exists t\in T:\exists k\in\N_0:ta\in a^{2k}+T$
\end{enumerate}
\end{cor}

\begin{proof}
\underline{(b)$\implies$(a)} is trivial.

\underline{(a)$\implies$(b)} follows by applying \ref{abstractstellensatz} on the ideal $(0)$, the
multiplicative set $\{1,a,a^2,\dots\}$ and the preorder $T-aT$.
\end{proof}

\begin{cor}[abstract real Nullstellensatz \cite{kri,du2,ris,efr}]\label{abstractrealnullstellensatz}
Let $A$ be a commutative ring, $I\subseteq A$ an ideal and $a\in A$. Then the following are equivalent:
\begin{enumerate}[\normalfont(a)]
\item $\widehat a=0$ on $\{P\in\sper A\mid I\subseteq\supp P\}$
\item $\exists k\in\N_0:\exists s\in\sum A^2:a^{2k}+s\in I$
\end{enumerate}
\end{cor}

\begin{proof}
\underline{(b)$\implies$(a)} is trivial.

\smallskip
\underline{(a)$\implies$(b)} follows by applying \ref{abstractstellensatz} on the ideal $I$, the
multiplicative set $\{1,a,a^2,\dots\}$ and the preorder $\sum A^2$.
\end{proof}

\section{The real radical ideal}

Throughout this section, we let $A$ be a commutative ring.

\begin{df}{}[$\to$ \ref{realchar}(c)] $A$ is called \emph{real} (or \emph{real reduced})
if \[\forall n\in\N:\forall a_1,\dots,a_n\in A:(a_1^2+\dots+a_n^2=0\implies a_1=0).\]
\end{df}

\begin{rem}
We have
\[A\ne\{0\}\text{ real}\implies-1\notin\sum A^2\overset{\ref{inmaxprimecone}}{\underset{\text{\ref{sqsm}(a)}}\iff}
\sper A\ne\emptyset.\]
Here ``$\Longrightarrow$'' cannot be replaced by ``$\iff$'' (in contrast to the case where
$A$ is a field [$\to$ \ref{realchar}]) as the example of $A=\R[X]/(X^2)$ shows.
\end{rem}

\begin{df}\label{dfrealideal}
An ideal $I\subseteq A$ is called \emph{real} (or \emph{real radical ideal}) if $A/I$ is real, i.e.,
$\forall n\in\N:\forall a_1,\dots,a_n\in A:(a_1^2+\ldots+a_n^2\in I\implies a_1\in I)$.
\end{df}

\begin{pro}\label{realidealchar}
Let $\p\in\spec A$. Then the following are equivalent:
\begin{enumerate}[\normalfont(a)]
\item $\p$ is real \emph{[$\to$ \ref{dfrealideal}]}
\item $\qf(A/\p)$ is real \emph{[$\to$ \ref{dfreal}]}
\item $\exists P\in\sper A:\p=\supp P$ \emph{[$\to$ \ref{newlang}(b)]}
\end{enumerate}
\end{pro}

\begin{proof} \underline{(a)$\implies$(b)}\quad Suppose (a) holds and let $n\in\N$, $a_1,\dots,a_n,s\in A/\p$
with $s\ne0$ such that $\left(\frac{a_1}s\right)^2+\ldots+\left(\frac{a_n}s\right)^2=0$. Then
$a_1^2+\ldots+a_n^2=0$. Since $A/\p$ is real, it follows that $a_1=0$ and therefore $\frac{a_1}s=0$.

\smallskip
\underline{(b)$\implies$(c)}\quad Suppose (b) holds. Then $\qf(A/\p)$ possesses an order $\le$.
According to Definition \ref{introrealspectrum}, we have $(\p,\le)\in\sper A$ and of course
$\p=\supp(\p,\le)$ by Definition \ref{supportmap}.

\smallskip
\underline{(c)$\implies$(a)}\quad Suppose $\p=\supp P$ for some $P\in\sper A$. Let $n\in\N$ and
$a_1,\dots,a_n\in A$ satisfy $a_1^2+\ldots+a_n^2\in\p$. Then
$\widehat a_1(P)^2+\ldots+\widehat a_n(P)^2=0$ and thus $\widehat a_1(P)=0$, i.e., $a_1\in\p$.
\end{proof}

\begin{df} The \emph{real radical} $\rrad I$ of an ideal $I\subseteq A$ is defined by
\[\rrad I:=\bigcap\left\{\p\in\rspec A\mid I\subseteq\p\right\}\]
where
$\rspec A:=\{\p\in\spec A\mid\p\text{ is real}\}$ and $\bigcap\emptyset:=A$.
\end{df}

\begin{rem}\label{intersectionreal}
Since every intersection of real ideals of $A$ is obviously again a real ideal of $A$, for every
ideal $I\subseteq A$, the set $\rrad I$ is a real ideal of $I$.
\end{rem}

\begin{thm}{}\emph{[$\to$ \ref{abstractrealnullstellensatz}]}\label{rradchar}
For every ideal $I$ of $A$,
\[\rrad I=\left\{a\in A\mid\exists k\in\N_0:\exists s\in\sum A^2:a^{2k}+s\in I\right\}.\]
\end{thm}

\begin{proof}
\ref{realidealchar} shows that this is just a reformulation of \ref{abstractrealnullstellensatz}.
\end{proof}

\begin{rem} Let $I\subseteq A$ be an ideal. Then by \ref{intersectionreal}, $\rrad I$ is the smallest real
ideal of $A$ containing $I$.
\end{rem}

\begin{df} We call $\rnil A:=\bigcap\rspec A=\rrad(0)$ the \emph{real nilradical} of $A$.
\end{df}

\begin{cor} We have
\[
\{a\in A\mid\widehat a=0\}=\rnil A=\{a\in A\mid\exists k\in\N:\exists s\in\sum A^2:a^{2k}+s=0\}.\]
\end{cor}

\section{Constructible sets}\label{sec:constructible}

In this section, we let $(A,T)$ always be a preordered ring [$\to$ \ref{preorderedring}(a)]. At the moment it
is a general one but after Proposition \ref{constructiblesetnormalform},
we will further specialize $(A,T)$.

\begin{df}{}[$\to$ \ref{introsemialg}]\label{introconstructible}
A Boolean combination [$\to$ \ref{booleanalgebra}(b)] of sets of the form
\[\{P\in\sper(A,T)\mid a\in P\}\qquad(a\in A)\]
 is called a \emph{constructible subset} of the real spectrum of $(A,T)$. We denote
the Boolean algebra of all constructible sets of $\sper(A,T)$ by $\mathcal C_{(A,T)}$. The analogous
definition remains in force for a commutative ring instead of a preordered ring $(A,T)$.
\end{df}

\begin{pro}{}\emph{[$\to$ \ref{sanf}]}\label{constructiblesetnormalform}
Every constructible subset of $\sper(A,T)$ is of the form
\[\bigcup_{i=1}^k
\left\{P\in\sper(A,T)\mid\widehat a_i(P)=0,\widehat b_{i1}(P)>0,\ldots,\widehat b_{im}(P)>0\right\}\]
for some $k,m\in\N_0$, $a_i,b_{ij}\in A$.
\end{pro}

\begin{proof} Completely analogous to \ref{sanf} using that
$a\in P\overset{\ref{sperprimecones}}\iff\widehat a(P)\ge0$ for all $a\in A$ and $P\in\sper A$.
\end{proof}

For the rest of this section, we fix an ordered field $(K,\le)$, denote by $R:=\overline{(K,\le)}$ its
real closure, we let $n\in\N_0$ and set $A:=K[\x]$ and $T:=\sum K_{\ge0}A^2$.
Then $(A,T)$ is a preordered ring and for all $P\in\sper(A,T)$ there is by \ref{unic} exactly
one homomorphism from $R$ to the representation field $R_P$ of $P$ extending $\rh_P|_K$
[$\to$ \ref{realrep}]. In virtue of this homomorphism, which is of course an embedding of ordered
fields, we interpret $R$ as an (ordered) subfield of $R_P$. In particular, we write $R=R_P$ if
it is an isomorphism.

\begin{pront}\label{rnassubsetofsper}
The correspondence
\begin{align*}
P&\mapsto x_P:=(\rh_P(X_1),\ldots,\rh_P(X_n))\\
\{f\in A\mid f(x)\ge0\}=:P_x&\mapsfrom x
\end{align*}
defines a bijection between $\{P\in\sper(A,T)\mid R_P=R\}$ and $R^n$.
\end{pront}

\begin{proof} We first show that both maps are well-defined.
For every $P\in\sper(A,T)$ with $R_P=R$, we have $x_P\in R^n$ under the identification
of $R_P$ and $R$. Conversely, let $x\in R^n$. Consider the ring homomorphism
\[\ph\colon A\to R,\ f\mapsto f(x).\]
Then $P_x=\ph^{-1}(R^2)=(\sper\ph)(R_{\ge0})\in\sper A$ [$\to$ \ref{sperfunctor}, \ref{newlang}(a)].
Obviously, $K_{\ge0}\subseteq P_x$ and therefore $P_x\in\sper(A,T)$. In order to show
$R_{P_x}=R$, we set $\p:=\supp P_x$ and consider the homomorphism of ordered fields
\[(\qf(A/\p),(P_x)_\p)\to(R,R^2),\ \frac{\cc a\p}{\cc s\p}\mapsto\frac{a(x)}{s(x)}\qquad(a\in A,s\in A\setminus\p)\]
induced by $\ph$ according to \ref{sperfunctor} taking into account \ref{newlang}. Since $R$ is real closed,
this homomorphism extends (uniquely) to a homomorphism of (ordered) fields
\[\ps\colon R_{P_x}=\overline{(\qf(A/\p),(P_x)_\p)}\to R.\]
We obviously have $\ps|_K=\id$ and therefore $\ps|_R$ is a $K$-endomorphism of the real closure $R$ of
$(K,\le)$ which can only be the identity by \ref{unic}. The injectivity of $\ps$ now implies
$R_{P_x}=R$ as desired. For later use we note that $\ps=\id_R$ implies
\begin{equation}
\tag{$*$}\cc f\p=\ps^{-1}(\ps(\cc f\p))=\ps^{-1}(f(x))=f(x)
\end{equation}
for all $f\in A$.

\smallskip
It remains to show that both maps are inverse to each other. This means:
\begin{enumerate}[(a)]
\item $P=P_{x_P}$ for all $P\in\sper(A,T)$ with $R_P=R$
\item $x=x_{P_x}$ for all $x\in R^n$
\end{enumerate}

To show (a), let $P\in\sper(A,T)$ such that $R_P=R$. Then
\begin{align*}
P_{x_P}&=\{f\in A\mid f(\rh_P(X_1),\dots,\rh_P(X_n))\ge0\text{ in $R$}\}\\
&=\{f\in A\mid\rh_P(f)\ge0\text{ in $R_P$}\}=
\{f\in A\mid\widehat f(P)\ge0\}=P.
\end{align*}

To show (b), we let $x\in R^n$. Then $\cc{X_i}{\supp P_x}\overset{(*)}=x_i\in R$ for all $i\in\{1,\dots,n\}$.
Consequently, $x_{P_x}=(\rh_{P_x}(X_1),\dots,\rh_{P_x}(X_n))=
(\cc{X_1}{\supp P_x},\dots,\cc{X_n}{\supp P_x})=(x_1,\dots,x_n)=x$.
\end{proof}

\begin{thmdef}{}\emph{[$\to$ \ref{introsn}, \ref{setification}]}\label{slimfatten}
Let $n\in\N_0$ and denote again
by $\mathcal S_{n,R}$ the Boolean algebra of all $K$-semialgebraic subsets of
$R^n$. Then
\[\slim\colon\mathcal C_{(A,T)}\to\mathcal S_{n,R},\ C\mapsto\{x\in R^n\mid P_x\in C\}\]
is an isomorphism of Boolean algebras. We call $\slim$ the \emph{despectrification} or \emph{slimming}
(in German: \emph{Entspeckung}) 
and \[\fatten:=\slim^{-1}\] the \emph{spectrification} or \emph{fattening}
(in German: \emph{Verspeckung}). For all $f\in A$, one has
\[\slim(\{P\in\sper(A,T)\mid f\in P\})=\{x\in R^n\mid f(x)\ge0\}.\]
\end{thmdef}

\begin{proof} It is obvious that $\slim$ is a homomorphism of Boolean algebras [$\to$ \ref{bahom}]
satisfying $\slim(\{P\in\sper(A,T)\mid f\in P\})=\{x\in R^n\mid f\in P_x\}=\{x\in R^n\mid f(x)\ge0\}$ for
all $f\in A$. Let $\mathcal R\supseteq\{R_P\mid P\in\sper(A,T)\}$ be a set of real closed fields
that are ordered extension fields of $(K,\le)$ [$\to$ \ref{rcfclass}(b)]. Let $\mathcal S_n$ again denote the
Boolean algebra of all $(K,\le)$-semialgebraic classes [$\to$ \ref{introsn}] and consider
\[\Ph\colon\mathcal S_n\to\mathcal C_{(A,T)},\ S\mapsto\{P\in\sper(A,T)\mid(R_P,(\rh_P(X_1),\dots,
\rh_P(X_n)))\in S\}.\]
It is obvious that $\Ph$ is a homomorphism of Boolean algebras satisfying
\begin{multline*}
\Ph(\{(R',x)\mid R'\in\mathcal R,x\in R'^n,f(x)\ge0\text{ in }R'\})\\
=\{P\in\sper(A,T)\mid f(\rh_P(X_1),\dots,\rh_P(X_n))\ge0\text{ in }R_P\}\\
=\{P\in\sper(A,T)\mid\rh_P(f)\ge0\text{ in }R_P\}\\
=\{P\in\sper(A,T)\mid\widehat f(P)\ge0\}=\{P\in\sper(A,T)\mid f\in P\}
\end{multline*}
for all $f\in A$. From this one sees, in the first place, that $\Ph$ is surjective and, secondly,
that $\slim\circ\,\Ph=\set_R$ [$\to$ \ref{introsn}]
which is an isomorphism of Boolean algebras by \ref{setification}.
Along with $\set_R$, $\Ph$ is also injective. We conclude that $\Ph$ is an isomorphism and
with it $\slim=(\slim\circ\,\Ph)\circ\;\Ph^{-1}$.
\end{proof}

\begin{ex}\label{sperrx}
In \ref{specsperex}, we have already described $\sper\R[X]$. Now we describe $\sper\R[X]$ as a set of
prime cones [$\to$ \ref{newlang}] while using \ref{ordersrx}: For $t\in\R$, we set
\begin{align*}
P_{t-}&:=\{f\in\R[X]\mid\exists\ep\in\R_{>0}:\forall x\in(t-\ep,t):f(x)\ge0\},\\
P_t&:=\{f\in\R[X]\mid f(t)\ge0\}\quad\text{and}\\
P_{t+}&:=\{f\in\R[X]\mid\exists\ep\in\R_{>0}:\forall x\in(t,t+\ep):f(x)\ge0\}
\end{align*}
Finally, we set
\begin{align*}
P_{-\infty}&:=\{f\in\R[X]\mid\exists c\in\R:\forall x\in(-\infty,c):f(x)\ge0\}\quad\text{and}\\
P_{\infty}&:=\{f\in\R[X]\mid\exists c\in\R:\forall x\in(c,\infty):f(x)\ge0\}.
\end{align*}
Then \[\sper\R[X]=\{P_{-\infty},P_\infty\}\cup\{P_{t-}\mid t\in\R\}\cup\{P_t\mid t\in\R\}\cup
\{P_{t+}\mid t\in\R\}.\] The fattening of the semialgebraic set $[0,1)\subseteq\R$ is the set
\begin{align*}
C:=\{P_0,P_{0+}\}&\cup\{P_{t-}\mid t\in(0,1)\}\cup\{P_t\mid t\in(0,1)\}\\
&\cup\{P_{t+}\mid t\in(0,1)\}\cup\{P_{1-}\}
\subseteq\sper\R[X].
\end{align*}
In particular, $C$ is constructible. In contrast to this, $C':=C\setminus\{P_{1-}\}$ is
not constructible for otherwise $C$ and $C'$ would have the same slimming in contradiction to
\ref{slimfatten}.
\end{ex}

\section{Real Stellensätze}\label{sec:realstellensaetze}

\begin{rem}\label{idealmultiplicativesetpreorder}
Let $A$ be a commutative ring.
\begin{enumerate}[(a)]
\item Since every intersection of \alalal{ideals}{multiplicative sets}{preorders} of $A$ is again
\alalal{an ideal}{a multiplicative set}{a preorder} of $A$, there exists for every subset $E\subseteq A$
\alalal{a smallest ideal}{a smallest multiplicative set}{a smallest preorder} of $A$ containing $E$.
It is called the \alalal{ideal}{multiplicative set}{preorder} \emph{generated} by $E$.
\item \alalal{An ideal}{A multiplicative set}{A preorder} of $A$ is called \emph{finitely generated}
if it is generated by a finite subset of $A$.
\item The \alalal{ideal}{multiplicative set}{preorder} generated by $a_1,\dots,a_m\in A$
(i.e., by $\{a_1,\dots,a_m\}\subseteq A$) is $\malalal{Aa_1+\ldots+Aa_m}
{\{a_1^{\al_1}\dotsm a_m^{\al_m}\mid\al\in\N_0^m\}}{\sum_{\de\in\{0,1\}^m}\sum A^2a_1^{\de_1}\dotsm
a_m^{\de_m}}$.
\item If \alalal{an ideal}{a multiplicative set}{a preorder} of $A$ is generated by $E\subseteq A$, then it is the union
over all \alalal{ideals}{multiplicative sets}{preorders} of $A$ generated by a finite subset of $E$.
\item If \alalal{an ideal $I$}{a multiplicative set $S$}{a preorder $T$}$\,\subseteq A$ is generated by
$E\subseteq A$ and if $P\in\sper A$, then
$\malalal{\forall a\in I:\widehat a(P)=0}{\forall s\in S:\widehat s(P)\ne0}{\forall t\in T:\widehat t(P)\ge0}\iff
\malalal{\forall a\in E:\widehat a(P)=0}{\forall s\in E:\widehat s(P)\ne0}{\forall t\in E:\widehat t(P)\ge0}$.
\end{enumerate}
\end{rem}

\begin{rem}\label{idealmultiplicativesetpreorderpolynomialring}
Let $(L,\le)$ be an ordered field and $K$ a subfield of $L$. If
\alalal{an ideal $I$}{a multiplicative set $S$}{a preorder $T$}$\subseteq K[\x]$ is generated
by $E\subseteq K[\x]$ and if $x\in L^n$, then
\[\malalal{\forall g\in I:g(x)=0}{\forall h\in S:h(x)\ne0}{\forall f\in T:f(x)\ge0}\iff
\malalal{\forall g\in E:g(x)=0}{\forall h\in E:h(x)\ne0}{\forall f\in E:f(x)\ge0}.\]
\end{rem}

\begin{remterm}
\begin{enumerate}[(a)]
\item ``over $B$ generated by $E$'' stands for ``generated by $B\cup E$''
\item ``over $B$ finitely generated'' stands for ``generated by $B\cup E$ for some finite set $E$''
\item If $(K,\le)$ is an ordered field and $n\in\N_0$, then the preorder generated by
$p_1,\dots,p_m\in K[\x]$ over $K_{\ge0}$ equals
$\sum_{\de\in\{0,1\}^m}\sum K_{\ge0}K[\x]^2p_1^{\de_1}\dotsm p_m^{\de_m}$
[$\to$ \ref{idealmultiplicativesetpreorder}(c)].
\end{enumerate}
\end{remterm}

\begin{pro} Let $(K,\le)$ be an ordered field, $R:=\overline{(K,\le)}$ and $n\in\N_0$
\emph{[$\to$ \ref{rnassubsetofsper}]}. Let $I$ be an ideal, $S$ a \emph{finitely generated}
multiplicative set and $T$ a preorder of $K[\x]$ \emph{finitely generated over $K_{\ge0}$}. Then
\[\{P\in\sper K[\x]\mid(\forall g\in I:\widehat g(P)=0),
(\forall h\in S:\widehat h(P)\ne0),(\forall f\in T:\widehat f(P)\ge0)\}\]
is a constructible subset of $\sper(K[\x],\sum K_{\ge0}K[\x]^2)$ whose
slimming is the $K$-semialgebraic set
\[\{x\in R^n\mid(\forall g\in I:g(x)=0),
(\forall h\in S:h(x)\ne0),(\forall f\in T:f(x)\ge0)\}.\]
\end{pro}

\begin{proof}
By Hilbert's basis theorem, $I$ is finitely generated as well. Now use
\ref{idealmultiplicativesetpreorder}, \ref{slimfatten} and
\ref{idealmultiplicativesetpreorderpolynomialring}.
\end{proof}

\begin{thm}[real Stellensatz \cite{kri,ste,pre}]\emph{[$\to$ \ref{abstractstellensatz}]}\label{stellensatz}
Let $(K,\le)$ be an ordered subfield of the real closed field $R$, $n\in\N_0$,
$I$ an ideal of $K[\x]$, $S$ a \emph{finitely generated}
multiplicative set of $K[\x]$ and $T$ a preorder of $K[\x]$ \emph{finitely generated over $K_{\ge0}$}.
Then the following are equivalent:
\begin{enumerate}[\normalfont(a)]
\item There does \emph{not} exist any $x\in R^n$ satisfying
\begin{align*}
\forall g\in I&:g(x)=0,\\
\forall h\in S&:h(x)\ne0\quad\text{and}\\
\forall t\in T&:t(x)\ge0.
\end{align*}
\item $0\in I+S^2+T$
\end{enumerate}
\end{thm}

\begin{proof}
\underline{(b)$\implies$(a)} is trivial.

\smallskip
\underline{(a)$\implies$(b)}\quad WLOG $R=\overline{(K,\le)}$ [$\to$ \ref{relcl}].
Because the fattening of the empty set is empty by \ref{slimfatten} [$\to$ \ref{bahom}], (a) implies
Condition \ref{abstractstellensatz}(a) from the abstract real Stellensatz applied to
$A:=K[\x]$.
\end{proof}

\begin{cor}[Positivstellensatz]\emph{[$\to$ \ref{abstractpositivstellensatz}]}\label{positivstellensatz}
Let $(K,\le)$ be an ordered subfield of the real closed field $R$, $n\in\N_0$,
$T$ a preorder of $K[\x]$ \emph{finitely generated over $K_{\ge0}$},
\[S:=\{x\in R^n\mid\forall p\in T:p(x)\ge0\}\] and $f\in K[\x]$.
Then the following are equivalent:
\begin{enumerate}[\normalfont(a)]
\item $f>0$ on $S$
\item $\exists t\in T:tf\in1+T$
\item $\exists t\in T:(1+t)f\in1+T$
\end{enumerate}
\end{cor}

\begin{proof} Alternatively from \ref{stellensatz} (as \ref{abstractpositivstellensatz} from
\ref{abstractstellensatz}) or from \ref{abstractpositivstellensatz}
(as \ref{stellensatz} from \ref{abstractstellensatz} using \ref{slimfatten}).
\end{proof}

\begin{cor}[Nichtnegativstellensatz]\emph{[$\to$ \ref{abstractnichtnegativstellensatz}]}\label{nichtnegativstellensatz}
Let $(K,\le)$ be an ordered subfield of the real closed field $R$, $n\in\N_0$,
$T$ a preorder of $K[\x]$ \emph{finitely generated over $K_{\ge0}$},
\[S:=\{x\in R^n\mid\forall p\in T:p(x)\ge0\}\] and $f\in K[\x]$.
Then the following are equivalent:
\begin{enumerate}[\normalfont(a)]
\item $f\ge0$ on $S$
\item $\exists t\in T:\exists k\in\N_0:tf\in f^{2k}+T$
\end{enumerate}
\end{cor}

\begin{proof} Alternatively from \ref{stellensatz} (as \ref{abstractnichtnegativstellensatz} from
\ref{abstractstellensatz}) or from \ref{abstractnichtnegativstellensatz}
(as \ref{stellensatz} from \ref{abstractstellensatz} using \ref{slimfatten}).
\end{proof}

\begin{rem} In the special case $T=\sum K_{\ge0}K[\x]^2$, the Nichtnegativstellensatz
\ref{nichtnegativstellensatz} is obviously a strengthening of Artin's solution \ref{artin}
to Hilbert's 17th problem in which Condition (b) is refined. This refinement has the advantage that
the proof of (b)$\implies$(a) does not require a real argument as it was the case in \ref{artin}.
The proof of \ref{nichtnegativstellensatz} requires prime cones of rings instead of just preorders of fields
and therefore is substantially more difficult as the proof of \ref{artin}.
\end{rem}

\begin{cor}[real Nullstellensatz  \cite{kri,du2,ris,efr}]\emph{[$\to$ \ref{abstractrealnullstellensatz}]}\label{realnullstellensatz}
Let $K$ be a Euclidean subfield of the real closed field $R$, $n\in\N_0$,
$I$ an ideal of $K[\x]$ and
\[V:=\{x\in R^n\mid\forall p\in I:p(x)=0\}.\]
Then $\{f\in K[\x]\mid f=0\text{ on }V\}=\rrad I$.
\end{cor}

\begin{proof} Using the description of $\rrad I$ from \ref{rradchar}, this follows alternatively
from \ref{stellensatz} (as \ref{abstractrealnullstellensatz} from
\ref{abstractstellensatz}) or from \ref{abstractrealnullstellensatz}
(as \ref{stellensatz} from \ref{abstractstellensatz} using \ref{slimfatten}).
\end{proof}

\begin{df}{}[$\to$ \ref{dfrealclosure}] Let $K$ be field. An extension field $R$ of $K$ is called \emph{a}
real closure of $K$ if $R$ is real closed and $R|K$ is algebraic. 
\end{df}

\begin{rem} For two fields $K$ and $R$, the following are equivalent:
\begin{enumerate}[(a)]
\item $R$ is a real closure of $K$.
\item There is an order $\le$ of $K$ such that $R=\overline{(K,\le)}$.
\end{enumerate}
\end{rem}

\begin{thm}[variant of the real Stellensatz]\emph{[$\to$ \ref{stellensatz}]}\label{altstellensatz}
Let $K$ be a field, $n\in\N_0$, $I$ an ideal of $K[\x]$, $S$ a \emph{finitely generated}
multiplicative set of $K[\x]$ and $T$ a \emph{finitely generated} preorder of $K[\x]$.
Then the following are equivalent:
\begin{enumerate}[\normalfont(a)]
\item There does \emph{not} exist a real closure $R$ of $K$ and an $x\in R^n$ such that
\begin{align*}
\forall g\in I&:g(x)=0,\\
\forall h\in S&:h(x)\ne0\quad\text{and}\\
\forall f\in T&:f(x)\ge0.
\end{align*}
\item $0\in I+S^2+T$
\end{enumerate}
\end{thm}

\begin{proof}
\underline{(b)$\implies$(a)} is trivial.

\smallskip
We show \underline{(a)$\implies$(b)} by contraposition. Suppose (b) does not hold. We have to show that
(a) does not hold.
By the abstract real Stellensatz \ref{abstractstellensatz}, there is some
$P\in\sper K[\x]$ such that
$\forall g\in I:\widehat g(P)=0$, $\forall h\in S:\widehat h(P)\ne0$ and $\forall f\in T:\widehat f(P)\ge0$
all hold at the same time. Now consider the real closure $R:=\overline{(K,K\cap P)}$ of $K$
and the preordered ring $(K[\x],\sum(K\cap P)K[\x]^2)$.
The set
\[U:=\left\{x\in R^n\mid(\forall g\in I:g(x)=0),(\forall h\in S:h(x)\ne0),
(\forall f\in T:f(x)\ge0)\right\}\]
is $K$-semialgebraic by \ref{idealmultiplicativesetpreorderpolynomialring} since $I$, $S$ and $T$ are finitely generated.
We will show that $U\ne\emptyset$. We have chosen
$P$ to be an element of the constructible subset
of the real spectrum of this preordered ring which is the
fattening of $U$, i.e., $P\in\fatten(U)$ in the notation of \ref{slimfatten}.
In particular, $\fatten(U)\ne\emptyset$ and thus $U\ne\emptyset$ by \ref{slimfatten}
\end{proof}

\begin{rem}\label{fingenimportant}
Wherever the hypothesis ``finitely generated'' appears in this section, it cannot be omitted.
For instance, assume that the Positivstellensatz \ref{positivstellensatz} holds with the weaker hypothesis
``$K_{\ge0}\subseteq T$'' instead of ``$T$ finitely generated over $K_{\ge0}$''. Consider then $K:=R:=\R$,
$n:=1$ and the preorder of $\R[X]$ generated by \[E:=\{X-N\mid N\in\N\}.\] Then
$S:=\{x\in\R\mid\forall p\in T:p(x)\ge0\}=\emptyset$ and thus $f:=-1>0$ on $S$.
It follows that $\exists t\in T:tf\in1+T$ and thus by \ref{idealmultiplicativesetpreorder}(d) even
$\exists t\in T':tf\in1+T'$ for a preorder $T'\subseteq T$ generated by a finite set $E'\subseteq E$.
The trivial direction of \ref{positivstellensatz} then yields $-1>0$ on
$S':=\{x\in\R\mid\forall p\in T':p(x)\ge0\}\overset{\ref{idealmultiplicativesetpreorderpolynomialring}}=
\{x\in\R\mid\forall p\in E':p(x)\ge0\}=[N,\infty)$ for some $N\in\N$. $\lightning$
\end{rem}

\begin{rem}{}[$\to$ \ref{artin-schreier}] Let $A$ be a commutative ring and $T\subseteq A$ a proper
preorder. Exactly as in the field case, there exists some $P\in\sper A$ such that $T\subseteq P$
[$\to$ \ref{inmaxprimecone}]. In sharp contrast, to the field case we do in general however not have that
$T=\bigcap\sper(A,T)$. As an example, take $A:=\R[X,Y]$, $T:=\sum\R[X,Y]^2$ and consider the Motzkin 
polynomial $f:=X^4Y^2+X^2Y^4-3X^2Y^2+1$. By \ref{motzkin}, we have $f\notin T$ and
$S:=\{(x,y)\in\R^2\mid f(x,y)\ge0\}=\R^2$. By \ref{slimfatten}, the fattening
\[C:=\{P\in\sper A\mid f\in P\}\subseteq\sper A=\sper(A,T)\] of $S$ equals the whole of $\sper A$, i.e.,
$f\in\bigcap\sper(A,T)$.
\end{rem}

\chapter{Schmüdgen's Positivstellensatz}

\section{The abstract Archimedean Positivstellensatz}

\begin{df}{}{[$\to$ \ref{unaryrem}(d)]}
A preordered ring $(A,T)$ is called \emph{Archimedean} if
\[\forall a\in A:\exists N\in\N:N+a\in T,\]
which is equivalent to $T-\N=A$ and also to $T+\Z=A$.
\end{df}

\begin{df}\label{dfarch}
Let $A$ be a commutative ring.
\begin{enumerate}[(a)]
\item A preorder $T$ of $A$ is called Archimedean if $(A,T)$ is Archimedean.
\item $A$ is called Archimedean if $(A,\sum A^2)$ is Archimedean.
\end{enumerate}
\end{df}

\begin{thm}[abstract Archimedean Positivstellensatz \cite{sto,kad,kri,du1}]
\emph{[$\to$ \ref{abstractpositivstellensatz}]}\label{abstractarchimedeanpositivstellensatz}
Let $(A,T)$ be an Archimedean preordered ring and $a\in A$. Then the following are equivalent:
\begin{enumerate}[\normalfont(a)]
\item $\widehat a>0$ on $\sper(A,T)$
\item $\exists N\in\N:Na\in1+T$
\end{enumerate}
\end{thm}

\begin{proof} \underline{(b)$\implies$(a)} is trivial

\smallskip\underline{(a)$\implies$(b)}\quad For the multiplicative set $S:=\N\cdot1\subseteq A$,
$(S^{-1}A,S^{-2}T)$ is again an Archimedean preordered ring [$\to$ \ref{preorderlocalize}] and we have
[$\to$ \ref{sperlocalize}]
\[\widehat a>0\text{ on }\sper(A,T)\iff\widehat{\left(\frac a1\right)}>0\text{ on }\sper(S^{-1}A,S^{-2}T).
\]
We can therefore suppose $\N\cdot1\subseteq A^\times$ and therefore have a homomorphism
\[\Q=\N^{-1}\Z\to A,\ \frac pq\mapsto\frac pq\qquad(p\in\Z,q\in\N).\]
Suppose now that (a) holds. By the abstract Positivstellensatz \ref{abstractpositivstellensatz}, there
is some $t\in T$ such that $ta\in1+T$. Since $T$ is Archimedean, there are $N\in\N$ with $N-t\in T$
and $r\in\N$ with $a+r\in T$. Now you can decrease $r\in\frac1N\N_0$ a finite number of times by $\frac1N$
until it gets negative since
\[a+\left(r-\frac1N\right)=\frac N{N^2}((\underbrace{N-t}_{\in T})(\underbrace{a+r}_{\in T})+
(\underbrace{ta-1}_{\in T})+\underbrace{rt}_{\in T})\in T\]
as long as $r\ge0$. It follows $a-\frac1N\in T$ and thus $Na\in1+T$.
\end{proof}

\begin{cor}\label{archmax}
Let $A$ be a commutative ring and $P\in\sper A$. Then the following are equivalent:
\begin{enumerate}[\normalfont(a)]
\item $P$ is Archimedean and a maximal prime cone. \emph{[$\to$ \ref{dfarch}(a)]}
\item $(\qf(A/\p),P_\p)$ is Archimedean where $\p:=\supp P$. \emph{[$\to$ \ref{primeconeinfield}]}
\item $R_P$ is Archimedean. \emph{[$\to$ \ref{realrep}]}
\item There exists a homomorphism $\ph\colon A\to\R$ such that $P=(\sper\ph)(\R_{\ge0})$.
\end{enumerate}
\end{cor}

\begin{proof}
\underline{(a)$\implies$(b)}\quad Suppose (a) holds and let $a\in A$ and $s\in A\setminus\p$.
To show: $\exists N\in\N:\frac{\cc a\p}{\cc s\p}+N\in P_\p$. WLOG $s\in P$. Since $P$ is maximal, we have
$\sper(A,P)=\{P\}$ and thus $\widehat s>0$ on $\sper(A,P)$. By \ref{abstractarchimedeanpositivstellensatz},
there is $N'\in\N$ such that $N's\in1+P$. Choose $N''\in\N$ such that $a+N''\in P$ and set $N:=N'N''$. Then
$a+Ns=a+N'N''s\in a+N''+P\subseteq P+P\subseteq P$ and thus $(a+Ns)s\in PP\subseteq P$. It follows that
$\frac{\cc a\p}{\cc s\p}+N=\frac{\cc{a+Ns}\p}{\cc s\p}\in P_\p$.

\smallskip\underline{(b)$\implies$(c)}\quad If (b) holds, then
$(\qf(A/\p),P_\p)\hookrightarrow(\R,\R_{\ge0})$ by \ref{archsubfieldreals} [$\to$ \ref{ordfieldhom}] and
\[R_P=\overline{(\qf(A/\p),P_\p)}\hookrightarrow(\R,\R_{\ge0})\]
by \ref{unic}.

\smallskip\underline{(c)$\implies$(d)}\quad Suppose that (c) holds. Choose an embedding $\io\colon R_P\hookrightarrow\R$
according to \ref{archsubfieldreals}. We have $\io^{-1}(\R_{\ge0})=(R_P)_{\ge0}$ because
$\io$ is an embedding of ordered fields. Now set $\ph:=\io\circ\rh_P$. Then
$\ph^{-1}(\R_{\ge0})=\rh_P^{-1}(\io^{-1}(\R_{\ge0}))=\rh_P^{-1}((R_P)_{\ge0})\overset{\ref{kernelrep}}=P$.

\smallskip\underline{(d)$\implies$(a)}\quad Suppose $\ph\colon A\to\R$ is a homomorphism with
$P=\ph^{-1}(\R_{\ge0})$. Then $P$ is Archimedean for if $a\in A$, then one can choose $N\in\N$ with
$\ph(a)+N\ge0$ and it follows that $a+N\in\ph^{-1}(\R_{\ge0})=P$. In order to show that $P$ is maximal,
let $Q\in\sper A$ with $P\subseteq Q$. To show: $P=Q$. If we had $a\in Q\setminus P$, then $\ph(a)<0$
and thus $\ph(Na)\le-1$ for some $N\in\N$ from which it would follow that $\ph(-1-Na)\ge0$ and thus
$-1-Na\in P\subseteq Q$ and $-1=(-1-Na)+Na\in Q+Q\subseteq Q$ $\lightning$.
\end{proof}

\section{The Archimedean Positivstellensatz [$\to$ §\ref{sec:realstellensaetze}]}

\begin{lem}\label{evashom}
Suppose $(K,\le)$ is an ordered subfield of $\R$, $n\in\N_0$ and
$K_{\ge0}\subseteq T\subseteq K[\x]$. Then the correspondence
\begin{align*}
x&\mapsto\ev_x\colon K[\x]\to\R,\ p\mapsto p(x)\\
(\ph(X_1),\dots,\ph(X_n))&\mapsfrom\ph
\end{align*}
defines a bijection between $S:=\{x\in\R^n\mid\forall p\in T:p(x)\ge0\}$ and the set of all ring homomorphisms
$\ph\colon K[\x]\to\R$ satisfying $\ph(T)\subseteq\R_{\ge0}$.
\end{lem}

\begin{proof}
It is obviously enough to show that every ring homomorphism $\ph\colon K[\x]\to\R$ with
$\ph(T)\subseteq\R_{\ge0}$ is the identity on $K$. But this is clear by \ref{archemb} since the identity
is the \emph{only} embedding of ordered fields $(K,\le)\hookrightarrow(\R,\R_{\ge0})$.
\end{proof}

\begin{thm}[Archimedean Positivstellensatz]\emph{[$\to$ \ref{abstractarchimedeanpositivstellensatz},
\ref{positivstellensatz}]}\label{archimedeanpositivstellensatz}
Suppose $(K,\le)$ is an ordered subfield of $\R$, $n\in\N_0$,
$T\subseteq K[\x]$ is an \emph{Archimedean} preorder containing $K_{\ge0}$,
$S:=\{x\in\R^n\mid\forall p\in T:p(x)\ge0\}$ and $f\in K[\x]$. Then the following are equivalent:
\begin{enumerate}[\normalfont(a)]
\item $f>0$ on $S$
\item  $\exists N\in\N:f\in\frac1N+T$
\end{enumerate}
\end{thm}

\begin{proof}
\underline{(b)$\implies$(a)} is trivial.

\smallskip\underline{(a)$\implies$(b)}\quad Suppose that (a) holds. It is enough to show that
$\widehat f>0$ on $\sper(K[\x],T)$ due to the abstract Archimedian Positivstellensatz
\ref{abstractarchimedeanpositivstellensatz} using $\frac1N=N\left(\frac1N\right)^2$.
To this end, let $P\in\sper(K[\x],T)$. We show that $\widehat f(P)>0$, i.e., $\widehat f(P)\not\le0$ which is equivalent to $f\notin-P$. Choose a maximal prime cone $Q$ of $K[\x]$ such that
$P\subseteq Q$ by \ref{inmaxprimecone}. We even show that $f\notin-Q$.
By \ref{archmax}(d) and \ref{evashom}, there is some $x\in S$ satisfying
$Q=\ev_x^{-1}(\R_{\ge0})=\{p\in K[\x]\mid p(x)\ge0\}$. From $f(x)>0$, we deduce now $f\notin-Q$
as desired.
\end{proof}

\begin{rem}
If $T$ is finitely generated over $K_{\ge0}$ in the situation of \ref{archimedeanpositivstellensatz}, then
one can reduce \ref{archimedeanpositivstellensatz} alternatively by fattening to 
\ref{abstractarchimedeanpositivstellensatz}. This ultimately uses unnecessarily the heavy artillery
of real quantifier elimination \ref{elim} and is not applicable if $T$ is not finitely generated over $K_{\ge0}$.
The principal reason why the real quantifier elimination is not needed here is \ref{archsubfieldreals}.
\end{rem}

\section{Schmüdgen's characterization of Archimedean preorders of the polynomial ring}

\begin{dfpro}\label{arithmbounded}
Let $(A,T)$ be a preordered ring. Then
\[B_{(A,T)}:=\{a\in A\mid\exists N\in\N:N\pm a\in T\}\]
is a subring of $A$ which we call the ring of with respect to $T$
\emph{arithmetically bounded} elements of $A$. 
\end{dfpro}

\begin{proof}
One sees immediately that $B_{(A,T)}$ is a subgroup of the additive group of $A$. It is clear that
$1\in B_{(A,T)}$. Finally, we have $B_{(A,T)}B_{(A,T)}\subseteq B_{(A,T)}$ as one sees immediately
from the identity
\[3N^2\pm ab=(N+a)(N\pm b)+N(N-a)+N(N\mp b)\qquad(N\in\N,a,b\in A).\]
\end{proof}

\begin{lem}\label{squarerootsarithmeticallybounded}
Let $(A,T)$ be a preordered ring such that $\frac12\in A$. Then
\[a^2\in B_{(A,T)}\implies a\in B_{(A,T)}\] for all $a\in A$.
\end{lem}

\begin{proof}
Choose $N\in\N$ with $(N-1)-a^2\in T$. Then
\[N\pm a=(N-1)-a^2+\left(\frac12\pm a\right)^2+3\left(\frac12\right)^2\in T.\]
\end{proof}

\begin{rem}\label{archb}
If $(A,T)$ is a preordered ring, then $T$ is Archimedean if and only if $B_{(A,T)}=A$.
\end{rem}

\begin{lem}\label{archchar}
Suppose $(K,\le)$ is an ordered subfield of $\R$, $n\in\N_0$ and $T\subseteq K[\x]$
is a preorder containing $K_{\ge0}$. Then the following are equivalent:
\begin{enumerate}[\normalfont(a)]
\item $T$ is Archimedean.
\item $\exists N\in\N:N-\sum_{i=1}^nX_i^2\in T$
\item $\exists N\in\N:\forall i\in\{1,\dots,n\}:N\pm X_i\in T$
\end{enumerate}
\end{lem}

\begin{proof}
\underline{(a)$\implies$(b)} is trivial.

\smallskip\underline{(b)$\implies$(c)}\quad If (b) holds, then $N-X_i^2\in T$ and thus $X_i^2\in B_{(A,T)}$
for all $i\in\{1,\dots,n\}$. Now apply \ref{squarerootsarithmeticallybounded}.

\smallskip\underline{(c)$\implies$(a)}\quad Since $(K,\le)$ is Archimedean and $K_{\ge0}\subseteq T$,
we have $K\subseteq B_{(A,T)}$. If now moreover (c) holds, then $K[\x]=B_{(A,T)}$.
\end{proof}

\begin{thm}[Schmüdgen's Theorem \cite{sch,bw}]\label{schmuedgen}
Suppose $(K,\le)$ is an ordered subfield of $\R$,
$n\in\N_0$ and $T$ a preorder of $K[\x]$ which is finitely generated \emph{over $K_{\ge0}$}. Write
\[S:=\{x\in\R^n\mid\forall p\in T:p(x)\ge0\}.\] Then
\[\text{$T$ Archimedean}\iff\text{$S$ compact}.\]
\end{thm}

\begin{proof} \cite{bw} ``$\Longrightarrow$'' Let $T$ be Archimedean. By \ref{archchar}(b), there is some
$N\in\N$ with $N-\sum_{i=1}^nX_i^2\in T$. Then $S$ is contained in the ball of radius $\sqrt N$
centered at the origin and thus bounded. Anyway $S$ is already closed. Consequently, $S$ is compact.

\smallskip ``$\Longleftarrow$'' Let $S$ be compact. Write $r:=\sum_{i=1}^nX_i^2$ and choose $N\in\N$
such that $N-r>0$ on $S$. By the Positivstellensatz \ref{positivstellensatz}, we find $t\in T$ such that
$(1+t)(N-r)\in1+T\subseteq T$. We know that $T':=T+(N-r)T$ is a preorder of $K[\x]$ that is Archimedean
by \ref{archchar}. We have $(1+t)T'\subseteq T$ and $N-r+Nt=(1+t)(N-r)+tr\in T+T\subseteq T$.
Choose $N'\in\N$ with $N'-t\in T'$. Then
\[(1+N')(N'-t)=(1+t)(N'-t)+(N'-t)^2\in(1+t)T'+T\subseteq T+T\subseteq T\]
from which $N'-t\in T$ follows because of $\frac1{1+N'}=(1+N')\left(\frac1{1+N'}\right)^2\in T$.
We conclude that
\[N(N'+1)-r=NN'+N-r=(N-r+tN)+N(N'-t)\in T+T\subseteq T.\]
Now \ref{archchar} implies that $T$ is Archimedean.
\end{proof}

\begin{cor}[Schmüdgen's Positivstellensatz]\emph{[$\to$ \ref{archimedeanpositivstellensatz}]}
\label{schmuedgenpositivstellensatz}
Suppose $(K,\le)$ is an ordered subfield of $\R$,
$n\in\N_0$, $T$ a preorder of $K[\x]$ which is finitely generated \emph{over $K_{\ge0}$}. Suppose
$S:=\{x\in\R^n\mid\forall p\in T:p(x)\ge0\}$ is \emph{compact} and $f\in K[\x]$. Then the following are equivalent:
\begin{enumerate}[\normalfont(a)]
\item $f>0$ on $S$
\item $\exists N\in\N:f\in\frac1N+T$
\end{enumerate}
\end{cor}

\begin{proof}
 By Schmüdgen's Theorem \ref{schmuedgen}, $T$ is Archimedean. But then the Archimedean Positivstellensatz
 \ref{archimedeanpositivstellensatz} proves the equivalence of (a) and (b).
\end{proof}

\begin{rem}\label{needep}
\begin{enumerate}[(a)] \item
Exactly as in \ref{fingenimportant}, one sees that the hypothesis ``$T$ finitely
generated over $K_{\ge0}$'' cannot be replaced by the weaker hypothesis ``$K_{\ge0}\subseteq T$''.
\item If one drops the requirement that $S$ is compact, then \ref{schmuedgenpositivstellensatz}
gets wrong as the example $K:=\R$, $n:=1$, $T:=\sum\R[X]^2+\sum\R[X]^2X^3$ and $f:=X+1$ shows:
We have $f>0$ on $S=[0,\infty)$ but $f\notin T$ for degree reasons as one sees from \ref{soslongrem}(b).
\item In the situation of \ref{schmuedgenpositivstellensatz}, we unfortunately do not have in general
\[f\ge0\text{ on }S\iff f\in T.\] 
For this, consider $K:=\R$, $n:=1$, $T:=\sum\R[X]^2+\sum\R[X]^2X^3(1-X)$ and $f:=X$. Then
$f\ge0$ on $S=[0,1]$. Assume $f\in T$. Write $f=\sum_ip_i^2+\sum_jq_j^2X^3(1-X)$ for some
$p_i,q_j\in\R[X]$. Evaluating in $0$, yields $0=\sum_ip_i(0)^2$ and thus $p_i(0)=0$ for all $i$.
Write $p_i=Xp_i'$ for some $p_i'\in\R[X]$. Then $X=f=X^2\left(\sum_ip_i'^2+\sum_jq_j^2X(1-X)\right)\ 
\lightning$.
\end{enumerate}
\end{rem}

\chapter{The real spectrum as a topological space}

\section{Tikhonov's theorem}

\begin{rem}
Any finite intersection of unions of certain sets is a union of finite intersections of such sets
[$\to$ \ref{unionsection}].
\end{rem}

\begin{reminder}\mbox{}[$\to$ \ref{booleanalgebra}]\label{topreminder}
Let $M$ be a set.
\begin{enumerate}[(a)]
\item A set $\O \subseteq\mathcal P(M)$ is called a \emph{topology} on $M$ if
\begin{itemize}
\item $M\in\O $,
\item $\forall A_1,A_2\in\O :A_1\cap A_2\in\O $ and
\item $\forall\mathcal A\subseteq\O :\bigcup\mathcal A\in\O $.
\end{itemize}
In this case, $(M,\O )$ is called a \emph{topological space} and the elements of $\O $ are called
its \emph{open sets}.
\item Let $\mathcal G\subseteq\mathcal P(M)$. Then the set of all unions of finite intersections of elements of
$\mathcal G$ (where $\bigcap\emptyset:=M$) is obviously the smallest topology $\O $ on $M$ such that
$\mathcal G\subseteq\O $. It is called the topology \emph{generated by $\mathcal G$} (on $M$).
\item If $\O $ and $\O '$ are topologies on $M$, then $\O $ is called
\emph{\alal{coarser}{finer}}
 than $\O '$ if $\malal{\O \subseteq\O '}
{\O \supseteq\O '}$.
\item The finest topology on $M$ is the \emph{discrete topology} $\O :=\mathcal P(M)$.
\item The coarsest topology on $M$ is the \emph{trivial topology} (in German: \emph{Klumpentopologie})
$\O :=\{\emptyset,M\}$. 
\end{enumerate}
\end{reminder}

\begin{reminder}\label{conti}
Let $(M,\O )$ and $(N,\mathcal P)$ be topological spaces and $f\colon M\to N$ be a map.
Then $f$ is called \emph{continuous} if $f^{-1}(B)\in\O $ for all $B\in\mathcal P$. If $\mathcal P$ is generated by $\mathcal G$, then $f$ is continuous if and only if
$f^{-1}(B)\in\mathcal\O$ for all $B\in\mathcal G$.
\end{reminder}

\begin{reminder}\label{initialtop}
Let $M$ be a set, $(N_i,\mathcal P_i)_{i\in I}$ a family of topological spaces and $(f_i)_{i\in I}$ a family
of maps $f_i\colon M\to N_i$ ($i\in I$). Then there exists a coarsest topology $\O $ on $M$
making all $f_i$ ($i\in I$) continuous. One calls $\O $ the
\emph{initial topology}
(or \emph{weak topology}) with respect to $(f_i)_{i\in I}$. If $I=\{1,\dots,n\}$, then $\O $ is also called the
initial topology with respect to $f_1,\dots,f_n$. This topology $\O $ is generated by
$\{f_i^{-1}(B_i)\mid i\in I,B\in\mathcal P_i\}$. More generally, the following holds: If $\mathcal P_i$ is generated
by $\mathcal G_i$ for $i\in I$, then $\O $ is generated by $\{f_i^{-1}(B)\mid i\in I,B\in\mathcal G_i\}$.
It holds that $\O $ is the unique topology on $M$ having the following property:
For every topological space $(M',\O ')$ and every $g\colon M'\to M$, the map $g$
is continuous if and only if all the maps $f_i\circ g$ with $i\in I$ are continuous.
\end{reminder}

\begin{ex}\label{subspaceproductspace}
\begin{enumerate}[(a)]
\item Let $(N,\mathcal P)$ be a topological space and $M\subseteq N$. Then one endows $M$ with the
initial topology $\O $ with respect to $M\to N,\ x\mapsto x$. One calls $\O $ the
topology \emph{induced} by $\mathcal P$ on $M$ and $(M,\O )$ a \emph{subspace}
of $(N,\mathcal P)$. We have \[\O =\{M\cap B\mid B\in\mathcal P\}.\]
\item Let $(N_i,\mathcal P_i)_{i\in I}$ be a family of topological spaces. Then there exists a coarsest topology
$\O $ on $N:=\prod_{i\in I}N_i$ making all projections $\pi_i\colon N\to N_i,\ (x_j)_{j\in I}\mapsto x_i$
($i\in I$) continuous. One calls $\O $ the \emph{product topology} of the $\mathcal P_i$ ($i\in I$) on $N$
and $(N,\O )$ the \emph{product space} of the $(N_i,P_i)$ ($i\in I$).
The elements of $\O $ are exactly the unions of sets of the form
$\prod_{i\in I}B_i$ where $B_i\in\mathcal P_i$ for $i\in I$ and $\#\{i\in I\mid B_i\ne N_i\}<\infty$.
\end{enumerate}
\end{ex}

\begin{rem}\label{inducedproductcommute}
The constructions (a) and (b) in Example \ref{subspaceproductspace} commute
in the following sense:
Let $(N_i,\mathcal P_i)_{i\in I}$ be a family of topological spaces and
$(N,\O )$ its product. Furthermore, let
$(M_i)_{i\in I}$ be a family of sets such that $M_i\subseteq N_i$ for each $i\in I$
and set $M:=\prod_{i\in I}M_i$. Then $\O $ induces on $M$
the product topology of the topologies induced on the $M_i$ by the $\mathcal P_i$.
\end{rem}

\begin{df}\label{dffilter}
Let $M$ be a set and $\mathcal S$ a Boolean algebra on $M$ [$\to$ \ref{booleanalgebra}]
(for instance $\mathcal S=\mathcal P(M)$). A set $\mathcal F\subseteq\mathcal S$ is called a
\emph{filter} in $\mathcal S$ (or filter on $M$ in case $\mathcal S=\mathcal P(M)$) if
\begin{itemize}
\item $\emptyset\notin\mathcal F,M\in\mathcal F$,
\item $\forall U,V\in\mathcal F:U\cap V\in\mathcal F$ and
\item $\forall U\in\mathcal F:\forall V\in\mathcal S:(U\subseteq V\implies V\in\mathcal F)$.
\end{itemize}
If in addition $\forall U\in\mathcal S:(U\in\mathcal F\text{ or }\complement U\in\mathcal F)$, then $\mathcal F$
is called an \emph{ultrafilter}.
\end{df}

\begin{pro}
Let $\mathcal S$ be a Boolean algebra on the set $M$ and $\mathcal F$ a filter in $\mathcal S$.
Then the following are equivalent:
\begin{enumerate}[\normalfont(a)]
\item $\mathcal F$ is an ultrafilter.
\item $\forall U,V\in\mathcal S:(U\cup V\in\mathcal F\implies(U\in\mathcal F\text{ or }V\in\mathcal F))$
\end{enumerate}
\end{pro}

\begin{proof}
\underline{(a)$\implies$(b)}\quad Suppose that (a) holds and let $U,V\in\mathcal S$ such that
$U\cup V\in\mathcal F$ and $U\notin\mathcal F$. To show: $V\in\mathcal F$. Since $\mathcal F$ is an ultrafilter,
we have $\complement U\in\mathcal F$ and thus $(U\cup V)\cap\complement U\in\mathcal F$.
Because of $(U\cup V)\cap\complement U\subseteq V$ it then also holds that $V\in\mathcal F$.

\smallskip\underline{(b)$\implies$(a)} is trivial.
\end{proof}

\begin{ex}\label{neighbor}
Let $(M,\O )$ be a topological space and $x\in M$. Then
\[\mathcal U_x:=\{U\in\mathcal P(M)\mid\exists A\in\O :x\in A\subseteq U\}\]
is a filter on $M$.
One calls $\mathcal U_x$ the \emph{neighborhood filter} of $x$ and its elements the
\emph{neighborhoods} of $x$. In general, $\mathcal U_x$ is not an ultrafilter since
$[-1,1]=[-1,0]\cup[0,1]$ is a neighborhood of $0$ in $\R$ as opposed to $[-1,0]$ and $[0,1]$.
\end{ex}

\begin{df}
Let $(M,\O )$ be a topological space, $\mathcal F$ a filter on $M$ and $x\in M$. One says that
$\mathcal F$ \emph{converges} to $x$ and writes $\mathcal F\to x$ if \[\mathcal U_x\subseteq\mathcal F.\]
If $\mathcal F$ converges to exactly one point $x$, one calls this the \emph{limit} of $\mathcal F$ and writes
\[x=\lim\mathcal F.\]
\end{df}

\begin{df}
Let $(M,\O )$ be a topological space, $(a_n)_{n\in\N}$ a sequence in $M$ and $x\in M$. We call
\[\mathcal F:=\{U\in\mathcal P(M)\mid\exists N\in\N:\{a_n\mid n\ge N\}\subseteq U\}\]
the filter associated to $(a_n)_{n\in\N}$. It is clearly a filter on $M$.
One says that
$(a_n)_{n\in\N}$ \emph{converges} to $x$ and writes \[a_n\overset{n\to\infty}\longrightarrow x\] if $\mathcal F\to x$.
If $\mathcal F$ converges to exactly one point $x$, one calls this the \emph{limit} of $(a_n)_{n\in\N}$ and writes
\[x=\lim_{n\to\infty}a_n.\]
\end{df}

\begin{dflem}\label{ultraimage}
Suppose $f\colon M\to N$ is a map and $\mathcal F$ a filter on $M$. Then the \emph{image filter}
\[f(\mathcal F):=\{V\in\mathcal P(N)\mid\exists U\in\mathcal F:f(U)\subseteq V\}\] is a filter on $N$.
If $\mathcal F$ is an ultrafilter, then so is $f(\mathcal F)$.
\end{dflem}

\begin{proof}
One sees immediately that $f(\mathcal F)$ is a filter. Now let $\mathcal F$ be an ultrafilter.
Suppose $V\subseteq N$ and $V\notin f(\mathcal F)$.
To show: $\complement V\in f(\mathcal F)$. For $U:=f^{-1}(V)$, one has $f(U)\subseteq V$ and thus
$U\notin\mathcal F$.
But then $f^{-1}(\complement V)=\complement U\in\mathcal F$. From
$f(\complement U)\subseteq\complement V$, we obtain thus $\complement V\in f(\mathcal F)$.
\end{proof}

\begin{lem}\label{ultraprodconv}
Let $M$ be a set endowed with the initial topology with respect to a family $(f_i)_{i\in I}$ of maps
$f_i\colon M\to N_i$ into topological spaces $N_i$ ($i\in I$). Let $\mathcal F$ be a filter on $M$ and $x\in M$.
Then \[\mathcal F\to x\iff\forall i\in I:f_i(\mathcal F)\to f_i(x).\]
\end{lem}

\begin{proof}
``$\Longrightarrow$'' Suppose $\mathcal F\to x$ and let $i\in I$. To show: $f_i(\mathcal F)\to f_i(x)$.
Let $V\in\mathcal U_{f_i(x)}$. To show: $V\in f_i(\mathcal F)$. Since $f_i$ is continuous, we have
$U:=f_i^{-1}(V)\in\mathcal U_x$ and thus $U\in\mathcal F$. From $f_i(U)\subseteq V$, we get
$V\in f_i(\mathcal F)$.

\smallskip
``$\Longleftarrow$'' Suppose $f_i(\mathcal F)\to f_i(x)$ for all $i\in I$. Let $U\in\mathcal U_x$.
To show: $U\in\mathcal F$. Choose $n\in\N_0$, $i_1,\dots,i_n\in I$ and
$V_k$ open in $N_{i_k}$ ($k\in\{1,\dots,n\}$) such that
\[x\in f_{i_1}^{-1}(V_1)\cap\ldots\cap f_{i_n}^{-1}(V_n)\subseteq U.\] Since $\mathcal F$ is a filter, it is enough to
show that $f_{i_k}^{-1}(V_k)\in\mathcal F$ for all $k\in\{1,\dots,n\}$. Fix therefore $k\in\{1,\dots,n\}$.
Since $V_k$ is an (open) neighborhood of $f_{i_k}(x)$ in $N_{i_k}$,
the hypothesis yields $V_k\in f_{i_k}(\mathcal F)$. Hence there is $U_0\in\mathcal F$ such that
$f_{i_k}(U_0)\subseteq V_k$. Now everything follows from
$U_0\subseteq f_{i_k}^{-1}(f_{i_k}(U_0))\subseteq f_{i_k}^{-1}(V_k)$.
\end{proof}

\begin{df}\label{dfcomp}
Let $(M,\O )$ be a topological space. Then $(M,\O )$ is called a \emph{Hausdorff} space
if every two distinct points of $M$ can be separated by disjoint neighborhoods, i.e.,
\[\forall x,y\in M:(x\ne y\implies\exists U\in\mathcal U_x:\exists V\in\mathcal U_y:U\cap V=\emptyset).\]
We call $(M,\O )$ \emph{quasicompact} if every open cover of $M$ possesses a finite subcover, i.e.,
\[\forall\mathcal A\subseteq\O :\left(M=\bigcup\mathcal A\implies
\exists\mathcal B\subseteq\mathcal A:\left(\#\mathcal B<\infty~\et~M=\bigcup\mathcal B\right)\right).\]
Furthermore, we call a quasicompact Hausdorff space \emph{compact}.
\end{df}

\begin{pro}\label{maxultra}
Suppose $M$ is a set, $\mathcal S$ a Boolean algebra on $M$ and $\mathcal U$ a filter in $\mathcal S$.
Then the following are equivalent:
\begin{enumerate}[\normalfont(a)]
\item $\mathcal U$ is an ultrafilter in $\mathcal S$.
\item $\mathcal U$ is a maximal filter in $\mathcal S$.
\end{enumerate}
\end{pro}

\begin{proof}
\underline{(a)$\implies$(b)}\quad Suppose that (a) holds and let $\mathcal F$ be a filter in $\mathcal S$
such that $\mathcal U\subseteq\mathcal F$. In order to show $\mathcal F\subseteq\mathcal U$, we fix
$U\in\mathcal F$. If we had $U\notin\mathcal U$, we would get
$\complement U\in\mathcal U\subseteq\mathcal F$ and thus $\emptyset=U\cap\complement U\in\mathcal F$
$\lightning$.

\smallskip\underline{(b)$\implies$(a)}\quad Suppose that (b) holds and let $U\in\mathcal S$ satisfy
$U\notin\mathcal U$. To show: $\complement U\in\mathcal U$. It is enough to show that
$\mathcal F:=\{V\in\mathcal S\mid\exists A\in\mathcal U:A\cap\complement U\subseteq V\}$ is a filter in
$\mathcal S$ because then it follows from $\mathcal U\subseteq\mathcal F$ that
$\complement U\in\mathcal F=\mathcal U$. For this, it suffices to show $\emptyset\notin\mathcal F$.
If we had $\emptyset\in\mathcal F$, then there would be an $A\in\mathcal U$ satisfying
$A\cap\complement U=\emptyset$ and from $A\subseteq U$ it would follow that $U\in\mathcal U$ 
$\lightning$.
\end{proof}

\begin{thm}\label{makeultra}
Let $M$ be a set, $\mathcal S$ a Boolean algebra on $M$ and $\mathcal F$ a filter in $\mathcal S$.
Then there is an ultrafilter $\mathcal U$ in $\mathcal S$ such that $\mathcal F\subseteq\mathcal U$.
\end{thm}

\begin{proof}
By \ref{maxultra}, it suffices to show that the set
$\{\mathcal F'\mid\mathcal F'\text{ Filter in }\mathcal S,\mathcal F\subseteq\mathcal F'\}$
partially ordered by inclusion has a maximal element. This follows from Zorn's lemma since the union of
a nonempty chain of filters in $\mathcal S$ is again a filter in $\mathcal S$.
\end{proof}

\begin{thm}\label{ultracompact}
A topological space $M$ is quasicompact if and only if each ultrafilter on the set $M$ converges in $M$.
\end{thm}

\begin{proof}
Let $M$ be a topological space. We show the equivalence of the following statements:
\begin{enumerate}[(a)]
\item $M$ is not quasicompact.
\item There is an ultrafilter on $M$ that does not converge.
\end{enumerate}

\underline{(a)$\implies$(b)}\quad Suppose that (a) holds. Then for each $x\in M$, there is obviously
an open set $A_x\subseteq M$ with $x\in A_x$ in such a way that
$\bigcup_{x\in M}A_x=M$ and $A_{x_1}\cup\ldots\cup A_{x_n}\ne M$ for all $n\in\N$ and $x_1,\dots,x_n\in M$.
Then
\[\mathcal F:=\left\{U\in\mathcal P(M)\mid\exists n\in\N:\exists x_1,\ldots,x_n\in M:
\left(\complement A_{x_1}\right)\cap
\ldots\cap\left(\complement A_{x_n}\right)\subseteq U\right\}\]
is a filter on $M$ that can be extended by \ref{makeultra} to an ultrafilter $\mathcal U$ on $M$.
Let $x\in M$. We show that $\mathcal U$ does not converge to $x$. If we had $\mathcal U\to x$,
then $A_x\in\mathcal U$ in contradiction to $\complement A_x\in\mathcal U$.

\smallskip\underline{(b)$\implies$(a)}\quad Suppose that (b) holds. Choose an ultrafilter $\mathcal U$ on $M$
that does not converge. Then for every $x\in M$ there is an $U_x\in\mathcal U_x$ such that
$U_x\notin\mathcal U$. WLOG $U_x$ is open for every $x\in M$. Of course $M=\bigcup_{x\in M}U_x$.
If $n\in\N$ and $x_1,\dots,x_n\in M$, then
$\complement(U_{x_1}\cup\ldots\cup U_{x_n})=\left(\complement U_{x_1}\right)\cap\ldots\cap
\left(\complement U_{x_n}\right)\in\mathcal U$ and thus
$\emptyset\ne\complement(U_{x_1}\cup\ldots\cup U_{x_n})$, i.e., $M\ne U_{x_1}\cup\ldots\cup U_{x_n}$.
\end{proof}

\begin{thm}[Tikhonov]\label{tikhonov}
Products of quasicompact topological spaces are quasicompact.
\end{thm}

\begin{proof} Let $(N_i)_{i\in I}$ be a family of quasicompact topological spaces and
$M:=\prod_{i\in I}N_i$ the product space [$\to$ \ref{subspaceproductspace}(b)]. Consider for each $i\in I$
the canonical projection $\pi_i\colon M\to N_i$. According to \ref{ultracompact} it suffices to show that
every ultrafilter on $M$ converges. For this purpose, let $\mathcal U$ be an ultrafilter on $M$.
By \ref{ultraimage}, the image filters $\pi_i(\mathcal U)$ ($i\in I$) are again ultrafilters and therefore converge.
Accordingly, we choose $(x_i)_{i\in I}$ satisfying $\pi_i(\mathcal U)\to x_i$ for each $i\in I$. From
\ref{ultraprodconv}, we now get $\mathcal U\to(x_i)_{i\in I}$.
\end{proof}

\begin{cor}\label{tikhonovcor}
Products of compact spaces are compact.
\end{cor}

\begin{rem}
Let $M$ be a topological space. \alal{In \ref{ultracompact}, we have shown}{Using \ref{makeultra}, one shows as
an exercise} that $M$ is \alal{quasicompact}{a Hausdorff space} if and only if every ultrafilter on $M$
converges to \alal{at least}{at most} one point of $M$. Therefrom, $M$ is compact if and only if each
ultrafilter on $M$ converges to exactly one point of $M$.
\end{rem}

\begin{reminder}\label{compactsubspace}
Let $M$ be a topological space and $A\subseteq M$. Then $A$ is called \emph{closed} in $M$
if $\complement A$ is open in $M$. We call $A$ \alal{quasicompact}{compact} if $A$ furnished with
the subspace topology [$\to$ \ref{subspaceproductspace}(a)] is a \alal{quasicompact}{compact}
topological space. Consequently, $A$ is quasicompact if and only if each open cover of $A$ in $M$
possesses a finite subcover, i.e.,
\[\forall\mathcal A\subseteq\O :\left(A\subseteq\bigcup\mathcal A\implies
\exists\mathcal B\subseteq\mathcal A:\left(\#\mathcal B<\infty~\et~A\subseteq\bigcup\mathcal B\right)\right).\]
It follows immediately that closed subsets of \alal{quasicompact}{compact} topological spaces are again
\alal{quasicompact}{compact}.
\end{reminder}

\section{Topologies on the real spectrum}

\begin{df}\label{spectraltop}
Let $A$ be a commutative ring. We call the
topology generated by \[\{\{P\in\sper A\mid\widehat a(P)>0\}\mid a\in A\}\] on $\sper A$
the \emph{spectral topology} (or \emph{Harrison-topology}) on $\sper A$. 
Moreover, we call the topology generated by
$\mathcal C_A$ [$\to$ \ref{introconstructible}] or, equivalently [$\to$ \ref{constructiblesetnormalform}], by
\[\{\{P\in\sper A\mid\widehat a(P)=0\}\mid a\in A\}\cup\{\{P\in\sper A\mid\widehat a(P)>0\}\mid a\in A\},\]
the \emph{constructible topology} on $\sper A$. Unless otherwise indicated, we endow $\sper A$ always with
the spectral topology. It is coarser than the constructible topology.
\end{df}

\begin{reminder}\label{dfhomeo}
Let $M$ and $N$ be topological spaces. A bijection $f\colon M\to N$ is called a \emph{homeomorphism}
if both $f$ and $f^{-1}$ are continuous. One calls $M$ and $N$ \emph{homeomorphic} if there exists a
homeomorphism from $M$ to $N$.
\end{reminder}

\begin{thm}\label{constrcompact}
Let $A$ be a commutative ring. Then $\sper A$ is compact with respect to the constructible topology.
\end{thm}

\begin{proof}
We endow the two-element set $\{0,1\}$ with the discrete topology [$\to$ \ref{topreminder}(d)]. Then
$\{0,1\}$ is compact and so is $\{0,1\}^A=\prod_{i\in A}\{0,1\}$ with respect to the product topology
by Tikhonov's Theorem \ref{tikhonov}. For every $B\subseteq A$, we denote by
\[1_B\colon A\to\{0,1\},\ a\mapsto\begin{cases}0&\text{if }a\notin B\\1&\text{if }a\in B\end{cases}\]
the corresponding characteristic function. Consider $S:=\{1_P\mid P\in\sper A\}\subseteq\{0,1\}^A$
endowed with the subspace topology of the product topology. Obviously,
\[\sper A\to S,\ P\mapsto 1_P\] is a homeomorphism. Since $\{0,1\}^A$ is compact by \ref{tikhonovcor},
it suffices to show that $S$ is closed in $\{0,1\}^A$ since then $S$ and consequently $\sper A$ is compact
[$\to$ \ref{compactsubspace}]. Encoding \ref{dfprimecone} in characteristic functions, we obtain
\begin{align*}
S=&\bigcap_{a,b\in A}\left\{\ch\in\{0,1\}^A\mid\ch(a)=0\text{ or }\ch(b)=0\text{ or }\ch(a+b)=1\right\}\cap\\
&\bigcap_{a,b\in A}\left\{\ch\in\{0,1\}^A\mid\ch(a)=0\text{ or }\ch(b)=0\text{ or }\ch(ab)=1\right\}\cap\\
&\bigcap_{a\in A}\left\{\ch\in\{0,1\}^A\mid\ch(a)=1\text{ or }\ch(-a)=1\right\}\cap\\
&\left\{\ch\in\{0,1\}^A\mid\ch(-1)=0\right\}\cap\\
&\bigcap_{a,b\in A}\left\{\ch\in\{0,1\}^A\mid\ch(ab)=0\text{ or }\ch(a)=1\text{ or }\ch(-b)=1\right\}.
\end{align*}
Being thus an intersection of closed sets, $S$ is itself closed.
\end{proof}

\begin{cor} Let $A$ be a commutative ring. Then $\sper A$ is quasicompact.
\end{cor}

\begin{proof} Every open cover of $\sper A$ is in particular an open cover with respect to the finer
constructible topology. By \ref{constrcompact}, it possesses therefore a finite subcover.
\end{proof}

\begin{reminder}\label{interiorclosure}
Let $M$ be a topological space and $A\subseteq M$.
\alal{The \emph{interior} $A^\circ$}{The \emph{closure} $\overline A$}
of $A$ (in $M$) is the \alal{union}{intersection} over all \alal{open subsets}{closed supersets} of $A$ in $M$,
i.e., the \alal{largest open subset}{smallest closed superset} of $A$ in $M$. One shows immediately
\begin{align*}
A^\circ&=\{x\in M\mid\exists U\in\mathcal U_x:U\subseteq A\}\quad\text{and}\\
\overline A&=\{x\in M\mid\forall U\in\mathcal U_x:U\cap A\ne\emptyset\}.
\end{align*}
Therefore one calls the elements
of $\malal{A^\circ}{\overline A}$ also \emph{\alal{interior}{adherent} points} of $A$.
One says that $A$ is \emph{dense}
in $M$ if $\overline A=M$ or, equivalently, if every nonempty open subset of $M$ contains an element of $A$.
\end{reminder}

\begin{rem}
Let $A$ be a commutative ring and let $P,Q\in\sper A$. Then
\begin{align*}
P\subseteq Q&\iff\forall a\in A:(\widehat a(P)\ge0\implies\widehat a(Q)\ge0)\\
&\iff\forall a\in A:(\widehat a(Q)<0\implies\widehat a(P)<0)\\
&\iff\forall a\in A:(\widehat{-a}(Q)<0\implies\widehat{-a}(P)<0)\\
&\iff\forall a\in A:(\widehat a(Q)>0\implies\widehat a(P)>0)\\
&\iff\forall U\in\mathcal U_Q:P\in U\\
&\iff\forall U\in\mathcal U_Q:U\cap\{P\}\ne\emptyset\\
&\iff Q\in\overline{\{P\}}.
\end{align*}
Thus if there are $P,Q\in\sper A$ with $P\subset Q$,
then $\sper A$ ist not a Hausdorff space. For example, $\sper\R[X]$ is not a Hausdorff space
[$\to$ \ref{sperrx}].
\end{rem}

\begin{rem}\label{spercontinuity}
Suppose $A$ and $B$ are commutative rings and $\ph\colon A\to B$ is a ring homomorphism. Then
\[\sper\ph\colon\sper B\to\sper A,\ Q\mapsto\ph^{-1}(Q)\]
is continuous with respect to the spectral topologies on both sides as well as with respect to the constructible
topologies on both sides because for $a\in A$, we have
\begin{align*}
(\sper\ph)^{-1}(\{P\in\sper A\mid\widehat a(P)>0\})&=
\{Q\in\sper B\mid\widehat a((\sper\ph)(Q))>0\}\\
&=\{Q\in\sper B\mid\widehat a(\ph^{-1}(Q))>0\}\\
&=\{Q\in\sper B\mid a\in\ph^{-1}(Q)\setminus-\ph^{-1}(Q)\}\\
&=\{Q\in\sper B\mid a\in\ph^{-1}(Q\setminus-Q)\}\\
&=\{Q\in\sper B\mid \ph(a)\in Q\setminus-Q\}\\
&=\{Q\in\sper B\mid\widehat{\ph(a)}(Q)>0\}
\end{align*}
and analogously
\[(\sper\ph)^{-1}(\{P\in\sper A\mid\widehat a(P)\ge0\})=\{Q\in\sper B\mid\widehat{\ph(a)}(Q)\ge0\}.\]
\end{rem}

\begin{rem}\label{sperpreordcomp}
Let $(A,T)$ be a preordered ring [$\to$ \ref{preorderedring}]. Then
\[\sper(A,T)=\bigcap_{t\in T}\left\{P\in\sper A\mid\widehat t(P)\ge0\right\},\]
as an intersection of closed sets, is again closed in $\sper A$, namely with respect to the spectral but also
with respect to the constructible topology on $\sper A$. By \ref{compactsubspace}, $\sper(A,T)$ is thus
quasicompact with respect to the spectral and compact with respect to the constructible topology.
\end{rem}

\section{The real spectrum of polynomial rings}

As in §\ref{sec:constructible}, we fix in this section an ordered field $(K,\le)$, we denote by
$R:=\overline{(K,\le)}$ its real closure, we let $n\in\N_0$, $A:=K[\x]=K[X_1,\dots,X_n]$ and
$T:=\sum K_{\ge0}A^2$. Moreover, we denote by $\mathcal S:=\mathcal S_{n,R}$ the Boolean
algebra of all $K$-semialgebraic subsets of $R^n$ [$\to$~\ref{introsemialg}, \ref{introsn}] and by
$\mathcal C:=\mathcal C_{(A,T)}$ the Boolean algebra of all constructible subsets of $\sper(A,T)$
[$\to$ \ref{introconstructible}]. Consider again the isomorphisms of Boolean algebras
\[\slim\colon\mathcal C\to\mathcal S, C\mapsto\{x\in R^n\mid P_x\in C\}\]
and $\fatten:=\slim^{-1}$ [$\to$ \ref{slimfatten}].

\begin{thm}\label{fatclo}
Let $S\in\mathcal S$. Then $\fatten(S)$ is the closure of $\{P_x\mid x\in S\}$ in $\sper(A,T)$
(or equivalently in $\sper A$ \emph{[$\to$ \ref{sperpreordcomp}]}) with respect to the constructible topology.
\end{thm}

\begin{proof}[Proof (simplified by Jakob Everling)]
For the duration of this proof, we endow $\sper(A,T)$ with the constructible topology.
Since $\complement\fatten(S)\in\mathcal C$ is open, $\fatten(S)$ is closed. Because of
\[S=\slim(\fatten(S))\overset{\ref{slimfatten}}=\{x\in R^n\mid P_x\in\fatten(S)\},\] we have
$\{P_x\mid x\in S\}\subseteq\fatten(S)$ and thus
$\overline{\{P_x\mid x\in S\}}\subseteq\fatten(S)$. In order to show $\fatten(S)\subseteq
\overline{\{P_x\mid x\in S\}}$, we let $P\in\fatten(S)$ and $U\in\mathcal U_P$.
To show: $U\cap\{P_x\mid x\in S\}\ne\emptyset$. WLOG $U$ is open. WLOG $U\subseteq\fatten(S)$
(because $\fatten(S)$ is open and $P\in\fatten(S)$, one can otherwise replace $U$ by
$U\cap\fatten(S)\in\mathcal U_P)$. Since $\slim$ is an isomorphism of Boolean algebras by \ref{slimfatten}, it follows from
 $\emptyset \ne U \subseteq\fatten(S)$ that $\emptyset \ne \slim(U)\subseteq\slim(\fatten(S)) = S$.
But since $\slim(U)= \{ x\in R^n \mid P_x \in U \}$, this means that there is an $x\in S$ such that $P_x \in U$.
\end{proof}

\begin{cor}
$\{P_x\mid x\in R^n\}$ lies dense in $\sper(A,T)$ with respect to the constructible topology and thus also
with respect to the spectral topology.
\end{cor}

\begin{lem}\label{sectfat}
Let $\mathcal F$ be \alal{a filter}{an ultrafilter} in $\mathcal S$. Then
$\{\fatten(S)\mid S\in\mathcal F\}$ is \alal{a filter}{an ultrafilter} in $\mathcal C$ and
$\bigcap\{\fatten(S)\mid S\in\mathcal F\}$ is \alal{nonempty}{a singleton}.
\end{lem}

\begin{proof}
The first part follows immediately from the fact that $\fatten$ is according to \ref{slimfatten} an isomorphism
of Boolean algebras combined with the definition of \alal{a filter}{an ultrafilter} \ref{dffilter}.
Since $\fatten(S)$ is for each $S\in\mathcal F$ closed with respect to the constructible topology, it would
follow from
$\bigcap\{\fatten(S)\mid S\in\mathcal F\}=\emptyset$ together with the compactness of $\sper(A,T)$ with
respect to the constructible topology [$\to$ \ref{sperpreordcomp}] that there would be
$n\in\N$ and $S_1,\dots,S_n\in
\mathcal F$ such that $\fatten(S_1)\cap\ldots\cap\fatten(S_n)=\emptyset$, which would imply
$\fatten(S_1\cap\ldots\cap S_n)=\emptyset$ and thus $\emptyset=S_1\cap\ldots\cap S_n\in\mathcal F\
\lightning$. Hence $\bigcap\{\fatten(S)\mid S\in\mathcal F\}\ne\emptyset$.
Finally, let $\mathcal F$ and thus $\{\fatten(S)\mid S\in\mathcal F\}$ be an ultrafilter
and fix $P,Q\in\bigcap\{\fatten(S)\mid S\in\mathcal F\}$.
Assume $P\ne Q$. Since $\sper(A,T)$ is a Hausdorff space with respect to the constructible topology,
there is some $C\in\mathcal C$ such that $P\in C$ but $Q\notin C$.
Since $\{\fatten(S)\mid S\in\mathcal F\}$ is an ultrafilter in $\mathcal C$, we obtain
$C=\fatten(S)$ or $\complement C=\fatten(S)$ for some $S\in\mathcal F$.
In the first case, it follows that $Q\notin\fatten(S)\ \lightning$, in the second that $P\notin\fatten(S)\
\lightning$.
\end{proof}

\begin{lem}\label{ultra2cone}
Let $\mathcal U$ be an ultrafilter in $\mathcal S$. Then
\[P_{\mathcal U}:=\{f\in A\mid\{x\in R^n\mid f(x)\ge 0\}\in\mathcal U\}\in
\sper(A,T)\] and $\bigcap\{\fatten(S)\mid S\in\mathcal U\}=\{P_{\mathcal U}\}$.
\end{lem}

\begin{proof}
By Lemma \ref{sectfat}, there is some $Q\in\sper(A,T)$ satisfying
\[\bigcap\{\fatten(S)\mid S\in\mathcal U\}=\{Q\}.\]
We show $P_{\mathcal U}=Q$. If $f\in P_{\mathcal U}$,
then $Q\in\fatten(\{x\in R^n\mid f(x)\ge0\})$, i.e., $\widehat f(Q)\ge0$ and
hence $f\in Q$. If on the other hand $f\in A\setminus P_{\mathcal U}$,
then $\{x\in R^n\mid f(x)<0\}\in\mathcal U$ (since $\mathcal U$ is an ultrafilter)
and thus $Q\in\fatten(\{x\in R^n\mid f(x)<0\})$, i.e., $\widehat f(Q)<0$ and hence
$f\notin Q$.
\end{proof}

\begin{lem}\label{cone2ultra}
Let $P\in\sper(A,T)$. Then
\begin{align*}
\mathcal U_P:=\{S\in\mathcal S\mid~&\exists f\in\supp P:\exists m\in\N:
\exists g_1,\dots,g_m\in P\setminus-P:\\
&\{x\in R^n\mid f(x)=0,g_1(x)>0,\dots,g_m(x)>0\}\subseteq S\}
\end{align*}
is an ultrafilter in $\mathcal S$ and we have
$\{S\in\mathcal S\mid P\in\fatten(S)\}=\mathcal U_P$.
\end{lem}

\begin{proof}
Since $\{C\in\mathcal C\mid P\in C\}$ is an ultrafilter in $\mathcal C$ and
$\slim\colon\mathcal C\to\mathcal S$ is an isomorphism of Boolean algebras,
$\{S\in\mathcal S\mid P\in\fatten(S)\}$ is an ultrafilter in $\mathcal S$.
From the description of $K$-semialgebraic subsets of $R^n$ implied by
Theorem \ref{sanf}, one gets that this ultrafilter equals
\begin{align*}
\{S\in\mathcal S\mid&\exists S'\subseteq S:\exists f,g_1,\dots,g_m\in A:\\
&S'=\{x\in R^n\mid f(x)=0,g_1(x)>0,\dots,g_m(x)>0\}\et P\in\fatten(S')\}=\mathcal U_P
\end{align*}
since $\fatten$ is an isomorphism of Boolean algebras.
\end{proof}

\begin{thm}[Bröcker's ultrafilter theorem \cite{bro}]\label{broecker}
The correspondence
\[
\begin{array}{rcll}
\mathcal U&\mapsto&P_{\mathcal U}&[\to \ref{ultra2cone}]\\
\mathcal U_P&\mapsfrom&P&[\to \ref{cone2ultra}]
\end{array}
\]
defines a bijection between the set of ultrafilters in $\mathcal S$ and $\sper(A,T)$.
\end{thm}

\begin{proof}
To show: (a) If $\mathcal U$ is an ultrafilter in $\mathcal S$, then
$\mathcal U=\mathcal U_{P_{\mathcal U}}$.

(b) If $P\in\sper(A,T)$, then $P=P_{{\mathcal U}_P}$.

\medskip
In order to show (a), we let $\mathcal U$ be an ultrafilter in $\mathcal S$. By 
\ref{cone2ultra}, we have to show that $\{S\in\mathcal S\mid
P_{\mathcal U}\in\fatten(S)\}=\mathcal U$. Since $\fatten$ is an isomorphism of
Boolean algebras by \ref{slimfatten}, $\{S\in\mathcal S\mid P_{\mathcal U}\in
\fatten(S)\}$ is a filter in $\mathcal S$. Since $\mathcal U$ is a maximal filter in
$\mathcal S$ [$\to$ \ref{maxultra}], it suffices to show that
$\mathcal U\subseteq\{S\in\mathcal S\mid P_{\mathcal U}\in\fatten(S)\}$.
To this end, let $S\in\mathcal U$. Then $\{P_{\mathcal U}\}\subseteq
\fatten(S)$ by \ref{ultra2cone} and thus $P_{\mathcal U}\in\fatten(S)$.

\medskip
For (b), we let $P\in\sper(A,T)$. By \ref{ultra2cone},
$\bigcap\{\fatten(S)\mid S\in\mathcal U_P\}$ consists of exactly one element, namely
$P_{\mathcal U_P}$. Therefore it is enough to show $P\in\bigcap
\{\fatten(S)\mid S\in\mathcal U_P\}$. Thus fix $S\in\mathcal U_P$. By
\ref{cone2ultra}, we then obtain $P\in\fatten(S)$.
\end{proof}

\begin{pro}\label{semialgk}
Every semialgebraic subset of $R^n$ \emph{[$\to$ \ref{introsemialg}]} is even
$K$-semialgebraic.
\end{pro}

\begin{proof}
To begin with, we show that all one-element subsets of $R$ are $K$-semialgebraic.
For this, let $a\in R$. To show: $\{a\}$ is $K$-semialgebraic. Since $R|K$ is algebraic,
there is $f\in K[X]\setminus\{0\}$ with $f(a)=0$. Set $k:=\#\{x\in R\mid f(x)=0\}$
and choose $j\in\{1,\dots,k\}$ such that $a$ is the $j$-th root of $f$ when the
roots of $f$ in $R$ are arranged in increasing order with respect to the order
$\le_R$ of $R$. By applying the real quantifier elimination \ref{elim}
$k$ times, we obtain that
\[\{a\}=\{y\in R\mid\exists x_1,\dots,x_k\in R:
(x_1<_R\ldots<_Rx_k\et f(x_1)=\ldots=f(x_k)=0\et x_j=y)\}\]
is $K$-semialgebraic.
Now consider an arbitrary $p\in R[\x]$. It suffices to show that
$\{x\in R^n\mid p(x)\ge0\}$ is $K$-semialgebraic. Write
$p=\sum_{\substack{\al\in\N^n\\|\al|\le d}}a_\al\x^\al$
[$\to$ \ref{monomnotation}] with
$d:=\deg p$ and $a_\al\in R$. Since all $\{a_\al\}$ are $K$-semialgebraic by
what has already been shown, real quantifier elimination yields that
\begin{align*}
\{x\in R^n\mid p(x)\ge0\}=\Bigg\{x\in R^n\mid&\exists\text{ family }
(y_\al)_{|\al|\le d}\text{ in $R$}:\\
&\left(\biget_{|\al|\le d}y_\al\in\{a_\al\}\et\sum_{|\al|\le d}y_\al x_1^{\al_1}\dotsm
x_n^{\al_n}\ge0\right)\Bigg\}
\end{align*}
is $K$-semialgebraic.
\end{proof}

\begin{thm}\label{rcsper}
$\sper R[\x]\to\sper(A,T),\ P\mapsto P\cap A$ is bijective.
\end{thm}

\begin{proof}
Because of \ref{semialgk}, we obtain from applying
the ultrafilter theorem of Bröcker twice
(once in the special case $K=R$) that
\begin{align*}
\sper R[\x]&\to\sper(A,T)\\
\{f\in R[\x]\mid\{x\in R^n\mid f(x)\ge0\}\in\mathcal U\}&\mapsto
\{f\in A\mid\{x\in R^n\mid f(x)\ge0\}\in\mathcal U\}\\
&\text{($\mathcal U$ an ultrafilter in $\mathcal S$)}
\end{align*}
is a bijection.
\end{proof}

\begin{cor}\label{rcspercor}
$\sper R(\x)\to\sper(K(\x),\sum K_{\ge0}K(\x)^2),\ P\mapsto P\cap K(\x)$ is
bijective.
\end{cor}

\begin{proof}
In the commutative diagram
\begin{center}
\begin{tikzpicture}
  \matrix (m) [matrix of math nodes,row sep=5em,column sep=8em]
  {\sper R(\x) & \sper(K(\x),\sum K_{\ge0}K(\x)^2)\\
     \{P\in\sper R[\x]\mid\supp P=(0)\} &\{P\in\sper(A,T)\mid\supp P=(0)\}\\};
  \path[->]
    (m-1-1) edge node [left] {$P\mapsto P\cap R[\x]$} (m-2-1)
            edge node [above] {$P\mapsto P\cap K(\x)$} (m-1-2)
    (m-2-1) edge node [below] {$P\mapsto P\cap A$} (m-2-2)
    (m-1-2) edge node [right] {$P\mapsto P\cap A$} (m-2-2);
\end{tikzpicture}
\end{center}
both vertical arrows represent bijections by Proposition
\ref{sperlocalize} or by the very definition of the real spectrum \ref{introrealspectrum}. It therefore suffices to show that the lower horizontal arrow
represents a bijection. Because of the bijection from \ref{rcsper},
it therefore suffices to show that every $P\in\sper R[\x]$ with
$\supp P\ne(0)$ satisfies even $\supp(P\cap A)\ne(0)$. Thus fix
$P\in\sper R[\x]$ and $f\in\p:=\supp P$ with $f\ne0$. Since $K$ has
characteristic $0$, there exists an extension field $L$ of $K$ containing
all coefficients of $f$ such that $L|K$ is a finite Galois extension.
If $C$ denotes the algebraic closure of $R$ (and therefore of $K$),
then we can of course suppose that $L$ is a subfield of $C$.
By extending
the action of the Galois group $\Aut(L|K)$ from $L$ to $L[\x]$, we obtain
$h:=\prod_{g\in\Aut(L|K)}gf\in A\setminus\{0\}$. Clearly, $f$ divides $h$ in $L[\x]$
and therefore in $C[X]$. Translating this divisibility into a system of affine linear
equations (whose variables correspond to the coefficients of the corresponding multiplier polynomial),
we see by linear algebra that the same system must have a solution over the field $L\cap R$.
This means that $f$ divides $h$ in $(L\cap R)[\x]$ and therefore in $R[\x]$.
Since $f\in\p$ and $\p$ is an ideal in $R[\x]$, we get now
$h\in\p\cap A=\supp(P\cap A)$.
\end{proof}

\begin{thm}\label{ultrasurjective}
Let $(L,\le')$ be an ordered extension field of $(K,\le)$. Then
\[\sper\left(L[\x],\sum L_{\ge'0}L[\x]^2\right)\to\sper(A,T),\ P\mapsto P\cap A\]
is surjective.
\end{thm}

\begin{proof}
Let $\mathcal S''$ denote the Boolean algebra of all $L$-semialgebraic subsets
of $R'^n$ where $R':=\overline{(L,L_{\ge'0})}$. The Boolean algebra
$\mathcal S'\subseteq\mathcal S''$ of all $K$-semialgebraic subsets of $R'^n$
is isomorphic to $\mathcal S$ in virtue of the
$\transfer_{R,R'}\colon\mathcal S\to\mathcal S'$ [$\to$ \ref{transfer}].
Now let $Q\in\sper(A,T)$ be given. We show that there is
$P\in\sper(L[\x],\sum L_{\ge'0}L[\x]^2)$ with $Q=P\cap A$.
By \ref{cone2ultra}, $\mathcal U_Q$ is an ultrafilter in
$\mathcal S$. Since $\mathcal U_Q$ is a filter in $\mathcal S$,
\[\mathcal F:=\{S''\in\mathcal S''\mid\exists S\in\mathcal U_Q:
\transfer_{R,R'}(S)\subseteq S''\}\]
is a filter in $\mathcal S''$. Choose by \ref{makeultra} an ultrafilter
$\mathcal U$ in $\mathcal S''$ such that $\mathcal F\subseteq\mathcal U$.
By Bröcker's ultrafilter theorem \ref{broecker}, there is
$P\in\sper(L[\x],\sum L_{\ge'0}L[\x]^2)$ such that
$\mathcal U=\mathcal U_P$. We have
\begin{align*}
Q\overset{\ref{broecker}}=P_{\mathcal U_Q}&=
\{f\in A\mid\{x\in R^n\mid f(x)\ge0\}\in\mathcal U_Q\}\\
&=\{f\in A\mid\transfer_{R,R'}(\{x\in R^n\mid f(x)\ge0\})\in
\{\transfer_{R,R'}(S)\mid S\in\mathcal U_Q\}\}\\
&\overset!=\{f\in A\mid\{x\in R'^n\mid f(x)\ge''0\}\in\mathcal U\}=P_\mathcal U\cap A
=P_{\mathcal U_P}\cap A\overset{\ref{broecker}}=P\cap A
\end{align*}
where $\le''$ denotes the unique order on $R'$ and
the equality flagged with an exclamation mark follows from the claim
\[\mathcal U\cap\mathcal S'=\{\transfer_{R,R'}(S)\mid S\in\mathcal U_Q\}.\]
The inclusion ``$\supseteq$'' in this claim is trivial. The other inclusion
``$\subseteq$'' follows from the fact that
$\{\transfer_{R,R'}(S)\mid S\in\mathcal U_Q\}$ is an ultrafilter and thus a maximal
filter in $\mathcal S'$ and that $\mathcal U\cap\mathcal S'$ is a filter in
$\mathcal S'$.
\end{proof}

\section{The finiteness theorem for semialgebraic classes}

In this section, we fix a real closed field $R_0$ (in the applications, one
mostly has $R_0=\R$ or $R_0=\R_{\text{alg}}$ [$\to$ \ref{ralg}]).
Moreover, we let $\mathcal R$ denote the class of all real closed extension fields
of $R_0$ [$\to$ \ref{rcfclass}(b)] (that is the class of all real closed fields
in case $R_0=\R_{\text{alg}}$). Whoever gets vertiginous from this
[$\to$ \ref{rcfclass}(c)] can take for $\mathcal R$ a set of real closed extension
fields of $R_0$ that is sufficiently large to contain all representation fields
$R_P$ of prime cones $P\in\sper R_0[\x]$ [$\to$ \ref{realrep}]
(which we perceive as an extension fields of $R_0$ in virtue of the
representation $\rh_P$ of $P$, confer the discussion before
\ref{rnassubsetofsper}).

\begin{thm}[Finiteness theorem for semialgebraic classes]
\label{finiteness}
Let $n\in\N_0$ and $\mathcal E$ a set of $n$-ary $R_0$-semialgebraic
classes. Then the following are equivalent:
\begin{enumerate}[\normalfont(a)]
\item $\bigcup\mathcal E=\mathcal R_n$
\item $\exists k\in\N:\exists S_1,\dots,S_k\in\mathcal E:
S_1\cup\ldots\cup S_k=\mathcal R_n$.
\item $\exists k\in\N:\exists S_1,\dots,S_k\in\mathcal E:
\set_{R_0}(S_1)\cup\ldots\cup\set_{R_0}(S_k)=R_0^n$
\quad\emph{[$\to$ \ref{introsn}]}.
\end{enumerate}
\end{thm}

\begin{proof}
\underline{(b)$\iff$(c)} is clear because the setification
$\set_{R_0}\colon\mathcal S_n\to\mathcal S_{n,R_0}$ [$\to$ \ref{introsn}]
and thus also the classification
$\class_{R_0}=\set_{R_0}^{-1}\colon\mathcal S_{n,R_0}\to\mathcal S_n$ is an
isomorphism of Boolean algebras [$\to$ \ref{setification}].

\smallskip\underline{(b)$\implies$(a)} is trivial.

\smallskip
\underline{(a)$\implies$(b)}\quad Suppose that (a) holds. In the proof of \ref{slimfatten}, we
have shown that
\[\Ph\colon\mathcal S_n\to\mathcal C_{R_0[\x]},\ S\mapsto\{P\in\sper 
R_0[\x]\mid(R_P,(\rh_P(X_1),\dots,
\rh_P(X_n)))\in S\}\]
is an isomorphism of Boolean algebras. Moreover, we have
\[\bigcup\{\Ph(S)\mid S\in\mathcal E\}=\sper R_0[\x]\] by the definition of
$\Ph$.
From \ref{constrcompact}, we get the existence of $k\in\N$ and
$S_1,\dots,S_k\in\mathcal E$ satisfying
$\Ph(S_1)\cup\ldots\cup\Ph(S_k)=\sper R_0[\x]$. Since
$\Ph$ is an isomorphism, we deduce
$S_1\cup\ldots\cup S_k=\mathcal R_n$.
\end{proof}

\begin{cor}\label{finitenesscor}
Let $n\in\N_0$ and $\mathcal E$ a set of
$n$-ary $R_0$-semialgebraic classes satisfying
\[\forall S_1,S_2\in\mathcal E:
\exists S_3\in\mathcal E:S_1\cup S_2\subseteq S_3.\] Then the following
are equivalent:
\begin{enumerate}[\normalfont(a)]
\item $\bigcup\mathcal E=\mathcal R_n$
\item $S=\mathcal R_n$ for some $S\in\mathcal E$
\item $\set_{R_0}(S)=R_0^n$ for some $S\in\mathcal E$
\end{enumerate}
\end{cor}

\begin{rem}
In practice, \ref{finitenesscor} is mostly applied in the following context:
One has a certain true statement about real numbers (for example that
$\R$ is Archimedean [$\to$ \ref{archetcdef}(a)]).
Now one is interested in one of the
following questions:
\begin{enumerate}[(a)]
\item Does the statement hold for all real closed extension fields of $\R$?
(In our example: Is every real closed field extension of $\R$ Archimedean?)
\item Does the statement hold in a strengthened form
(with certain quantitative additional information, so called ``bounds'')
for every real closed extension of $\R$?
(In our example: Is there an $N\in\N$ such that we have for all real closed field
extensions $R$ of $\R$ and all $a\in R$ that $|a|\le N$?)
\item Does the statement hold in the strengthened from (that is ``with bounds'')
for the real numbers? (In our example: Is there some $N\in\N$ such that for all
$a\in\R$ one has $|a|\le N$?)
\end{enumerate}
\ref{finitenesscor} establishes under certain circumstances a connection between
these three questions. For this aim, one tries to express the statement in such a way
that for $n$ numbers a certain ``semialgebraic event'' occurs where the event is the
existence of a bound. The set of events is $\mathcal E$.
\end{rem}

\begin{ex}\label{boundex}
For $n:=1$, $R_0:=\R$ and
$\mathcal E:=\{\{(R,a)\in\mathcal R_1\mid-N\le a\le N\}\mid N\in\N\}$,
\ref{finitenesscor} says that the following are equivalent:
\begin{enumerate}[(a)]
\item For every real closed extension field $R$ of $\R$ and every $a\in R$, there
is some $N\in\N$ with $|a|\le N$, i.e., every real closed extension field $R$
of $\R$ is Archimedean.
\item There is some $N\in\N$ such that for every real closed extension field $R$
of $\R$ and every $a\in R$ we have $|a|\le N$.
\item There is some $N\in\N$ such that for every $a\in\R$ we have $|a|\le N$.
\end{enumerate}
Since (c) obviously fails, we see that (a) also fails. Thus we see (once more)
that there are non-Archimedean real closed (extension) fields (of $\R$).
\end{ex}

\begin{thm}[Existence of degree bounds for Hilbert's 17th problem]
\label{h17bound}
For all $n,d\in\N_0$, there is some $D\in\N$ such that for every real closed field
$R$ and every $f\in R[\x]_d$ \emph{[$\to$ \ref{degnot}]} with $f\ge0$ on $R^n$,
there are
$p_1,\dots,p_D\in R[\x]_D$ and $q\in R[\x]\setminus\{0\}$ with
$f=\sum_{i=1}^D\left(\frac{p_i}q\right)^2$.
\end{thm}

\begin{proof}
Let $n,d\in\N_0$. Set $N:=\dim\R[\x]_d$ and write
$\{\al\in\N_0^n\mid|\al|\le d\}=\{\al_1,\dots,\al_N\}$. Set $R_0:=\R_{\text{alg}}$
and
\[S_D:=\left\{(R,(a_1,\dots,a_N))\in\mathcal R_N~\middle|~
\begin{aligned}
&\left(\forall x\in R^n:\sum_{i=1}^Na_ix_1^{\al_{i1}}\dotsm x_n^{\al_{in}}\ge0\right)\implies\\
&\text{There are families $(b_{i\al})_{\substack{1\le i\le D\\|\al|\le D}}$
and $(c_\al)_{|\al|\le D}\ne0$ in $R$}\\
&\text{such that}\\
&\left(\sum_{i=1}^Na_i\x^{\al_i}\right)\left(\sum_{|\al|\le D}c_\al\x^\al\right)^2
=\sum_{i=1}^D\left(\sum_{|\al|\le D}b_{i\al}\x^\al\right)^2
\end{aligned}
\right\}
\]
for each $D\in\N$. Obviously, $S_D$ is for each $D\in\N$ an
$R_0$-semialgebraic class
since the polynomial identity in the last part of its specification can for example
be expressed by finitely many polynomial equations in the $a_i$, $b_{i\al}$
and $c_\al$, the requirement on the existence of the two finite families and the
quantification ``$\forall x\in R^n$'' is allowed because of the real
quantifier elimination \ref{elim}. Set $\mathcal E:=\{S_D\mid D\in\N\}$ and
observe that $\forall D_1,D_2\in\N:\exists D_3\in\N:S_{D_1}\cup S_{D_2}\subseteq
S_{D_3}$ (take $D_3:=\max\{D_1,D_2\}$). By Artin's solution to Hilbert's 17th
problem \ref{artin}, we have $\bigcup\mathcal E=\mathcal R_N$. Now
\ref{finitenesscor} yields $S_D=\mathcal R_N$ for some $D\in\N$.
\end{proof}

\begin{rem} Recently, Lombardi, Perrucci and Roy \cite{lpr}
managed to prove that
one can choose in \ref{h17bound}
\[D:=2^{2^{2^{d^{4^n}}}}.\]
We will neither use nor prove this in this lecture.
\end{rem}

\begin{dfpro}\label{ordval}
Let $(K,\le)$ be an ordered extension field of $\R$. Then
\[\O_{(K,\le)}:=B_{(K,K_{\ge0})}=\{a\in K\mid\exists N\in\N:|a|\le N\}\]
is a subring of $K$ \emph{[$\to$ \ref{arithmbounded}]}
with a single maximal ideal
\[\m_{(K,\le)}:=\left\{a\in K\mid\forall N\in\N:|a|\le\frac 1N\right\}\]
with group of units
\[\O_{(K,\le)}^\times=\O_{(K,\le)}\setminus\m_{(K,\le)}
=\left\{a\in K\mid\exists N\in\N:\frac 1N\le|a|\le N\right\}.\]
We call the elements of
$\malalal{\O_{(K,\le)}}{\m_{(K,\le)}}{K\setminus\O_{(K,\le)}}$
the \emph{\alalal{finite}{infinitesimal}{infinite}} elements of $(K,\le)$. For every
$a\in\O_{(K,\le)}$, there is exactly one $\st(a)\in\R$, called
the \emph{standard part} of $a$, such that \[a-\st(a)\in\m_{(K,\le)}.\]
The map $\O_{(K,\le)}\to\R,\ a\mapsto\st(a)$ is a ring homomorphism with
kernel $\m_{(K,\le)}$. If $a,b\in\O_{(K,\le)}$ satisfy $\st(a)<\st(b)$, then
$a<b$. The standard part $\st(p)$ of a polynomial $p\in\O_{(K,\le)}[\x]$ arises
by replacing each coefficient of $p$ by its standard part. Also
$\O_{(K,\le)}[\x]\to\R[\x],\ p\mapsto\st(p)$ is a ring homomorphism.
\end{dfpro}

\begin{proof}
The existence of the standard part follows easily from the completeness of $\R$
[$\to$~\ref{introduce-the-reals}] and its uniqueness is trivial.
The rest is also easy. We show exemplarily:
\begin{enumerate}[(a)]
\item $\st(ab)=(\st(a))(\st(b))$ for all $a,b\in\O_{(K,\le)}$
\item $\st(a)<\st(b)\implies a<b$ for all $a,b\in\O_{(K,\le)}$
\end{enumerate}

To show (a), let $a,b\in\O_{(K,\le)}$. Because of
$a-\st(a),b-\st(b)\in\m_{(K,\le)}$, we have
\begin{align*}
ab-(\st(a))(\st(b))&=
(ab-(\st(a))b)+((\st(a))b-(\st(a))(\st(b)))\\
&=(a-\st(a))b+(\st(a))(b-\st(b))\in\m_{(K,\le)}+\m_{(K,\le)}\subseteq\m_{(K,\le)}
\end{align*}

For (b), we fix again $a,b\in\O_{(K,\le)}$ with $\st(a)<\st(b)$. Choose
$N\in\N$ with \[\st(b)-\st(a)>\frac1N.\]
Then $|a-\st(a)|\le\frac1{2N}$ and $|b-\st(b)|\le\frac1{2N}$ and thus
\begin{align*}
a&=a-\st(a)+\st(a)\le|a-\st(a)|+\st(a)\le\frac1{2N}+\st(a)-\st(b)+\st(b)\\
&<\frac1{2N}-\frac1N+\st(b)=-\frac1{2N}+\st(b)-b+b\\
&\le-\frac1{2N}+|b-\st(b)|+b\le-\frac1{2N}+\frac1{2N}+b=b
\end{align*}
\end{proof}

\begin{ex}[Nonexistence of degree bounds for Schmüdgen's
Positivstellensatz {[$\to$ \ref{schmuedgenpositivstellensatz}]}]\label{nonexdegbounds}
For every $\ep\in\R_{>0}$, we have $X+\ep>0$ on $[0,1]$ so that
Schmüdgen's Positivstellensatz
\ref{schmuedgenpositivstellensatz} together with \ref{so2s}(b) yields
$p_1,p_2,q_1,q_2\in\R[X]$ such that
\[(*)\qquad X+\ep=p_1^2+p_2^2+(q_1^2+q_2^2)X^3(1-X).
\]
One can ask the question
if there is in analogy to \ref{h17bound} a $D\in\N$ such that
for all $\ep\in\R_{>0}$ there are $p_1,p_2,q_1,q_2\in\R[X]_D$ satisfying
$(*)$. To this end, consider for each $D\in\N$
\[S_D:=\left\{(R,\ep)\in\mathcal R_1~\middle|~
\begin{aligned}
&\ep>0\implies
\exists b_0,\dots,b_D,b_0',\dots,b_D',c_0,\dots,c_D,c_0',\dots,c_D'\in R:
\\
&X+\ep=
\left(\sum_{i=0}^Db_iX^i\right)^2+\left(\sum_{i=0}^Db_i'X^i\right)^2+\\
&\qquad\qquad\qquad\qquad\left(\left(\sum_{i=0}^Dc_iX^i\right)^2+\left(\sum_{i=0}^Dc_i'X^i\right)^2\right)
X^3(1-X)
\end{aligned}
\right\}
\]
As in the proof of \ref{h17bound}, one shows that $S_D$ is for each $D\in\N$
an $\R$-semialgebraic class. Set $\mathcal E:=\{S_D\mid D\in\N\}$. We claim
that the answer to the above question is no. Assume it would be yes.
Then $\set_\R(S_D)=\R$ for some $D\in\N$ and thus
$\bigcup\mathcal E=\mathcal R_1$ by \ref{finitenesscor}.
Choose a non-Archimedean real closed extension field $R$ of $\R$ and
an $\ep>0$ which is infinitesimal in $R$. Then there are
$p_1,p_2,q_1,q_2\in R[X]$ satisfying $(*)$. It suffices to show that all
coefficients of these four polynomials are finite in $R$ [$\to$ \ref{ordval}] since
then
$X=\st(X+\ep)=(\st(p_1))^2+(\st(p_2))^2+((\st(q_1))^2+(\st(q_2))^2)X^3(1-X)$
in contradiction to \ref{needep}. It therefore suffices to show that
the coefficient $c$ of biggest absolute value among all coefficients
of the four polynomials is finite. Assume it were infinite. Then $\frac 1c$ would
be infinitesimal and
\begin{align*}
0&=\st\left(\frac{X+\ep}{c^2}\right)=\st\left(\left(\frac{p_1}c\right)^2+
\left(\frac{p_2}c\right)^2+\left(\left(\frac{q_1}c\right)^2+\left(\frac{q_2}c\right)^2
\right)X^3(1-X)\right)\\
&=\Big(\underbrace{\st\left(\frac{p_1}c\right)}_{\widetilde p_1}\Big)^2+
\Big(\underbrace{\st\left(\frac{p_2}c\right)}_{\widetilde p_2}\Big)^2+
\Big(\Big(\underbrace{\st\left(\frac{q_1}c\right)}_{\widetilde q_1}\Big)^2+
\Big(\underbrace{\st\left(\frac{q_2}c\right)}_{\widetilde q_2}\Big)^2\Big)X^3(1-X).
\end{align*}
It follows that $\widetilde p_1=\widetilde p_2=\widetilde q_1=\widetilde q_2=0$
on $(0,1)$ and thus $\widetilde p_1=\widetilde p_2=\widetilde q_1=
\widetilde q_2=0$, contradicting the choice of $c$ $\lightning$.
\end{ex}

\begin{rem} Completely analogous to \ref{h17bound}, one can prove the existence
of degree bounds for the real Stellensätze \ref{stellensatz}, \ref{positivstellensatz}
and \ref{nichtnegativstellensatz} \emph{in the case $K=R$}.
\end{rem}

\chapter{Semialgebraic geometry}

Throughout this chapter, we let $R$ be a real closed field and $K$ a subfield of
$R$. Moreover, $\mathcal S_n$ denotes for each $n\in\N_0$ the Boolean algebra
of all $K$-semialgebraic subsets of $R^n$ [$\to$ \ref{introsn}, \ref{introsemialg}].

\section{Semialgebraic sets and functions}

\begin{reminder}{}[$\to$ \ref{sanf}, \ref{rcfclass}(a)]\label{sasanf}
Every $K$-semialgebraic
subset of $R^n$ is of the form
\[\bigcup_{i=1}^k\left\{x\in R^n\mid f_i(x)=0,g_{i1}(x)>0,\dots,g_{im}(x)>0\right\}\]
for some $k,m\in\N_0$, $f_i,g_{ij}\in K[X_1,\dots,X_n]$.
\end{reminder}

\begin{reminder}{}[$\to$ \ref{elim}]\label{saproj}
For all $n\in\N_0$ and $S\in\mathcal S_{n+1}$,
\[\{x\in R^n\mid\exists y\in R:(x,y)\in S\},\{x\in R^n\mid\forall y\in R:(x,y)\in S\}\in
\mathcal S_n.\]
\end{reminder}

\begin{df} Let $m,n\in\N_0$ and $A\subseteq R^m$. A map $f\colon A\to R^n$
is called \emph{$K$-semialgebraic} if its graph
\[\Ga_f:=\{(x,y)\in A\times R^n\mid y=f(x)\}\subseteq R^{m+n}\]
is $K$-semialgebraic. We
say ``semialgebraic'' for ``$R$-semialgebraic''.
\end{df}

\begin{rem}\label{domainsa}
The domains of $K$-semialgebraic functions are
$K$-semialgebraic. Indeed, if $A\subseteq R^m$ and $f\colon A\to R^n$
is $K$-semialgebraic, then by \ref{saproj} also
\[\{x\in R^m\mid\exists y\in R^n:(x,y)\in\Ga_f\}=A\]
is $K$-semialgebraic.
\end{rem}

\begin{df}\label{ordertop}
We equip $R$ with the \emph{order topology} which is generated
[$\to$ \ref{topreminder}(b)] by the intervals $(a,b)_R$ with $a,b\in R$
[$\to$ \ref{intervals}(b)]. Moreover, we endow $R^n$ with the corresponding
product topology [$\to$ \ref{subspaceproductspace}(b)] which is generated
according to \ref{initialtop} by the sets $\prod_{i=1}^n(a_i,b_i)_R$ with
$a_i,b_i\in R$.
\end{df}

\begin{rem}
For $R=\R$, the topology introduced in \ref{initialtop} on $R^n=\R^n$ is obviously
the usual Euclidean topology on $\R^n$.
\end{rem}

\begin{exo}\label{continfnorm}
Let $m,n\in\N_0$, $A\subseteq R^m$ and $f\colon A\to R^n$ a map.
Then $f$ is continuous [$\to$ \ref{conti}, \ref{subspaceproductspace}(a)] if and only
if
\[\forall x\in A:\forall\ep\in R_{>0}:\exists\de\in R_{>0}:\forall y\in A:
(\|x-y\|_\infty<\de\implies\|f(x)-f(y)\|_{\infty}<\ep)\]
where
\[\|x\|_\infty:=
\begin{cases}
0&\text{if $k=0$}\\
\max\{|x_1|,\dots,|x_k|\}&\text{if $k>0$}
\end{cases}
\]
for $x\in R^k$.
\end{exo}

\begin{pro}\label{fieldopsa}
The maps
\begin{align*}
R^2\to R,\ &(a,b)\mapsto a+b,\\
R^2\to R,\ &(a,b)\mapsto ab,\\
R\setminus\{0\}\to R,\ &a\mapsto a^{-1},\\
R\to R,\ &a\mapsto|a|\qquad\emph{\text{[$\to$ \ref{introabssgn}]}},\\
R_{\ge0}\to R,\ &a\mapsto\sqrt a\qquad\emph{\text{[$\to$ \ref{notremsqrt}]}}
\end{align*}
are $\Q$-semialgebraic and continuous.
\end{pro}

\begin{proof}
It is clear that these maps are $\Q$-semialgebraic. Because of the real
quantifier elimination \ref{elim}, the class of all real closed fields for which the
claim holds is semialgebraic [$\to$ \ref{introsemialg}]. Since the claim is known to
hold for $R=\R$, it holds also for all real closed fields
[$\to$ \ref{nothingorall}].
\end{proof}

\begin{cor}\label{polycont}
Polynomial maps $R^m\to R^n$ are continuous.
\end{cor}

\begin{cor}\label{normcont}
$R^n\to R,\ x\mapsto\|x\|:=\|x\|_2:=\sqrt{x_1^2+\ldots+x_n^2}$ is continuous.
\end{cor}

\begin{rem}\label{inf22}
Because of \ref{normcont} and \ref{continfnorm}, there is to every
$\ep\in R_{>0}$ some $\de\in R_{>0}$ such that
$\forall x\in R^n:(\|x\|_\infty<\de\implies\|x\|<\ep)$.
On the other hand, $\|x\|_\infty\le\|x\|$ for all $x\in R^n$.
It follows that the topology on $R^n$ is also generated by the open balls
$\{x\in R^n\mid\|x-y\|<\ep\}$ ($y\in R^n,\ep>0$) and that \ref{continfnorm}
holds also with $\|.\|$ instead of $\|.\|_\infty$.
\end{rem}

\begin{rem}\label{toprn}
\begin{enumerate}[(a)]
\item By \ref{polycont}, $R^n$ is obviously endowed with the
initial topology with respect to all maps $R^n\to R,\ x\mapsto p(x)$
$(p\in R[\x]$) [$\to$ \ref{initialtop}].
\item Because of (a), the topology on $R^n$ is obviously generated by the sets
\[\{x\in R^n\mid p(x)>0\}\qquad(p\in R[\x]).\]
\item Viewing $R^n$ in virtue of the injective map
\[R^n\to\sper R[\x],\ x\mapsto P_x=\{f\in R[\x]\mid f(x)\ge0\}\]
as a subset of $\sper R[\x]$, the topology on $R^n$ is due to (b) induced
by the spectral topology [$\to$ \ref{spectraltop}]  on $\sper R[\x]$.
\end{enumerate}
\end{rem}

\begin{thm}\label{saclo}
\begin{enumerate}[\normalfont(a)]
\item If $A\subseteq R^m$ and $f\colon A\to R^n$ is $K$-semialgebraic, then
$f(B)\in\mathcal S_n$ for all $B\in\mathcal S_m$ with $B\subseteq A$
and $f^{-1}(C)\in\mathcal S_m$ for all $C\in\mathcal S_n$.
\item If $A\subseteq R^\ell$, $B\subseteq R^m$, $f\colon A\to B$ and
$g\colon B\to R^n$ are $K$-semialgebraic, then $g\circ f\colon A\to R^n$
is again $K$-semialgebraic.
\item If $A\in\mathcal S_n$, then the $K$-semialgebraic functions
$A\to R$ form a subring of the ring $R^A$ of all functions $A\to R$.
\end{enumerate}
\end{thm}

\begin{proof}
\begin{enumerate}[(a)]
\item Let $A\subseteq R^m$ and $f\colon A\to R^n$ be $K$-semialgebraic. By
\ref{saproj}, with $\Ga_f$ also
$f(B)=\{y\in R^n\mid\exists x\in R^m:(x\in B\et (x,y)\in\Ga_f)\}$ is
for all $B\in\mathcal S_m$ with
$B\subseteq A$ $K$-semialgebraic, and
$f^{-1}(C)=\{x\in R^m\mid\exists y\in R^n:(y\in C\et (x,y)\in\Ga_f)\}$ is for all
$C\in\mathcal S_n$ also $K$-semialgebraic.
\item Suppose $A\subseteq R^\ell$, $B\subseteq R^m$ and $f\colon A\to B$
as well as $g\colon B\to R^n$ are $K$-semialgebraic.
Then $\Ga_f\in\mathcal S_{\ell+m}$ and $\Ga_g\in\mathcal S_{m+n}$ and thus
\[\Ga_{g\circ f}=\{(x,z)\in A\times R^n\mid\exists y\in R^m:((x,y)\in\Ga_f\et(y,z)\in
\Ga_g)\}\in\mathcal S_{\ell+n}.\]
Hence $g\circ f$ is $K$-semialgebraic.
\item If $A\in\mathcal S_n$ and $f_1,f_2\colon A\to R$ are $K$-semialgebraic,
then also \[A\to R^2,\ x\mapsto(f_1(x),f_2(x))\]
is $K$-semialgebraic. Now apply
\ref{fieldopsa} and (b).
\end{enumerate}
\end{proof}

\begin{ex}
If $R$ is a non-Archimedean (real closed) extension of $\R$,
then $[0,1]_R$ is not compact
[$\to$ \ref{dfcomp}]. Indeed, if $\ep\in\m_R$ [$\to$ \ref{ordval}]
with $\ep>0$, then
\[[0,1]_R\subseteq\bigcup_{a\in[0,1]_R}(a-\ep,a+\ep)_R,\]
but there is no $N\in\N$ and $a_1,\dots,a_N\in[0,1]_R$ with
$[0,1]_R\subseteq\bigcup_{k=1}^N(a_k-\ep,a_k+\ep)_R$
(for otherwise $[0,1]_\R=\st([0,1]_R)\subseteq\{\st(a_1),\dots,\st(a_N)\}\ \lightning$).
\end{ex}

\begin{df}\label{dfsacompact}
Let $A\subseteq R^n$. We call $A$ \emph{bounded} if there is
$b\in R$ with $\|x\|\le b$ for all $x\in A$ [$\to$ \ref{normcont}]. Moreover, $A$ is called
\emph{$K$-semialgebraically compact} if $A\in\mathcal S_n$ and $A$
is bounded and closed. We simply say ``semialgebraically compact'' instead
of ``$R$-semialgebraically compact''.
\end{df}

\begin{rem}\label{sacompreals}
From analysis, one knows for $R=\R$: A $K$-semialgebraic set
$A\subseteq\R^n$ is compact if and only if it is $K$-semialgebraically compact.
\end{rem}

\begin{pro}\label{ksemalgcomp}
Let $A\in\mathcal S_n$. Then the following are equivalent:
\begin{enumerate}[\normalfont(a)]
\item $A$ is bounded.
\item $\exists b\in R:\forall x\in A:\|x\|\le b$
\item $\exists b\in R:\forall x\in A:\|x\|_\infty\le b$
\item $\exists b\in K:\forall x\in A:\|x\|\le b$
\item $\exists b\in K:\forall x\in A:\|x\|_\infty\le b$
\end{enumerate}
\end{pro}

\begin{proof}
WLOG $A\ne\emptyset$. We have
(a)$\overset{\ref{dfsacompact}}\iff$(b)$\overset{\ref{inf22}}\iff$(c) $\Longleftarrow$ (e) $\Longleftarrow$ (d).
It remains to show (b)$\implies$(d). Suppose therefore that (b) holds.
The set
\[S:=\{\|x\|\mid x\in A\}\subseteq R_{\ge0}\]
is $K$-semialgebraic [$\to$ \ref{saproj}].
Hence $S$ can be defined by finitely many polynomials [$\to$ \ref{sanf}] with coefficients in $K$ and
by Lemma \ref{sgnbounds}(a) we find some $b\in K_{>1}$ such that each of these polynomials
has constant sign on the interval $(b,\infty)_R$. Then either $(b,\infty)_R\cap S=\emptyset$ or
$(b,\infty)_R\subseteq S$. But the latter is impossible due to (b). Hence $\forall x\in A:\|x\|\le b$.
\end{proof}

\begin{thm}\label{compactimage}
Let $A\subseteq R^m$ and suppose $f\colon A\to R^n$ is
$K$-semialgebraic and continuous. Then for every $K$-semialgebraically compact
set $B\subseteq A$, the set $f(B)$ is also $K$-semialgebraically compact.
\end{thm}

\begin{proof}
If $B\in\mathcal S_m$ with $B\subseteq A$, then $f(B)\in\mathcal S_n$ by
\ref{saclo}(a) since $f$ is $K$-semialgebraic. For the rest of the claim we
can suppose that $K=R$. We fix a ``complexity bound'' $N\in\N$ and fix
$m,n\in\N_0$ but no longer fix $A$ and $f$. By \ref{sasanf}, it suffices to show
the following:

$(*)$ For all $f_1,\dots,f_N,g_{11},g_{12},\dots,g_{NN}\in
R[X_1,\dots,X_m,Y_1,\dots,Y_n]_N$ and\\
$\widetilde f_1,\dots,\widetilde f_N,\widetilde g_{11},\widetilde g_{12},\dots,
\widetilde g_{NN}\in
R[X_1,\dots,X_m]_N$, if we set
\begin{align*}
\Ga&:=\bigcup_{i=1}^N\{(x,y)\in R^m\times R^n\mid f_i(x,y)=0,g_{i1}(x,y)>0,\dots,
g_{iN}(x,y)>0\},\\
A&:=\{x\in R^m\mid\exists y\in R^n:(x,y)\in\Ga\}\text{ and}\\
B&:=\bigcup_{i=1}^N\{x\in R^m\mid\widetilde f_i(x)=0,\widetilde g_{i1}(x)>0,
\ldots,\widetilde g_{iN}(x)>0\},
\end{align*}
then
\begin{itemize}
\item $\Ga$ is not the graph of a continuous function from $A$ to $R^n$ or
\item $B$ is not a subset of $A$ or
\item $B$ is not closed in $R^m$ or
\item $B$ is not bounded in $R^m$ or
\item $\{y\in R^n\mid\exists x\in R^m:(x\in B\et(x,y)\in\Ga)\}$ is closed and
bounded in $R^n$.
\end{itemize}
We now in addition no longer fix $R$. One can easily figure out why the class
of all real closed fields $R$ for which $(*)$ holds is semialgebraic. For this aim,
one applies many times the real quantifier elimination \ref{elim}, for example for 
introducing the finitely many coefficients of the $f_i,g_{ij},\widetilde f_i,
\widetilde g_{ij}$ by universal quantifiers. By \ref{nothingorall}, $(*)$ now holds
either for all or for no real closed field $R$. Therefore it is enough to show
$(*)$ for $R=\R$. But we know this from analysis due to \ref{sacompreals}.
\end{proof}

\begin{exo}\label{dim1}
\begin{enumerate}[(a)]
\item The \alal{open}{closed} semialgebraic subsets of $R$ are exactly the
finite unions of pairwise disjoint sets of the form
$\malal{(-\infty,\infty)_R,(-\infty,a)_R,(a,\infty)_R\text{ and }(a,b)_R}
{(-\infty,\infty)_R,(-\infty,a]_R,[a,\infty)_R\text{ and }[a,b]_R}$
with $a,b\in R$.
\item The semialgebraically compact
subsets of $R$ are exactly the finite unions of pairwise disjoint
sets of the form $[a,b]_R$ with $a,b\in R$.
\end{enumerate}
\end{exo}

\section{The \L ojasiewicz inequality}

\begin{pro}\label{polgrowth}
Let $a\in K$ and suppose
$h\colon(a,\infty)_R\to R$ is $K$-semialgebraic. Then there is
$b\in K\cap[a,\infty)_R$ and $N\in\N$ such that
$|h(x)|\le x^N$ for all $x\in(b,\infty)_R$.
\end{pro}

\begin{proof}
Using \ref{sasanf}, we write
\[\Ga_h=\bigcup_{i=1}^k\{(x,y)\in R^2\mid f_i(x,y)=0,g_{i1}(x,y)>0,\ldots,
g_{im}(x,y)>0\}\]
with $k,m\in\N_0$ and $f_i,g_{ij}\in K[X,Y]$ where we suppose each of the $k$
sets contributing to this union to be nonempty. We must have $k>0$ and
$\deg_Yf_i>0$ for all $i\in\{1,\dots,k\}$
(for otherwise there would be $x,c,d\in R$ with $c<d$ and $\{x\}\times
(c,d)_R\subseteq\Ga_h$ which is impossible since $\Ga_h$ is the graph of a
function). Write $\prod_{i=1}^kf_i=\sum_{i=0}^dp_iY^i$ with
$d>0$, $p_0,\ldots,p_d\in K[X]$ and $p_d\ne0$.
By rescaling one of the $f_i$ if necessary, we can suppose that the leading coefficient of $p_d$ is
greater than $1$. Choose $c\in K\cap[a,\infty)_R$ such
that $p_d>1$ on $(c,\infty)_R$ [$\to$ \ref{sgnbounds}(a)]. Because of
$\sum_{i=0}^dp_i(x)h(x)^i=0$ and $p_d(x)\ne0$ for all $x\in(c,\infty)_R$, we
have
\[|h(x)|\le\max\left\{1,\frac{|p_0(x)|+\ldots+|p_{d-1}(x)|}{|p_d(x)|}
\right\}\le1+|p_0(x)|+\ldots+|p_{d-1}(x)|
\]
for all $x\in(c,\infty)_R$ [$\to$ \ref{sgnbounds}(a)].
Now the existence of $b$ is easy to see.
\end{proof}

\begin{thm}[\L ojasiewicz inequality]\label{lojasiewicz}
Let $n\in\N_0$ and
suppose $A\subseteq R^n$ is $K$-semialgebraically compact
and $f,g\colon A\to R$ are continuous $K$-semialgebraic functions satisfying
\[\forall x\in A:(f(x)=0\implies g(x)=0).\] Then there is $N\in\N$ and $C\in K_{\ge0}$
such that \[\forall x\in A:|g(x)|^N\le C|f(x)|.\]
\end{thm}

\begin{proof}
With $A$ also $A_t:=\{x\in A\mid|g(x)|=\frac1t\}$ is $K$-semialgebraically compact
for each $t\in R_{>0}$. Set $I:=\{t\in R_{>0}\mid A_t\ne\emptyset\}$. For each
$t\in I$, \[f_t:=\min\{|f(x)|\mid x\in A_t\}\]
exists by \ref{compactimage} and \ref{dim1}(b).
Apparently, we have to show that there exist
$N\in\N$ and $C\in K_{\ge0}$ such that
$\forall t\in I:\left(\frac1t\right)^N\le Cf_t$.
By hypothesis, we have $f_t>0$ for all $t\in I$. Furthermore,
\[R_{>0}\to R,\ t\mapsto\begin{cases}0&\text{if $t\notin I$}\\
\frac1{f_t}&\text{if $t\in I$}
\end{cases}\]
is $K$-semialgebraic. Thus, by \ref{polgrowth} there are $b\in K_{>0}$ and
$N\in\N$ such that
\[(*)\qquad\frac1{f_t}\le t^N\]
for all $t\in I\cap(b,\infty)_R$. Since
\[B:=\left\{x\in A\mid|g(x)|\ge\frac1b\right\}=\bigcup_{t\in I\cap(0,b]_R}A_t\]
is $K$-semialgebraically compact, we can choose according to \ref{compactimage} and \ref{ksemalgcomp}
some $C\in K_{\ge1}$ satisfying
\[\frac{|g(x)|^N}{|f(x)|}\le C\]
for all $x\in B$ (note that $f(x)\ne0$ for all $x\in B$).
We deduce
\[(**)\qquad\frac1{f_t}\le Ct^N\]
for all $t\in I\cap(0,b]_R$. Together with $(*)$, we obtain $(**)$ even
for all $t\in I$ as desired.
\end{proof}

\begin{lem}(``shrinking map'', in German: ``Schränkungstranformation'')
\label{shrink}
Let $n\in\N_0$,
$B:=\{x\in R^n\mid\|x\|<1\}$ and $S:=\{x\in R^n\mid\|x\|=1\}$.
The maps
\begin{align*}
\ph\colon R^n\to B,\ &x\mapsto\frac x{\sqrt{1+\|x\|^2}}\qquad\text{and}\\
\ps\colon B\to R^n,\ &y\mapsto\frac y{\sqrt{1-\|y\|^2}}
\end{align*}
are $\Q$-semialgebraic, continuous and inverse to each other. For all
$A\in\mathcal S_n$, we have
\[\text{$A$ closed}\iff\text{$\ph(A)\cup S$ is $K$-semialgebraically compact.}\]
\end{lem}

\begin{proof}
From \ref{fieldopsa}, the $\Q$-semialgebraicity and the continuity are clear.
For all $x\in R^n$, we have
\[\ps(\ph(x))=\frac{\frac x{\sqrt{1+\|x\|^2}}}{\sqrt{1-\frac{\|x\|^2}{1+\|x\|^2}}}=
\frac{\frac{x}{\sqrt{1+\|x\|^2}}}{\frac1{\sqrt{1+\|x\|^2}}}=x.\]
For all $y\in B$, we have
\[\ph(\ps(y))=\frac{\frac y{\sqrt{1-\|y\|^2}}}{\sqrt{1+\frac{\|y\|^2}{1-\|y\|^2}}}=
\frac{\frac y{\sqrt{1-\|y\|^2}}}{\sqrt{\frac1{1-\|y\|^2}}}=y.\]
Now let $A\in\mathcal S_n$. To show:
$A$ closed $\iff$ $\ph(A)\cup S$ closed.

\smallskip
``$\Longleftarrow$'' Suppose $\ph(A)\cup S$ is closed. Then
$\ph(A)=(\ph(A)\cup S)\cap B$ is closed in $B$ (with respect to the
topology induced from $R^n$) and thus also $A=\ph^{-1}(\ph(A))$ in $R^n$.

\smallskip
``$\Longrightarrow$'' Let $A$ be closed. Then $\ph(A)=\ps^{-1}(A)$ is closed in $B$
and hence $\ph(A)=C\cap B$ for some closed set $C\subseteq R^n$.
WLOG $C\subseteq B\cup S$ (otherwise replace $C$ by $C\cap(B\cup S)$).
WLOG $S\subseteq C$ (otherwise replace $C$ by $C\cup S$).
Now $\ph(A)\cup S\subseteq C\subseteq(C\cap B)\cup(C\cap S)=
\ph(A)\cup S$. Hence $\ph(A)\cup S=C$ is closed. 
\end{proof}

\begin{cor}\label{lojasiewiczcor}
Let $n\in\N_0$ and suppose that $A\subseteq R^n$ is
closed and
$f,g\colon A\to R$ are continuous $K$-semialgebraic functions satisfying
\[\forall x\in A:(f(x)=0\implies g(x)=0).\]
Then there are $N,k\in\N$ and
$C\in K_{\ge0}$ such that
\[\forall x\in A:|g(x)|^N\le C(1+\|x\|^2)^k|f(x)|.\]
\end{cor}

\begin{proof}
By \ref{domainsa}, $A$ is $K$-semialgebraic.
If $A$ is bounded, then $A$ is $K$-semialgebraically compact and
the claim follows (with $k:=1$) from the
\L ojasiewicz inequality \ref{lojasiewicz}. Now suppose that $A$ is unbounded.
Since $\{\|x\|\mid x\in A\}\subseteq R$ is $K$-semialgebraic, there is then some
$a\in K$ such that $(a,\infty)_R\subseteq\{\|x\|\mid x\in A\}$. The functions
\begin{align*}
\mathring f\colon(a,\infty)_R\to R,\ &t\mapsto\max\{|f(x)|\mid x\in A,\|x\|=t\}
\qquad\text{and}\\
\mathring g\colon(a,\infty)_R\to R,\ &t\mapsto\max\{|g(x)|\mid x\in A,\|x\|=t\}
\end{align*}
are semialgebraic. By \ref{polgrowth}, there are $b\in K\cap[a,\infty)_R$ with $b\ge1$
and $\ell\in\N$ such that $\mathring f(t)\le(1+t^2)^\ell$ and $\mathring g(t)\le(1+t^2)^\ell$ for all $t\in(b,\infty)_R
\subseteq R_{\ge1}$. Now consider the continuous $K$-semialgebraic functions
\[f_0\colon A\to R,\ x\mapsto\frac{f(x)}{(1+\|x\|^2)^{\ell+1}}\qquad\text{and}\qquad
g_0\colon A\to R,\ x\mapsto\frac{g(x)}{(1+\|x\|^2)^{\ell+1}}.\]
We have $\forall x\in A:(f_0(x)=0\implies g_0(x)=0)$ and obviously it is enough
to show that there are $N\in\N$ and $C\in K_{\ge0}$ such that
$\forall x\in A:|g_0(x)|^N\le C|f_0(x)|$ (set then
$k:=\max\{1,(N-1)(\ell+1)\}$). The advantage of $f_0$ and $g_0$ over $f$ and $g$
is that there is for all $\ep\in R_{>0}$ a semialgebraically compact set
$B\subseteq A$ such that $|f_0(x)|<\ep$ and $|g_0(x)|<\ep$ for all
$x\in A\setminus B$. With the notation of Lemma \ref{shrink}, the $K$-semialgebraic
functions
\begin{align*}
\widetilde f\colon\ph(A)\cup S\to R,\ &y\mapsto
\begin{cases}
0&\text{if $y\in S$}\\
f_0(\ps(y))&\text{if $y\in\ph(A)$}
\end{cases}
\qquad\text{and}\\
\widetilde g\colon\ph(A)\cup S\to R,\ &y\mapsto
\begin{cases}
0&\text{if $y\in S$}\\
g_0(\ps(y))&\text{if $y\in\ph(A)$}
\end{cases}
\end{align*}
are continuous.
For example, for $\widetilde f$ one sees this as follows:
Since $f_0\circ\ps|_{\ph(A)}$ is continuous and $\ph(A)=(\ph(A)\cup S)\cap B$
is open in $\ph(A)\cup S$, it suffices to show by \ref{continfnorm} and
\ref{inf22} that
\[\forall y_0\in S:\forall\ep\in R_{>0}:\exists\de\in R_{>0}:\forall y\in\ph(A):
(\|y_0-y\|<\de\implies|f_0(\ps(y))|<\ep).\]
To this end, let $y_0\in S$ and $\ep\in R_{>0}$. Choose a semialgebraically
compact set $B\subseteq A$ with $|f_0(x)|<\ep$ for all $x\in A\setminus B$.
Then $\ph(B)$ is semialgebraically compact by \ref{compactimage} and
consequently $S\cup\ph(A\setminus B)=(S\cup\ph(A))\setminus\ph(B)$
is open in $\ph(A)\cup S$. Thus there is $\de\in R_{>0}$ with
$\{y\in\ph(A)\cup S\mid\|y_0-y\|<\de\}\subseteq S\cup\ph(A\setminus B)$, i.e.,
\[\{y\in\ph(A)\mid\|y_0-y\|<\de\}\subseteq\ph(A\setminus B).\] Now let
$y\in\ph(A)$ with $\|y_0-y\|<\de$. Then $y\in\ph(A\setminus B)$ and
thus $\ps(y)\in A\setminus B$. Hence $|f_0(\ps(y))|<\ep$. This shows the
continuity of $\widetilde f$. For all $y\in\ph(A)$, we have obviously
\[\widetilde f(y)=0\implies f_0(\ps(y))=0\implies g_0(\ps(y))=0\implies\widetilde
g(y)=0.\]
Altogether, $\forall y\in\ph(A)\cup S:(\widetilde f(y)=0\implies\widetilde g(y)=0)$.
Since $\ph(A)\cup S$ is $K$-semialgebraically compact by \ref{shrink},
we get from the \L ojasiewicz inequality \ref{lojasiewicz} $N\in\N$ and
$C\in R_{\ge0}$ with $\forall y\in\ph(A)\cup S:|\widetilde g(y)|^N\le C|\widetilde f
(y)|$. In particular, we obtain
$\forall y\in\ph(A):|g_0(\ps(y))|^N\le C|f_0(\ps(y))|$ which means $\forall x\in A:|g_0(x)|^N\le C|f_0(x)|$ as desired.
\end{proof}

\section{The finiteness theorem for semialgebraic sets}

\begin{df}\label{dfbasicopenclosed}
Let $n\in\N_0$. A subset $S$ of $R^n$ is called
\emph{$K$-basic \alal{open}{closed}} if there are $m\in\N_0$ and
$g_1,\dots,g_m\in K[\x]$ satisfying
$S=\{x\in R^n\mid g_1(x)\malal>\ge0,\dots,g_m(x)\malal>\ge0\}$.
\end{df}

\begin{rem}\label{basicclosedisclosed}
Every $K$-basic \alal{open}{closed} subset of $R^n$ is $K$-semialgebraic
and \alal{open}{closed} in $R^n$.
\end{rem}

\begin{thm}[Finiteness theorem for semialgebraic sets]
\label{finthm}
Let $n\in\N_0$ and $S\in\mathcal S_n$ \alal{open}{closed}. Then $S$ is a finite
union of $K$-basic \alal{open}{closed} subsets of $R^n$.
\end{thm}

\begin{proof}
\begin{align*}
&\text{$S$ is a finite union of $K$-basic open subsets of $R^n$}\\
\iff&\text{$S$ is a finite union of finite intersections of sets of the form
$\{x\in R^n\mid g(x)>0\}$}\\
&\text{\qquad\qquad\qquad\qquad\qquad\qquad\qquad
\qquad\qquad\qquad\qquad\qquad\qquad\qquad\qquad($g\in K[\x]$)}\\
\iff&\text{$\complement S$ is a finite intersection of finite unions of sets of the form
$\{x\in R^n\mid g(x)\ge0\}$}\\
&\text{\qquad\qquad\qquad\qquad\qquad\qquad\qquad
\qquad\qquad\qquad\qquad\qquad\qquad\qquad\qquad($g\in K[\x]$)}\\
\overset{\ref{unionsection}}\iff&
\text{$\complement S$ is a finite union of finite intersections of sets of the form
$\{x\in R^n\mid g(x)\ge0\}$}\\
&\text{\qquad\qquad\qquad\qquad\qquad\qquad\qquad
\qquad\qquad\qquad\qquad\qquad\qquad\qquad\qquad($g\in K[\x]$)}\\
\iff&\text{$\complement S$ is a finite union of $K$-basic closed subsets of $R^n$.}
\end{align*}
It is thus enough to show the claim for open $S$. Write
\[S=\bigcup_{i=1}^\ell\{x\in R^n\mid f_i(x)=0,g_{i1}(x)>0,\ldots,g_{im}(x)>0\}\]
according to \ref{sasanf} with $\ell,m\in\N_0$, $f_i,g_{ij}\in K[\x]$. Fix
$i\in\{1,\dots,\ell\}$. It is enough to find a $K$-basic open set $U\subseteq R^n$
such that
\[\{x\in R^n\mid f_i(x)=0,g_{i1}(x)>0,\ldots,g_{im}(x)>0\}\subseteq U
\subseteq S.\] Consider the closed set $A:=R^n\setminus S\in\mathcal S_n$
and the continuous $K$-semialgebraic functions
\begin{align*}
f\colon A\to R,&\ x\mapsto(f_i(x))^2\qquad\text{and}\\
g\colon A\to R,&\ x\mapsto\prod_{j=1}^m(|g_{ij}(x)|+g_{ij}(x)).
\end{align*}
We have $\forall x\in A:(f(x)=0\implies g(x)=0)$.
By \ref{lojasiewiczcor}, there thus exist
$N,k\in\N$ and $C\in K_{\ge0}$ such that
$\forall x\in A:|g(x)|^N\le C(1+\|x\|^2)^kf(x)$. For all $x\in A$ satisfying
$g_{i1}(x)>0,\ldots,g_{im}(x)>0$, we thus have
$(2^m\prod_{j=1}^mg_{ij}(x))^N\le C(1+\sum_{j=1}^nx_j^2)^kf_i(x)^2$.
Set
\[U:=\left\{x\in R^n\mid C\left(1+\sum_{j=1}^nx_j^2\right)^kf_i(x)^2<
\left(2^m\prod_{j=1}^mg_{ij}(x)\right)^N,g_{i1}(x)>0,\dots,g_{im}(x)>0\right\}.\]
Then $U\cap A=\emptyset$ and $\{x\in R^n\mid f_i(x)=0,g_{i1}(x)>0,\ldots,
g_{im}(x)>0\}\subseteq U\subseteq S$.
\end{proof}

\begin{ex}
The ``slashed square'' $S:=(-1,1)_R^2\setminus([0,1]_R\times\{0\})$ is
$K$-semialgebraic and open. By \ref{finthm}, it is thus a finite union
of $K$-basic open subsets of $R^2$. Indeed,
\begin{align*}
S=&\{(x,y)\in R^2\mid-1<x<1,-(y+1)y^2(y-1)>0\}\cup\\
&\left\{(x,y)\in R^2\mid\left(x+\frac12\right)^2+y^2<\left(\frac12\right)^2\right\}
\end{align*}
is a union of two $K$-basic open sets. However, $S$ is not $K$-basic open.
To show this, we assume
\[S=\{(x,y)\in R^2\mid g_1(x,y)>0,\ldots,g_m(x,y)>0\}\]
with $m\in\N_0$, $g_i\in K[X,Y]$.
For continuity reasons, we have $g_i(x,0)\ge0$ for all $x\in[0,1]_R$ and
$i\in\{1,\ldots,m\}$. Because of $([0,1]_R\times\{0\})\cap S=\emptyset$, we
have thus $[0,1]_R=\bigcup_{i=1}^m\{x\in[0,1]_R\mid g_i(x,0)=0\}$.
WLOG $\#\{x\in[0,1]_R\mid g_1(x,0)=0\}=\infty$. Then $g_1(X,0)=0$
and consequently $(R\times\{0\})\cap S=\emptyset$ in contradiction to
$(-1,0)_R\times\{0\}\subseteq S$.
\end{ex}

\begin{thm}[Abstract version of the finiteness theorem for semialgebraic sets]
\label{finthmabstract}
Let $R|K$ be algebraic, i.e., $R$ be the real closure of $(K,K\cap R^2)$.
Let $n\in\N_0$ and write $A:=K[\x]$ and $T:=\sum K_{\ge0}A^2$
so that we are in the setting described before \ref{rnassubsetofsper}.
Denote by \[\fatten\colon\mathcal S_n\to\mathcal C:=\mathcal C_{(A,T)}\]
again the fattening \emph{[$\to$ \ref{slimfatten}, \ref{fatclo}]}.
Let $S\in\mathcal S_n$.
Then
\[
S \malal{\text{open}}{\text{closed}} \text{in $R^n$}
\iff
\fatten(S)\malal{\text{open}}{\text{closed}} \text{in $\sper(A,T)$}. 
\]
\end{thm}

\begin{proof}
It is enough to show: $S$ open $\iff$ $\fatten(S)$ open.

\smallskip
``$\Longleftarrow$'' By definition of the spectral topology
[$\to$ \ref{spectraltop}], $\fatten(S)$ is a union of sets of the form
$\{P\in\sper(A,T)\mid\widehat g_1(P)>0,\ldots,\widehat g_m(P)>0\}$
($m\in\N_0,g_1,\dots,g_m\in A$). By \ref{constrcompact} and 
\ref{compactsubspace}, $\fatten(S)$ is quasicompact [$\to$ \ref{dfcomp}]
with respect to the constructible topology [$\to$ \ref{spectraltop}]. Hence
$\fatten(S)$ is a finite union of sets of the described form, i.e.,
\[(**)\qquad \fatten(S)=\bigcup_{i=1}^k\{P\in\sper(A,T)\mid\widehat g_{i1}(P)>0,\ldots,
\widehat g_{im}(P)>0\}\]
with $k,m\in\N_0$, $g_{ij}\in A$. It follows by \ref{slimfatten} that
\[(*)\qquad S=\bigcup_{i=1}^k\{x\in R^n\mid g_{i1}(x)>0,\ldots,g_{im}(x)>0\}.\]
In particular, $S$ is open.

\smallskip
``$\Longrightarrow$''
By the finiteness theorem for semialgebraic sets \ref{finthm},
we can find $k,m\in\N_0$ and $g_{ij}\in A$ such that $(*)$ holds.
It follows that $(**)$ holds.
In particular, $\fatten(S)$ is open.
\end{proof}

\begin{rem} Suppose we are in the situation of Theorem \ref{finthmabstract}.
\begin{enumerate}[(a)]
\item From the very definition of the slimming $\slim$ \ref{slimfatten}, one sees immediately that it is (unlike in general the fattening!)
compatible even with arbitrary unions instead of just finite ones: If $(C_i)_{i\in I}$ is a family of constructible subsets of $\sper(A,T)$ whose union $\bigcup_{i\in I}C_i$ is
again constructible, then
\[\slim\left(\bigcup_{i\in I}C_i\right)=\bigcup_{i\in I}\slim\left(C_i\right).\]
\item From (a) it is clear that the proof of the easy direction ``$\Longleftarrow$'' of Theorem \ref{finthmabstract} could be considerably simplified
by ignoring the quasicompactness of $\fatten(S)$ and instead replacing $(*)$ and $(**)$ by similar conditions with a possibly infinite union instead of a finite one.
\item In the special case $R=K$, the easy direction ``$\Longleftarrow$'' of Theorem \ref{finthmabstract} follows also simply from \ref{toprn}(c) which says that
\[S=\slim(\fatten(S))=\{x\in R^n\mid P_x\in\fatten(S)\}\] is open.
\item If one had already \ref{toprn2}(c) available, the easy direction ``$\Longleftarrow$'' of Theorem \ref{finthmabstract} would follow exactly as in the preceding item (c) also in the general case. However, our proof of \ref{toprn2}(c) will use Corollary \ref{finthmabstractcor} and therefore Theorem \ref{finthmabstract}.
\end{enumerate}
\end{rem}

\begin{rem}\label{finthmabstractmotivates}
The description of \ref{finthmabstract} as an abstract version of \ref{finthm}
is motivated by the fact that one can easily retrieve the latter from
the first: Note first that one can reduce in \ref{finthm} to the case
where $R|K$ is algebraic by using the transfer between $R$ and
$\overline{(K,K\cap R_{\ge0})}$ [$\to$ \ref{transfer}].
For this, one has to argue that this transfer preserves openness which can be
accomplished by real quantifier elimination \ref{elim}. Thus let now $R|K$ be
algebraic, $n\in\N_0$ and $S\in\mathcal S_n$ open (by the first part of the
proof of Theorem \ref{finthm}, it suffices to treat the case of open sets).
We have to show that $S$ is a finite union of $K$-basic open subsets of $R^n$.
As seen in the easy part ``$\Longleftarrow$'' of the proof of \ref{finthmabstract},
for this purpose, it suffices to show that $\fatten(S)$ is open. This follows from
the difficult part ``$\Longrightarrow$'' of \ref{finthmabstract}.
\end{rem}

\begin{cor}[Strengthening of \ref{rcsper}]\label{finthmabstractcor}
Let $R|K$ be algebraic, i.e.,
$R$ be the real closure of $(K,K\cap R_{\ge0})$. Let $n\in\N_0$
and write $A:=K[\x]$ and $T:=\sum K_{\ge0}A^2$. Then
\[\sper R[\x]\to\sper(A,T),\ P\mapsto P\cap A
\]
is a homeomorphism with respect to both, the spectral as well as the constructible
topology on both sides.
\end{cor}

\begin{proof}
The map is continuous with respect to both topologies by 
\ref{spercontinuity} and bijective by \ref{rcsper}. According to the definition of
a homeomorphism \ref{dfhomeo} and the definition of both topologies in
\ref{spectraltop}, it suffices to show that for all $C\in\mathcal C_{R[\x]}$ we have
$\{P\cap A\mid P\in C\}\in\mathcal C_{(A,T)}$ and that this latter set is open
in $\sper(A,T)$ whenever $C$ is open in $\sper R[\x]$. For this purpose, let
$C\in\mathcal C_{R[\x]}$. The slimming $\{x\in R^n\mid P_x\in C\}$
[$\to$ \ref{slimfatten}] of $C$ is then a semialgebraic subset of $R^n$
and thus even $K$-semialgebraic by \ref{semialgk} since $R|K$ is algebraic.
By \ref{sasanf}, we thus find $k,m\in\N_0$ and $f_i,g_{ij}\in K[\x]$ such that
\[\{x\in R^n\mid P_x\in C\}=\bigcup_{i=1}^k
\{x\in R^n\mid f_i(x)=0,g_{i1}(x)>0,\ldots,
g_{im}(x)>0\},\]
where one can even choose $f_1=\ldots=f_k=0$ by the finiteness theorem for
semialgebraic sets \ref{finthm} in the case where $C$ is open. Fattening this,
we obtain
\[C=\bigcup_{i=1}^k\{P\in\sper R[\x]\mid\widehat f_i(P)=0,\widehat g_{i1}(P)>0,
\ldots,\widehat g_{im}(P)>0\}\]
and therefore [$\to$ \ref{spercontinuity}]
\begin{multline*}
\{P\cap A\mid P\in C\}=\bigcup_{i=1}^k\{P\in\sper(A,T)\mid\widehat f_i(P)=0,
\widehat g_{i1}(P)>0,\ldots,\widehat g_{im}(P)>0\}\\
\in\mathcal C_{(A,T)}.
\end{multline*}
If $C$ is open, then so is $\{P\cap A\mid P\in C\}$ because of the choice of
$f_i=0$.
\end{proof}

\begin{rem}\label{toprn2}
In the situation of \ref{finthmabstract}, one can now generalize \ref{toprn} as follows:
\begin{enumerate}[(a)]
\item $R^n$ is equipped with the
initial topology with respect to all maps $R^n\to R,\ x\mapsto p(x)$
$(p\in A$).
\item The topology on $R^n$ is generated by the sets
$\{x\in R^n\mid p(x)>0\}$ ($p\in A$).
\item Viewing $R^n$ in virtue of the injective map [$\to$ \ref{rnassubsetofsper}]
\[R^n\to\sper A,\ x\mapsto P_x=\{f\in A\mid f(x)\ge0\}\]
as a subset of $\sper A$, the topology on $R^n$ is induced
by the spectral topology on $\sper A$ [$\to$ \ref{finthmabstractcor}].
\end{enumerate}
Indeed, (a) is again obvious.
Unlike in \ref{finthmabstract}, (b) does not immediately follow from (a) anymore.
Instead one first proves (c) by using the corresponding item from \ref{toprn} together with Corollary \ref{finthmabstractcor}.
Finally, one easily deduces (b) from (c).
\end{rem}

\chapter{Convex sets in vector spaces}

In this chapter, $K$ denotes always a subfield of $\R$ equipped with the order
and the subspace topology [$\to$ \ref{subspaceproductspace}(a)] induced by $\R$
unless otherwise specified.

\section{The isolation theorem for cones}

\begin{df}\label{defcone}
Let $V$ be a $K$-vector space. A subset $C\subseteq V$ is called a
\emph{(convex) cone} (in $V$) if $0\in C$, $C+C\subseteq C$
and $K_{\ge0}C\subseteq C$ [$\to$ \ref{divnot}]. A cone $C\subseteq V$ is called
\emph{proper} if $C\ne V$.
\end{df}

\begin{ex}\label{poiscone}
Let $T$ be a preorder [$\to$ \ref{defpreorder}] of $K[\x]$ with
$K_{\ge0}\subseteq T$. Then $T$ is a cone. Moreover, $T$ is proper as a preorder
[$\to$ \ref{preproper}] if and only if $T$ is proper as a cone.
\end{ex}

\begin{pro}
Let $V$ be a $K$-vector space
and $C\subseteq V$. Then the following are equivalent:
\begin{enumerate}[\normalfont(a)]
\item $C$ is a cone.
\item $C$ is convex \emph{[$\to$ \ref{dfconv}]}, $C\ne\emptyset$ and
$K_{\ge0}C\subseteq C$.
\end{enumerate}
\end{pro}

\begin{proof}
\underline{(a)$\implies$(b)} is trivial.

\smallskip\underline{(b)$\implies$(a)}\quad Suppose that (b) holds.
From $C\ne\emptyset$ and $0C\subseteq C$, we get $0\in C$.
To show: $C+C\subseteq C$. Let $x,y\in C$. Then $\frac x2+\frac y2\in C$
and thus $x+y=2\left(\frac x2+\frac y2\right)\in C$.
\end{proof}

\begin{df}\label{defunit}
Let $C$ be a cone in the $K$-vector space $V$ and $u\in V$.
Then $u$ is called a \emph{unit} for $C$ (in $V$) if for every $x\in V$ there is some
$N\in\N$ with $Nu+x\in C$.
\end{df}

\begin{ex}{}[$\to$ \ref{poiscone}]
Let $T$ be a preorder of $K[\x]$ with $K_{\ge0}\subseteq T$. Then $T$ is
Archimedean [$\to$ \ref{dfarch}(a)] if and only if $1$ is a unit for $T$.
\end{ex}

\begin{pro}\label{unitchar}
Let $C$ be a cone on the $K$-vector space $V$ and $u\in V$. Then the following
are equivalent:
\begin{enumerate}[\normalfont(a)]
\item $u$ is a unit for $C$.
\item $V=C-\N u$
\item $V=C-K_{\ge0}u$
\item $u\in C$ and $V=C+\Z u$
\item $u\in C$ and $V=C+Ku$
\item $\forall x\in V:\exists\ep\in K_{>0}:u+\ep x\in C$
\end{enumerate}
\end{pro}

\begin{proof}
\underline{(a)$\implies$(b)$\implies$(c)} is clear.

\smallskip\underline{(c)$\implies$(d)}\quad Suppose that (c) holds.
Then $u\in C-K_{\ge0}u$ and thus
$(1+\la)u\in C$ for some $\la\in K_{\ge0}$ and so $u\in C$. Fix now $x\in V$.
To show: $x\in C+\Z u$. Choose $\la\in K_{\ge0}$ with $x\in C-\la u$.
Choose $N\in\N$ with $\la\le N$. Then $(N-\la)u\in C$ and hence
\begin{multline*}
x=(x-(N-\la)u)+(N-\la)u\in(C-\la u-(N-\la)u)+C\\
\subseteq C-Nu\subseteq C-\N u\subseteq C+\Z u.
\end{multline*}

\smallskip\underline{(d)$\implies$(e)} is trivial.

\smallskip\underline{(e)$\implies$(f)}\quad Suppose that (e) holds and let
$x\in V$. Choose $\la\in K$ such that $x\in C-\la u$. If $\la\le0$,
then $x\in C$ and consequently $u+\ep x=u+x\in C+C\subseteq C$ with $\ep:=1$.
If $\la>0$, then set $\ep:=\frac1\la>0$. Then $u+\ep x\in\ep C\subseteq C$.

\smallskip\underline{(f)$\implies$(a)}\quad Suppose that (f) holds and let $x\in V$.
To show: $\exists N\in\N:Nu+x\in C$. Choose $\ep\in K_{>0}$ with $u+\ep x\in C$.
Choose $N\in\N$ with $\frac1\ep\le N$. From (f), it follows also that $u\in C$
and hence $(N-\frac1\ep)u\in C$. Now $Nu+x=(N-\frac1\ep)u+\frac1\ep u+x\in
C+\frac1\ep(u+\ep x)\subseteq C+\frac1\ep C\subseteq C+C\subseteq C$.
\end{proof}

\begin{cor} Let $u$ be a unit for the cone $C$ in the $K$-vector space
$V$. Then $u\in C$ and $V=C-C$.
\end{cor}

\begin{rem}\label{unitinteriorpoint}
The units for a cone in $K^n$ are exactly its interior points
[$\to$ \ref{interiorclosure}, \ref{unitchar}(f)].
\end{rem}

\begin{df}\label{defstate}
Let $V$ be a $K$-vector space, $C\subseteq V$ and $u\in V$. A \emph{state}
of $(V,C,u)$ is a $K$-linear function $\ph\colon V\to\R$ satisfying
$\ph(C)\subseteq\R_{\ge0}$ and $\ph(u)=1$. We refer to the set
$S(V,C,u)\subseteq\R^V$ of all states of $(V,C,u)$ as the \emph{state space}
of $(V,C,u)$.
\end{df}

\begin{ex}\label{badexample}
Set $K:=\R$, $V:=\R[X]$, $C:=P_\infty\in\sper\R[X]$. Then the cone $C$ does
not possess a unit in $V$ and we have $S(V,C,u)=\emptyset$ for all $u\in V$.
Indeed, let $u\in V$. Choose $d\in\N$ with $d>\deg u$. Then
$u-\ep X^d\notin C$ for all $\ep>0$. By \ref{unitchar}(f), $u$ is thus not a unit for
$C$. Assume $\ph\in S(V,C,u)$. Then $\ep\ph(X^d)-1=\ph(\ep X^d-u)\in
\ph(C)\subseteq\R_{\ge0}$ for all $\ep>0$ $\lightning$.
\end{ex}

\begin{ex}
Set $K:=\Q$, $V:=\Q^2$, $C:=\{(x,y)\in\Q^2\mid y\ge\sqrt2x\}$. All elements of
$C$ except $0$ are units for $C$ [$\to$ \ref{unitinteriorpoint}]. There is no
$\ph\in V^*\setminus\{0\}$ satisfying $\ph(C)\subseteq\Q_{\ge0}$ but for each
$u\in C\setminus\{0\}$, we have $\#S(V,C,u)=1$.
\end{ex}

\begin{lem}\label{sublinear}
Let $u$ be a unit for a proper cone $C$ in the $K$-vector space $V$. Then
\[\rh\colon V\to\R,\ x\mapsto\sup\{\la\in K\mid x-\la u\in C\}\]
is well-defined  and we have $\rh(x)+\rh(y)\le\rh(x+y)$ as well as
$\rh(\la x)=\la\rh(x)$ for all $x,y\in V$ and $\la\in K_{\ge0}$.
\end{lem}

\begin{proof}
Let $x,y\in V$ and $\la\in K_{\ge0}$. For the well-definedness of $\rh$, we have to
show that $I:=\{\la\in K\mid x-\la u\in C\}$ is nonempty and bounded from above
[$\to$ \ref{archetcdef}, \ref{introduce-the-reals}]. Since $u$ is a unit for $C$, we
have $I\ne\emptyset$ and furthermore there is $N\in\N$ such that $-x+Nu\in C$.
Then $\la<N+1$ for all $\la\in I$ since otherwise, if $\la\in I$ satisfied $\la\ge N+1$,
then
\begin{align*}
-u&=Nu-(N+1)u=(-x+Nu)+x-(N+1)u\\
&\in C+x-\la u+(\la-(N+1))u\\
&\subseteq C+C+K_{\ge0}u\subseteq C.
\end{align*}
But now $-u\notin C$ for otherwise $C\overset{\text{\ref{unitchar}(b)}}=V$.
Now choose sequences $(\la_n)_{n\in\N}$ and $(\mu_n)_{n\in\N}$ in $K$
such that $x-\la_nu,y-\mu_nu\in C$ for all $n\in\N$ and
$\lim_{n\to\infty}\la_n=\rh(x)$ as well as $\lim_{n\to\infty}\mu_n=\rh(y)$.
Then we have $(x+y)-(\la_n+\mu_n)u\in C+C\subseteq C$ and thus
$\la_n+\mu_n\le\rh(x+y)$ for all $n\in\N$. It follows that
\[\rh(x)+\rh(y)=\left(\lim_{n\to\infty}\la_n\right)+\left(\lim_{n\to\infty}\mu_n\right)
=\lim_{n\to\infty}(\la_n+\mu_n)\le\rh(x+y).\]
Moreover, $\la x-\la\la_n u\in\la C\subseteq C$ and thus $\la\la_n\le\rh(\la x)$ for
all $n\in\N$. It follows that $\la\rh(x)=\la\lim_{n\to\infty}\la_n=\lim_{n\to\infty}\la\la_n
\le\rh(\la x)$ and analogously $\frac1\la\rh(\la x)\le\rh\left(\frac1\la(\la x)\right)$
if $\la\ne0$, i.e., $\la\rh(x)=\rh(\la x)$.
\end{proof}

\begin{thm}[Isolation theorem for cones]\label{isolation}
Let $u$ be a unit for the proper cone $C$
in the $K$-vector space $V$. Then $S(V,C,u)\ne\emptyset$.
\end{thm}

\begin{proof}
Since the union of a nonempty chain of cones in $V$ is again a cone in $V$,
we can use Zorn's lemma to enlarge $C$ to a cone of $V$ that is maximal with
respect to the property of not containing $-u$. WLOG suppose that $C$ has
already this maximality property.

\medskip
\textbf{Claim 1:} $C\cup-C=V$

\smallskip
\emph{Explanation.} Let $x\in V$ with $x\notin-C$. To show: $x\in C$.
Due to the maximality of $C$ it is enough to show that the cone
$C+K_{\ge0}x$ does not contain $-u$. But if we had $-u=y+\la x$
for some $y\in C$ and $\la\in K_{\ge0}$, then $\la>0$ and
$x=\frac1\la(-u-y)\in-C$ $\lightning$.

\medskip\noindent
Consider for each $x\in V$, the sets
\[I_x:=\{\la\in K\mid x-\la u\in C\}\text{ and }J_x:=\{\la\in K\mid x-\la u\in-C\}.\]

\medskip
\textbf{Claim 2:} $\forall x\in V:\forall\la\in I_x:\forall\mu\in J_x:\la\le\mu$

\smallskip
\emph{Explanation.} Let $x\in V$, $\la\in I_x$ and $\mu\in J_x$. Then
$x-\la u\in C$ and $\mu u-x\in C$. Thus,
$(\mu-\la)u=(\mu u-x)+(x-\la u)\in C+C\subseteq C$. If we had $\mu<\la$, then
we had $-u\in C\ \lightning$.

\medskip\noindent
Consider now $\ph\colon V\to\R,\ x\mapsto\sup I_x$ [$\to$ \ref{sublinear}].

\medskip
\textbf{Claim 3:} $-\ph(x)=\sup\{\la\in K\mid x-\la(-u)\in-C\}$ for all $x\in V$

\smallskip
\emph{Explanation.} Let $x\in V$. From $I_x\cup J_x\overset{\text{Claim 1}}
=K$ and Claim 2, we get \[\ph(x)=\sup I_x=\inf J_x\] and hence
\[-\ph(x)=-\inf J_x=\sup\{-\la\mid\la\in K,x-\la u\in -C\}=
\sup\{\la\in K\mid x+\la u\in -C\}.\] 

\medskip\noindent
From \ref{sublinear}, we obtain $\ph(x)+\ph(y)\le\ph(x+y)$ and $\ph(\la x)=\la\ph(x)$
for all $x,y\in V$ and $\la\in K_{\ge0}$. Since $-u$ is a unit for the proper cone
$-C$, \ref{sublinear} and Claim 3 yield also $-\ph(x)-\ph(y)\le-\ph(x+y)$ for all
$x,y\in V$. It follows that \[\ph(x)+\ph(y)\le\ph(x+y)\le\ph(x)+\ph(y)\] and
therefore $\ph(x)+\ph(y)=\ph(x+y)$ for all $x,y\in V$. In particular,
$\ph(x)+\ph(-x)=\ph(0)=0$ and hence $\ph(-x)=-\ph(x)$ for all $x\in V$ from which
we deduce \[\ph((-\la)x)=\ph(-\la x)=-\ph(\la x)=-\la\ph(x)=(-\la)\ph(x)\]
for all $x\in V$ and $\la\in K_{\ge0}$. Altogether, $\ph(\la x)=\la\ph(x)$ for all
$x\in V$ and $\la\in K_{\ge0}\cup K_{\le0}=K$, i.e., $\ph$ is $K$-linear. Obviously,
$\ph(C)\subseteq\R_{\ge0}$ and $\ph(u)=1$. Therefore $\ph\in S(V,C,u)$.
\end{proof}

\begin{lem}\label{xelc}
Let $C$ be a cone in the $K$-vector space $V$ and $x\in V$. Then
\[x\in C\iff x\in C-K_{\ge0}x.\]
\end{lem}

\begin{proof}
``$\Longrightarrow$'' is trivial.

\smallskip
``$\Longleftarrow$'' Let $x\in C-K_{\ge0}x$, for instance $x=y-\la x$ with
$y\in C$ and $\la\in K_{\ge0}$. Then \[x=\frac1{1+\la}y\in C.\]
\end{proof}

\begin{cor}\label{conemembership}
Suppose $u$ is a unit for the cone $C$ in the $K$-vector space $V$
and $x\in V$. If $\ph(x)>0$ for all $\ph\in S(V,C,u)$, then $x\in C$.
\end{cor}

\begin{proof}
Suppose $x\notin C$. To show: $\exists\ph\in S(V,C,u):\ph(x)\le0$.
By \ref{xelc}, the cone $C-K_{\ge0}x$ is proper. Since $u$ is a unit for $C$,
it is of course also a unit for $C-K_{\ge0}x$. By the isolation theorem
\ref{isolation}, there is $\ph\in S(V,C-K_{\ge0}x,u)$. We have $\ph\in S(V,C,u)$ and
$\ph(x)\le0$.
\end{proof}

\begin{exo}{}[$\to$ \ref{defstate}]\label{statespacetop}
Let $V$ be a $K$-vector space, $C\subseteq V$
and $u\in V$. We equip the $\R$-vector space $\R^V$ of all functions from $V$ to
$\R$ with the product topology [$\to$ \ref{subspaceproductspace}(b)]. Then
$S(V,C,u)$ is a closed convex subset of $\R^V$ which we equip with the
subspace topology [$\to$ \ref{subspaceproductspace}(a)]. Using \ref{initialtop}, one shows that this topology is
at the same time also the initial topology with
respect to the functions \[S(V,C,u)\to\R,\ \ph\mapsto\ph(x)\qquad(x\in V).\]
\end{exo}

\begin{thm}\label{statespacecompact}
Let $u$ be a unit for the cone $C$ in the $K$-vector space $V$.
Then the state space $S(V,C,u)$ is compact \emph{[$\to$ \ref{dfcomp}]}.
\end{thm}

\begin{proof} Choose for each $x\in V$ an $N_x\in\N$ such that
$\pm x+N_xu\in C$. Then we have for all $\ph\in S(V,C,u)$ and $x\in V$
that $\pm\ph(x)+N_x=\ph(\pm x+N_xu)\ge0$ and thus \[\ph(x)\in[-N_x,N_x].\]
Thus $S(V,C,u)\subseteq\prod_{x\in V}[-N_x,N_x]$. From analysis
(cf. \ref{sacompreals}) and Tikhonov's theorem \ref{tikhonov},
$\prod_{x\in V}[-N_x,N_x]$ is compact with respect to the product topology.
But the product topology on $\prod_{x\in V}[-N_x,N_x]$ is induced
by the topology of $\R^V$ [$\to$ \ref{inducedproductcommute}]. By \ref{statespacetop}, $S(V,C,u)$ is thus closed
in the compact space $\prod_{x\in V}[-N_x,N_x]$ and hence is compact itself
[$\to$ \ref{compactsubspace}].
\end{proof}

\begin{exo}\label{quasicompactimage}
Let $M$ and $N$ be topological spaces and $f\colon M\to N$
be continuous. If $M$ is quasicompact [$\to$ \ref{dfcomp}], then so is $f(M)$
[$\to$ \ref{compactsubspace}]
\end{exo}

\begin{cor}\label{takeson}
Let $M$ be a nonempty quasicompact topological space and $f\colon M\to\R$
be continuous. Then $f$ takes on a minimum and a maximum, i.e., there are
$x,y\in M$ with \[f(x)\le f(z)\le f(y)\] for all $z\in M$.
\end{cor}

\begin{proof}
$f(M)$ is compact by \ref{quasicompactimage}. Hence $f(M)$ is nonempty,
bounded and closed. From the first two properties, it follows that $\inf f(M),
\sup f(M)\in\R$ exist [$\to$ \ref{archetcdef}(c), \ref{introduce-the-reals}].
The last property yields $\inf f(M)=\min f(M)$ and $\sup f(M)=\max f(M)$.
\end{proof}

\begin{thm}[Strengthening of \ref{conemembership}]\emph{[$\to$ \ref{archimedeanpositivstellensatz}]}\label{conemembershipunit}
Let $u$ be a unit for the cone $C$ in the $K$-vector space $V$ and $x\in V$.
Then the following are equivalent:
\begin{enumerate}[\normalfont(a)]
\item $\forall\ph\in S(V,C,u):\ph(x)>0$
\item $\exists N\in\N:x\in\frac1Nu+C$
\item $x$ is a unit for $C$.
\end{enumerate}
\end{thm}

\begin{proof}
\underline{(b)$\implies$(a)} is trivial.

\smallskip
\underline{(a)$\implies$(b)}\quad Suppose that (a) holds.
If $S(V,C,u)=\emptyset$, then $C=V$ by \ref{isolation} and we can choose
$N\in\N$ arbitrarily. Suppose therefore that $S(V,C,u)\ne\emptyset$. Then the
continuous function $S(V,C,u)\to\R,\ \ph\mapsto\ph(x)$ takes on
by \ref{statespacecompact} and \ref{takeson}
a minimum $\mu$ for which $\mu>0$ holds by (a). Choose $N\in\N$ such that
$\frac1N<\mu$. Then $\ph\left(x-\frac1Nu\right)=\ph(x)-\frac1N\ge\mu-\frac1N>0$
for all $\ph\in S(V,C,u)$. Now \ref{conemembership} yields that $x-\frac 1Nu\in C$.

\smallskip
\underline{(b)$\implies$(c)}\quad
Suppose that (b) holds and let $y\in V$. To show: $\exists N\in\N:Nx+y\in C$.
Choose $N',N''\in\N$ with $x\in\frac1{N'}u+C$ and $N''u+y\in C$.
Setting $N:=N'N''$, we obtain
$Nx+y\in N''N'\left(\frac1{N'}u+C\right)+y\subseteq N''(u+C)+y\subseteq
N''u+y+C\subseteq C+C\subseteq C$.

\smallskip
\underline{(c)$\implies$(a)}\quad Suppose that (c) holds and let $\ph\in S(V,C,u)$.
To show: $\ph(x)>0$. Choose $N\in\N$ with $Nx-u\in C$. Then
$N\ph(x)-1=\ph(Nx-u)\ge0$ and thus $\ph(x)\ge\frac1N>0$ for all
$\ph\in S(V,C,u)$.
\end{proof}

\section{Separating convex sets in topological vector spaces}

\begin{df}\label{deftopvs}
A $K$-vector space $V$ together with a topology on $V$ [$\to$
\ref{topreminder}(a)] is called a \emph{topological $K$-vector space}
if $V\times V\to V, (x,y)\mapsto x+y$ and
$K\times V\to V,\ (\la,x)\mapsto\la x$ are continuous and $\{0\}$ is a closed
set in $V$.
\end{df}

\begin{ex}\label{topvsex}
\begin{enumerate}[(a)]
\item If $I$ is a set, then $K^I$ (endowed with the product topology
[$\to$ \ref{subspaceproductspace}(b)]) is a topological $K$-vector space.
\item A $K$-vector space $V$ together with the discrete topology on $V$ is a
topological vector space if and only if $V=\{0\}$. Indeed, if $y\in V\setminus\{0\}$,
then \[\{(\la,x)\in K\times V\mid \la x= y\}=\{(\la,\la^{-1}y)\mid\la\in K^\times\}\]
is not open in $K\times V$.
\item From analysis, one knows that every normed $\R$-vector space,
in particular every $\R$-vector space with scalar product, is a topological
$\R$-vector space.
\end{enumerate}
\end{ex}

\begin{lem}\label{convgencone}
Let $V$ be a $K$-vector space and $A\subseteq V$ be convex.
If $0\notin A\ne\emptyset$, then $A$ generates a proper convex cone, i.e.,
$\sum_{x\in A}K_{\ge0}x\ne V$.
\end{lem}

\begin{proof}
Suppose that $A\ne\emptyset$ and $\sum_{x\in A}K_{\ge0}x=V$. We show
$0\in A$. Choose $y\in A$ and write $-y=\sum_{i=1}^m\la_ix_i$ with $\la_1,\dots,
\la_m\in K_{\ge0}$ and $x_1,\dots,x_m\in A$. Setting $\mu:=1+\sum_{i=1}^m
\la_i>0$, we have then $0=\frac1\mu y+\sum_{i=1}^m\frac{\la_i}\mu x_i\in A$
since $\frac1\mu+\sum_{i=1}^m\frac{\la_i}\mu=\frac\mu\mu=1$.
\end{proof}

\begin{lem}{}[$\to$ \ref{unitinteriorpoint}]\label{interiorisunit}
Let $V$ be a topological $K$-vector space, $C\subseteq V$ a convex cone and
$u\in C^\circ$ [$\to$ \ref{interiorclosure}]. Then $u$ is a unit for $C$
[$\to$ \ref{defunit}].
\end{lem}

\begin{proof}
We show $\forall x\in V:\exists\ep\in K_{>0}:u+\ep x\in C$ [$\to$ \ref{unitchar}(f)].
For this aim, fix $x\in V$. From Definition \ref{deftopvs}, it follows that
$K\to V, \la\mapsto u+\la x$ is continuous. Choose an open set $A\subseteq V$
such that $u\in A\subseteq C$. Then $\{\la\in K\mid u+\la x\in A\}$ is open and
contains $0$. In particular, there is $\ep\in K_{>0}$ such that $u+\ep x\in A
\subseteq C$.
\end{proof}

\begin{ex}\label{zigzag}
Consider the $\R$-vector space $V:=C([0,1],\R)$ of all
continuous real valued functions on the interval $[0,1]\subseteq\R$
together with the scalar product defined by
\[\langle f,g\rangle:=\int_0^1f(x)g(x)dx\qquad(f,g\in V).\]
By \ref{topvsex}(c), this is a topological vector space. The constant function
$u\colon[0,1]\to\R,\ x\mapsto 1$ is a unit for the cone
$C:=C([0,1],\R_{\ge0})$ of all functions nonnegative on $[0,1]$ by
\ref{takeson} (since $[0,1]$ is compact by \ref{sacompreals}).
But $u$ does not lie in $C^\circ$ since for every $\ep>0$ there is some
$f\in V$ with $\|u-f\|=\sqrt{\int_0^1(u(x)-f(x))^2dx}<\ep$ and $f\notin C$.
\end{ex}

\begin{rem}\label{transhomeo}
From Definition \ref{deftopvs}, it follows that for every topological
$K$-vector space $V$ the maps
$V\to V,\ x\mapsto\la x+y$ ($\la\in K^\times,y\in V)$ are homeomorphisms
[$\to$ \ref{dfhomeo}].
\end{rem}

\begin{lem}\label{conthalfspace}
Suppose $V$ is a topological $K$-vector space
and $\ph\colon V\to\R$ is $K$-linear. Then the following are equivalent:
\begin{enumerate}[(a)]
\item $\ph$ is continuous.
\item $\ph^{-1}(\R_{>0})$ is open.
\item $\ph^{-1}(\R_{\ge0})$ is closed.
\end{enumerate}
\end{lem}

\begin{proof}
\underline{(b)$\iff$(c)} follows from
$\ph^{-1}(\R_{\ge0})=-\ph^{-1}(\R_{\le0})=-(V\setminus\ph^{-1}(\R_{>0}))$ since
$V\to V,\ x\mapsto -x$ is a homeomorphism by \ref{transhomeo}.

\smallskip
\underline{(a)$\implies$(b)} is trivial.

\smallskip
\underline{(b)$\implies$(a)} \quad WLOG $\ph\ne0$. WLOG choose $u\in V$ in
such a way that $\ph(u)=1$ (otherwise scale $\ph$). Suppose that (b) holds.
Then the set $\ph^{-1}(\R_{>a})=au+\ph^{-1}(\R_{>0})$ is open
and hence
$\ph^{-1}(\R_{<-a})=-\ph^{-1}(\R_{>a})$ is open for all $a\in K$
[$\to$ \ref{transhomeo}]. 
So the set $\ph^{-1}((a,b)_\R)=\ph^{-1}(\R_{>a})\cap\ph^{-1}(\R_{<b})$ is open
for all $a,b\in K$. Since every open subset of $\R$ is a union of
intervals $(a,b)_\R$ with $a,b\in K$, the continuity of $\ph$ follows.
\end{proof}

\begin{lem}\label{contiffinterior}
Let $V$ be a topological $K$-vector space and $\ph\colon V\to\R$ be $K$-linear
map. Then $\ph$ is continuous if and only if
$\ph^{-1}(\R_{\ge0})$ has an interior point.
\end{lem}

\begin{proof}
WLOG $\ph\ne0$.
If $\ph$ is continuous, then $\ph^{-1}(\R_{>0})$ is open and because of $\ph\ne0$
nonempty. Conversely, let $u$ be an interior point of $\ph^{-1}(\R_{\ge0})$.
By \ref{conthalfspace}, it is enough to show that $\ph^{-1}(\R_{>0})$ is open.
For this, consider $x\in\ph^{-1}(\R_{>0})$. We have to show that there is an
open set $A\subseteq V$ such that $x\in A\subseteq\ph^{-1}(\R_{>0})$.
Choose an open
set $B\subseteq V$ with $u\in B\subseteq\ph^{-1}(\R_{\ge0})$.
Choose $\la\in K_{>0}$ such that $\la\ph(u)<\ph(x)$.
Then $A:=x+\la(B-u)$ is open by
\ref{transhomeo}, and we have $x=x+\la(u-u)\in A$ and
\[\ph(A)=\ph(x)+\la(\ph(B)-\ph(u))\subseteq\ph(x)+\R_{\ge0}-\la\ph(u)
\subseteq\R_{>0}.\]
\end{proof}

\begin{ex}
Let $V:=C([0,1],\R)$ be the topological $K$-vector space from \ref{zigzag}
and $x\in[0,1]$. Then $V\to\R,\ f\mapsto f(x)$ is not continuous.
\end{ex}

\begin{thm}[Separation theorem for topological vector spaces]
\label{septopvs}
Let $A$ and $B$ be convex sets in the topological $K$-vector space $V$
with $A^\circ\ne\emptyset\ne B$ and $A\cap B=\emptyset$. Then there is
a continuous $K$-linear function $\ph\colon V\to\R$ with $\ph\ne0$ and
$\ph(x)\le\ph(y)$ for all $x\in A$ and $y\in B$.
\end{thm}

\begin{proof}
Since $A$ is convex, also $-A$ is convex and thus the Minkowski sum
$B-A=B+(-A)$ [$\to$ \ref{minkowskisum}] is also convex.
By hypothesis, we have $0\notin B-A\ne\emptyset$, for which reason there is according to \ref{convgencone} a proper
cone $C\subseteq V$ such that $B-A\subseteq C$.
Due to $A^\circ\ne\emptyset$ and $B\ne\emptyset$, \ref{transhomeo}
yields $(B-A)^\circ\ne\emptyset$ and thus $C^\circ\ne\emptyset$. Choose
$u\in C^\circ$. By \ref{interiorisunit}, $u$ is a unit for $C$. By the isolation theorem
\ref{isolation}, there exists a state $\ph$ of $(V,C,u)$. Because of $\ph(u)=1$,
we have $\ph\ne0$ and because of $\ph(B-A)\subseteq\R_{\ge0}$, we have
$\ph(x)\le\ph(y)$ for all $x\in A$ and $y\in B$.
Finally, $\ph$ is continuous by \ref{contiffinterior} since $u$ is an interior point
of $C$ and a fortiori of $\ph^{-1}(\R_{\ge0})$.
\end{proof}

\begin{cor}\label{septopvscor}
Let $A$ and $B$ be convex sets in the topological
$K$-vector space $V$ satisfying $A\cap B=\emptyset$. Suppose $A$ is open.
Then there is a continuous $K$-linear function $\ph\colon V\to\R$
and an $r\in\R$ such that $\ph(x)<r\le\ph(y)$ for all $x\in A$ and $y\in B$.
\end{cor}

\begin{proof}
If $A=\emptyset$ or $B=\emptyset$ then we can set $\ph:=0\in V^*$ and choose $r\in\R$ arbitrarily since
the statement $\forall x\in A:\forall y\in B:\ph(x)<r\le\ph(y)$ is empty. WLOG $A\ne\emptyset$ and
$B\ne\emptyset$.
Choose by \ref{septopvs} a continuous $K$-linear function $\ph\colon V\to\R$
with $\ph\ne0$ and $\ph(x)\le\ph(y)$ for all $x\in A$ and $y\in B$. The set
$\{\ph(x)\mid x\in A\}\subseteq\R$ is nonempty because of $A\ne\emptyset$
and bounded from above because of $B\ne\emptyset$. It thus possesses a
supremum $r\in\R$. We have $\ph(x)\le r\le\ph(y)$ for all $x\in A$ and $y\in B$.
Let $x\in A$. It remains to show that $\ph(x)<r$. For this purpose, choose
$z\in V$ such that $\ph(z)>0$. The function $K\to V,\ \la\mapsto x+\la z$ is 
continuous and together with $0$, a whole neighborhood of $0$ lies in the
preimage of $A$ under this function. In particular, there is an $\ep\in K_{>0}$
such that $x+\ep z\in A$. Then $\ph(x)<\ph(x)+\ep\ph(z)=\ph(x+\ep z)\le r$.
\end{proof}

\begin{lem}\label{stayinside}
Let $V$ be a topological $K$-vector space, $A\subseteq V$ be convex,
$x\in A^\circ$, $y\in A$ and $\la\in K$ with $0<\la\le1$. Then
$\la x+(1-\la)y\in A^\circ$.
\end{lem}

\begin{proof}
Choose an open neighborhood $B$ of $x$ with $B\subseteq A$.
Setting $z:=\la x+(1-\la)y$, $C:=z+\la(B-x)$ is by \ref{transhomeo}
an open neighborhood of $z$. It is enough to show $C\subseteq A$.
To this end, let $c\in C$. Because of $B=x+\frac1\la(C-z)$, we have then
$b:=x+\frac1\la(c-z)\in B\subseteq A$. Consequently,
$c=\la(b-x)+z=\la b-\la x+\la x+(1-\la)y=\la b+(1-\la)y\in A$.
\end{proof}

\begin{pro}\label{intcloconvex}
Suppose $V$ is a topological $K$-vector space and $A\subseteq V$ is convex.
Then both $A^\circ$ and $\overline A$ are convex.
\end{pro}

\begin{proof}
It follows immediately from Lemma \ref{stayinside} that
$A^\circ$ is convex. In order to show that $\overline A$ is convex,
fix $x,y\in\overline A$ and $\la\in[0,1]_K$. To show:
$z:=\la x+(1-\la)y\in\overline A$. Let $B$ be a neighborhood of $z$ in $V$.
To show: $B\cap A\ne\emptyset$. Since \[V\times V\to V,\ (x',y')\mapsto
\la x'+(1-\la)y'\] is continuous, there are neighborhoods $C$ of $x$ and
$D$ of $y$ in $V$ such that \[\la C+(1-\la)D\subseteq B.\] Due to
$x,y\in\overline A$, we find $x_0\in C\cap A$ and $y_0\in D\cap A$.
Then \[z_0:=\la x_0+(1-\la)y_0\in B\cap A.\]
\end{proof}

\begin{df}\label{defbalanced}
Let $V$ be a $K$-vector space and $A\subseteq V$ a set. Then $A$ is called
\emph{balanced} if $\la x\in A$ for all $x\in A$ and $\la\in K$ with $|\la|\le1$.
\end{df}

\begin{pro}\label{balanced}
Suppose $V$ be a topological $K$-vector space
and $B$ is a neighborhood of $0$ in $V$.
Then there is a balanced open neighborhood $A$ of $0$ in $V$ with
$A\subseteq B$.
\end{pro}

\begin{proof}
WLOG $B$ is open [$\to$ \ref{neighbor}]. Since the scalar multiplication is continuous by \ref{deftopvs}, there is an $\ep\in K_{>0}$ and an open
neighborhood $C$ of $0$ in $V$ such that
\[\forall\la\in(-\ep,\ep)_K:\forall x\in C:\la x\in B.\]
By \ref{transhomeo}, each $\la C$ with $\la\in K^\times$ is open. Thus
also $A:=\bigcup_{\la\in(-\ep,\ep)_K\setminus\{0\}}\la C\subseteq B$ is open.
Moreover, we have $0\in A$ and $A$ is obviously balanced.
\end{proof}

\begin{exo}\label{comclo}
In a Hausdorff space [$\to$ \ref{dfcomp}], every compact subset
[$\to$ \ref{compactsubspace}] is closed.
\end{exo}

\begin{df}\label{defvstop}
Let $V$ be a $K$-vector space. We call a topology on $V$ making $V$ into
a topological vector space [$\to$ \ref{deftopvs}] a \emph{vector space topology}
on $V$.
\end{df}

\begin{rem}
Up to now the condition $\overline{\{0\}}=\{0\}$ from Definition \ref{deftopvs} has
been used nowhere. From now on, we will however need it. We will show that
each finite-dimensional $\R$-vector space carries exactly one vector space
topology which would be false without the condition $\overline{\{0\}}=\{0\}$
since otherwise the trivial topology [$\to$ \ref{topreminder}(e)] would also be a
vector space topology.
\end{rem}

\begin{pro}\label{topvshausdorff}
Every topological $K$-vector space is a Hausdorff space.
\end{pro}

\begin{proof}
Let $V$ be a topological $K$-vector space [$\to$ \ref{deftopvs}] and let $x,y\in V$
with $x\ne y$. Set $z:=x-y\ne0$. By Definition \ref{deftopvs}, $\{0\}$ and thus 
by \ref{transhomeo} also $\{z\}$ is closed. Hence $V\setminus\{z\}$ is an open
neighborhood of $0$. Since $V\times V\to V,\ (v,w)\mapsto v-w$ is continuous by
\ref{deftopvs}, there is a neighborhood $U$ of $0$ such that
$U-U\subseteq V\setminus\{z\}$. Then $(x+U)\cap(y+U)=\emptyset$ for otherwise
there would be $u,v\in U$ with $x+u=y+v$ from which it would follow
$z=x-y=v-u\in U-U$ $\lightning$.
\end{proof}

\begin{pro}\label{onetop}
Let $V$ be a finite-dimensional $\R$-vector space. Then there is
exactly one vector space topology \emph{[$\to$ \ref{defvstop}]} on $V$.
\end{pro}

\begin{proof}
Choose a basis $v_1,\dots,v_n$ of $V$. Then
$f\colon\R^n\to V,\ x\mapsto\sum_{i=1}^nx_iv_i$ is a vector space isomorphism.
With $\R^n$ [$\to$ \ref{topvsex}] also $V$ possesses therefore
a vector space topology. This shows existence. For uniqueness, endow now
$V$ with any vector space topology. We show that $f$ is a homeomorphism.
By \ref{deftopvs}, $f$ is certainly continuous. It is enough to show that images of
open sets under $f$ are again open. For this purpose, it suffices to show that for
all open balls in $\R^n$ the image of their center is an interior point of their image
because if $A\subseteq\R^n$ is open then every point in $f(A)$ is the image of
the center of an open ball contained in $A$. Due to \ref{transhomeo}, it
suffices to consider the ball $B:=\{x\in\R^n\mid\|x\|<1\}$ around the origin of
radius $1$. In order to show that $0\in(f(B))^\circ$, we take the sphere
$S:=\{x\in\R^n\mid\|x\|=1\}$. By \ref{sacompreals}, $S$ is compact and hence so is
by \ref{quasicompactimage} and \ref{topvshausdorff} also $f(S)$.
According to \ref{comclo}, $f(S)$ is thus closed in $V$. Hence $V\setminus f(S)$
is a neighborhood of $0$ in $V$. By \ref{balanced}, there is a balanced open
neighborhood $A$ of $0$ in $V$ with $A\subseteq V\setminus f(S)$, i.e.,
$A\cap f(S)=\emptyset$. Since $f$ is continuous, $f^{-1}(A)$ is an open 
neighborhood of $0$ in $\R^n$. Due to the linearity of $f$, with $A$ also
$f^{-1}(A)$ is balanced according to Definition \ref{defbalanced}. Since
$f^{-1}(A)$ is disjoint to $S$, it follows that $f^{-1}(A)\subseteq B$ and thus
$A\subseteq f(B)$. Hence $0\in(f(B))^\circ$ as desired.
\end{proof}

\section{Convex sets in locally convex vector spaces}

\begin{df}\label{defloccon}
A \emph{locally convex $K$-vector space} is a topological
$K$-vector space [$\to$ \ref{deftopvs}] in which for every $x\in V$ each 
neighborhood of $x$ contains a convex neighborhood of $x$.
\end{df}

\begin{rem}
Because of \ref{transhomeo}, one can restrict oneself in \ref{defloccon}
to $x=0$.
\end{rem}

\begin{ex}{}[$\to$ \ref{topvsex}]
\begin{enumerate}[(a)]
\item If $I$ is a set, then $K^I$ is a locally convex $K$-vector space.
\item If a $K$-vector space $V$ is endowed with the initial topology
 [$\to$ \ref{initialtop}] with respect
to a family $(f_i)_{i\in I}$ of $K$-linear functions $f_i\colon V\to\R$ in such a way
that to each $x\in V\setminus\{0\}$ there is some $i\in I$ with $f_i(x)\ne0$,
then $V$ is a locally convex $K$-vector space.
\item Every normed $\R$-vector space $V$, in particular every $\R$-vector space
with scalar product, is a locally convex $\R$-vector space since
\[\|\la x+(1-\la)y\|\le\la\|x\|+(1-\la)\|y\|<\la\ep+(1-\la)\ep=\ep\]
for all $\ep>0$, $x,y\in V$ satisfying $\|x\|,\|y\|<\ep$ (``balls are convex'') and $\la\in[0,1]_\R$.
\end{enumerate}
\end{ex}

\begin{lem}\label{closedminkowskisum}
Suppose $V$ is a topological $K$-vector space, $A\subseteq V$ is closed and
$C\subseteq V$ is compact. Then $A+C$ is closed.
\end{lem}

\begin{proof} Let $x\in V\setminus(A+C)$. We have to show that there is a
neighborhood $U$ of the origin satisfying $(x+U)\cap(A+C)=\emptyset$.

\medskip
\textbf{Claim:}
For each $y\in C$, there exists a neighborhood $U_y$ of the origin such that
\[(x+U_y)\cap(y+U_y+A)=\emptyset.\]

\smallskip
\emph{Explanation.} Let $y\in C$. Then
$V\times V\to V,\ (x',y')\mapsto x-y+x'-y'$ is continuous and $(0,0)$ lies in the
preimage of the open set $V\setminus A$ since
$x-y\notin A$ (otherwise we would have $x\in A+y\subseteq A+C$). Hence there is a
neighborhood $U_y$ with \[x-y+U_y-U_y\subseteq V\setminus A,\]
i.e., $(x+U_y-y-U_y)\cap A=\emptyset$.

\smallskip\noindent
By compactness of $C$, there is a finite subset $D\subseteq C$ such that
$C\subseteq\bigcup_{y\in D}(y+U_y)$. Now $U:=\bigcap_{y\in D}U_y$ is a
neighborhood of the origin. In order to show that \[(x+U)\cap(A+C)=\emptyset,\]
it is enough to prove that $(x+U)\cap(A+y+U_y)=\emptyset$ for all $y\in D$.
For this purpose, it suffices to show that
$(x+U_y)\cap(y+U_y+A)=\emptyset$ for all $y\in D$. But this holds even for all
$y\in C$ by the above claim.
\end{proof}

\begin{thm}[Separation theorem for locally convex vector spaces]{}
\emph{[$\to$ \ref{septopvs}, \ref{septopvscor}]}
\label{seplocconvs}
Let $A$ and $C$ be convex sets in the locally convex $K$-vector space
$V$ with $A\cap C=\emptyset$. Let $A$ be closed and $C$ be compact. Then
there is a continuous $K$-linear function $\ph\colon V\to\R$ and
$r,s\in\R$ with $\ph(x)\le r<s\le\ph(y)$ for all $x\in A$ and $y\in C$.
\end{thm}

\begin{proof}
If $A=\emptyset$ or $C=\emptyset$, we can take $\ph:=0\in V^*$ and choose $r,s\in\R$ arbitrary since the statement
$\forall x\in A:\forall y\in C:\ph(x)\le r<s\le\ph(y)$ is empty. WLOG $A\ne\emptyset$ and $C\ne\emptyset$. We have that
$B:=C-A$ is by \ref{closedminkowskisum} closed and by hypothesis we have
$0\notin B$. Since $V$ is locally convex, there is in view of \ref{intcloconvex}
a convex open set $D\subseteq V$ with $0\in D$ and $D\cap B=\emptyset$.
Since $B$ is also convex, there is by Corollary \ref{septopvscor} a
continuous $K$-linear function $\ph\colon V\to\R$ and an $\ep\in\R$ such that
$\ph(x)<\ep\le\ph(y)$ for all $x\in D$ and $y\in B$. In particular, $\ep>\ph(0)=0$
and $\ph(x)+\ep\le\ph(y)$ for all $x\in A$ and $y\in C$. Because of
$A\ne\emptyset\ne C$, $r:=\sup\{\ph(x)\mid x\in A\}\in\R$ and
$s:=\inf\{\ph(y)\mid y\in C\}\in\R$ exist. Moreover, we have $r+\ep\le s$, i.e.,
$r<s$.
\end{proof}

\begin{df}\label{dfface}
Let $V$ be a $K$-vector space and $A\subseteq V$ be convex. Then a convex set $F\subseteq A$ is called a \emph{face} of $A$ if for all $x,y\in A$
with $\frac{x+y}2\in F$, we have also $x,y\in F$.
\end{df}

\begin{pro}\label{extremepointface}
Suppose $V$ is a $K$-vector space, $A\subseteq V$ is convex and
$x\in A$. Then $x$ is an extreme point of $A$ \emph{[$\to$ \ref{dfconv}]}
if and only if $\{x\}$ is a face of $A$.
\end{pro}

\begin{proof}
\begin{align*}
&x\text{ is an extreme point of }A\\
\overset{\ref{dfconv}}\iff&\nexists y,z\in A:\left(y\ne z\et x=\frac{y+z}2\right)\\
\iff&\forall y,z\in A:\left(x=\frac{y+z}2\implies y=z\right)\\
\iff&\forall y,z\in A:\left(x=\frac{y+z}2\implies y=z=x\right)\\
\iff&\forall y,z\in A:\left(\frac{y+z}2\in\{x\}\implies y,z\in\{x\}\right)
\end{align*}
\end{proof}

\begin{pro}{}\emph{[$\to$ \ref{extremeexo}]}\label{ratiootherthan2}
Suppose $V$ is a $K$-vector space, $A\subseteq V$ is convex, $F\subseteq A$
is convex and $\la\in(0,1)_K$. Then the following are equivalent:
\begin{enumerate}[\normalfont(a)]
\item $F$ is a face of $A$
\item $\forall x,y\in A:(\la x+(1-\la)y\in F\implies x,y\in F)$
\end{enumerate}
\end{pro}

\begin{proof}
\underline{(b)$\implies$(a)} is an easy exercise.

\smallskip
\underline{(a)$\implies$(b)} \quad Assume that $F$ is a face of $A$ but there are
$x,y\in A$ such that \[\la x+(1-\la)y\in F\] and WLOG
(otherwise permute $x$ and $y$ and replace $\la$ by $1-\la$)
$x\notin F$. If $\la<\frac12$, one then can replace
$(x,\la)$ by $(x',\la')$ where $x':=\frac{x+y}2$ and $\la':=2\la\in(0,1)_K$ because
we then have $x'\in A\setminus F$ (since $A$ is convex and $F$ is a face of $A$),
$\la'\in(0,1)_K$ and \[\la'x'+(1-\la')y=2\la\frac{x+y}2+(1-2\la)y=\la x+(1-\la)y\in F.\]
By iterating this in case of need finitely many times, one can suppose
$\la\ge\frac12$. Then
\[z:=x+2((\la x+(1-\la)y)-x)=(2\la-1)x+2(1-\la)y\in A\] since
$2\la-1\ge0$, $2(1-\la)\ge0$ and $(2\la-1)+2(1-\la)=1$.
Now \[\frac{x+z}2=x+(\la x+(1-\la)y)-x=\la x+(1-\la)y\in F\] and thus
$x,z\in F$ since $F$ is a face of $A$ $\lightning$.
\end{proof}

\begin{ex}\label{exforfaces}
\begin{enumerate}[(a)]
\item If $V$ is a $K$-vector space and $A\subseteq V$ is convex, then both
$\emptyset$ and $A$ are faces of $A$. We call these the \emph{trivial}
faces of $A$.
\item The faces of $[0,1]^2\subseteq\R^2$ are
$\emptyset$, $\{(0,0)\}$, $\{(0,1)\}$, $\{(1,0)\}$, $\{(1,1)\}$, $\{0\}\times[0,1]$,
$\{1\}\times[0,1]$, $[0,1]\times\{0\}$, $[0,1]\times\{1\}$, $[0,1]^2$.
\item The faces of $B:=\{x\in\R^2\mid\|x\|\le1\}$ are $\emptyset$,
$\{x\}$ ($x\in B\setminus B^\circ$) and $B$.
\item $\{x\in\R^2\mid\|x\|<1\}$ has only the trivial faces.
\end{enumerate}
\end{ex}

\begin{dfpro}\label{defexposed}
Let $V$ be a $K$-vector space and suppose $A\subseteq V$ is convex.
We call $F$ an \emph{exposed face} of $A$ if there is a $K$-linear function
$\ph\colon V\to\R$ such that \[F=\{x\in A\mid\forall y\in A:\ph(x)\le\ph(y)\}.\]
Every exposed face of $A$ is a face of $A$.
\end{dfpro}

\begin{proof}
Let $F$ be an exposed face of $A$. To show: $F$ is a face of $A$.
It is easy to show that $F$ is convex.
Choose a $K$-linear $\ph\colon V\to\R$ such that
$F=\{x\in A\mid\forall y\in A:\ph(x)\le\ph(y)\}$.
Let $x,y\in A$ such that $\frac{x+y}2\in F$. To show: $x,y\in F$.
It is obviously enough to show that $\ph(x)=\ph\left(\frac{x+y}2\right)=\ph(y)$.
We have that
\[\ph(x)+\ph(y)=\ph\left(\frac{x+y}2\right)+\ph\left(\frac{x+y}2\right)
\overset{x,y\in A}{\underset{\frac{x+y}2\in F}\le}\ph(x)+\ph(y)
\]
where the inequality would be strict if one of $\ph(x)$ and $\ph(y)$ were different
from $\ph\left(\frac{x+y}2\right)$.
\end{proof}

\begin{ex}{}[$\to$ \ref{exforfaces}]
\begin{enumerate}[(a)]
\item If $V$ is a $K$-vector space and $A\subseteq V$ is convex, then $A$ is an
exposed face of $A$ while $\emptyset$ might be exposed
[$\to$ \ref{exforfaces}(d)] or non-exposed [\ref{exforfaces}(c)].
\item All faces of $[0,1]^2\subseteq\R^2$ are exposed except $\emptyset$.
\item All faces of $\{x\in\R^2\mid\|x\|\le1\}$ are exposed except $\emptyset$.
\item All faces of $\{x\in\R^2\mid\|x\|<1\}$ are exposed.
\item $((-\infty,0]\times[0,\infty))\cup\{(x,y)\in\R_{\ge0}^2\mid y\ge x^2\}$ has exactly 
one nonexposed face, namely $\{0\}$.
\end{enumerate}
\end{ex}

\begin{pro}\label{faceofface}
Suppose $V$ is a $K$-vector space, $A\subseteq V$ is convex, $F$ is a face of $A$
and $G\subseteq F$. Then the following holds:
\[\text{$G$ is a face of $F$}\iff\text{$G$ is a face of $A$}.\]
\end{pro}

\begin{proof}
``$\Longrightarrow$'' Let $G$ be a face of $F$ and let $x,y\in A$ with
$\frac{x+y}2\in G$. To show: $x,y\in G$. Because of $\frac{x+y}2\in G\subseteq F$,
we have $x,y\in F$. Since $G$ is a face of $F$, it follows that $x,y\in G$.

\smallskip
``$\Longleftarrow$'' Let $G$ be a face of $A$ and let $x,y\in F$ with
$\frac{x+y}2\in G$. Because of $x,y\in F\subseteq A$, we then have $x,y\in G$.
\end{proof}

\begin{rem}
Every intersection of faces (except possible $\bigcap\emptyset:=V$) of a convex set in a $K$-vector space $V$ is obviously
again a face of this convex set.
\end{rem}

\begin{lem}\label{hasextreme}
Let $C\ne\emptyset$ be a compact convex
subset of a locally convex $K$-vector
space $V$. Then $C$ possesses an extreme point.
\end{lem}

\begin{proof}
Every intersection of a nonempty chain of closed nonempty faces of $C$ is
again a closed nonempty face of $C$. Indeed, if the intersection were empty,
then a finite subintersection would be empty by the compactness of $C$
[$\to$ \ref{dfcomp}] which is impossible since we dealt with a chain. By
Zorn's lemma there is thus a minimal closed nonempty face $F$ of $C$.
Being a closed subset of a compact set, $F$ is compact itself [$\to$ \ref{compactsubspace}]. By
\ref{extremepointface}, it suffices to show that $\#F=1$. Due to
$F\ne\emptyset$, it suffices to exclude $\#F\ge2$. Assume $x,y\in F$ such that
$x\ne y$. By \ref{seplocconvs}, there is a continuous $K$-linear function
$\ph\colon V\to\R$ such that $\ph(x)<\ph(y)$. Then
\[G:=\{v\in F\mid\forall w\in F:\ph(v)\le\ph(w)\}\]
is nonempty by \ref{takeson}
because $F$ is compact and nonempty and $\ph$ is continuous.
According to \ref{defexposed}, $G$ is an (exposed) face of $F$.
Hence $G$ is a face of $C$ by \ref{faceofface}. From the continuity of $\ph$,
we deduce that \[G=F\cap\bigcap_{w\in F}\ph^{-1}((-\infty,\ph(w)])\] is closed.
Moreover, $y\notin G$ since $\ph(y)\not\le\ph(x)$. Therefore $G$ is a closed
nonempty face of $C$ that is properly contained in $F$, contradicting the
minimality of $F$.
\end{proof}

\begin{notation}
Let $A$ be a convex set in a $K$-vector space $V$. Then we write
\[\extr A\] for the set of extreme points of $A$.
\end{notation}

\begin{thm}{}\emph{[$\to$ \ref{takeson}]}\label{takesonextr}
Suppose $C$ is a nonempty compact convex subset of a locally convex
$K$-vector space $V$ and $\ph\colon V\to\R$ is a continuous $K$-linear function.
Then $\ph$ attains on $C$ a minimum and a maximum in an extreme point of $C$.
In other words, there are $x,y\in\extr C$ such that
\[\ph(x)\le\ph(z)\le\ph(y)\]
for all $z\in C$.
\end{thm}

\begin{proof}
Since one could replace $\ph$ by $-\ph$, we show only the existence of
$x\in\extr C$ such that $\ph(x)\le\ph(z)$ for all $z\in C$. By \ref{takeson},
there is $y\in C$ such that $\ph(y)\le\ph(z)$ for all $z\in C$, i.e., the exposed face
[$\to$ \ref{defexposed}]
\[F:=\{y\in C\mid\forall z\in C:\ph(y)\le\ph(z)\}\]
of $C$ is nonempty. Since $\ph$ is continuous,
\[F=C\cap\bigcap_{z\in C}\ph^{-1}((-\infty,\ph(z)]_\R)\]
is a closed subset of the
compact set $C$ and hence compact itself. By Lemma \ref{hasextreme},
$F$ possesses an extreme point $x$ which is by \ref{faceofface} and
\ref{extremepointface} also an extreme point of $C$.
\end{proof}

\begin{cor}[Krein–Milman theorem]\label{kreinmilman}
Suppose $C$ is a compact convex subset of a locally convex $K$-vector space
$V$. Then $C$ is the closure of the convex hull of the set of its extreme points, 
i.e.,
\[C=\overline{\conv(\extr C)}.\]
\end{cor}

\begin{proof}
``$\supseteq$'' From $\extr C\subseteq C$ and the convexity of $C$, we get
$\conv(\extr C)\subseteq C$. Being a compact subset of a Hausdorff space
[$\to$ \ref{topvshausdorff}], $C$ is closed [$\to$ \ref{comclo}] which entails even
$\overline{\conv(\extr C)}\subseteq C$.

``$\subseteq$'' WLOG $C\ne\emptyset$.
$A:=\overline{\conv(\extr C)}$ is closed, nonempty by Lemma
\ref{hasextreme} and convex by \ref{intcloconvex}. We show
$V\setminus A\subseteq V\setminus C$. Let $x\in V\setminus A$. By the
separation theorem for locally convex vector spaces \ref{seplocconvs},
there is a continuous $K$-linear function $\ph\colon V\to\R$ such that
$\ph(y)<\ph(x)$ for all $y\in A$. By \ref{takesonextr}, there is
$y\in\extr C\subseteq A$ satisfying $\ph(z)\le\ph(y)$ for all $z\in C$.
It follows that $\ph(z)\le\ph(y)<\ph(x)$ for all $z\in C$. Therefore $x\notin C$.
\end{proof}

\begin{df} Let $V$ be a $K$-vector space, $C\subseteq V$ and $u\in V$.
We call an extreme point [$\to$ \ref{dfconv}] of the state space $S(V,C,u)$
[$\to$ \ref{defstate}, \ref{statespacetop}] a \emph{pure state} of $(V,C,u)$.
\end{df}

\begin{thm}[Strengthening of \ref{conemembershipunit}]
\label{conemembershipunitextreme}
Suppose $u$ is a unit for the cone $C$ in the $K$-vector space $V$ and
let $x\in V$. Then the following are equivalent:
\begin{enumerate}[\normalfont(a)]
\item $\forall\ph\in\extr S(V,C,u):\ph(x)>0$
\item $\forall\ph\in S(V,C,u):\ph(x)>0$
\item $\exists N\in\N:x\in\frac1Nu+C$
\item $x$ is a unit for $C$.
\end{enumerate}
\end{thm}

\begin{proof}
\underline{(b)$\iff$(c)$\iff$(d)} is \ref{conemembershipunit}.

\smallskip
\underline{(b)$\implies$(a)} is trivial.

\smallskip
\underline{(a)$\implies$(b)} \quad WLOG $S(V,C,u)\ne\emptyset$. It suffices to
show that the function \[S(V,C,u)\to\R,\ \ph\mapsto\ph(x)\] attains a minimum in an
extreme point of $S(V,C,u)$. But this follows from \ref{takesonextr} because
$S(V,C,u)$ is a nonempty compact [$\to$ \ref{statespacecompact}] convex
[$\to$ \ref{statespacetop}] subset of the locally convex [$\to$ \ref{topvsex}(a)]
$\R$-vector space $\R^V$ and \[\R^V\to\R,\ \ph\mapsto\ph(x)\] is continuous
[$\to$ \ref{subspaceproductspace}(b)].
\end{proof}

\begin{cor}[Strengthening of \ref{conemembership}]\label{conemembershipextr}
Suppose $u$ is a unit for the cone $C$ in the $K$-vector space $V$ and
let $x\in V$. If $\ph(x)>0$ for all pure states $\ph$ of $(V,C,u)$, then
$x\in C$.
\end{cor}

\section{Convex sets in finite-dimensional vector spaces}

\begin{lem}\label{findimunit}
Let $C$ be a cone in a finite-dimensional $K$-vector space $V$.
Then $U:=C-C$ is a subspace of $V$ and $C$ possesses \emph{in $U$}
a unit [$\to$ \ref{defunit}].
\end{lem}

\begin{proof}
On the basis of Definition \ref{defcone}, it is easy to see that $U$ is a subspace of
$V$. Choose a basis $u_1,\dots,u_m$ of $U$ and write $u_i=v_i-w_i$ with
$v_i,w_i\in C$ for $i\in\{1,\dots,m\}$. We show that
$u:=\sum_{i=1}^mv_i+\sum_{i=1}^mw_i\in C$ is a unit for $C$ \emph{in $U$}.
For this purpose, fix $v\in U$. To show: $\exists N\in\N:Nu+v\in C$. Write
$v=\sum_{i=1}^m\la_iu_i$ with $\la_i\in K$ for $i\in\{1,\dots,m\}$. Choose
$N\in\N$ with $N\ge|\la_i|$ for $i\in\{1,\dots,m\}$. Then
\[Nu+v=\sum_{i=1}^m\underbrace{(N+\la_i)}_{\ge0}v_i+\sum_{i=1}^m
\underbrace{(N-\la_i)}_{\ge0}w_i\in C.\]
\end{proof}

\begin{thm}[Finite-dimensional isolation theorem]\emph{[$\to$ \ref{isolation}]}\label{findimisolation}
Let $C$ be a proper cone in the finite-dimensional $K$-vector space $V$.
Then there is a $K$-linear function $\ph\colon V\to\R$ with $\ph\ne0$ and
$\ph(C)\subseteq\R_{\ge0}$.
\end{thm}

\begin{proof}
$U:=C-C$ is by \ref{findimunit} a subspace of $V$.

\medskip
\textbf{Case 1:}  $C=U$

\smallskip
Then $U$ is a proper subspace of $V$ and by linear algebra it is easy to see that
there is some $\ph\in V^*\setminus\{0\}$ such that $\ph(U)=\{0\}$. 

\medskip
\textbf{Case 2:}  $C\ne U$

\smallskip
By \ref{findimunit}, there exists a unit $u$ for $C$ in $U$. The isolation theorem
\ref{isolation} provides us with some $\ph_0\in S(U,C,u)$. Extend $\ph_0$
by linear algebra to a $K$-linear function $\ph\colon V\to\R$.
\end{proof}

\begin{rem}
Example \ref{badexample} shows that one cannot omit the hypothesis
$\dim V<\infty$ in \ref{findimunit} and \ref{findimisolation}.
\end{rem}

\begin{thm}[Separation theorem for finite-dimensional vector spaces]
\emph{[$\to$ \ref{septopvs}]}\label{sepfindimvs}
Let $A$ and $B$ be convex sets in the finite-dimensional $K$-vector space $V$
such that $A\ne\emptyset\ne B$ and $A\cap B=\emptyset$. Then there is a
$K$-linear function $\ph\colon V\to\R$ such that $\ph\ne0$ and
$\ph(x)\le\ph(y)$ for all $x\in A$ and $y\in B$.
\end{thm}

\begin{proof}
Completely analogous to the proof of \ref{septopvs}.
\end{proof}

\begin{df}{}[$\to$ \ref{dfconv}]\label{affinehull}
Let $V$ be a $K$-vector space and $A\subseteq V$. Then $A$ is called an
\emph{affine subspace of $V$} if $\forall x,y\in A:\forall\la\in K:\la x+(1-\la)y\in A$.
The smallest affine subspace of $V$ containing $A$ is obviously
\[\aff A:=\left\{\sum_{i=1}^m\la_ix_i\mid m\in\N,\la_i\in K,x_i\in A,\sum_{i=1}^m\la_i=1\right\},
\]
called the affine subspace generated by $A$ or the \emph{affine hull} of $A$.
\end{df}

\begin{dfpro}\label{dfdimaff}
Let $V$ be a $K$-vector space. Then for each $A\subseteq V$,
the following are equivalent:
\begin{enumerate}[\normalfont(a)]
\item $A$ is a nonempty affine subspace of $V$.
\item There is an $x\in V$ and a subspace $U$ of $V$ such that $A=x+U$.
\end{enumerate}
If these conditions are met, then $U$ as in \emph{(b)}
is uniquely determined and is called
the \emph{direction} of $A$. Then $\dim A:=\dim U\in\N_0\cup\{\infty\}$ is the
\emph{dimension} of $A$. We set $\dim\emptyset:=-1$.
\end{dfpro}

\begin{proof}
\underline{(b)$\implies$(a)} is easy.

\smallskip
\underline{(a)$\implies$(b)}\quad
Suppose that (a) holds. Choose $x\in A$. Set
$U:=A-x$. To show: $U+U\subseteq U$ and $KU\subseteq U$. Let $u,v\in U$
and $\la\in K$. To show: $u+v\in U$ and $\la u\in U$. Choose $a,b\in A$ such that
$u=a-x$ and $v=b-x$. Then $u+v=(1a+1b+(-1)x)-x\in(\aff A)-x=A-x=U$ and
$\la u=(\la a+(1-\la)x)-x\in(\aff A)-x=A-x=U$.

\smallskip
\underline{Uniqueness claim}\quad Whenever $x,y\in V$  and $U$ and $W$ are
subspaces of $V$ satisfying $x+U=y+W$, then $x-y\in W$ and thus
$U=(y-x)+W=W$.
\end{proof}

\begin{df}\label{dfdimconv}
Let $V$ be a $K$-vector space and $A\subseteq V$ be convex. Then
\[\dim A:=\dim\aff A\in\{-1\}\cup\N_0\cup\{\infty\}\]
is the \emph{dimension} of $A$.
\end{df}

\begin{pro}{}\emph{[$\to$ \ref{ratiootherthan2}]}\label{face234}
Suppose that $V$ is a $K$-vector space, $A\subseteq V$ is convex and
$F$ is a face of $A$. Let $m\in\N$, $x_1,\dots,x_m\in A$ and
$\la_1,\dots,\la_m\in K_{>0}$ such that $\sum_{i=1}^m\la_i=1$ and
$\sum_{i=1}^m\la_ix_i\in F$. Then $x_1,\dots,x_m\in F$.
\end{pro}

\begin{proof} WLOG $m\ge2$. Let $i\in\{1,\dots,m\}$. To show: $x_i\in F$.
WLOG $i=1$. We have $0<\la_1<1$ and $y:=\sum_{i=2}^m\frac{\la_i}{1-\la_1}x_i
\in A$ since $\sum_{i=2}^m\frac{\la_i}{1-\la_1}=\frac{1-\la_1}{1-\la_1}=1$.
From $\sum_{i=1}^m\la_ix_i=\la_1x_1+(1-\la_1)y$ it follows thus by
\ref{ratiootherthan2} that $x_1,y\in F$.
\end{proof}

\begin{pro}\label{faceaff}
Suppose $V$ is a $K$-vector space, $A\subseteq V$ is convex and $F$ is a
face of $A$. Then $F=(\aff F)\cap A$.
\end{pro}

\begin{proof}
``$\subseteq$'' is trivial.

``$\supseteq$'' Let $x\in(\aff F)\cap A$. To show: $x\in F$. Write
$x=\sum_{i=1}^m\la_iy_i-\sum_{j=1}^n\mu_jz_j$ with $m,n\in\N_0$,
$\la_i,\mu_j\in K_{>0}$, $y_i,z_j\in F$ and
$\sum_{i=1}^m\la_i-\sum_{j=1}^n\mu_j=1$. Setting $\la:=\sum_{i=1}^m\la_i$ and
$\mu:=\sum_{j=1}^n\mu_j$, it follows that
$\frac1{1+\mu}x+\sum_{j=1}^n\frac{\mu_j}{1+\mu}z_j=
\sum_{i=1}^m\frac{\la_i}\la y_i\in F$
and thus $x\in F$ by \ref{face234}.
\end{proof}

\begin{pro}\label{dimensiondrop}
Let $V$ be a finite-dimensional $K$-vector space.
\begin{enumerate}[\normalfont(a)]
\item If $A$ and $B$ are affine subspaces of $V$ with $A\subseteq B$, then
\[A=B\iff\dim A=\dim B.\]
\item If $F$ and $G$ are faces of the convex set $A\subseteq V$ with
$F\subseteq G$, then
\[F=G\iff\dim F=\dim G.\]
\end{enumerate}
\end{pro}

\begin{proof}
(a) follows from \ref{dfdimaff} by linear algebra and (b) follows hereof by
\ref{dfdimconv} and \ref{faceaff}.
\end{proof}

\begin{rem}\label{topsvs}
Suppose $V$ is a topological $K$-vector space, $K'$ is a subfield of $K$ and $V'$
a $K'$-vector subspace of the $K'$-vector space $V$. Then $V$ induces on
$V'$ a vector space topology. This is easy to see since $V\times V$ induces
on $V'\times V'$ the product topology of the induced topologies and
$K\times V$ induces on $K'\times V'$ the product topology of the induced topologies.
\end{rem}

\begin{dfpro}\label{dfrelint}
Let $A$ be a convex set in the topological $K$-vector space $V$. The interior of
$A$ in the topological space $\aff A$ (endowed with the topology induced by $V$)
is called the \emph{relative interior} of $A$, denoted by
$\relint A$. This is a convex set.
\end{dfpro}

\begin{proof}
WLOG $A\ne\emptyset$. Write $\aff A=x+U$ for some $x\in V$ and some 
subspace $U$ of $V$ [$\to$ \ref{dfdimaff}]. WLOG $x=0$ [$\to$ \ref{transhomeo}].
WLOG $U=V$ [$\to$ \ref{topsvs}]. Then
$\relint A=A^\circ$ is convex by \ref{intcloconvex}.
\end{proof}

\begin{rem}\label{topfindimrem}
Let $V$ be a finite-dimensional $K$-vector space. Choose a basis $v_1,\dots,v_n$
of $V$. Then $f\colon K^n\to V,\ x\mapsto\sum_{i=1}^nx_iv_i$ is a vector space
isomorphism that is continuous with respect to every vector space topology of
$V$ [$\to$ \ref{defvstop}] and that is a homeomorphism with respect to exactly
one vector space topology of $V$ [$\to$ \ref{topvsex}(a)]. Consequently, there
is a finest [$\to$ \ref{topreminder}(c)] vector space topology on $V$
(cf. also \ref{onetop}). With respect to this topology on $V$, we have for all
$A\subseteq V$ that
\[\text{$A$ open in $V$}\iff\text{$f^{-1}(A)$ open in $K^n$},\]
independently of the choice of the basis $v_1,\dots,v_n$. It is easy to see that
$K^n$ carries the initial topology with respect to all
\alal{linear forms on $K^n$}{$K$-linear functions $K^n\to\R$}.
The finest vector space topology on $V$ is therefore also the initial topology
[$\to$ \ref{initialtop}] with respect to all
\alal{linear forms on $V$}{$K$-linear functions $V\to\R$}.
If $U$ is a subspace of $V$, then the finest vector space topology of $V$ induces
on $U$ again the finest vector space topology because one can extend every
linear form on $U$ to one on $V$.
\end{rem}

\begin{ex}
Since $\R$ is a topological $\R$-vector space and thus also a topological
$\Q$-vector space, also $\Q+\sqrt 2\Q\subseteq\R$ is a topological
$\Q$-vector space with respect to the induced topology [$\to$ \ref{topsvs}].
One sees easily that
\[\Q+\sqrt2\Q\to\Q,\ x+\sqrt2y\mapsto x\qquad(x,y\in\Q)\] is not continuous.
\end{ex}

\begin{lem}\label{reddim}
Let $A\subseteq K^n$ be convex. Then
$A^\circ=\emptyset\implies\aff A\ne K^n$.
\end{lem}

\begin{proof}
Suppose that $\aff A=K^n$. We show that $A^\circ\ne\emptyset$. Denote by
$e_1,\dots,e_n$ the standard basis of $K^n$ and set $e_0:=0\in K^n$.
Write $e_i=\sum_{j=1}^m\la_{ij}x_{ij}$ with $m\in\N$, $\la_{ij}\in K$,
$x_{ij}\in A$ and $\sum_{j=1}^m\la_{ij}=1$ for $i\in\{0,\dots,n\}$.
We show that \[x:=\sum_{i=0}^n\sum_{j=1}^m\frac1{m(n+1)}x_{ij}\in A^\circ.\]
Since $A$ is convex, we have $x\in A$ and it suffices to show that for each
$i\in\{1,\dots,n\}$, there is an $\ep\in K_{>0}$ such that $x\pm\ep e_i\in A$
(cf. also \ref{unitinteriorpoint}). For this purpose, fix $i\in\{1,\dots,n\}$.
From $e_i=e_i-0=e_i-e_0=
\sum_{j=1}^m\la_{ij}x_{ij}+\sum_{j=1}^m(-\la_{0j})x_{0j}$
and $\sum_{j=1}^m\la_{ij}-\sum_{j=1}^m\la_{0j}=1-1=0$, the existence of such an $\ep>0$ easily follows.
\end{proof}

\begin{thm}\label{relintclosure}
Suppose $V$ is a finite-dimensional topological
$K$-vector space that is equipped
with the finest vector space topology \emph{[$\to$ \ref{topfindimrem}]}
and $A\subseteq V$ is convex. Then
$A\subseteq\overline{\relint A}$.
\end{thm}

\begin{proof}
WLOG $A\ne\emptyset$. Write $\aff A=x+U$ for some $x\in V$ and some
subspace $U$ of $V$. Obviously, $\aff(A-x)\overset{\ref{affinehull}}=
(\aff A)-x=U$, $\relint(A-x)\overset{\ref{transhomeo}}=
(\relint A)-x$ and
$\overline{\relint(A-x)}=
\overline{\relint A}-x$.
Replacing $A$ by $A-x$, we can thus suppose that $\aff A=U$.
Using the last remark from \ref{topfindimrem},
we can therefore suppose that
$\aff A=V$ (otherwise replace $V$ by $U$).
According to \ref{topfindimrem},
we can reduce to the case where $V=K^n$ (with the product topology).
We have to show $A\subseteq
\overline{A^\circ}$. Choose $y\in A^\circ$ with Lemma \ref{reddim}.
Let $x\in A$. To show: $x\in\overline{A^\circ}$. By \ref{stayinside}, we have
$(1-\la)x+\la y\in A^\circ$ for all $\la\in(0,1]_K$. Obviously, we have
$x\in\overline{\{(1-\la)x+\la y\mid\la\in(0,1]_K\}}\subseteq\overline{A^\circ}$.
\end{proof}

\begin{thm}\label{hasaface}
Let $V$ be a finite-dimensional $K$-vector space that is equipped
with the finest vector space topology \emph{[$\to$ \ref{topfindimrem}]}.
Let $A\subseteq V$
be convex and $x\in A\setminus\relint A$. Then there is an exposed face $F$ of
$A$ satisfying $\dim F<\dim A$ and $x\in F$.
\end{thm}

\begin{proof} Similarly to the proof of \ref{relintclosure}, we reduce again
to the case $\aff A=V$. Note that $A^\circ$ is convex
[$\to$ \ref{dfrelint}] and nonempty [$\to$~\ref{relintclosure}]. The separation 
theorem \ref{sepfindimvs} yields a $K$-linear function $\ph\colon V\to\R$
with $\ph\ne0$ and $\ph(x)\le\ph(y)$ for all $y\in A^\circ$.
Since $\ph$ is continuous [$\to$ \ref{topfindimrem}], the set
$\ph^{-1}([\ph(x),\infty)_\R)$ is closed and contains with
$A^\circ$ also $\overline{A^\circ}$ and hence by \ref{relintclosure} also
$A$, i.e., $\ph(x)\le\ph(y)$ for all $y\in A$.
In other words, $x$ is an element of the exposed face 
[$\to$~\ref{defexposed}] $F:=\{z\in A\mid\forall y\in A:\ph(z)\le\ph(y)\}$ of $A$.
By \ref{dimensiondrop}, it is enough to show $F\ne A$. If we had
$F=A$, we would have $\ph|_A=\ph(x)$ and hence by linearity of $\ph$
also $\ph=\ph|_{\aff A}=\ph(x)$, i.e., $\ph=0$ $\lightning$.
\end{proof}

\begin{rem}\label{tacitly}
If we use topological notions in a finite-dimensional $\R$-vector space $V$, then
we tacitly furnish $V$ with its unique vector space topology [$\to$ \ref{onetop}]
which is the initial topology with respect to the family of all linear forms on $V$
[$\to$ \ref{topfindimrem}].
\end{rem}

\begin{thm}[Minkowski's theorem]
\emph{[$\to$ \ref{polyisconvexhull}, \ref{kreinmilman}]}\label{minkowski}
Let $V$ be a finite-dimensional $\R$-vector space. Let
$A\subseteq V$ be convex and compact. Then
\[A=\conv(\extr A).\]
\end{thm}

\begin{proof}
Since $A$ is closed [$\to$ \ref{comclo}], all faces of $A$ are also closed
[$\to$ \ref{faceaff}, \ref{dfdimaff}, \ref{transhomeo}] and therefore compact
[$\to$ \ref{compactsubspace}]. By induction, we can thus
assume that $F=\conv(\extr F)$ for all faces $F$ of $A$ that satisfy
$\dim F<\dim A$.

``$\supseteq$'' is trivial.

``$\subseteq$'' Let $x\in A$. To show: $x\in\conv(\extr A)$. WLOG $x\notin\extr A$.
Choose $y,z\in A$ with $y\ne z$ and $x\in\conv\{y,z\}$. Because of the 
assumptions on $A$, WLOG $y,z\in A\setminus\relint A$. By \ref{hasaface},
there are (exposed) faces $F$ and $G$ of $A$ such that $\dim F<\dim A$,
$\dim G<\dim A$, $y\in F$ and $z\in G$.
From \ref{extremepointface} and \ref{faceofface}, we get $\extr F\subseteq\extr A$
and $\extr G\subseteq\extr A$. Consequently,
$y\in F=\conv(\extr F)\subseteq
\conv(\extr A)$ and
$z\in G=\conv(\extr G)\subseteq
\conv(\extr A)$ where the equalities follow from the induction hypothesis.
Finally, $x\in\conv\{y,z\}\subseteq\conv(\extr A)$.
\end{proof}

\begin{thm}\label{fundthmlin}
Let $(K,\le)$ be an arbitrary ordered field, let $V$ be a $K$-vector space
with $n:=\dim V<\infty$. Suppose that
$E\subseteq V$ is a finite set that generates $V$ and $x\in V$.
Then exactly one of the following conditions occurs:
\begin{enumerate}[\normalfont(a)]
\item $x$ is a nonnegative linear combination of elements of $E$ that form a basis of $V$.
\item There is some $\ell\in V^*$ with $\ell(E)\subseteq K_{\ge0}$ and $\ell(x)<0$
and a linearly independent set $F\subseteq E\cap\ker\ell$ with $\#F=n-1$.
\end{enumerate}
\end{thm}

\begin{proof}
It is easy to see that (a) and (b) cannot occur both at the same time. Indeed, from
(a) it follows that $x\in\sum_{v\in E}K_{\ge0}v$ which is not compatible with (b)
because if $\ell\in V^*$ with $\ell(E)\subseteq K_{\ge0}$, then
$\ell(x)\in\ell\left(\sum_{v\in E}K_{\ge0}v\right)\subseteq K_{\ge0}$.

\medskip
We choose an order $\le$ on $E$ [$\to$ \ref{ordered-set}] and a basis
$B\subseteq E$ of $V$. We show that the following algorithm always terminates:
\begin{enumerate}[(1)]
\item Write $x=\sum_{v\in B}\la_vv$ with $\la_v\in K$ for all $v\in B$.
\item If $\la_v\ge0$ for all $v\in B$, then stop since (a) occurs.
\item $u:=\min\{v\in B\mid\la_v<0\}$
\item Define $\ell\in V^*$ by $\ell(u)=1$ and $\ell(v)=0$ for all
$v\in B\setminus\{u\}$ (so that $\ell(x)=\la_u<0$).
\item If $\ell(E)\subseteq K_{\ge0}$, then stop since (b) occurs.
\item $w:=\min\{v\in E\mid\ell(v)<0\}$
\item Replace $B$ by the new basis $(B\setminus\{u\})\cup\{w\}$ and
go to (1).
\end{enumerate}
Observe first of all that in Step (7) the set
$(B\setminus\{u\})\cup\{w\}$ is again a basis
since $B$ is one. Indeed, $w$ does not lie in the subspace generated by
$B\setminus\{u\}$ since $\ell$ vanishes according to its choice in (4) on this
subspace while it does not vanish on $w$ by the choice of $w$ in (6).

\medskip
To show that this algorithm terminates, we assume that this is not the case. Let
then denote by $(B_k,u_k,w_k,\ell_k)$ the value of $(B,u,w,\ell)$ after Step (6)
in the $k$-th iteration of the loop. We first argue that the existence of
$s,t\in\N$ with
\[(*)\qquad u_t\le u_s=w_t\text{ and }
\{v\in B_s\mid v>u_s\}=\{v\in B_t\mid v>u_s\}\]
causes a contradiction. For this purpose, let $x=\sum_{v\in B_s}\la_vv$ with
$\la_v\in K$ for all $v\in B_s$ be the representation of $x$ from the $s$-th
iteration of the loop. We will apply $\ell_t$ to this representation of $x$.
For that matter, observe the following:
\begin{itemize}
\item For all $v\in B_s$ with $v<u_s=w_t$, we have by the assignment to $u_s$
in (3) that $\la_v\ge0$.
\item For all $v\in E\supseteq B_s$ with $v<u_s=w_t$, we have by the
assignment to $w_t$ in (6) that $\ell_t(v)\ge0$.
\item $\la_{u_s}<0$ according to (3)
\item $\ell_t(u_s)=\ell_t(w_t)<0$ according to (6)
\item For all $v\in B_s$ with $v>u_s=w_t$, we have $\ell_t(v)=0$ since for these
$v$ we have by $(*)$ that $v\in B_t\setminus\{u_t\}$
(using that $u_t\le u_s$) and thus $\ell_t(v)=0$ by (4).
\end{itemize}
It thus follows that
\[
0\overset{(4)}>\ell_t(x)=
\underbrace{
\sum_{\substack{v\in B_s\\v<u_s}}\underbrace{\la_v}_{\ge0}\underbrace{\ell_t(v)}_{\ge0}
}_{\ge0}+
\underbrace{
\underbrace{\la_{u_s}}_{<0}\underbrace{\ell_t(u_s)}_{<0}
}_{>0}
+\underbrace{
\sum_{\substack{v\in B_s\\v>u_s}}\la_v\underbrace{\ell_t(v)}_{=0}
}_{=0}
>0
\]
which is the desired contradiction.

\medskip
Finally, we show the existence of $s,t\in\N$ with $(*)$. For clarity, we first
generalize the algorithm by looking at the following more abstract version of it:

\medskip
Suppose $E$ is a finite set, $\le$ an order on $E$ and $B$ a subset of $E$.
\begin{enumerate}[(1')]
\item Choose $u\in B$.
\item Choose $w\in E\setminus B$.
\item Replace $B$ by $(B\setminus\{u\})\cup\{w\}$ and go to (1').
\end{enumerate}
Denote by $(B_k,u_k,w_k)$ the value of $(B,u,w)$ after Step (2')
in the $k$-th iteration of the algorithm. We show the existence of $s,t\in\N$
satisfying $(*)$. Since $E$ is finite, the power set of $E$ is also finite.
Consequently, there are $p,q\in\N$ such that $p<q$ and $B_p=B_q$.
Because of (3'), it then obviously holds that $\{u_s\mid p\le s<q\}
=\{w_t\mid p\le t<q\}$. Set $v_0:=\max\{u_s\mid p\le s<q\}=
\max\{w_t\mid p\le t<q\}$. Then \[\{v\in B_s\mid v>v_0\}=\{v\in B_t\mid v>v_0\}\]
for all $s,t\in\{p,\dots,q-1\}$. Choose $s,t\in\{p,\dots,q-1\}$ with $u_s=v_0=w_t$
(note that $s<t$ or $t<s$ but certainly not $s=t$). Now
$(*)$ holds.
\end{proof}

\begin{cor}\emph{[$\to$ \ref{sospsd2}]}\label{linhomnichtnegativstellensatz}
Let $(K,\le)$ be an arbitrary ordered field. Let $m,n\in\N_0$,
$f,\ell_1,\dots,\ell_m\in K[X_1,\dots,X_n]$ be linear forms
\emph{[$\to$ \ref{longremi}(a)]} and set
\[S:=\{x\in K^n\mid\ell_1(x)\ge0,\ldots,\ell_m(x)\ge0\}.\] Then the following are
equivalent:
\begin{enumerate}[\normalfont(a)]
\item $f\ge0$ on $S$
\item $f\in K_{\ge0}\ell_1+\ldots+K_{\ge0}\ell_m$
\item There are $i_1,\dots,i_s\in\{1,\dots,m\}$ such that $\ell_{i_1},\dots,
\ell_{i_s}$ are linearly independent and
\[f\in K_{\ge0}\ell_{i_1}+\ldots+K_{\ge0}\ell_{i_s}.\]
\end{enumerate}
\end{cor}

\begin{proof}
\underline{(c)$\implies$(b)$\implies$(a)} is trivial.

\smallskip
\underline{(a)$\implies$(c)}\quad Suppose that (a) holds.

\medskip
\textbf{Claim:} $f\in V:=K\ell_1+\ldots+K\ell_m$

\smallskip
\emph{Explanation.} Assume $f\notin V$. Then there is some
$\ph\in(KX_1+\ldots+KX_n)^*$ such that $\ph(\ell_1)=\ldots=\ph(\ell_m)=0$
and $\ph(f)=-1$. Set $x:=(\ph(X_1),\ldots,\ph(X_n))\in K^n$. Then
$\ell_i(x)=\ph(\ell_i)=0$ for all $i\in\{1,\dots,m\}$. Hence $x\in S$ and
$f(x)=\ph(f)=-1<0$. $\lightning$

\smallskip\noindent Now apply \ref{fundthmlin} to $V$ and
$E:=\{\ell_1,\dots,\ell_m\}$ (taking account of the claim). Then it suffices to show
that for all $\ph\in V^*$ with $\ph(E)\subseteq K_{\ge0}$ also $\ph(f)\ge0$ holds.
Thus let $\ph\in V^*$ with $\ph(E)\subseteq K_{\ge0}$. Choose
$\ps\in(KX_1+\ldots+KX_n)^*$ with $\ps|_V=\ph$.
Set $x:=(\ps(X_1),\ldots,\ps(X_n))\in K^n$. Then
$\ell_i(x)=\ps(\ell_i)=\ph(\ell_i)\ge0$ for all $i\in\{1,\dots,m\}$ and thus
$x\in S$ and $\ph(f)=\ps(f)=f(x)\ge0$.
\end{proof}

\begin{cor}[Linear Nichtnegativstellensatz]
\emph{[$\to$ \ref{son1s}, \ref{nichtnegativstellensatz}]}
\label{linnichtnegativstellensatz}
Let $(K,\le)$ be an arbitrary ordered field.
Let $m,n\in\N_0$, $f,\ell_1,\dots,\ell_m\in K[X_1,\dots,X_n]_1$
\emph{[$\to$ \ref{degnot}]} and suppose
\[S:=\{x\in K^n\mid\ell_1(x)\ge0,\ldots,\ell_m(x)\ge0\}\ne\emptyset.\]
Then the following are equivalent:
\begin{enumerate}[\normalfont(a)]
\item $f\ge0$ on $S$
\item $f\in K_{\ge0}+K_{\ge0}\ell_1+\ldots+K_{\ge0}\ell_m$
\item There are $i_1,\dots,i_s\in\{0,\dots,m\}$ such that $\ell_{i_1},\ldots,\ell_{i_s}$
are linearly independent and \[f\in K_{\ge0}\ell_{i_1}+\ldots+K_{\ge0}\ell_{i_s}\]
where $\ell_0:=1$.
\end{enumerate}
\end{cor}

\begin{proof}
\underline{(c)$\implies$(b)$\implies$(a)} is trivial.

\smallskip
\underline{(a)$\implies$(c)}\quad Suppose that (a) holds. Due to
\ref{linhomnichtnegativstellensatz}, it suffices to show that
$f^*\ge0$ on
$S^*:=\{x=(x_0,\dots,x_n)\in K^{n+1}\mid x_0\ge0,\ell_1^*(x)\ge0,\ldots,
\ell_m^*(x)\ge0\}$ [$\to$ \ref{introhom}(c)(d), \ref{homdehom}(e)].
To this end, let $(x_0,\dots,x_n)\in S^*$.

\medskip
\textbf{Case 1:}  $x_0>0$

\smallskip
Then $\left(1,\frac{x_1}{x_0},\dots,\frac{x_n}{x_0}\right)=
\frac1{x_0}(x_0,\dots,x_n)\in
S^*$ and hence $\left(\frac{x_1}{x_0},\ldots,\frac{x_n}{x_0}\right)\in S$.
From (a), it follows that
$f^*\left(1,\frac{x_1}{x_0},\dots,\frac{x_n}{x_0}\right)=
f\left(\frac{x_1}{x_0},\dots,\frac{x_n}{x_0}\right)\ge0$
and hence also
$f^*(x_0,\ldots,x_n)=x_0f^*\left(1,\frac{x_1}{x_0},\dots,\frac{x_n}{x_0}\right)\ge0$

\medskip
\textbf{Case 2:}  $x_0=0$

\smallskip
Then $(\lf(\ell_i))(x_1,\dots,x_n)\overset{\text{\ref{homdehom}(a)}}=
\ell_i^*(x_0,\dots,x_n)\ge0$ and therefore
\[(\lf(\ell_i))(\la x_1,\dots,\la x_n)\ge0\] for all $i\in\{1,\dots,m\}$ and
$\la\in K_{\ge0}$. Because of $S\ne\emptyset$, we can choose
$(y_1,\dots,y_n)\in S$. Then $\ell_i(y_1+\la x_1,\dots,y_n+\la x_n)\ge0$
for all $\la\in K_{\ge0}$ and $i\in\{1,\dots,m\}$.
Due to (a), we have thus
$f(y_1+\la x_1,\dots,y_n+\la x_n)\ge0$ for all $\la\in K_{\ge0}$.
It follows that $(\lf(f))(x_1,\dots,x_n)\ge0$. Hence
\[f^*(x_0,\dots,x_n)\overset{\text{\ref{homdehom}(a)}}=(\lf(f))(x_1,\dots,x_n)\ge0.\]
\end{proof}

\begin{df}\label{defray}
Let $V$ be a $K$-vector space and $C\subseteq V$ a cone
[$\to$ \ref{defcone}].
\begin{enumerate}[(a)]
\item The sets of the form $K_{\ge0}x$ with $x\in C\setminus\{0\}$ are called
the \emph{rays} of $C$.
\item Rays of $C$ that are at the same time faces [$\to$ \ref{dfface}] of $C$ are
called \emph{extreme rays} of $C$.
\item A set $B\subseteq C\setminus\{0\}$ is called a \emph{base} of $C$,
if for each $x\in C\setminus\{0\}$ there is exactly one $\la\in K_{>0}$ such that
$\la x\in B$, i.e., if every ray of $C$ hits the set $B$ in exactly one point.
\end{enumerate}
\end{df}

\begin{pro}\label{extrbase}
Suppose $V$ is a $K$-vector space and $C\subseteq V$ is a cone with convex
base $B$. Then for all $x\in V$,
\[\text{$K_{\ge0}x$ is an extreme ray of $C$}\iff\exists\la\in K_{>0}:\la x\in\extr B.\]
\end{pro}

\begin{proof} Let $x\in V$.

\smallskip
``$\Longrightarrow$'' Let $K_{\ge0}x$ be an extreme ray of $C$. Then it follows
that $x\in C\setminus\{0\}$. Hence there is exactly one $\la\in K_{>0}$ such that
$\la x\in B$. We claim $\la x\in\extr B$. For this purpose, consider $y,z\in B$ with
$\frac{y+z}2=\la x$. To show: $y=z=\la x$. From $y,z\in C$ and
$\frac{y+z}2\in K_{\ge0}x$, we deduce $y,z\in K_{\ge0}x$. Due to $y,z\in B$ and
$0\notin B$, we get $y,z\in K_{>0}x$. Again from $y,z\in B$ and the uniqueness of
$\la$, we get $y=\la x=z$.

\smallskip
``$\Longleftarrow$''  WLOG let $x\in\extr B$. To show: $K_{\ge0}x$ is an extreme
ray of $C$. Since $x\in B\subseteq C\setminus\{0\}$, $K_{\ge0}x$ is a ray of $C$.
Let $y,z\in C$ with $\frac{y+z}2\in K_{\ge0}x$. To show: $y,z\in K_{\ge0}x$.
WLOG $y\ne0$ and $z\ne0$. If we had $y+z=0$, then one could easily show
$0\in B\ \lightning$.
WLOG $y+z=x$. Choose $\mu,\nu\in K_{>0}$ such that $\mu y,\nu z\in B$.
Then
\[
x=y+z=(\mu^{-1}+\nu^{-1})
\underbrace{\left(\frac{\mu^{-1}}{\mu^{-1}+\nu^{-1}}(\mu y)+\frac{\nu^{-1}}{\mu^{-1}+\nu^{-1}}
(\nu z)\right)}_{\in B}
\]
and thus $\mu^{-1}+\nu^{-1}=1$. Since
$x=\mu^{-1}(\mu y)+\nu^{-1}(\nu z)$, $\mu y,\nu z\in B$ and $x\in\extr B$,
we have $\mu y=x=\nu y$.
\end{proof}

\begin{thm}{}\emph{[$\to$ \ref{minkowski}]}\label{conicminkowski}
Every convex cone with compact \emph{[$\to$ \ref{tacitly}]} convex base in a
finite-dimensional $\R$-vector space is the sum of its extreme rays.
\end{thm}

\begin{proof}
Suppose $V$ is a finite-dimensional $\R$-vector space and
$C\subseteq V$ is a convex cone with compact convex base $B$. Let $x\in C$.
To show: $x$ is a sum of elements of extreme rays of $C$. WLOG $x\in B$.
By Minkowski's theorem \ref{minkowski}, we have $x\in\conv(\extr B)$, say
$x=\sum_{i=1}^m\la_ix_i$ with $m\in\N$, $\la_1,\dots,\la_m\in K_{\ge0}$,
$\la_1+\ldots+\la_m=1$ and $x_i\in\extr B$. According to \ref{extrbase},
$K_{\ge0}x_i$ is for all $i\in\{1,\dots,m\}$ an extreme ray of $C$.
\end{proof}

\begin{pro}\label{compactbaseclosed}
Every convex cone with compact \emph{[$\to$ \ref{tacitly}]}
base in a finite-dimensional $\R$-vector space is closed.
\end{pro}

\begin{proof}
Suppose $V$ is a finite-dimensional $\R$-vector space and
$C\subseteq V$ is a convex cone with compact base $B$. By
Tikhonov's theorem \ref{tikhonov}, also $[0,1]_\R\times B$ is compact.
From \ref{quasicompactimage} together with the continuity of the
scalar multiplication, we obtain
that \[A:=\{\la x\mid\la\in[0,1]_\R,x\in B\}\]
is again compact. WLOG $V=\R^n$ by \ref{onetop}. WLOG $B\ne\emptyset$.
Set
\[d:=\min\{\|y\|\mid y\in B\}>0.\] In order to show that $C$ is closed,
we now let $x\in V\setminus C$. WLOG $\|x\|<\frac d2$ [$\to$ \ref{transhomeo}].
Since $A$ is closed by \ref{comclo}, there is an $\ep>0$ such that
$\{y\in V\mid\|x-y\|<\ep\}\cap A=\emptyset$. From $0\in A$, we get
$\ep\le\|x\|<\frac d2$. Then $\{y\in V\mid\|x-y\|<\ep\}\cap C=\emptyset$
for if $y\in C\setminus A$, then there is $\la\in K$ with $0<\la<1$ and
$\la y\in B$ and it follows that $\|y\|=\frac1\la\|\la y\|\ge\frac1\la d>d$ which is 
incompatible with $\|x-y\|<\frac d2$ (which would imply contrarily
$\|y\|\le\|y-x\|+\|x\|<\frac d2+\frac d2=d$). This shows that $C$ is closed.
\end{proof}

\section{Application to ternary quartics}

A \emph{ternary quartic} is a $4$-form (also called quartic form [$\to$ \ref{quintic}])
in $3$ variables.

\begin{lem}\label{densarg}
Let $(K,\le)$ be an ordered field and $G\in SK^{m\times m}$.
Then $G$ is psd [$\to$ \ref{psdpd}] if and
only if $x^TGx\ge0$ for all $x\in(K^\times)^m$.
\end{lem}

\begin{proof}
Suppose $x^TGx\ge0$ for all $x\in(K^\times)^m$. Let $z\in K^m$. We
have to show  that \[z^TGz\ge0.\]
Choose $y\in(K^\times)^m$ arbitrary. Then $z+\la y\in(K^\times)^m$ and
therefore \[z^TGz+2\la y^TGz+\la^2y^TGy=(z+\la y)^TG(z+\la y)\ge0\] for
all but finitely many $\la\in K$. For example, by \ref{sgnbounds}(b) applied to
the polynomial \[z^TGz+2y^TGzT+y^TGyT^2\in K[T],\]
it follows that $z^TGz\ge0$.
\end{proof}

\begin{lem}\label{h1888a}
Let $K$ be a Euclidean field and $f\in K[X,Y,Z]$ a $4$-form. Suppose that there
are linearly independent $v_1,v_2,v_3\in K^3$ such that
$f(v_1)=f(v_2)=f(v_3)=0$. Then the following are equivalent:
\begin{enumerate}[(a)]
\item $f$ is psd [$\to$ \ref{psdpd}(a)]
\item $f\in\sum K[X,Y,Z]^2$
\item $f$ is a sum of $3$ squares of quadratic forms in $K[X,Y,Z]$.
\end{enumerate}
\end{lem}

\begin{proof}
Denote by $e_1,e_2,e_3$ the standard basis of $K^3$. Set
$A:=\begin{pmatrix}v_1&v_2&v_3\end{pmatrix}
\in\GL_3(K)$ and $g:=f\left(A\left(\begin{smallmatrix}X\\Y\\Z
\end{smallmatrix}\right)\right)\in K[X,Y,Z]$. Then $g$ is a $4$-form
satisfying $g(e_1)=g(e_2)=g(e_3)=0$. Since $A$ defines a permutation
(even a vector space isomorphism)
\[K^3\to K^3,\ 
\left(\begin{smallmatrix}x\\y\\z
\end{smallmatrix}\right)\mapsto A
\left(\begin{smallmatrix}x\\y\\z
\end{smallmatrix}\right)\]
on $K^3$, we have that
\[\text{$f$ is psd}\iff\text{$g$ is psd}.\]
Since $A$ induces on the other hand a ring automorphism
\[K[X,Y,Z]\to K[X,Y,Z],\ h\mapsto h\left(A
\left(\begin{smallmatrix}X\\Y\\Z
\end{smallmatrix}\right)\right),\]
we obtain
\[f\in\sum K[X,Y,Z]^2\iff g\in\sum K[X,Y,Z]^2.\]
Since this ring automorphism permutes the quadratic forms
in $K[X,Y,Z]$, we have that
\[\text{(c)}\iff\text{$g$ is a sum of $3$ squares of quadratic forms}.\]
Replacing $f$ by $g$, we can henceforth suppose that $v_1=e_1$,
$v_2=e_2$ and $v_3=e_3$.

\smallskip
\underline{(c)$\implies$(b)$\implies$(a)} is trivial.

\smallskip
\underline{(a)$\implies$(c)} \quad It is easy to see that each polynomial $g\in K[T]$
with $g\ge0$ on $K$ and $g(0)=0$ lies in the ideal $(T^2)$
[$\to$ \ref{sgnbounds}(b)].
Suppose now that $(a)$ holds. The vanishing at $0$ and the nonnegativity
of the polynomials
\[f(1,T,0),\ f(1,0,T),\ f(T,1,0),\ f(0,1,T),\ f(0,T,1),\ f(T,0,1)\in K[T]\]
therefore forces the coefficients of
\[X^4,\quad X^3Y,\quad X^3Z,\quad Y^4,\quad Y^3X,\quad Y^3Z,
\quad Z^4,\quad Z^3X,\quad Z^3Y\]
in $f$ to vanish. For example, the first polynomial forces
the coefficients of $X^4$ and $X^3Y$ to vanish, and
the second one the coefficients of again $X^4$ and of $X^3Z$.
It follows that
\begin{align*}
N(f)&\subseteq\conv\{(2,2,0),\xcancel{(2,1,1)},(2,0,2),(0,2,2),
\xcancel{(1,2,1)},\xcancel{(1,1,2)}\}\text{, i.e.,}\\
\frac12N(f)&\subseteq\conv\{(1,1,0),(1,0,1),(0,1,1)\}\quad\text{and thus}\\
\frac12N(f)\cap\N_0^3&\subseteq\{(1,1,0),(1,0,1),(0,1,1)\}.
\end{align*}
By the Gram matrix method \ref{gram}, we have to show that there is a
\emph{psd} matrix $G\in SK^{3\times3}$ satisfying
\[(*)\qquad f=\begin{pmatrix}XY&XZ&YZ\end{pmatrix}G
\begin{pmatrix}XY\\XZ\\YZ\end{pmatrix}.\]
Since every monomial occurring in $f$ is a product of two entries
of $\begin{pmatrix}XY&XZ&YZ\end{pmatrix}$, there is certainly a
$G\in SK^{3\times3}$ satisfying $(*)$ (actually one sees easily that there is a
\emph{unique} such $G$ which does however not play an immediate role).
But from $(*)$ it follows
\emph{automatically} that $G$ is \emph{psd} since $f$ is psd.
In order to see this, let $v\in K^3$. We have to show that
$v^TGv\ge0$. Using \ref{densarg}, one reduces to the case
$v\in(K^\times)^3$. Then set $\la:=v_1 v_2v_3$
and $x:=\frac1{v_3}$, $y:=\frac1{v_2}$ and $z:=\frac1{v_1}$.
Now $v=\la\begin{pmatrix}xy\\xz\\yz\end{pmatrix}$ and therefore
\[v^TGv=\la^2\begin{pmatrix}xy&xz&yz\end{pmatrix}G
\begin{pmatrix}xy\\xz\\yz\end{pmatrix}\overset{(*)}=\la^2f(x,y,z)\ge0.\]
\end{proof}

\begin{lem}\label{h1888b}
Let $K$ be a Euclidean field and $f\in K[X,Y,Z]$ a $4$-form.
Suppose there are linearly independent $v_1,v_2,v_3\in K^3$ satisfying
$f(v_1+Tv_2)\in(T^3)$ and $f(v_3)=0$. Then the following are equivalent:
\begin{enumerate}[(a)]
\item $f$ is psd
\item $f\in\sum K[X,Y,Z]^2$
\item $f$ is a sum of $3$ squares of quadratic forms in $K[X,Y,Z]$.
\end{enumerate}
\end{lem}

\begin{proof}
Almost exactly as in the proof of \ref{h1888a}, one sees that one can suppose
WLOG $v_1=e_1$, $v_2=e_2$ and $v_3=e_3$.

\smallskip
\underline{(c)$\implies$(b)$\implies$(a)} is again trivial.

\smallskip
\underline{(a)$\implies$(c)} \quad One sees easily that a polynomial
$g\in K[T]$ with $g\ge0$ on $K$ and $g\in(T^{2k-1})$ lies in $(T^{2k})$
for $k\in\N$. Suppose now that (a) holds. By considering the polynomials
\[f(1,T,0),\ f(1,0,T),\ f(0,T,1),\ f(T,0,1)\in K[T],\]
one sees easily that the coefficients of
\[X^4,\quad X^3Y,\quad X^2Y^2,\quad XY^3,\quad X^3Z,\quad Z^4,
\quad Z^3Y,\quad Z^3X\]
in $f$ must vanish.  More precisely, the first polynomial is responsible for the
first four of these coefficients, the second for the coefficients of $X^4$ (again)
and $X^3Z$, the third for the coefficients of $Z^4$ and $Z^3Y$, and the last
for the coefficients of $Z^4$ (again) and $Z^3X$.
It follows that
\begin{align*}
N(f)&\subseteq\conv\{(2,0,2),(2,1,1),\xcancel{(1,1,2)},\xcancel{(1,2,1)},
(0,2,2),\xcancel{(0,3,1)},(0,4,0)\}\text{, i.e.,}\\
\frac12N(f)&\subseteq\conv\left\{(1,0,1),\left(1,\frac12,\frac12\right),(0,1,1),(0,2,0)\right\}\quad\text{and thus}\\
\frac12N(f)\cap\N_0^3&\subseteq\{(1,0,1),(0,1,1),(0,2,0)\}.
\end{align*}
By the Gram matrix method \ref{gram}, we have to show that there is a
\emph{psd} matrix $G\in SK^{3\times3}$ satisfying
\[(*)\qquad f=\begin{pmatrix}XZ&YZ&Y^2\end{pmatrix}G
\begin{pmatrix}XZ\\YZ\\Y^2\end{pmatrix}.\]
If the monomial $X^2YZ$ actually appeared in $f$, we would now run into
a big problem that we did not have in the proof of \ref{h1888a} because
this monomial is not a product of two entries of
$\begin{pmatrix}XZ&YZ&Y^2\end{pmatrix}$.
But this coefficient vanishes as one easily shows since for all $y\in K$, the leading
coefficient of $f(X,y,1)\in K[X]$ is nonnegative since this polynomial
is nonnegative on $K$. 
As in the proof of \ref{h1888a}, one sees again that there exists
$G\in SK^{3\times3}$ satisfying $(*)$ (one could again see easily
that $G$ is unique). From $(*)$ it follows \emph{automatically} that $G$ is psd
since $f$ is psd. To see this, let $v\in K^3$. To show: $v^TGv\ge0$.
Using \ref{densarg}, one reduces to the case
$v\in K\times(K^\times)^2$. Then set $\la:=v_2^2v_3$ and
$x:=\frac{v_1}{v_2^2}$, $y:=\frac1{v_2}$, $z:=\frac1{v_3}$.
Now $v=\la\begin{pmatrix}xz\\yz\\y^2\end{pmatrix}$ and therefore
\[v^TGv=\la^2\begin{pmatrix}xz&yz&y^2\end{pmatrix}G
\begin{pmatrix}xz\\yz\\y^2\end{pmatrix}\overset{(*)}=\la^2f(x,y,z)\ge0.\]
\end{proof}

\begin{lem}\label{h1888c}
Let $K$ be a Euclidean field and $f\in K[X,Y,Z]$ a $4$-form.
Suppose there are linearly independent $v_1,v_2\in K^3$ satisfying
$f(v_1+Tv_2)\in(T^3)$ and $f(v_2)=0$. Then the following are equivalent:
\begin{enumerate}[(a)]
\item $f$ is psd
\item $f\in\sum K[X,Y,Z]^2$
\item $f$ is a sum of $3$ squares of quadratic forms in $K[X,Y,Z]$.
\end{enumerate}
\end{lem}

\begin{proof}
One can again suppose WLOG $v_1=e_1$ and $v_2=e_2$.

\smallskip
\underline{(c)$\implies$(b)$\implies$(a)} is again trivial.

\smallskip
\underline{(a)$\implies$(c)} \quad One uses again that a polynomial
$g\in K[T]$ with $g\ge0$ on $K$ and $g\in(T^{2k-1})$ lies in $(T^{2k})$
for $k\in\N$. Suppose now that (a) holds. By considering the polynomials
\[f(1,T,0),\ f(1,0,T),\ f(T,1,0),\ f(0,1,T)\in K[T],\]
one sees easily that the coefficients of
\[X^4,\quad X^3Y,\quad X^2Y^2,\quad XY^3,\quad X^3Z,\quad Y^4,
\quad Y^3Z\]
in $f$ must vanish.  More precisely, the first polynomial is responsible for the
first four of these coefficients, the second for the coefficients of $X^4$ (again)
and $X^3Z$, the third for the coefficients of $Y^4$ and $XY^3$ (again),
and the last for the coefficients of $Y^4$ (again) and $Y^3Z$.
It follows that
\begin{align*}
N(f)&\subseteq\conv\{(2,0,2),(2,1,1),(1,2,1),(0,2,2),\xcancel{(0,1,3)},(0,0,4),
\xcancel{(1,0,3)},\xcancel{(1,1,2)}\}\text{, i.e.,}\\
\frac12N(f)&\subseteq\conv\left\{(1,0,1),\left(1,\frac12,\frac12\right),\left(\frac12,1,\frac12\right),(0,1,1),(0,0,2)\right\}\quad\text{and thus}\\
&\frac12N(f)\cap\N_0^3\subseteq\{(1,0,1),(0,1,1),(0,0,2)\}.
\end{align*}
By the Gram matrix method \ref{gram}, we have to show that there is a
\emph{psd} matrix $G\in SK^{3\times3}$ satisfying
\[(*)\qquad f=\begin{pmatrix}XZ&YZ&Z^2\end{pmatrix}G
\begin{pmatrix}XZ\\YZ\\Z^2\end{pmatrix}.\]
If one of the monomials $X^2YZ$ and $XY^2Z$
actually appeared in $f$, we would have trouble since
these monomials are not a product of two entries of
$\begin{pmatrix}XZ&YZ&Z^2\end{pmatrix}$.
But these coefficients vanish
as one easily shows since for all $x,y\in K$, the leading
coefficients of $f(X,y,1)\in K[X]$ and $f(x,Y,1)\in K[Y]$ are nonnegative since these
polynomials are nonnegative on $K$.
One sees again that there exists
$G\in SK^{3\times3}$ satisfying $(*)$ (one could again see easily
that $G$ is unique). From $(*)$ it follows \emph{automatically} that $G$ is psd
since $f$ is psd. To see this, let $v\in K^3$. To show: $v^TGv\ge0$.
Using \ref{densarg}, one reduces to the case
$v\in K\times(K^\times)^2$. Then set $\la:=v_2^2v_3$ and
$x:=\frac{v_1}{v_2v_3}$, $y:=\frac1{v_3}$, $z:=\frac1{v_2}$.
Now $v=\la\begin{pmatrix}xz\\yz\\z^2\end{pmatrix}$ and therefore
\[v^TGv=\la^2\begin{pmatrix}xz&yz&z^2\end{pmatrix}G
\begin{pmatrix}xz\\yz\\z^2\end{pmatrix}\overset{(*)}=\la^2f(x,y,z)\ge0.\]
\end{proof}

\begin{lem}\label{subtractlin4}
Let $f\in\R[X,Y,Z]$ be a psd $4$-form that is not a sum of $3$ squares of quadratic forms in
$\R[X,Y,Z]$ and that has two linearly independent zeros in $\R^3$. Then there is a linear form
$\ell\in\R[X,Y,Z]\setminus\{0\}$ such that $f-\ell^4$ is psd.
\end{lem}

\begin{proof}
By Lemma \ref{h1888a}, the zeros of $f$ span a two-dimensional subspace of $\R^3$.
By a change of coordinates, we can thus achieve that $f(e_2)=f(e_3)=0$ and
\[f>0\text{ on }\R^\times\times\R^2.\]
We now claim that there is some $\ep\in\R_{>0}$ such that
f$-\ep X^4$ is psd. By homogeneity, it suffices to find $\ep>0$ such that $f-\ep X^4\ge0$ holds
on the compact set \[[-1,1]_\R^3\setminus(-1,1)_\R^3.\]
For this purpose, it is enough show that for each two-dimensional face $F$ of the polytope
$[-1,1]^3$ (i.e., for each side of the cube $[-1,1]^3$) there is some $\ep>0$
such that $f-\ep X^4\ge0$ on $F$. On the two sides $\{-1\}\times[-1,1]^2$ and
$\{1\}\times[-1,1]^2$, $f$ is positive so that the existence of such an $\ep$ for them follows
from \ref{takeson}. After a further change of coordinates, it suffices to consider from the
remaining four sides just $[-1,1]^2\times\{1\}$. Consider therefore
$\widetilde f:=f(X,Y,1)\in\R[X,Y]$ [$\to$ \ref{introhom}(d)]. From Lemma \ref{h1888c}, we deduce
\[\frac{\partial^2\widetilde f}{\partial Y^2}(0,y)>0\] for all $y\in\R$ that satisfy $\widetilde f(0,y)=0$
(apply \ref{h1888c} to $f$, $v_1:=(0,y,1)$ and $v_2:=(0,1,0)$, taking into account that
$\frac{\partial\widetilde f}{\partial Y}(0,y)=0$ due to $\widetilde f\ge0$ on $\R^2$).
In the same way, Lemma \ref{h1888b} implies that for each $y\in\R$ satisfying
$\widetilde f(0,y)=0$ all other directional derivatives of $\widetilde f$ in $(0,y)$ are also positive.
Altogether, $\widetilde f$ has thus only zeros in $\R^2$ at which the second derivative
(i.e., the Hessian) is \emph{pd} (recall that all zeros of $\widetilde f$ lie on the $y$-axis).
From analysis we know that each zero of the nonnegative polynomial $\widetilde f$
(in $\R^2$, or equivalently $\{0\}\times\R$ since all zeros lie on the $y$-axis) is an isolated
\emph{global} minimizer. Therefore
\[\{(x,y)\in\R^2\mid\widetilde f(x,y)=0\}=\{(0,y_1),\dots,(0,y_m)\}\]
for some $m\in\N$ and $y_1,\dots,y_m\in\R$ (one of the $y_i$ is $0$).
Since $-X^4$ as well as its first and second derivative vanishes on the $y$-axis
(since $\frac{\partial X^4}{\partial X}=4X^3$, $\frac{\partial X^4}{\partial Y}=0$,
$\frac{\partial^2X^4}{\partial X^2}=12X^2$, $\frac{\partial^2X^4}{\partial X\partial Y}=0$
and $\frac{\partial^2X^4}{\partial Y^2}=0$), every $(0,y_i)$ is a zero and an isolated
\emph{local} minimizer of $\widetilde f-X^4$. Choose for each $i\in\{1,\dots,m\}$ an open
neighborhood $U_i$ of $(0,y_i)$ such that $\widetilde f-X^4>0$ on $U_i\setminus\{(0,y_i)\}$.
Then of course also $\widetilde f-\ep X^4>0$ on $U_i\setminus\{(0,y_i)\}$
for all $\ep\le1$ and $i\in\{1,\dots,m\}$. Since $\widetilde f$ is positive on the compact set
$[-1,1]^2\setminus(U_1\cup\dots\cup U_m)$, there is an $\ep\in(0,1)_\R$ such that
$\widetilde f-\ep X^4>0$ on $[-1,1]^2\setminus(U_1\cup\dots\cup U_m)$. Altogether,
$\widetilde f-\ep X^4>0$ on $[-1,1]^2\setminus\{(0,y_1),\dots,(0,y_m)\}$ and
$\widetilde f-\ep X^4=0$ on $\{(0,y_1),\dots,(0,y_m)\}$.
\end{proof}

\begin{lem}\label{extremehas2zeros}
Suppose $f$ lies on an extreme ray [$\to$ \ref{defray}(b)] of the cone $P$ of
the psd $4$-forms in $\R[X,Y,Z]$. Then there are linearly independent $v_1,v_2\in\R^3$ such
that $f(v_1)=f(v_2)=0$.
\end{lem}

\begin{proof}
If $f$ were pd, then the forms $f\pm\ep X^4$ would be psd for some $\ep>0$
(choose $\ep$ for instance as the minimum of $f$ on the compact unit sphere of $\R^3$) and
because of $f=\frac12(f-\ep X^4)+\frac12(f+\ep X^4)$ it would follow that
$f+\ep X^4\in\R_{\ge0}f$ and thus $f\in\R X^4$ $\lightning$. Hence $f$ has at least one
zero $v_1\in\R^3\setminus\{0\}$. After a change of coordinates, we can without loss of generality
achieve $v_1=e_1$. Since $(0,0)$ is a local (even a global) minimizer of $f(1,Y,Z)\in\R[Y,Z]$,
we know from analysis that $\frac{\partial f}{\partial Y}(1,0,0)=\frac{\partial f}{\partial Z}(1,0,0)=0$.
It follows that there are a quadratic form $a\in\R[Y,Z]$, a cubic form $b\in\R[Y,Z]$ and a quartic form $c\in\R[Y,Z]$ such that
\[f=aX^2+bX+c.\]
The quadratic form $a$ is positive semidefinite since either $a(y,z)=0$ or $a(y,z)$ is the leading coefficient of
$f(X,y,z)\in\R[X]$ which is nonnegative on $\R$.
We  have to show that there
exists $v_2\in\R\times(\R^2\setminus\{0\})$ such that $f(v_2)=0$. We make a case distinction
by $\rk(a)$ [$\to$ \ref{longremi}(h)].

\medskip
\textbf{Case 1:}  $\rk(a)=0$

\smallskip
Then $a=0$ and thus $b(y,z)=0$ for all $(y,z)\in\R^2$ from which $b=0$ follows by \ref{pol0}.
If $f=c\in\R[Y,Z]$ was pd, then $c\pm\ep Y^4\in\R[Y,Z]$ would be psd for some $\ep>0$
and it would follow that $c+\ep Y^4\in\R_{\ge0}c$ and thus $c\in\R Y^4$ $\lightning$.
Now choose $(y,z)\in\R^2\setminus\{0\}$ such that $c(y,z)=0$ and set $v_2:=(0,y,z)$. Then $f(v_2)=c(y,z)=0$ and $v_1$ and $v_2$ are linearly independent. 

\medskip
\textbf{Case 2:}  $\rk(a)=1$

\smallskip
By a coordinate change in the $y$-$z$-plane WLOG $a=Y^2$. Then $b(0,z)=0$ for all $z\in\R$
and hence $b(0,Z)=0$, i.e., $b=Yb'$ for some $b'\in\R[Y,Z]$. It follows that
$f=X^2Y^2+b'XY+c=\left(XY+\frac{b'}2\right)^2+\left(c-\frac{b'^2}4\right)$.
For all $(y,z)\in\R^\times\times\R$,
we find some $x\in\R$ satisfying $xy+\frac{b'(y,z)}2=0$ from which
$c(y,z)-\frac{b'(y,z)^2}4=f(x,y,z)\ge0$ follows. Hence $c-\frac{b'^2}4\in P$.
Aside from that, we have of course $(XY+\frac{b'}2)^2\in P$. Since $f$ lies on an extreme
ray of $P$, it follows that $(XY+\frac{b'}2)^2\in\R f$ (and $c-\frac{b'^2}4\in\R f$). Now choose
$(y,z)\in\R^\times\times\R$ arbitrary and with it $x\in\R$ such that
$xy+\frac{b'(y,z)}2=0$. Then $f(x,y,z)=0$.

\medskip
\textbf{Case 3:}  $\rk(a)=2$

\smallskip
By a coordinate change in the $y$-$z$-plane WLOG $a=Y^2+Z^2$. Since $f$ is psd, also
the $6$-form $4ac-b^2\in\R[Y,Z]$ is psd. We have to show that there is
$(y,z)\in\R^2\setminus\{0\}$ such that there exists $x\in\R$ satisfying
$a(y,z)x^2+b(y,z)x+c(y,z)=0$. Because of $a(y,z)\ne0$ for all $(y,z)\in\R^2\setminus\{0\}$,
this is equivalent to the existence of $(y,z)\in\R^2\setminus\{0\}$ with $(b^2-4ac)(y,z)\ge0$,
i.e., $(4ac-b^2)(y,z)=0$ (since $4ac-b^2$ is psd). We have thus to show that $4ac-b^2$
is not pd. Aiming for a contradiction, assume that $4ac-b^2$ is pd. Then also the
$6$-forms $4a(c\pm\ep Y^4)-b^2$ are psd for some $\ep>0$ (choose for example
$4\ep$ as the minimum of $4ac-b^2$ on the compact unit sphere of $\R^2$ and
take into account that $a=Y^2+Z^2$). It follows that $f\pm\ep Y^4\in P$.
From $f=\frac12(f+\ep Y^4)+\frac12(f-\ep Y^4)$, we obtain $f+\ep Y^4\in\R_{\ge0}f$
and thus $f\in\R Y^4$ $\lightning$.
\end{proof}

\begin{lem}\label{base}
Let $d,n\in\N_0$ and let $V$ be the $\R$-vector space of all $2d$-forms in $\R[\x]=
\R[X_1,\dots,X_n]$ and $P\subseteq V$ be the cone of all psd forms in $V$. Then $P$
is a closed cone with compact convex base [$\to$ \ref{defray}(c)].
\end{lem}

\begin{proof}
As an intersection of closed sets, $P=\bigcap_{x\in\R^n}\{p\in V\mid p(x)\ge0\}$ is closed.
By \ref{pol0},
\[\|p\|:=\sum_{x_1=-d}^d\ldots\sum_{x_n=-d}^d|p(x_1,\dots,x_n)|\qquad(p\in V)\]
defines a norm on $V$. Then
\[B:=\{p\in P\mid\|p\|=1\}=\left\{p\in V\mid\sum_{x_1=-d}^d\ldots
\sum_{x_n=-d}^dp(x_1,\dots,x_n)=1\right\}\]
is a compact convex base of $P$.
\end{proof}

\begin{lem}\label{extremequadratic2}
Let $V$ denote the $\R$-vector space of all $4$-forms in $\R[X,Y,Z]$ and $P\subseteq V$
the cone of all psd forms in $V$. Suppose that $f$ lies on an extreme ray of $P$. Then $f$ is
a square of a quadratic form.
\end{lem}

\begin{proof}
It is enough to show that $f$ is a \emph{sum of} squares of quadratic forms for if
$f=\sum_{i=1}^mq_i^2\ne0$ with $2$-forms $q_i\in\R[X,Y,Z]$, then
\[f=\frac12\underbrace{2q_1^2}_{\in P}+\frac12
\underbrace{2\sum_{i=2}^mq_i^2}_{\in P}\]
and thus $q_1^2\in\R_{\ge0}f$. If there is a linear form $\ell\in\R[X,Y,Z]\setminus\{0\}$ such that
$f-\ell^4$ is psd, then $\ell^4\in\R_{\ge0}f$ and $f=(c\ell^2)^2$ for some $c\in\R^\times$ so that
we are done. From now on therefore suppose that such a linear form does not exist.
From the Lemmata \ref{subtractlin4} and \ref{extremehas2zeros}, it follows now that $f$ is a
sum of $3$ squares of $2$-forms in $\R[X,Y,Z]$.
\end{proof}

\begin{thm}\label{h1888}
Let $R$ be a real closed field and $f\in R[X,Y,Z]$ a $4$-form. Then the following are equivalent:
\begin{enumerate}[\normalfont(a)]
\item $f$ is psd.
\item $f\in\sum R[X,Y,Z]^2$
\item $f$ is a sum of squares of quadratic forms in $R[X,Y,Z]$.
\end{enumerate}
\end{thm}

\begin{proof}
\underline{(c)$\implies$(b)$\implies$(a)} is trivial. 

\smallskip
\underline{(a)$\implies$(c)} follows for $R=\R$ from \ref{extremequadratic2} together with
the conic version \ref{conicminkowski} of Minkowski's theorem and \ref{base}. Using the Gram matrix method
\ref{gram} (or \ref{fundthmlin}), one sees that the class of all real closed fields $R$ for which
(a)$\implies$(c) holds for all $4$-forms $f\in R[X,Y,Z]$, is semialgebraic. By \ref{nothingorall}, every
real closed field belongs to this class. In short, the statement follows thus from the case $R=\R$
by the Tarski principle \ref{tprinciple}.
\end{proof}

\begin{cor}[dehomogenized version of \ref{h1888}]\label{dh1888}
Let $R$ be a real closed field and $f\in R[X,Y]_4$. Then
\[\text{$f\ge0$ on $R^2$}\iff f\in\sum R[X,Y]^2.\]
\end{cor}

\begin{proof}
``$\Longleftarrow$'' is trivial.

\smallskip
``$\Longrightarrow$'' Suppose $f\ge0$ on $R^2$. WLOG $f\notin R$. Then $\deg f=2$
or $\deg f=4$ by \ref{soslongrem}(b). For $\deg f=2$, the claim follows from \ref{son1s}.
Suppose therefore $\deg f=4$. Then $f^*:=Z^4f\left(\frac XZ,\frac YZ\right)\in R[X,Y,Z]$
is the homogenization of $f$ with respect to $Z$ [$\to$~\ref{introhom}(c)] and $f^*$ is psd
by \ref{psdpsdhom}(a). Now \ref{h1888} yields $f^*\in\sum R[X,Y,Z]^2$. By
dehomogenization [$\to$ \ref{introhom}(d), \ref{homdehom}], it follows that $f\in\sum R[X,Y]^2$.
\end{proof}

\begin{rem}
A posteriori, we see now that in the situation of Lemma \ref{extremehas2zeros}, there actually exist even infinitely many pairwise linearly independent zeros of $f$. This follows from
\ref{extremequadratic2}. Indeed, if $f=q^2$ with a $2$-form $q\in\R[X,Y,Z]$,
then WLOG $\sg q\ge0$ (otherwise replace $q$ by $-q$) and thus $\sg q\in\{0,1,2,3\}$.

If $\sg q=3$, then $q$ and thus $f$ is positive definite which is of course impossible by \ref{extremehas2zeros}.

If $\sg q=2$, then after a linear change of coordinates we have WLOG $f=X^2+Y^2$ which contradicts again \ref{extremehas2zeros} since any zero of $q$ and hence of $f$
in $\R^3$ lies in $\{(0,0)\}\times\R$.

If $\sg q=1$, then WLOG $q\in\{X^2,X^2+Y^2-Z^2\}$.
If $q=X^2$, then 
for example the $(0,y,1)$ where $y\in\R$ are pairwise linearly
independent zeros of $q$ and therefore also of $f$. If $q=X^2+Y^2-Z^2$, then the $(x,1,\sqrt{x^2+1})$ where $x\in\R$
are pairwise linearly independent zeros of $q$ and therefore of $f$. Indeed even the projections of these vectors onto their first two components are already linearly independent as
we have already seen.

If $\sg q=0$, then WLOG $q\in\{0,X^2-Y^2\}$. The case $q=0$ is trivial. In the case $q=X^2-Y^2$
for example the $(x,x,1)$ where $x\in\R$ are pairwise linearly
independent zeros of $q$ and therefore also of $f$.
Again even the projections of these vectors onto their second and third components are already linearly independent.
\end{rem}

\begin{rem}
We will neither use nor prove the following:
\begin{enumerate}[(a)]
\item In 1888, Hilbert showed a strengthening of \ref{h1888}
(``sum of \emph{three} squares'' instead of ``sum of squares'', cf. also
\ref{h1888a}, \ref{h1888b}, \ref{h1888c} and \ref{subtractlin4}) \cite{hil}.
A very long and tedious elementary proof for this has been given by Scheiderer and
Pfister in 2012 \cite{ps}..
\item Scheiderer showed in 2016 that
\[X^4+XY^3+Y^4-3X^2YZ-4XY^2Z+2X^2Z^2+XZ^3+YZ^3+Z^4\]
is psd but does not belong to $\sum\Q[X,Y,Z]^2$ \cite{s2}. In the same year,
Henrion, Naldi, Safey El Din gave an elementary proof for this \cite{hns}.
\end{enumerate}
\end{rem}

\chapter{Nonnegative polynomials with zeros}

Throughout this chapter, $K$ denotes again always a subfield of $\R$
with the induced order. Moreover, we let $A$ always be a commutative ring (e.g.,
$A=K[X_1,\dots,X_n]$).

\section{Modules over semirings}

\begin{df}\label{deftmodule}
Let $T\subseteq A$. Then we call $T$ a \emph{semiring} of $A$ if $\{0,1\}\subseteq T$,
$T+T\subseteq T$ and $TT\subseteq T$ [$\to$ \ref{defpreorder}]. If $T$ is a semiring of $A$,
then $M\subseteq A$ is called a \emph{$T$-module} of $A$ if $0\in M$, $M+M\subseteq M$
and $TM\subseteq M$.
\end{df}

\begin{rem}\label{semiringrem}
\begin{enumerate}[(a)]
\item $\text{$T$ is a preorder of $A$}\iff(\text{$T$ is a semiring of $A$} \et A^2\subseteq T)$
\item If $T$ is a semiring of $A$, then $T-T$ is a subring of $A$.
\item If $T$ is a semiring of $A$ and $M$ a $T$-module of $A$, then
$M-M$ is a $(T-T)$-module of $A$.
\item If $T$ is a semiring of $A$, then $T$ is a $T$-module of $A$. 
\end{enumerate}
\end{rem}

\begin{df}\label{defarchsemiring}
Let $T$ be a semiring of $A$ and $M$ a $T$-module of $A$.
Then $M$ is called \emph{Archimedean} (in $A$) if
$\forall a\in A:\exists N\in\N:N+a\in M$ [$\to$ \ref{dfarch}(a)].
\end{df}

\begin{rem}
Due to \ref{semiringrem}(d), the notion of an Archimedean semiring is also defined by
\ref{defarchsemiring}. Because of \ref{semiringrem}(a), this generalizes the notion of an
Archimedean preorder of $A$ [$\to$ \ref{dfarch}(a)].
\end{rem}

\begin{df}{}[$\to$ \ref{arithmbounded}]\label{bamu}
Let $T$ be a semiring of $A$, $M$ a $T$-module of $A$ and $u\in A$. Then
\[B_{(A,M,u)}:=\{a\in A\mid\exists N\in\N:Nu\pm a\in M\}\]
the set of with respect to $M$ by $u$ \emph{arithmetically bounded} elements of $A$.
If $u=1$, then we write $B_{(A,M)}:=B_{(A,M,u)}$ and omit the specification ``by $u$''.
\end{df}

\begin{pro}\label{mboundedtimesmbounded}
Suppose $T$ is a semiring of $A$, $M_1$ and $M_2$ are $T$-modules of $A$,
$u_1\in M_1$ and $u_2\in M_2$. Then $\sum M_1M_2$ is also a $T$-module of $A$
and we have
\[B_{(A,M_1,u_1)}B_{(A,M_2,u_2)}\subseteq B_{(A,\sum M_1M_2,u_1u_2)}.\]
\end{pro}

\begin{proof}
Let $a_i\in B_{(A,M_i,u_i)}$, say $Nu_i\pm a_i\in M_i$ for $i\in\{1,2\}$ with $N\in\N$.
Then (cf. the proof of \ref{arithmbounded})
\[3N^2u_1u_2\pm a_1a_2=(Nu_1+a_1)(Nu_2\pm a_2)+Nu_2(Nu_1-a_1)+
Nu_1(Nu_2\mp a_2).\]
\end{proof}

\begin{cor} Let $T$ be a semiring of $A$, $M$ a $T$-module of $A$, $u\in T$ and $v\in M$.
Then $B_{(A,T,u)}B_{(A,M,v)}\subseteq B_{(A,M,uv)}$.
\end{cor}

\begin{proof}
Apply \ref{mboundedtimesmbounded} to $M_1:=T$, $M_2:=M$, $u_1:=u$,
$u_2:=v$ and observe $\sum M_1M_2=\sum TM=M$.
\end{proof}

\begin{cor}{}\emph{[$\to$ \ref{arithmbounded}]}\label{bmodule}
Let $T$ be a semiring of $A$. Then $B_{(A,T)}$ is a subring of $A$.
Moreover, if $M$ a $T$-module of $A$ and $u\in M$, then $B_{(A,M,u)}$
is a $B_{(A,T)}$-module of $A$.
\end{cor}

\begin{rem}{}[$\to$ \ref{defarchsemiring}, \ref{archb}]\label{archmoduleb}
If $T\subseteq A$ is a semiring and $M\subseteq A$ a $T$-module with $1\in M$, then $M$ is Archimedean if
and only if $B_{(A,M)}=A$.
\end{rem}

\begin{thm}{}\emph{[$\to$ \ref{archchar}]}\label{archsemiringchar}
Let $n\in\N_0$ and $T\subseteq K[\x]$ a semiring with $K_{\ge0}\subseteq T$.
Then the following are equivalent:
\begin{enumerate}[\normalfont(a)]
\item $T$ is Archimedean.
\item $\exists N\in\N:\forall i\in\{1,\dots,n\}:N\pm X_i\in T$
\item $\exists m\in\N:\exists\ell_1,\dots,\ell_m\in T\cap K[\x]_1:\exists N\in\N:\\
\emptyset\ne\{x\in K^n\mid\ell_1(x)\ge0,\dots,\ell_m(x)\ge0\}\subseteq[-N,N]_K^n$
\end{enumerate}
\end{thm}

\begin{proof} Write $A:=K[\x]$. From $K_{\ge0}\subseteq T$, it follows that
$K\subseteq B_{(A,T)}$. Hence we have $B_{(A,T)}=A\iff X_1,\dots,X_n\in B_{(A,T)}$ which
shows (a)$\iff$(b). The implication (b)$\implies$(c) is trivial and (c)$\implies$(b) is an easy
consequence of the linear Nichtnegativstellensatz \ref{linnichtnegativstellensatz}.
\end{proof}

\begin{lem}{}[$\to$ \ref{squarerootsarithmeticallybounded}]\label{rootqmb}
Suppose that $\frac12\in A$ (i.e., $2\in A^\times$), let $M\subseteq A$ be a
$(\sum A^2)$-module with $1\in M$ and let $a\in A$. Then
\[a^2\in B_{(A,M)}\iff a\in B_{(A,M)}.\]
\end{lem}

\begin{proof}
``$\Longrightarrow$'' If $N\in\N$ with $(N-1)-a^2\in M$, then
\[N\pm a=(N-1)-a^2+\left(\frac12\pm a\right)^2+3\left(\frac12\right)^2\in M\]
(exactly as in the proof of \ref{squarerootsarithmeticallybounded}).

``$\Longleftarrow$'' If $N\in\N$ with $(2N-1)\pm a\in M$, then
\[N^2(2N-1)-a^2=2\left(\frac12\right)^2((N-a)^2(2N-1+a)+(N+a)^2(2N-1-a))\in M.\]
\end{proof}

\begin{pro}\label{bammodule}
Suppose $\frac12\in A$, $T\subseteq A$ is a preorder and $M\subseteq A$
is a $T$-module with $1\in M$. Then $B_{(A,M)}$ is a subring of $A$ and $B_{(A,M,u)}$ a
$B_{(A,M)}$-module of $A$ for each $u\in T$.
\end{pro}

\begin{proof}
It obviously suffices to show $B_{(A,M)}B_{(A,M,u)}\subseteq B_{(A,M,u)}$ for all $u\in T$
(since this means $B_{(A,M)}B_{(A,M)}\subseteq B_{(A,M)}$ for $u=1$).
If $a\in B_{(A,M)}$, then we have \[a=\left(\frac12\right)^2((a+1)^2-(a-1)^2)\] and
because of $1\in M$ also $a+1,a-1\in B_{(A,M)}$. Therefore it is enough to show
$a^2B_{(A,M,u)}\subseteq B_{(A,M,u)}$ for all $a\in B_{(A,M)}$ and $u\in T$. For this purpose,
fix $a\in B_{(A,M)}$, $u\in T$ and $b\in B_{(A,M,u)}$. To show: $a^2b\in B_{(A,M,u)}$.
From \ref{rootqmb}, we get $a^2\in B_{(A,M)}$. Choose $N\in\N$ such that
$N-a^2,Nu\pm b\in M$. Due to $a^2,u\in T$, we get now $Nu-ua^2,Nua^2\pm a^2b\in M$.
Consequently,
\[N^2u\pm a^2b=(N^2u-Nua^2)+(Nua^2\pm a^2b)\in M+M\subseteq M.\]
\end{proof}

\begin{thm}{}\emph{[$\to$ \ref{archchar}, \ref{archsemiringchar}]}
\label{archmodulechar}
Suppose $n\in\N_0$ and $M\subseteq K[\x]$ is a $(\sum K_{\ge0}K[\x]^2)$-module with
$1\in M$. Then the following are equivalent:
\begin{enumerate}[\normalfont(a)]
\item $M$ is Archimedean.
\item $\exists N\in\N:N-\sum_{i=1}^nX_i^2\in M$
\item $\exists N\in\N:\forall i\in\{1,\dots,n\}:N\pm X_i\in M$
\item $\exists m\in\N:\exists\ell_1,\dots,\ell_m\in M\cap K[\x]_1:\exists N\in\N:\\
\emptyset\ne\{x\in K^n\mid\ell_1(x)\ge0,\dots,\ell_m(x)\ge0\}\subseteq[-N,N]_K^n$
\end{enumerate}
\end{thm}

\begin{proof}
\underline{(a)$\implies$(b)} is trivial.

\smallskip
\underline{(b)$\implies$(c)}\quad If (b) holds, then $N-X_i^2\in M$ and thus $X_i^2\in B_{(K[\x],M)}$
for all $i\in\{1,\dots,n\}$. Apply now \ref{rootqmb}.

\smallskip
\underline{(c)$\implies$(d)} is trivial and \underline{(d)$\implies$(c)} follows again from the linear
Nichtnegativstellensatz \ref{linnichtnegativstellensatz}.

\smallskip
\underline{(c)$\implies$(a)} follows from \ref{bammodule}.
\end{proof}

\section{Pure states on rings and ideals}

In this section, we always suppose that the field $K$ is a subring of $A$. In particular,
$\Q\subseteq A$ and $A$ is a $K$-vector space.

\begin{rem}\label{abstractarchimedeanpositivstellensatz2}
Under the just made mild hypothesis $\Q\subseteq A$, one can reformulate the
abstract Archimedean Positivstellensatz \ref{abstractarchimedeanpositivstellensatz}
as follows:

\begin{center}
\fbox{
\begin{minipage}{35em}
For arbitrary $A$ and $K$ as above,
let $T$ be an Archimedean preorder of $A$ such that $K_{\ge0}\subseteq T$ and $a\in A$.
Then the following are equivalent:
\begin{enumerate}[(a)]
\item $\ph(a)>0$ for all $K$-linear ring homomorphisms $\ph\colon A\to\R$ with
$\ph(T)\subseteq\R_{\ge0}$.
\item $\exists N\in\N:a\in\frac1N+T$
\end{enumerate}
\end{minipage}
}
\end{center}
To see this, first note that in (a),
one can omit the $K$-linearity of $\ph$ since it just means that $\ph|_K=\id_K$
which follows from \ref{archemb} by $K_{\ge0}\subseteq T$ since the identity is the
\emph{only} embedding of ordered fields from $K$ to $\R$ (cf. the proof of \ref{evashom}).
But then the theorem becomes
strongest for $K=\Q$ and we can thus assume $K=\Q$ which makes
redundant the hypothesis $K_{\ge0}\subseteq T$ since for all $m,n\in\N$, we have
$\frac mn=mn\left(\frac1n\right)^2\in\sum A^2\subseteq T$. 
This last fact also shows (b)$\iff$(b') where we denote by (a') and (b') the
corresponding conditions from
\ref{abstractarchimedeanpositivstellensatz}, namely:
\begin{enumerate}[(a')]
\item $\widehat a>0$ on $\sper(A,T)$
\item $\exists N\in\N:Na\in1+T$
\end{enumerate}
It remains to show that (a)$\iff$(a'). To this end,
it suffices by \ref{archmax}(d) to show that (a') is equivalent to
\begin{enumerate}[(a'')]
\item $\widehat a(Q)>0$ for all maximal elements $Q$ of $\sper(A,T)$.
\end{enumerate}
It is clear that (a')$\implies$(a''). To show (a'')$\implies$(a'), suppose that (a'') holds and let
$P\in\sper(A,T)$. To show: $\widehat a(P)>0$. Using \ref{inmaxprimecone} or
\ref{spear}, we find a maximal element $Q$ of $\sper(A,T)$ such that
$P\subseteq Q$. By \ref{primeconeinclusion}, we have $Q=P\cup\supp(Q)$. Due to (a''), we
have $a\in Q\setminus-Q$, i.e., $a\in Q\setminus\supp(Q)\subseteq P$,
and because of $a\notin-P$ (for otherwise $a\in-Q$) it follows that $a\in P\setminus-P$,
i.e., $\widehat a(P)>0$. This shows (a')$\iff$(a''). These
arguments were implicitly present already in the proof of 
\ref{archimedeanpositivstellensatz}.
\end{rem}

\begin{rem}\label{semiringu1}
Suppose $T$ is a semiring of $A$ with $K_{\ge0}\subseteq T$ and $M$ a $T$-module of $A$.
Then $M$ is a cone in the $K$-vector space $A$ and we have:
\[\text{$M$ is Archimedean [$\to$ \ref{defarchsemiring}]}\iff
\text{$1$ is a unit for $M$ [$\to$ \ref{defunit}]}\]
\end{rem}

\begin{motivation}\label{puremotivation}
If $T$ is an Archimedean preorder of $A$ with $K_{\ge0}\subseteq T$, then the Archimedean
Positivstellensatz \ref{abstractarchimedeanpositivstellensatz} in the version of
\ref{abstractarchimedeanpositivstellensatz2} amounts to the equivalence of
\[\exists N\in\N:a\in\frac1N+T\] with
\[(*)\qquad\ph(a)>0\text{ for all ($K$-linear) ring homomorphisms }\ph\colon A\to\R
\text{ with $\ph(T)\subseteq\R_{\ge0}$}
\]
while \ref{conemembershipunitextreme}, paying attention to \ref{semiringu1},
tells that the same condition is equivalent to
\[(**)\qquad\ph(a)>0\text{ for all pure states }\ph\text{ of }(A,T,1).
\]
The following imprecise questions arise:
\begin{enumerate}[(a)]
\item What do pure states ``on rings'' have to do with ring homomorphisms?
\item Can the Archimedean Positivstellensatz be generalized from preorders to semirings
or even to modules over semirings?
\item If $(*)$ holds only with ``$\ge$'' instead of ``$>$'', then $\exists N\in\N:a\in\frac1N+T$
can of course not hold anymore but one would still want to prove that $a\in T$.
In this case, is it possible to find an ideal $I\subseteq A$
(e.g., the kernel of a ring homomorphism $\ph$ from $(*)$ with $\ph(a)=0$) such that
$I\cap T$ possesses in the $K$-vector space $I$ a unit $u$ in such a way that
$a\in I$ and $(**)$ holds for $(I,I\cap T,u)$ instead of $(A,T,1)$?
Then one could apply \ref{conemembershipunitextreme} or \ref{conemembershipextr}
in order to finally still show that $a\in T$ (even $a\in\frac1Nu+(I\cap T)$).
\item What can one say about pure states ``on ideals''? This question generalizes (a) and is
motivated by (c).
\end{enumerate}
\end{motivation}

\begin{reminder}\label{binomialseries}
For $z\in\C$ and $k\in\N_0$, the binomial coefficient
\[\binom zk:=\prod_{i=1}^k\frac{z-i+1}i\] is declared. From analysis, one knows that
\[\sqrt{1+t}=(1+t)^{\frac12}=\sum_{i=0}^\infty\binom{\frac12}it^i\] for all $t\in\R$ with $|t|<1$.
\end{reminder}

\begin{lem}\label{binomialmiracle}
For all $k\in\N$, the coefficients of
\[p_k:=\left(\sum_{i=0}^k\binom{\frac12}i(-T)^i\right)^2-(1-T)\in\Q[T]\]
are nonnegative.
\end{lem}

\begin{proof}
In the ring $\Q[[T]]$ of formal power series, we have because of \ref{binomialseries}
and the identity theorem for power series from analysis that
\[\left(\sum_{i=0}^\infty\binom{\frac12}i(-T)^i\right)^2=1-T.\]
Now let $k\in\N$ be fixed. For $i\in\N_0$ with $i\le k$, the coefficient of $T^i$ in $p_k$
obviously equals the coefficient of $T^i$ in
\[\left(\sum_{i=0}^\infty\binom{\frac12}i(-T)^i\right)^2-(1-T)\]
which is zero. The binomial coefficient $\binom{\frac12}i$ is positive for $i\in\{0,1,3,5,\dots\}$
and negative for $i\in\{2,4,6,\dots\}$. The only positive coefficient of
\[\sum_{i=0}^k\binom{\frac12}i(-T)^i\]
is thus the constant term. Hence,
for $i\in\N_0$ with $i>k$, the coefficient of $T^i$ in $p_k$ is thus a sum of products of two
nonpositive reals and therefore nonnegative.
\end{proof}

\begin{lem}\label{magiclemma}
Suppose $I$ is an ideal of $A$, $T$ is a preorder of $A$ with $K_{\ge0}\subseteq T$,
$M\subseteq I$ is a $T$-module of $A$, $u$ is a unit for $M$ in $I$, $a\in T$ and
$(1-2a)u\in M$. Then [$\to$ \ref{defstate}] \[S(I,M,u)\subseteq S(I,(1-a)M,u).\]
\end{lem}

\begin{proof}
Let $\ph\in S(I,M,u)$. To show: $\ph((1-a)M)\subseteq\R_{\ge0}$. Let $b\in M$.
To show: \[\ph((1-a)b)\ge0.\] WLOG $u-b\in M$
(otherwise choose $N\in\N$ with $Nu-b\in M$ and replace $b$ by $\frac1Nb\in M$).
We show $\ph((1-a)b)>-\ep$ for all $\ep>0$. To this end, let $\ep>0$. It is enough to show
that there is a $k\in\N$ satisfying
\[\ph((1-a)b)>\ph\left(\left(\sum_{i=0}^k\binom{\frac12}i(-a)^i\right)^2b\right)-\ep\]
since $A^2M\subseteq TM\subseteq M\subseteq\ph^{-1}(\R_{\ge0})$. Because of $a\in T$,
we have
\[(1-(2a)^i)u=\sum_{j=0}^{i-1}((2a)^j-(2a)^{j+1})u=\sum_{j=0}^{i-1}(2a)^j(1-2a)u\in M\]
for all $i\in\N_0$, i.e.,
\[(\square)\qquad\left(\frac1{2^i}-a^i\right)u\in M\]
for all $i\in\N_0$. By \ref{binomialseries}, we can choose $k\in\N$ such that
\[\left(\sum_{i=0}^k\binom{\frac12}i\left(-\frac12\right)^i\right)^2<\left(1-\frac12\right)+\ep,\]
i.e., $p_k\left(\frac12\right)<\ep$ with $p_k$ as in Lemma \ref{binomialmiracle}.
We show that $\ph(p_k(a)b)<\ep$ which is exactly our claim. Since $p_k(a)\in T$ holds by Lemma \ref{binomialmiracle},
it is enough to show that $\ph(p_k(a)u)<\ep$ since
$\ph(p_k(a)b)\le\ph(p_k(a)u)$ holds due to
$p_k(a)(u-b)\in M$. But we have
\[\ph(p_k(a)u)\le\ph(p_k\left(\frac12\right)u)
=p_k\left(\frac12\right)\ph(u)=p_k\left(\frac12\right)<\ep\]
due to $(p_k\left(\frac12\right)-p_k(a))u\in M$ 
(use \ref{binomialmiracle} and $(\square)$).
\end{proof}

\begin{thm}[Burgdorf, Scheiderer, Schweighofer \cite{bss}]\emph{[$\to$ \ref{puremotivation}(d)]}
\label{puremult}
Suppose that $I$ is an ideal of $A$, $T$ is a preorder or an Archimedean semiring of $A$,
$K_{\ge0}\subseteq T$, $M\subseteq I$ is a $T$-module of $A$, $u$ is a unit for $M$ in
$I$ and $\ph$ is a pure state of $(I,M,u)$. Then
\[(*)\qquad\ph(ab)=\ph(au)\ph(b)\]
for all $a\in A$ and $b\in I$.
\end{thm}

\begin{proof}
Due to $T-T=A$ [$\to$ \ref{diffsquare}, \ref{defarchsemiring}] it suffices to show $(*)$ for all
$a\in T$ and $b\in I$. If $T$ is \alal{an Archimedean semiring}{a preorder}, then one can
here suppose by scaling $a$ that $\malal{1-a\in T}{u-2au\in M}$ and thus because of
\alal{$TM\subseteq M$}{\text{Lemma \ref{magiclemma}}} that
\[S(I,M,u)\subseteq S(I,(1-a)M,u).\]
Moreover, we can suppose that $\ph(au)<1$. Fix therefore $a\in T$ with
$S(I,M,u)\subseteq S(I,(1-a)M,u)$ and $\ph(au)<1$. We have to show $(*)$ for all $b\in I$.

\medskip
\textbf{Case 1:}  $\ph(au)=0$

\smallskip
Then we have to show that $\ph(ab)=0$ for all $b\in I$. For this purpose, fix $b\in I$. Choose
$N\in\N$ such that $Nu\pm b\in M$. Then $Nau\pm ab\in TM\subseteq M$ and therefore
$|\ph(ab)|\le N\ph(au)=0$. Hence $\ph(ab)=0$.

\medskip
\textbf{Case 2:}  $\ph(au)\ne0$

\smallskip
Then $\ph(au)>0$ because of $au\in TM\subseteq M$. Furthermore, we have
$\ph((1-a)u)>0$ since $\ph(au)<1=\ph(u)$. For each $c\in A$ with $\ph(cu)>0$ and
$\ph\in S(I,cM,u)$,
\[\ph_c\colon I\to\R,\ b\mapsto\frac{\ph(cb)}{\ph(cu)}\]
is a state of $(I,M,u)$. In particular, $\ph_a,\ph_{1-a}\in S(I,M,u)$. Because of
\[\ph=\ph(au)\ph_a+\ph((1-a)u)\ph_{1-a},\] $\ph(au)>0$, $\ph((1-a)u)>0$ and
$\ph(au)+\ph((1-a)u)=\ph(u)=1$, we have by \ref{extremeexo} or
\ref{ratiootherthan2} that $\ph=\ph_a$ (and $\ph=\ph_{1-a}$).
\end{proof}

\begin{cor}{}\emph{[$\to$ \ref{puremotivation}(a)]}\label{pureringhom1}
Let $T$ be an Archimedean semiring of $A$ such
that $K_{\ge0}\subseteq T$ and $M$ a $T$-module of $A$ with $1\in M$. Then every pure state
of $(A,M,1)$ is a ring homomorphism.
\end{cor}

\begin{cor}{}\emph{[$\to$ \ref{puremotivation}(a)]}\label{pureringhom2}
Let $M$ be an Archimedean $\left(\sum K_{\ge0}A^2\right)$-module of $A$.
Then every pure state of $(A,M,1)$ is a ring homomorphism.
\end{cor}

\begin{cor}[Becker, Schwartz \cite{bs}, first generalization of the abstract Archimedean 
Positivstellensatz \ref{abstractarchimedeanpositivstellensatz}
in the version of \ref{abstractarchimedeanpositivstellensatz2}]
\emph{[$\to$ \ref{puremotivation}(b)]}
\label{beckerschwartzarchimedean}
Let $T$ be an Archimedean semiring of $A$ with $K_{\ge0}\subseteq T$, $M$ a $T$-module
of $A$ with $1\in M$ and $a\in A$. Then the following are equivalent:
\begin{enumerate}[\normalfont(a)]
\item $\ph(a)>0$ for all ($K$-linear) ring homomorphisms $\ph\colon A\to\R$ with
$\ph(M)\subseteq\R_{\ge0}$.
\item $\exists N\in\N:a\in\frac1N+M$
\end{enumerate}
\end{cor}

\begin{proof}
\ref{conemembershipunitextreme}, \ref{semiringu1}, \ref{pureringhom1}
\end{proof}

\begin{cor}[Jacobi \cite{jac}, second generalization of the abstract Archimedean 
Positivstellensatz \ref{abstractarchimedeanpositivstellensatz}
in the version of \ref{abstractarchimedeanpositivstellensatz2}]
\emph{[$\to$ \ref{puremotivation}(b)]}
\label{jacobiarchimedean}
Let $M$ be an Archimedean $\left(\sum K_{\ge0}A^2\right)$-module of $A$. Then (a) and (b)
from \ref{beckerschwartzarchimedean} are equivalent.
\end{cor}

\begin{rem}
Using Lemma \ref{evashom}, one gets for the polynomial ring
$K[\x]$ concrete geometric versions of \ref{beckerschwartzarchimedean} and
\ref{jacobiarchimedean} which are completely analogous to \ref{archimedeanpositivstellensatz}
(first and second generalization of the Archimedean Positivstellensatz). Instead of
stating them, we give immediately concrete examples.
\end{rem}

\begin{ex}{}[$\to$ \ref{beckerschwartzarchimedean}]\label{hman}
Let $\ell_1,\dots,\ell_m\in\R[\x]_1$ such that
\[\{x\in\R^n\mid\ell_1(x)\ge0,\dots,\ell_m(x)\ge0\}\]
is nonempty and compact. Moreover, let $g_1,\dots,g_\ell\in\R[\x]$ and set
\[S:=\{x\in\R^n\mid\ell_1(x)\ge0,\dots,\ell_m(x)\ge0,g_1(x)\ge0,\dots,g_\ell(x)\ge0\}.\]
Then for each $f\in\R[\x]$ with $f>0$ on $S$, we have
\[f\in\sum_{i=0}^\ell\sum_{\al\in\N_0^m}\R_{\ge0}\ell_1^{\al_1}\dotsm\ell_m^{\al_m}g_i=:M\]
where $g_0:=1$. This is because $M$ is a $T$-module with $1\in M$ for the
semiring \[T:=\sum_{\al\in\N_0^m}\R_{\ge0}\ell_1^{\al_1}\dotsm\ell_m^{\al_m}\] which is
Archimedean by \ref{archsemiringchar}(c).
\end{ex}

\begin{ex}[Putinar][$\to$ \ref{jacobiarchimedean}]\label{putinar}
Let $R\in\R_{\ge0}$ and let $g_1,\dots,g_m\in\R[\x]$. Set
\[S:=\{x\in\R^n\mid g_1(x)\ge0,\dots,g_m(x)\ge0,\|x\|\le R\}.\]
Then for every $f\in\R[\x]$ with $f>0$ on $S$, we have
\[f\in\sum_{i=0}^{m+1}\sum\R[\x]^2g_i\]
with $g_0:=1$ and $g_{m+1}:=R^2-\sum_{i=1}^nX_i^2$ [$\to$ \ref{archmodulechar}(b)].
\end{ex}

\begin{ex}[P\'olya \cite{pol}][$\to$ \ref{beckerschwartzarchimedean}]
Let $k\in\N_0$ and suppose $f\in\R[\x]$ a $k$-form such that $f(x)>0$ for all
$x\in\R_{\ge0}^n\setminus\{0\}$. Then there is some $N\in\N$ such that
\[(X_1+\dots+X_n)^Nf\in\sum_{\substack{\al\in\N_0^n\\|\al|=N+k}}
\R_{>0}\x^\al.\]
This can be shown as follows: We have $f>0$ on
$\De:=\{x\in\R_{\ge0}^n\mid x_1+\ldots+x_n=1\}$. By \ref{beckerschwartzarchimedean},
we obtain analogously to \ref{hman} that
\[f-\ep\in\sum_{\al\in\N_0^{n+2}}\R_{\ge0}X_1^{\al_1}\dotsm X_n^{\al_n}
(1-(X_1+\dots+X_n))^{\al_{n+1}}
(X_1+\dots+X_n-1)^{\al_{n+2}}.
\]
By substituting $X_i\mapsto\frac{X_i}{X_1+\dots+X_n}$ and clearing denominators, one gets the
claim due to homogeneity of $f$.
\end{ex}

\section{Dichotomy of pure states on ideals}

In this section, we let $K$ again be a subring of $A$ so that we consider $A$ also
as a $K$-vector space.

\begin{pro}\label{assringhom}
Let $I$ be a ideal of $A$ and $u\in I$. Let $\ph\in S(I,\emptyset,u)$
\emph{[$\to$ \ref{defstate}]}.
Then the following are equivalent:
\begin{enumerate}[\normalfont(a)]
\item $\forall a\in A:\forall b\in I:\ph(ab)=\ph(au)\ph(b)$ \emph{[$\to$ \ref{puremult}$(*)$]}
\item There is a ring homomorphism $\Ph\colon A\to\R$ such that
\[(**)\qquad\forall a\in A:\forall b\in I:\ph(ab)=\Ph(a)\ph(b).\]
\end{enumerate}
In Condition (b), $\Ph$ is uniquely determined since $(**)$ implies
$\Ph(a)=\ph(au)$ for all $a\in A$ and we call $\Ph$ the ring homomorphism \emph{belonging
to} or \emph{associated to}
$\ph$ (on $A$). Note that $\Ph$ does not depend on $u$ for if $v\in I$ with $\ph(v)=1$ then
$(**)$ of course also implies $\Ph(a)=\ph(av)$. Exactly one of the following alternatives occurs:
\begin{enumerate}[\normalfont(1)]
\item $\Ph(u)\ne0$ and $\forall b\in I:\ph(b)=\frac{\Ph(b)}{\Ph(u)}$
\item $\Ph|_I=0$
\end{enumerate}
\end{pro}

\begin{proof}
\underline{(a)$\implies$(b)}\quad If (a) holds, then
$\Ph\colon A\to\R,\ a\mapsto\ph(au)$ is a ring homomorphism since
$\Ph(a)\Ph(b)=\ph(au)\ph(bu)\overset{(*)}=\ph(abu)=\Ph(ab)$ holds for all $a,b\in A$.

\smallskip
\underline{(b)$\implies$(a)} is clear.

\medskip\noindent
Because of $u\in I$ it is clear that (1) and (2) exclude
each other. If $\ph(u^2)\ne0$, then (1) occurs since $(*)$ implies
$\ph(bu)=\ph(u^2)\ph(b)$ for all $b\in I$. If $\ph(u^2)=0$, then
$\ph(bu)=\ph(u^2)\ph(b)=0\ph(b)=0$ for all $b\in I$.
\end{proof}

\begin{thm}[Dichotomy]\label{dichotomy}
Under the hypotheses of \ref{puremult}, exactly one of the following cases occurs:
\begin{enumerate}[\normalfont(1)]
\item $\ph$ is the restriction of a scaled ring homomorphism: There is a ring homomorphism
$\Ph\colon A\to\R$ such that $\Ph(u)\ne0$ and $\ph=\frac1{\Ph(u)}\Ph|_I$.
\item There is a ring homomorphism $\Ph\colon A\to\R$ with $\Ph|_I=0$ such that
$(**)$ from \ref{assringhom}(b) holds.
\end{enumerate}
We have $\text{\normalfont(1)}\iff\ph(u^2)\ne0$ and $\text{\normalfont(2)}\iff\ph(u^2)=0$. 
In both {\normalfont(1)} and {\normalfont(2)}, $\Ph$ is uniquely determined, namely it is the ring homomorphism
that according to \ref{assringhom} belongs to $\ph$. We have $\Ph(T)\subseteq\R_{\ge0}$.
If $u\in T$, then additionally $\Ph(M)\subseteq\R_{\ge0}$.
\end{thm}

\begin{proof}
Easy with \ref{puremult} and \ref{assringhom}.
\end{proof}

\begin{cor}
Let $M$ be a $\left(\sum K_{\ge0}A^2\right)$-module of $A$ with $1\in M$. If $M$ has a unit in
$A$, then $M$ is Archimedean.
\end{cor}

\begin{proof}
Let $u$ be a unit for $M$ in $A$. By \ref{conemembershipunitextreme}, it is enough to show
that $\ph(1)>0$ for all $\ph\in\extr S(A,M,u)$. Now let $\ph$ be a pure state of
$(A,M,u)$ with the associated ring homomorphism $\Ph\colon A\to\R$. Due to
$\Ph(1)=1\ne0$, in the Dichotomy \ref{dichotomy}
only case (1) can occur, i.e., $\Ph(u)\ne0$ and $\ph=\frac1{\Ph(u)}\Ph$.
Because of $\Ph(u)=\ph(u^2)=\ph(u^2\cdot1)\in\ph(M)\subseteq\R_{\ge0}$, we have
$\Ph(u)>0$. It follows that $\ph(1)>0$.
\end{proof}

\begin{ex}\label{triangle1}
Consider the semiring $T:=\sum_{\al,\be,\ga\in\N_0}K_{\ge0}X^\al Y^\be(1-X-Y)^\ga$
of $K[X,Y]$ and
\[S:=\{(x,y)\in\R^2\mid\forall p\in T:p(x,y)\ge0\}=
\{(x,y)\in\R^2\mid x\ge0,y\ge0,x+y\le1\}.\] Since $S$ is bounded and
$X,Y,1-X-Y$ are linear, $T$ is Archimedean by \ref{archsemiringchar}(c).
Consider the ideal $I:=(X,Y)$ and the $T$-module $M:=T\cap I$ of $K[X,Y]$.
Then $u:=X+Y$ is a unit for $M$ in $I$ because $B_{(K[X,Y],T)}\overset{\ref{archmoduleb}}=K[X,Y]$
and thus by \ref{bmodule}
$B_{(K[X,Y],M,u)}$ is an ideal of
$K[X,Y]$ that contains $X$, $Y$ and thus
$I$ since $u\pm X,u\pm Y\in M$. The ring homomorphisms
\[\Ph\colon K[X,Y]\to\R\]
satisfying $\Ph(T)\subseteq\R_{\ge0}$ are obviously exactly the evaluations
$\ev_x$ in points $x\in S$ (compare Lemma \ref{evashom}). Now let $\ph$ be a pure state
of $(I,M,u)$. By the Dichotomy \ref{dichotomy}, exactly one of the following cases occurs:
\begin{enumerate}[(1)]
\item There is some $(x,y)\in S\setminus\{(0,0)\}$ with
$\ph(p)=\frac{p(x,y)}{x+y}$ for all $p\in I$.
\item $\ph(pX+qY)=\ph(pX)+\ph(qY)=p(0,0)\ph(X)+q(0,0)\ph(Y)$ for all $p,q\in K[X,Y]$.
\end{enumerate}
In Case (2), one can set $\la_1:=\ph(X)\ge0$ and $\la_2:=\ph(Y)\ge0$ and one obtains
$\la_1+\la_2=\ph(X+Y)=\ph(u)=1$ as well as $\ph=\la_1\ph_1+\la_2\ph_2$ with
$\ph_1\colon I\to\R,\ p\mapsto\frac{\partial p}{\partial X}(0,0)$ and
$\ph_2\colon I\to\R,\ p\mapsto\frac{\partial p}{\partial Y}(0,0)$. Since every polynomial
in $M$ vanishes in the origin and is nonnegative on $S$, we obtain
$\ph_1,\ph_2\in S(I,M,u)$. Because of $\ph\in\extr S(I,M,u)$, in Case~(2)
we have $\ph=\ph_1$ or $\ph=\ph_2$. Using \ref{conemembershipunitextreme}, we now
obtain: If $f\in K[X,Y]$ with $f>0$ on $S\setminus\{0\}$, $f(0)=0$,
$\frac{\partial f}{\partial X}(0)>0$ and $\frac{\partial f}{\partial Y}(0)>0$, then $f\in T$.
\end{ex}

\begin{ex}\label{triangle2}
Let $T$ and $S$ be as in Example \ref{triangle1}. Consider the ideal $I:=(X)$ and the
$T$-module $M:=T\cap I$ of $K[X,Y]$. Then $u:=X$ is a unit for $M$ in $I$ since
$B_{(K[X,Y],M,u)}$ is an ideal of $K[X,Y]$ by \ref{bmodule} that contains $X$ and thus $I$
because $u\pm X\in M$. Let $\ph$ be a pure state of $(I,M,u)$. By the Dichotomy
\ref{dichotomy}, exactly one of the following cases occurs:
\begin{enumerate}[(1)]
\item There is some $(x,y)\in S\setminus(\{0\}\times\R)$ with $\ph(p)=\frac{p(x,y)}x$ for all
$p\in I$.
\item There is some $y\in[0,1]$ such that $\ph(pX)=p(0,y)\ph(X)=p(0,y)\ph(u)=p(0,y)$ for
alle $p\in K[X,Y]$.
\end{enumerate}
In Case (2), there is obviously a $y\in[0,1]$ such that $\ph(p)=\frac{\partial p}{\partial X}(0,y)$
for all $p\in I$. Observe that each $f\in K[X,Y]$ with $f=0$ on $S\cap(\{0\}\times\R)$
satisfies $f(0,Y)=0$ and thus $f\in I$. Now \ref{conemembershipunitextreme} yields:
If $f\in K[X,Y]$ with $f>0$ on $S\setminus(\{0\}\times\R)$, $f=0$ on
$S\cap(\{0\}\times\R)$ and $\frac{\partial f}{\partial X}(0,y)>0$ for all $y\in[0,1]$, then
$f\in T$. At first glance, it might irritate that one would have to check here that
$\frac{\partial f}{\partial X}(0,1)>0$. However, note that for $y=1$ and in fact for every
$y\in\R$, $\frac{\partial f}{\partial X}(0,y)$ is the derivative of $f$ in \emph{every} direction
$(1,z)$ with $z\in\R$ since $\frac{\partial f}{\partial Y}(0,y)=0$.
\end{ex}

\begin{ex}
Let $T$ and $S$ again be as in \ref{triangle1} and \ref{triangle2}. Consider the ideal
$I:=(X^2,XY)$ and the $T$-module $M:=T\cap I$ of $K[X,Y]$. Then $u:=X^2+XY$ is a
unit for $M$ in $I$ since $u\pm X^2,u\pm XY\in M$. Let $\ph$ be a pure state of
$(I,M,u)$. By the Dichotomy \ref{dichotomy}, exactly one of the following cases occurs:
\begin{enumerate}[(1)]
\item There is some $(x,y)\in S\setminus(\{0\}\times\R)$ with
$\ph(p)=\frac{p(x,y)}{x(x+y)}$ for all $p\in I$.
\item There is some $y\in[0,1]$ such that $\ph(pX^2+qXY)=p(0,y)\ph(X^2)+q(0,y)\ph(XY)$
for all $p,q\in K[X,Y]$.
\end{enumerate}
Suppose now that (2) holds and fix $y\in[0,1]$ accordingly. Consider
$\la_1:=\ph(X^2)\ge0$, $\la_2:=\ph(XY)\ge0$. Then $\la_1+\la_2=\ph(u)=1$. 

\medskip
Consider first the case $y>0$. From $0=\ph(YX^2-X(XY))=\la_1y-\la_20=\la_1y$ we
get $\la_1=0$. Then
$\frac1y\frac{\partial(pX^2+qXY)}{\partial X}(0,y)=\frac1yq(0,y)y=q(0,y)=\la_1p(0,y)+\la_2q(0,y)
=\ph(pX^2+qXY)$
for all $p,q\in K[X,Y]$.
Hence $\ph=\ph_y$ with
\[\ph_y\colon I\to\R,\ p\mapsto\frac1y\frac{\partial p}{\partial X}(0,y).\]

\medskip
Consider now the case $y=0$. Then
$
\frac12\frac{\partial^2(pX^2+qXY)}{\partial X^2}(0,0)=p(0,0)=p(0,y)$ and
$\frac{\partial^2(pX^2+qXY)}{\partial X\partial Y}(0,0)=q(0,0)=q(0,y)$
for all $p,q\in K[X,Y]$. Hence $\ph=\la_1\ps_1+\la_2\ps_2$ with
\[\ps_1\colon I\to\R,\ p\mapsto\frac12\frac{\partial^2p}{\partial X^2}(0,0)\qquad\text{and}\qquad
\ps_2\colon I\to\R,\ p\mapsto\frac{\partial^2p}{\partial X\partial Y}(0,0).\]

\medskip\noindent
Before we give a summary, we observe that
\[I=\left\{f\in K[X,Y]\mid f=0\text{ on }S\cap(\{0\}\times\R),\frac{\partial f}{\partial X}(0)=0\right\}\]
where ``$\subseteq$'' is clear since the right hand side forms obviously an ideal and
``$\supseteq$'' can be seen as follows:
If $f\in K[X,Y]$ with $f=0$ on $S\cap(\{0\}\times\R)$, then $f(0,Y)=0$ and thus $f\in(X)$.
If $f=Xg\in K[X,Y]$ with $\frac{\partial f}{\partial X}(0)=0$, then $g(0)=0$, hence
$g\in(X,Y)$ and consequently $f\in(X^2,XY)$. Taking into account that each polynomial in $M$
is nonnegative on $S$, one obtains $\ph_y\in S(I,M,u)$ for all $y\in(0,1]_\R$ and
$\ps_1,\ps_2\in S(I,M,u)$. The above considerations therefore yield
\[\extr S(I,M,u)\subseteq
\{\ph_y\mid y\in(0,1]_\R\}\cup\{\ps_1,\ps_2\}\] from which one obtains with
\ref{conemembershipunitextreme}: If $f\in K[X,Y]$ with
\begin{itemize}
\item $f>0$ on $S\setminus(\{0\}\times\R)$,
\item $f=0$ on $S\cap(\{0\}\times\R)$,
\item $\frac{\partial f}{\partial X}(0,y)>0$ for $y\in(0,1]_\R$,
\item $\frac{\partial f}{\partial X}(0,0)=0$,
\item $\frac{\partial^2f}{\partial X^2}(0,0)>0$ and
\item $\frac{\partial^2f}{\partial X\partial Y}(0,0)>0$,
\end{itemize}
then $f\in T$.
\end{ex}

\section{A local-global-principle}

\begin{pro}\label{uaaa}
Let $T$ be a semiring of $A$ with $K_{\ge0}\subseteq T$, $M$ a $T$-module of
$A$, $n\in\N_0$ and $a_1,\dots,a_n\in A$. Set $I:=(a_1,\dots,a_n)$. Moreover, let $u$ be a
unit for $\malal TM$ in $A$ and suppose $a_1,\dots,a_n\in\malal MT$. Then
$u(a_1+\ldots+a_n)$ is a unit for $M\cap I$ in $I$.
\end{pro}

\begin{proof}
Let $b\in I$ and set $v:=u(a_1+\dots+a_n)$. To show: $\exists N\in\N:Nv+b\in M\cap I$.
Write $b=\sum_{i=1}^nc_ia_i$ with $c_1,\dots,c_n\in A$. Choose $N\in\N$ such that
$Nu\pm c_i\in\malal TM$ for $i\in\{1,\dots,n\}$. Then $Nv\pm b=\sum_{i=1}^n(Nua_i\pm c_ia_i)
=\sum_{i=1}^n(Nu\pm c_i)a_i\in M$.
\end{proof}

\begin{thm}[Burgdorf, Scheiderer, Schweighofer \cite{bss}]
Let $T$ be an Archimedean semiring of $A$ with $K_{\ge0}\subseteq T$ and $M$ a $T$-module
of $A$. Let $a\in A$ such that there is for each maximal ideal $\m$ of $A$ some
$t\in T\setminus\m$ with $ta\in M$. Then $a\in M$.
\end{thm}

\begin{proof}[Proof (simplified by Leonhard Nenno)]
The ideal \[I:=(\{t\in T\mid ta\in M\})\] is not contained in any maximal ideal $\m$ of $A$ for otherwise we find by our hypothesis some $s\in T\setminus\m$ with
$sa\in M$ which entails the contradiction $s\in I\subseteq\m$. It follows that $1\in I$, i.e., we can choose $m\in\N$ and $t_1,\dots,t_m\in T$ and
$d_1,\dots,d_m\in A$ with $t_1a,\dots,t_ma\in M$ such that
\[1=d_1t_1+\ldots+d_mt_m.\]
Multiplying with $a$, it follows that
\[a\in(at_1,\ldots,at_m)=:J.\]
By the (first version of) \ref{uaaa}, $u:=at_1+\ldots+at_m=at$
with $t:=t_1+\ldots+t_m\in T$ is a unit for $M\cap J$ in $J$. To show that $a\in M$, we will now
apply \ref{conemembershipextr}. So let $\ph$ be a pure state of $(J,J\cap M,u)$. To show:
$\ph(a)>0$. Denote by $\Ph$ the ring homomorphism that belongs to $\ph$ according to \ref{puremult} and
\ref{assringhom}. We have $\Ph(T)\subseteq\R_{\ge0}$ [$\to$ \ref{dichotomy}].
Now \[1=\ph(u)=\ph(at)=\ph(ta)=\underbrace{\Ph(t)}_{\ge0}\ph(a).\] Thus $\ph(a)>0$.
\end{proof}

\chapter{Pure states and nonnegative polynomials over real closed fields}

\section{Pure states and polynomials over real closed fields}

Throughout this section, we let $R$ be a real closed extension field of $\R$, we set
$\O:=\O_R$, $\m:=\m_R$ and we make extensive use of the standard part maps
$\O\to\R,\ a\mapsto\st(a)$, $\O[\x]\to\R[\x],\ p\mapsto\st(p)$ [$\to$ \ref{ordval}] and
$\O^n\to\R^n,\ x\mapsto\st(x):=(\st(x_1),\ldots,\st(x_n))$ which are
surjective ring homomorphisms.

\begin{df}{}[$\to$ \ref{defpreorder}, \ref{dfarch}(a)]\label{defqm}
Let $A$ be a commutative ring and $M\subseteq A$. Then $M$ is called a \emph{quadratic module} of $A$ if
$M$ is a $\sum A^2$-module of $A$ containing $1$ [$\to$ \ref{deftmodule}], or in other words, if
$\{0,1\}\subseteq M$, $M+M\subseteq M$ and $A^2M\subseteq M$.
We call a quadratic module $M$ of $A$ \emph{Archimedean} if
$B_{(A,M)}=A$ [$\to$ \ref{arithmbounded}, \ref{archb}, \ref{bamu}].
\end{df}

\begin{pro}{}\emph{[$\to$ \ref{archmodulechar}]}\label{archmodulecharrcf}
Suppose $n\in\N_0$ and $M$ is a quadratic module of $\O[\x]$. Then the following are equivalent:
\begin{enumerate}[\normalfont(a)]
\item $M$ is Archimedean.
\item $\exists N\in\N:N-\sum_{i=1}^nX_i^2\in M$
\item $\exists N\in\N:\forall i\in\{1,\dots,n\}:N\pm X_i\in M$
\end{enumerate}
\end{pro}

\begin{proof}
\underline{(a)$\implies$(b)} is trivial.

\smallskip
\underline{(b)$\implies$(c)}\quad If (b) holds, then $N-X_i^2\in M$ and thus $X_i^2\in B_{(\O[\x],M)}$
for all $i\in\{1,\dots,n\}$. Apply now \ref{rootqmb}.

\smallskip
\underline{(c)$\implies$(a)} follows from \ref{bammodule} since $\O\subseteq B_{(\O[\x],M)}$.
\end{proof}

\begin{rem}
In contrast to \ref{archmodulechar}(d), one cannot add
\begin{multline*}
\exists m\in\N:\exists\ell_1,\dots,\ell_m\in M\cap\O[\x]_1:\exists N\in\N:\\
\emptyset\ne\{x\in R^n\mid\ell_1(x)\ge0,\dots,\ell_m(x)\ge0\}\subseteq[-N,N]_R^n
\end{multline*}
as another equivalent condition
in \ref{archmodulecharrcf}. Indeed, choose $R$ non-Archimedean [$\to$ \ref{boundex}] and $\ep\in\m
\setminus\{0\}$.
Then $\emptyset\ne\{0\}=\{x\in R\mid\ep x\ge0,-\ep x\ge0\}\subseteq[-1,1]_R$ but
the quadratic module \[\sum\O[X]^2+\sum\O[X]^2\ep X+\sum\O[X]^2(-\ep X)
\overset{\ref{diffsquare}}=\sum\O[X]^2+\O[X]\ep X\]
generated by $\ep X$ and $-\ep X$
in $\O[X]$ is not Archimedean for if we had $N\in\N$ with
\[N-X^2\in\sum\O[X]^2+\O[X]\ep X,\]
then taking standard parts would yield
$N-X^2\in\sum\R[X]^2$ which contradicts \ref{soslongrem}(b).
\end{rem}

\begin{df}{}[$\to$ \ref{evashom}]\label{adamshom}
For every $x\in\O^n$, we define the ring homomorphism
\[\ev_x\colon\O[\x]\to\O,\ p\mapsto p(x)\]
and set $I_x:=\ker\ev_x$.
\end{df}

\begin{pro}
Let $x\in\O^n$. Then $I_x=(X_1-x_1,\dots,X_n-x_n)$.
\end{pro}

\begin{proof}
It is trivial that $J:=(X_1-x_1,\dots,X_n-x_n)\subseteq I_x$.
Conversely, $p\equiv_Jp(x)=0$ for all
$p\in I_x$. This shows the converse inclusion $I_x\subseteq J$.
\end{proof}

\begin{notation}
Suppose $A$ is a commutative ring and $I$ is an ideal of $A$.
As it is customary in commutative algebra, we will in the following often denote by
$I^2$ the product of the ideal $I$ with itself which in our suggestive notation
[$\to$ \ref{divnot}] would be written $\sum II$. From the context,
the reader should be able to avoid misinterpreting
$I^2$ as what it would mean in this suggestive notation, namely
$\{a^2\mid a\in I\}$. The same applies to $I^3$ and so on.
Another source of confusion could be that, we will often use the notation
$\m^n$ to denote the Cartesian power \[\underbrace{\m\times\ldots\times\m}_{\text{$n$ times}}.\]
\end{notation}

\begin{lem}\label{coprime}
Suppose $x,y\in\O^n$ with $\st(x)\ne\st(y)$. Then $I_x$ and $I_y$ are coprime, i.e.,
$1\in I_x+I_y$.
\end{lem}

\begin{proof}
WLOG $x_1-y_1\notin\m$. Then $x_1-y_1\in\O^\times$ and
\[1=\frac{x_1-X_1}{x_1-y_1}+\frac{X_1-y_1}{x_1-y_1}\in I_x+I_y.\]
\end{proof}

\begin{lem}\label{ixunit}
Let $M$ be an Archimedean quadratic module of $\O[\x]$ and $x\in\O^n$. Then
\[u_x:=(X_1-x_1)^2+\ldots+(X_n-x_n)^2\]
is a unit for $M\cap I_x^2$ in the real vector space $I_x^2$ [$\to$ \ref{defunit}].
\end{lem}

\begin{proof}
Using the ring automorphism \[\O[\x]\to\O[\x],\ p\mapsto p(X_1-x_1,\ldots,X_n-x_n),\]
which is also an isomorphism of real vector spaces, we can reduce to the case $x=0$.
Since $u_x\in I_0^2$, it suffices to show that $I_0^2\subseteq B_{(\O[\x],M,u_x)}$.
Since $M$ is Archimedean, \ref{bammodule} yields that
$B_{(\O[\x],M,u_x)}$ is an $\O[\x]$-module of $\O[\x]$ [$\to$~\ref{deftmodule}], i.e., an ideal of
$\O[\x]$. Because of
\[I_0^2=(X_iX_j\mid i,j\in\{1,\dots,n\}),\]
it suffices therefore to show that $X_iX_j\in B_{(\O[\x],M,u_x)}$ for all $i,j\in\{1,\dots,n\}$.
Thus fix $i,j\in\{1,\dots,n\}$. Then $\frac12(X_i^2+X_j^2)\pm X_iX_j=\frac12(X_i\pm X_j)^2\in M$
and thus $\frac12u_x\pm X_iX_j\in M$. Since $u_x\in M$, this implies
$u_x\pm X_iX_j\in M$.
\end{proof}

\begin{nt}
We use the symbols $\nabla$ and $\hess$ to denote the gradient and the Hessian of a real-valued function of $n$ real variables, respectively. For a \emph{polynomial} $p\in\R[\x]$, we understand its
gradient $\nabla p$ as a column vector from $\R[\x]^n$, i.e., as a vector of polynomials. Similarly, its Hessian
$\hess p$ is a symmetric matrix polynomial of size $n$, i.e., a symmetric matrix from $\R[\x]^{n\times n}$.
Using formal partial derivatives, we more generally define $\nabla p\in R[\x]^n$ and
$\hess p\in R[\x]^{n\times n}$ even for $p\in R[\x]$.
\end{nt}

\begin{lem}\label{secondtypelemma}
Let $x\in\O^n$ and $\ph\in S(I_x^2,\sum\O[\x]^2\cap I_x^2,u_x)$ [$\to$ \ref{defstate}] such that
$\ph|_{I_x^3}=0$. Then there exist
$v_1,\dots,v_n\in\R^n$ such that
$\sum_{i=1}^nv_i^Tv_i=1$ and
\[\ph(p)=\frac12\st\left(\sum_{i=1}^nv_i^T(\hess p)(x)v_i\right)\]
for all $p\in I_x^2$.
\end{lem}

\begin{proof}
As in the proof of Lemma \ref{ixunit}, one easily reduces to the case $x=0$.

\medskip
\textbf{Claim 1:} $\ph(au_x)=0$ for all $a\in\m$.

\smallskip
\emph{Explanation.} Let $a\in\m$. WLOG $a\ge0$. Then $a\in\O\cap R_{\ge0}=\O^2$ and thus
$au_x\in\sum\O[\x]^2\cap I_0^2$. This shows $\ph(au_x)\ge0$. It remains to show that $\ph(au_x)\le\frac1N$ for all
$N\in\N$. For this purpose, fix $N\in\N$. Then $\frac1N-a\in\O\cap R_{\ge0}=\O^2$ and thus
$\left(\frac1N-a\right)u\in\sum\O[\x]^2\cap I_0^2$. It follows that $\ph\left(\left(\frac1N-a\right)u_x\right)\ge0$, i.e.,
$\ph(au_x)\le\frac1N$.

\medskip
\textbf{Claim 2:} $\ph(aX_i^2)=0$ for all $a\in\m$ and $i\in\{1,\dots,n\}$.

\smallskip
\emph{Explanation.} Let $a\in\m$. WLOG $a\ge0$ and thus $a\in\O^2$.
Then \[\sum_{i=1}^n\underbrace{\ph(\overbrace{aX_i^2}^{\rlap{$\scriptstyle\in\O[\x]^2\cap I_0^2$}})}_{\ge0}=\ph(au_x)\overset{\text{Claim 1}}=0.\]

\medskip
\textbf{Claim 3:} $\ph(aX_iX_j)=0$ for all $a\in\m$ and $i,j\in\{1,\dots,n\}$.

\smallskip
\emph{Explanation.} Fix $i,j\in\{1,\dots,n\}$ and $a\in\m$. If $i=j$, then we are done by Claim 2. So suppose
$i\ne j$. WLOG $a\ge0$ and thus $a\in\O^2$.
Then \[a(X_i^2+X_j^2\pm 2X_iX_j)=a(X_i\pm X_j)^2\in\O[\x]^2\cap I_0^2\] and thus
$\pm 2\ph(aX_iX_j)\underset{\text{Claim 2}}=\ph(aX_i^2)+\ph(aX_j^2)\pm 2\ph(aX_iX_j)\ge 0$.

\medskip
\textbf{Claim 4:} $\ph(p)=\frac12\st\left(\tr\left((\hess p)(0)A\right)\right)$ for all $p\in I_0^2$ where
\[A:=\begin{pmatrix}\ph(X_1X_1)&\dots&\ph(X_1X_n)\\
\vdots&\ddots&\vdots\\
\ph(X_nX_1)&\dots&\ph(X_nX_n)
\end{pmatrix}.
\]

\smallskip
\emph{Explanation.} Let $p\in I_0^2$. By $\ph|_{I_0^3}=0$, we can reduce to the case
$p=aX_iX_j$ with $i,j\in\{1,\dots,n\}$ and $a\in\O$.
Using  Claim 3, we can assume $a=1$. Comparing both sides,
yields the result.

\medskip
\textbf{Claim 5:} $A$ is psd [$\to$ \ref{psdpd}(b)].

\smallskip
\emph{Explanation.} If $w\in\R^n$, then $w^TAw=\ph((w_1X_1+\ldots+w_nX_n)^2)\ge0$ since
\[(w_1X_1+\ldots+w_nX_n)^2\in\R[\x]^2\cap I_0^2\subseteq\sum\O[\x]^2\cap I_0^2.\]

\medskip\noindent
By Claim 5 and \ref{psdeq}(c), we can choose $B\in\R^{n\times n}$ such that $A=B^TB$. Denote by $v_i$ the $i$-th
row of $B$ for $i\in\{1,\dots,n\}$. Then by Claim 4, we get
\begin{multline*}
\ph(p)=\frac12\st(\tr((\hess p)(0)A))
=\frac12\st(\tr((\hess p)(0)B^TB))\\
=\frac12\st(\tr(B(\hess p)(0)B^T))=\frac12\st\left(\sum_{i=1}^nv_i^T(\hess p)(0)v_i\right)
\end{multline*}
for all $p\in I_0^2$.
In particular, we obtain $1=\ph(u_0)=\sum_{i=1}^nv_i^Tv_i$.
\end{proof}

\begin{lem}\label{stpointev}
Let $\Ph\colon\O[\x]\to\R$ be a ring homomorphism. Then there is some $x\in\R^n$ such that
$\Ph(p)=\st(p(x))$ for all $p\in\O[\x]$.
\end{lem}

\begin{proof}
By \ref{archemb}, we have $\Ph|_\R=\id_\R$.
It is easy to see that $\Ph|_\m=0$.
Indeed, for each
$N\in\N$ and $a\in\m$, we have $\frac1N\pm a\in R_{\ge0}\cap\O=\O^2$
and therefore
$\frac1N\pm\Ph(a)\in\R_{\ge0}$. 
Finally set \[x:=(\Ph(X_1),\dots,\Ph(X_n))\in\R^n\]
and use that $\Ph|_\R=\id_\R$, $\Ph|_\m=0$ and that $\Ph$ is a ring homomorphism.
\end{proof}

\begin{thm}{}\emph{[$\to$ \ref{dichotomy}]}\label{dicho}
Let $M$ be an Archimedean quadratic module of $\O[\x]$ and set
\[S:=\{x\in\R^n\mid\forall p\in M:\st(p(x))\ge0\}.\]
Moreover, suppose $k\in\N_0$ and let $x_1,\ldots,x_k\in\O^n$ satisfy $\st(x_i)\ne\st(x_j)$ for
$i,j\in\{1,\dots,k\}$ with $i\ne j$. Then
$u:=u_{x_1}\dotsm u_{x_k}$ is a unit for the cone $M\cap I$ in the real vector space
\[I:=I_{x_1}^2\dotsm I_{x_k}^2=I_{x_1}^2\cap\ldots\cap I_{x_k}^2\]
and for all pure states $\ph$ of $(I,M\cap I,u)$
exactly one of the following cases occurs:
\begin{enumerate}[\normalfont(1)]
\item There is an $x\in S\setminus\{\st(x_1),\dots,\st(x_k)\}$ such that
\[\ph(p)=\st\left(\frac{p(x)}{u(x)}\right)\]
for all $p\in I$.
\item There is an $i\in\{1,\dots,k\}$ and $v_1,\dots,v_n\in\R^n$ such that
$\sum_{\ell=1}^nv_\ell^Tv_\ell=1$ and
\[\ph(p)=\st\left(\frac{\sum_{\ell=1}^nv_\ell^T(\hess p)(x_i)v_\ell}{2\prod_{\substack{j=1\\j\ne i}}^ku_{x_j}(x_i)}\right)\]
for all $p\in I$.
\end{enumerate}
\end{thm}

\begin{proof}
The Chinese remainder theorem from commutative algebra shows that
\[I=I_{x_1}^2\dotsm I_{x_k}^2=I_{x_1}^2\cap\ldots\cap I_{x_k}^2\]
since $I_{x_i}$ and $I_{x_j}$ and thus also $I_{x_i}^2$ and $I_{x_j}^2$ are coprime for all
$i,j\in\{1,\dots,k\}$ with $i\ne j$.
By \ref{ixunit}, $u_{x_i}$ is a unit for $M\cap I_{x_i}^2$ in $I_{x_i}^2$
for each $i\in\{1,\dots,k\}$. To show that $u$ is a unit
for the cone $M\cap I$ in the real vector space $I$, it suffices to find for all
$a_1,b_1\in I_{x_1},\dots,a_k,b_k\in I_{x_k}$ an $N\in\N$ such that
$Nu+ab\in M$ where we set $a:=a_1\dotsm a_k$ and $b:=b_1\dotsm b_k$.
Because of $Nu+ab=(Nu-\frac12a^2-\frac12b^2)+\frac12(a+b)^2$, it is enough to find $N\in\N$ with
$Nu-a^2\in M$ and $Nu-b^2\in M$. By symmetry, it suffices to find $N\in\N$ with
$Nu-a^2\in M$. Choose $N_i\in\N$ with $N_iu_{x_i}-a_i^2\in M$ for $i\in\{1,\ldots,k\}$. We now claim that
$N:=N_1\dotsm N_k$ does the job. Indeed, the reader shows easily by induction that actually
\[N_1\dotsm N_iu_{x_1}\dotsm u_{x_i}-a_1^2\dotsm a_i^2\in M\]
for $i\in\{1,\dots,k\}$. Now let $\ph$ be a pure state of $(I,M\cap I,u)$.
Denote by $\Ph\colon\O[\x]\to\R$ the ring homomorphism belonging to
$\ph$, i.e.,
\[(*)\qquad\ph(pq)=\Ph(p)\ph(q)\]
for all $p\in\O[\x]$ and $q\in I$ [$\to$ \ref{assringhom}, \ref{dichotomy}].
By Lemma \ref{stpointev}, we can choose $x\in\R^n$ such that
\[\Ph(p)=\st(p(x))\]
for all $p\in\O[\x]$.
Since $u\in I\cap\sum\O[\x]^2$, we have
\[\st(p(x))=\Ph(p)=\Ph(p)\ph(u)\overset{(*)}=\ph(pu)=\ph(up)\overset{up\in M}\in\ph(M)\subseteq\R_{\ge0}\]
for all $p\in M$ which means $x\in S$.

\smallskip
Now first suppose that Case (1) in the Dichotomy \ref{dichotomy} occurs.
We show that $x$ satisfies (1). Note that $\Ph(u)\ne0$ by \ref{dichotomy}. This means $\st(u_{x_i}(x))\ne0$
and therefore $\st(x)\ne\st(x_i)$ for all $i\in\{1,\dots,k\}$. The rest follows from \ref{dichotomy}.

\smallskip
Now suppose that Case (2) in the Dichotomy \ref{dichotomy} occurs. We show that then (2) holds.
First note that $\prod_{i=1}^k\Ph(u_{x_i})=
\Ph(u)=0$ because $u\in I$ and $\Ph|_I=0$. Choose $i\in\{1,\dots,k\}$ such that $\st(u_{x_i}(x))=\Ph(u_{x_i})=0$.
Then $x=\st(x_i)$. Define
\[\ps\colon I_{x_i}^2\to\R,\ p\mapsto\ph\left(p\prod_{\substack{j=1\\j\ne i}}^ku_{x_j}\right).
\]
Since $u_{x_j}\in\sum\O[\x]^2\cap I_{x_j}^2$ for all $j\in\{1,\dots,k\}$, it follows that
$\ps\in S(I_{x_i}^2,M\cap I_{x_i}^2,u_{x_i})$. If $p\in I_{x_i}$ and $q\in I_{x_i}^2$, then
\[\ps(pq)=\ph\left(pq\prod_{\substack{j=1\\j\ne i}}^ku_{x_j}\right)
\overset{(*)}=
\Ph(p)\ph\left(q\prod_{\substack{j=1\\j\ne i}}^ku_{x_j}\right)=0\]
since $\Ph(p)=\st(p(x))=(\st(p))(x)=(\st(p))(\st(x_i))=\st(p(x_i))=\st(0)=0$.
It follows that $\ps|_{I_{x_i}^3}=0$.
We can thus apply Lemma \ref{secondtypelemma} to $\ps$ and obtain
$v_1,\dots,v_n\in\R^n$ such that
$\sum_{\ell=1}^nv_\ell^Tv_\ell=1$ and
\[\ps(p)=\frac12\st\left(\sum_{\ell=1}^nv_\ell^T(\hess p)(x_i)v_\ell\right)\]
for all $p\in I_{x_i}^2$.
Because of $\st(x_i)\ne\st(x_j)$ for $j\in\{1,\dots,k\}\setminus\{i\}$, we have
\[c:=\Ph\left(\prod_{\substack{j=1\\j\ne i}}^k u_{x_j}\right)=\prod_{\substack{j=1\\j\ne i}}^k\Ph(u_{x_j})=
\prod_{\substack{j=1\\j\ne i}}^k(\st(u_{x_j}))(\st(x_i))\ne0.\]
Hence we obtain
\[c\ph(p)\overset{(*)}=\ps(p)\]
for all $p\in I$.

\smallskip
It only remains to show that (1) and (2) cannot occur both at the same time. If (1) holds, then
we have obviously $\ph(u^2)\ne0$. If (2) holds, then $\ph(u^2)=0$ since $\hess(u^2)(x_i)=0$ for
all $i\in\{1,\dots,k\}$ as one easily shows.
\end{proof}

\begin{lem}\label{membershipix}
For all $x\in\O^n$, we have
\[I_x^2=\left\{p\in\O[\x]\mid p(x)=0,\nabla p(x)=0\right\}.\]
\end{lem}

\begin{proof}
For $x=0$ it is easy. One reduces the general case to the case $x=0$ as in the proof of
\ref{ixunit}.
\end{proof}

\begin{thm}\label{mainrep}
Let $M$ be an Archimedean quadratic module of $\O[\x]$ and set
\[S:=\{x\in\R^n\mid\forall p\in M:\st(p(x))\ge0\}.\]
Moreover, suppose $k\in\N_0$ and let $x_1,\ldots,x_k\in\O^n$ have pairwise distinct standard parts. Let \[f\in\bigcap_{i=1}^kI_{x_i}^2\]
such that \[\st(f(x))>0\] for all $x\in S\setminus\{\st(x_1),\ldots,\st(x_k)\}$ and
\[\st(v^T(\hess f)(x_i)v)>0\] for all $i\in\{1,\dots,k\}$ and $v\in\R^n\setminus\{0\}$.
Then $f\in M$.
\end{thm}

\begin{proof}
Define $I$ and $u$ as in Theorem \ref{dicho} so that $f\in I$.
We will apply \ref{conemembershipextr} to the real vector space $I$, the cone $M\cap I$ in $I$ and
the unit $u$ for $M\cap I$. From Theorem \ref{dicho}, we see indeed easily that $\ph(f)>0$ for all
$\ph\in\extr S(I,M\cap I,u)$.
\end{proof}

\begin{cor}\label{mainrep2}
Let $M$ be an Archimedean quadratic module of $\O[\x]$ and set
\[S:=\{x\in\R^n\mid\forall p\in M:\st(p(x))\ge0\}.\]
Moreover, let $k\in\N_0$ and $x_1,\ldots,x_k\in\O^n$ such that their standard parts are pairwise distinct and lie in the interior of $S$. 
Let \[f\in\bigcap_{i=1}^kI_{x_i}^2.\]
Set again $u:=u_{x_1}\dotsm u_{x_k}\in\O[\x]$.
Suppose there is $\ep\in\R_{>0}$ such
that \[f\ge\ep u\text{ on }S.\]
Then $f\in M$.
\end{cor}

\begin{proof}
By \ref{mainrep}, we have to show:
\begin{enumerate}[(a)]
\item $\forall x\in S\setminus\{\st(x_1),\ldots,\st(x_k)\}:\st(f(x))>0$
\item $\forall i\in\{1,\dots,k\}:\forall v\in\R^n\setminus\{0\}:\st(v^T(\hess f)(x_i)v)>0$
\end{enumerate}
It is easy to show (a). To show (b), fix $i\in\{1,\ldots,k\}$.
Because of $f-\ep u\ge0$ on $S$ and \[(f-\ep u)(x_i)=f(x_i)-\ep u(x_i)=0-0=0,\]
$\st(x_i)$ is a local minimum of $\st(f-\ep u)\in\R[\x]$ on $\R^n$. From
elementary analysis, we know therefore that $(\hess\st(f-\ep u))(\st(x_i))$ is psd.
Because of $u_{x_i}(x_i)=0$ and $\nabla u_{x_i}(x_i)=0$, we get
\[(\hess u)(x_i)=
\left(\prod_{\substack{j=1\\j\ne i}}^ku_{x_j}(x_i)\right)(\hess u_{x_i})(x_i)=
2\left(\prod_{\substack{j=1\\j\ne i}}^ku_{x_j}(x_i)\right)I_n.
\]
Therefore
\[\st(v^T(\hess f)(x_i)v)\ge\ep\st(v^T(\hess u)(x_i)v)=2\ep v^Tv\st\left(\prod_{\substack{j=1\\j\ne i}}^ku_{x_j}(x_i)\right)>0\]
for all $v\in\R^n\setminus\{0\}$.
\end{proof}

\begin{cor}\label{mainrep3}
Let $n,m\in\N_0$ and suppose $g_1,\ldots,g_m\in\R[\x]$ generate an Archimedean quadratic module
in $\R[\x]$ \emph{[$\to$ \ref{archmodulechar}]}. Set
\[S:=\{x\in R^n\mid g_1(x)\ge0,\ldots,g_m(x)\ge0\}.\]
Moreover, let $k\in\N_0$ and $x_1,\ldots,x_k\in\O^n$ and $\ep\in\R_{>0}$ such that
the sets $x_1+\ep B,\ldots,x_k+\ep B$ are pairwise disjoint and all contained in $S$
where \emph{[$\to$ \ref{normcont}]} \[B:=\{x\in R^n\mid\|x\|_2<1\}\subseteq\O^n.\]
Set once more $u:=u_{x_1}\dotsm u_{x_k}\in\O[\x]$.
Let $f\in\O[\x]$ such that $f\ge\ep u$ on $S$ and
\[f(x_1)=\ldots=f(x_k)=0.\]
Then $f$ lies in the quadratic module generated by $g_1,\ldots,g_m$ in $\O[\x]$.
\end{cor}

\begin{proof}
The quadratic module $M$ generated by $g_1,\ldots,g_m$ in $\O[\x]$ is clearly also Archimedean [$\to$ 9.1.2].
Moreover, it is easy to see that $\{x\in\R^n\mid \forall p\in M:\st(p(x))\ge0\}=S\cap\R^n$.
Hence $f\in M$ follows from \ref{mainrep2} once we show that \[\nabla f(x_1)=\ldots=\nabla f(x_k)=0.\]
Choose $d\in\N_0$ with $f\in\R[\x]_d$.
Since $f\ge\ep u\ge0$ on $S$ and thus $f\ge0$ on $x_i+\ep B$ for all $i\in\{1,\ldots,k\}$, it suffices
to prove the following for $R'=R$: If $p\in R'[\x]_d$, $x\in R'^n$, $\de\in R'_{>0}$ such that $p\ge0$ on
$x+\de B'$ where $B':=\{x\in R'^n\mid\|x\|_2<1\}$ and
$p(x)=0$, then $\nabla p(x)=0$. To see this, we employ the Tarski principle [$\to$ \ref{tprinciple}]:
The class
of all $R'\in\mathcal R$ [$\to$ \ref{introsemialg}] such that this holds true for all $p\in R'[\x]_d$ is
obviously a $0$-ary semialgebraic class by real quantifier elimination.
By elementary analysis, $\R$ is an element of this class.
We conclude thus by \ref{nothingorall}.
\end{proof}

\section{Degree bounds and quadratic modules}

\begin{df}\label{dftrunc}
Let $d,m\in\N_0$, $g_1,\dots,g_m\in\R[\x]$ and set $g_0:=1\in\R[\x]$. For $i\in\{0,\dots,m\}$, set
$r_i:=\frac{d-\deg g_i}2$ if $g_i\ne0$ and $r_i:=-\infty$ if $g_i=0$.
Then we denote by $M(g_1,\dots,g_m)$ the quadratic module generated by
$g_1,\dots,g_m$ in $\R[\x]$.
Moreover, we define the \emph{$d$-truncated quadratic module} $M_d(g_1,\dots,g_m)$
associated to $g_1,\dots,g_m$ by
\[
M_d(g_1,\dots,g_m):=
\left\{\sum_{i=0}^m\sum_jp_{ij}^2g_i\mid p_{ij}\in\R[\x]_{r_i}\right\}\subseteq M(g_1,\dots,g_m)\cap\R[\x]_d.
\]
\end{df}

\begin{rem}
Let $m\in\N_0$ and $g_1,\dots,g_m\in\R[\x]$. Set again $g_0:=1\in\R[\x]$.
\begin{enumerate}[(a)]
\item $M(g_1,\dots,g_m)=\bigcup_{d\in\N_0}M_d(g_1,\dots,g_m)$
\item For all $d\in\N_0$,
\[M_d(g_1,\dots,g_m)=\sum_{i=0}^m\left(\left(\sum\R[\x]^2g_i\right)\cap\R[\x]_d\right)\]
by \ref{soslongrem}(b).
\item In general, the inclusion $M_d(g_1,\dots,g_m)\subseteq M(g_1,\dots,g_m)\cap\R[\x]_d$ is proper
as \ref{nonexdegbounds} shows.
In fact, the validity of Schmüdgen's and Putinar's Positivstellensätze \ref{schmuedgen} and
\ref{putinar} strongly relies on this.
\end{enumerate}
\end{rem}

\begin{thm}[Putinar's Positivstellensatz with zeros and degree bounds]
\label{putinarzerosdegreebound}
Let $n,m\in\N_0$ and $g_1,\ldots,g_m\in\R[\x]$ such that $M(g_1,\ldots,g_m)$ is Archimedean. Set
\begin{align*}
B&:=\{x\in\R^n\mid\|x\|<1\}\qquad\text{and}\\
S&:=\{x\in\R^n\mid g_1(x)\ge0,\ldots,g_m(x)\ge0\}.
\end{align*}
Moreover, let $k\in\N_0$, $N\in\N$ and $\ep\in\R_{>0}$.
Then there exists \[d\in\N_0\] such that for all $f\in\R[\x]_N$ with all coefficients in $[-N,N]_\R$
and $\#\{x\in S\mid f(x)=0\}=k$, we have: Denoting by $x_1,\dots,x_k$ the distinct zeros of $f$ on $S$,
if the sets $x_1+\ep B,\ldots,x_k+\ep B$ are pairwise disjoint and contained in $S$ and if we have
$f\ge\ep u$ on $S$
where $u:=u_{x_1}\dotsm u_{x_k}\in\R[\x]$ \emph{[$\to$ \ref{ixunit}]}, then \[f\in M_d(g_1,\dots,g_m).\]
\end{thm}

\begin{proof}{}(cf. the proof of Theorem \ref{h17bound})
Set $\nu:=\dim\R[\x]_N$. For each $d\in\N_0$, the class $S_d$ of
all pairs $(R,a)$ where $R$ is a real closed extension field of $\R$
and $a\in R^\nu$ such that the following holds
is obviously a $\nu$-ary $\R$-semialgebraic class [$\to$ \ref{introsemialg}]:
If $a\in[-N,N]_R^\nu$ and if
$a$ is the vector of coefficients (in a certain fixed order) of a polynomial $f\in R[\x]_N$ with
exactly $k$ zeros $x_1,\dots,x_k$ on $S':=\{x\in R^n\mid g_1(x)\ge0,\dots,g_m(x)\ge0\}$, then
at least one of the following conditions (a), (b) and (c) is fulfilled:
\begin{enumerate}[(a)]
\item The sets $x_1+\ep B',\ldots,x_k+\ep B'$ are not pairwise disjoint or not all contained in $S'$ where
$B':=\{x\in R^n\mid\|x\|_2<1\}$.
\item $f\ge\ep u$ on $S'$ is violated where $u:=u_{x_1}\dotsm u_{x_k}\in R[\x]$.
\item $f$ is a sum of $d$ elements from $R[\x]$ where each term in the sum is of degree at most $d$
and is of the form $p^2g_i$ with $p\in R[\x]$ and $i\in\{0,\dots,m\}$ where $g_0:=1\in R[\x]$.
\end{enumerate}
Set $\mathcal E:=\{S_d\mid d\in\N_0\}$ and
observe that $\forall d_1,d_2\in\N_0:\exists d_3\in\N_0:S_{d_1}\cup S_{d_2}\subseteq
S_{d_3}$ (take $d_3:=\max\{d_1,d_2\}$). By \ref{mainrep3},
we have $\bigcup\mathcal E=\mathcal R_\nu$. Now
\ref{finitenesscor} yields $S_d=\mathcal R_\nu$ for some $d\in\N_0$.
\end{proof}

\begin{cor}[Putinar's Positivstellensatz with degree bounds \cite{pre,ns,kri'}]\emph{[$\to$~\ref{putinar}]}
Let $n,m\in\N_0$ and $g_1,\ldots,g_m\in\R[\x]$ such that $M(g_1,\ldots,g_m)$ is Archimedean. Set
\[S:=\{x\in\R^n\mid g_1(x)\ge0,\ldots,g_m(x)\ge0\}.\]
Moreover, let $N\in\N$ and $\ep\in\R_{>0}$.
Then there exists \[d\in\N_0\] such that for all $f\in\R[\x]_N$ with all coefficients in $[-N,N]_\R$
and with $f\ge\ep$ on $S$, we have \[f\in M_d(g_1,\dots,g_m).\]
\end{cor}

\begin{pro}\label{squeeze}
Suppose $S\subseteq\R^n$ is compact, $x_1,\dots,x_k\in S^\circ$ are pairwise distinct,
$u:=u_{x_1}\dotsm u_{x_k}\in\R[\x]$ \emph{[$\to$ \ref{ixunit}]} and $f\in\R[\x]$ with $f(x_1)=\ldots=f(x_k)=0$.
Then the following are equivalent:
\begin{enumerate}[\normalfont(a)]
\item $f>0$ on $S\setminus\{x_1,\ldots,x_k\}$ and $\hess f(x_1),\dots,\hess f(x_k)$ are pd.
\item There is some $\ep\in\R_{>0}$ such that $f\ge\ep u$ on $S$.
\end{enumerate}
\end{pro}

\begin{proof}
\underline{(b)$\implies$(a)} is easy to show (cf. the proof of \ref{mainrep2}).

\smallskip
\underline{(a)$\implies$(b)} \quad It is easy to show that one can WLOG
assume that $S=\bigdotcup_{i=1}^k(x_i+\ep B)$ for some $\ep>0$ where $B$ is the closed unit ball
in $\R^n$. Then one finds easily an Archimedean quadratic module $M$ of $\R[\x]$ such that
\[S=\{x\in\R^n\mid\forall p\in M:p(x)\ge0\}.\]
A strengthened version of Theorem \ref{mainrep} now yields $f-\ep u\in M$ for some $\ep\in\R_{>0}$
and thus $f-\ep u\ge0$ on $S$. One gets this strengthened version of Theorem \ref{mainrep}
by applying (a)$\implies$(c) from \ref{conemembershipunitextreme} instead of
\ref{conemembershipextr} in its proof. Alternatively, we leave it as an exercise to the reader to
give a direct proof using only basic multivariate analysis.
\end{proof}

\begin{cor}[Putinar's Positivstellensatz with zeros \cite{s1}]\label{putinarzeros}
Let $g_1,\ldots,g_m\in\R[\x]$ such that $M(g_1,\ldots,g_m)$ is Archimedean. Set
\[S:=\{x\in\R^n\mid g_1(x)\ge0,\ldots,g_m(x)\ge0\}.\]
Moreover, suppose $k\in\N_0$ and $x_1,\ldots,x_k\in S^\circ$ are pairwise distinct.
Let $f\in\R[\x]$ such that $f(x_1)=\ldots=f(x_k)=0$, $f>0$ on $S\setminus\{x_1,\ldots,x_k\}$ and
$\hess f(x_1),\dots,\hess f(x_k)$ are pd.
Then \[f\in M(g_1,\ldots,g_m).\]
\end{cor}

\begin{proof}
This follows from \ref{putinarzerosdegreebound} by Proposition \ref{squeeze}.
\end{proof}

\begin{rem}
Because of Proposition \ref{squeeze}, Theorem \ref{putinarzerosdegreebound} is really a
quantitative version of Corollary \ref{putinarzeros}.
\end{rem}

\begin{rem}\label{commentonbounds}
\begin{enumerate}[(a)]
\item In Condition (c) from the proof of Theorem \ref{putinarzerosdegreebound}, we speak of
``a sum of $d$ elements'' instead of ``a sum of elements'' (which would in general
be strictly weaker).
Our motivation to do this was that this is the easiest way to make sure that we can formulate
(c) in a ``semialgebraic way''. A second motivation could have been to formulate
Theorem \ref{putinarzerosdegreebound} in stronger way, namely by letting $d$ be a bound
not only on the degree of the quadratic module representation but also on the number of terms in it.
This second motivation is however not interesting because we get also from the Gram matrix method
\ref{gram} a bound on this number of terms (a priori bigger than $d$ but after readjusting $d$ we can
again assume it to be $d$). We could have used the Gram matrix method already to see that
``a sum of elements'' (instead of ``a sum of $d$ elements'') can also be expressed semialgebraically.
\item We could strengthen condition (c) from the proof of
Theorem \ref{putinarzerosdegreebound}, by writing ``with $p\in R[\x]$ all of whose coefficients lie
in $[-d,d]_R$'' instead of just ``with $p\in R[\x]$''. Then $\bigcup\mathcal E=\mathcal R_\nu$ would
still hold since Corollary \ref{mainrep3} states that
$f$ lies in the quadratic module generated by $g_1,\ldots,g_m$ even in $\O[\x]$ not just in $\R[\x]$.
This would lead to a real strengthening of Theorem \ref{putinarzerosdegreebound}, namely we
could ensure that $d$ is a bound
not only on the degree of the quadratic module representation but also on the size of the coefficients
in it. However, we do currently not know of any application of this and therefore renounced to carry this
out.
\end{enumerate}
\end{rem}

\chapter{Linearizing systems of polynomial inequalities}

\section{The Lasserre hierachy for a system of polynomial inequalities}

By a \emph{system of polynomial inequalities},
we understand a (finite) system of (real non-strict) polynomial inequalities in several variables, i.e., a condition of the form
\[g_1(x)\ge0,\ldots,g_m(x)\ge0\qquad(x\in\R^n)\]
where $m,n\in\N_0$ and $\g=(g_1,\ldots,g_m)\in\R[\x]^m$. Its sets of solutions are exactly the basic closed semialgebraic subsets of
$\R^n$ [$\to$ \ref{dfbasicopenclosed}]. We now introduce notation for these.

\begin{df}
For $\g=(g_1,\ldots,g_m)\in\R[\x]^m$, we consider the basic closed semialgebraic set [$\to$ \ref{dfbasicopenclosed}]
\[S(\g):=\{x\in\R^n\mid g_1(x)\ge0,\ldots,g_m(x)\ge0\}.\]
\end{df}

\noindent
One can easily imagine that many ``horribly complicated'' computational problems can be easily translated into the problem of solving
systems of polynomial inequalities. It is outside of the scope of these lecture notes to make this statement more formal. Consequently,
it would be no surprise if it were in general very hard to deal with polynomial inequalities algorithmically. The proof of real quantifier elimination
[$\to$ Theorem \ref{elim}, §\ref{sec:qe}]
(by means of which one can for example decide whether a given system of polynomial inequalities is solvable) can be read as an algorithm
of horrible complexity which is basically useless for practical purposes.
We try to identify cases where efficient algorithms for dealing with systems of polynomial inequalities might possibly exist.
Our approach will be to relate systems of polynomial inequalities to \emph{linear matrix inequalities}. Before we introduce the latter, we need
more notation.

\begin{df}{}[$\to$ \ref{psdpd}(b)]
Let $A\in S\R^{k\times k}$. We write
$A\malal\succeq\succ0$ to express that
$A$ is \alal{psd}{pd}, i.e., $A$ is symmetric and $x^TMx\malal\ge>0$ for all $x\in\malal{\R^k}{\R^k\setminus\{0\}}$.
If $B\in\R^{k\times k}$ is another matrix, we write $A\malal\succeq\succ B$ or
$B\malal\preceq\prec A$ to express that
$A-B\malal\succeq\succ0$.
We say that $A$ is
\emph{\alal{negative semidefinite (nsd)}{negative definite (nd)}} if $A\preceq 0$.
\end{df}

\noindent
By a \emph{linear matrix inequality (LMI)},
we understand a condition of the form
\[A_0+x_1A_1+\ldots+x_nA_n\succeq0\qquad(x\in\R^n)\]
where $n\in\N_0$, $k\in\N_0$ and $A_0,\ldots,A_n\in S\R^{k\times k}$. An equivalent way of writing this is by saying it is of the form
\[\begin{pmatrix}\ell_{11}(x)&\ldots&\ell_{1k}(x)\\\vdots&&\vdots\\\ell_{k1}(x)&\ldots&\ell_{kk}(x)\end{pmatrix}\qquad(x\in\R^n)
\]
where each $\ell_{ij}\in\R[\x]_1$ is a linear polynomial [$\to$ \ref{quintic}] and $\ell_{ij}=\ell_{ji}$ for all $i,j\in\{1,\ldots,k\}$.
We speak of a \emph{diagonal} LMI if, in the above, each $A_i$ is diagonal or, equivalently, $\ell_{ij}=0$ for all $i$ different from $j$.
A diagonal LMI obviously just corresponds to a finite system of (non-strict) linear inequalities. Solution sets of LMIs are called 
\emph{spectrahedra}. They generalize the solution sets of diagonal LMIs that are called \emph{polyhedra}.

\begin{ex}
Consider the LMIs
\[\begin{pmatrix}1+x_1&x_2\\x_2&1-x_1\end{pmatrix}\succeq0\qquad(x_1,x_2\in\R)\]
and
\[\begin{pmatrix}1&x_1&x_2\\x_1&1&0\\x_2&0&1\end{pmatrix}\succeq0\qquad(x_1,x_2\in\R).\]
\noindent
By \ref{psdeq}(f), we have that the first LMI is equivalent to
\[1\ge0\et1+x_1\ge0\et1-x_1\ge0\et(1+x_1)(1-x_1)-x_2^2\ge0\qquad(x_1,x_2\in\R)\]
and the second one to
\[1\ge0\et1-x_1^2\ge0\et1-x_2^2\ge0\et1-x_1^2-x_2^2\ge0\qquad(x_1,x_2\in\R).\]
Hence both are actually equivalent to
\[x_1^2+x_2^2\le1\qquad(x_1,x_2\in\R)\]
so that they define the closed unit disk in the plane.
\end{ex}

There are very efficient algorithms to deal with finite systems of linear inequalities (for example, to decide whether such a system is
solvable). We catched already a glimpse of this in the algorithmic proof of \ref{fundthmlin}. Finite systems of linear inequalities
correspond to diagonal LMIs. As a matter of fact which goes beyond our scope, there are still very efficient algorithms to deal with
arbitrary LMIs. In fact, solving a system of linear inequalities in such a way that a given linear objective function is maximized or minimized is the
problem of \emph{linear programming} (LP). LP is ubiquitous in science and engineering and there are extremely efficient
LP solvers. The more general problem of solving an LMI in such a way that a given linear objective function is maximized or minimized is the
problem of \emph{semidefinite programming} (SDP). SDP is still in its infancy but gets more and more appreciated by many people
within mathematics and its applications. For SDP there are still very efficient solvers although they cannot yet compete with LP solvers.

\begin{ex}\label{x4y4} Consider the following toy system of polynomial inequalities:
\[(*)\qquad\quad 1-x_1+x_2\ge0,\quad 1-x_1^4-x_2^4\ge0\qquad(x_1,x_2\in\R).\]
The terms $x_1^4$ and $x_2^4$ make that it is not a system of linear inequalities.
A first idea which is in general way too naive is to simply replace them by
new unknowns $y_1$ and $y_2$ and consider
\[(**)\qquad1-x_1+x_2\ge0,\quad1-y_1-y_2\ge0\qquad(x_1,x_2,y_1,y_2\in\R).\]
Every solution of $(*)$ clearly gives rise to a solution of $(**)$ by setting $y_1:=x_1^4$ and $y_2:=x_2^4$. This implies that the projection
\[\{(x_1,x_2)\in\R^2\mid\exists y_1,y_2:(x_1,x_2,y_1,y_2)\text { solves }(**)\}\]
of the set of solutions of $(**)$ onto $x$-space contains the set of solutions of $(*)$. The converse is however obviously not 
true and this does not come as a surprise since one lost too much information about $(*)$.

The idea is now to add a whole bunch of redundant inequalities to $(*)$ before replacing all nonlinear terms by new variables.
For example, we can simply add for each $a,b,c,d,e,f\in\R$ the inequality
\[(***)\qquad(a+bx_1+cx_2+dx_1^2+ex_1x_2+fx_2^2)^2\ge0.\]
We could now expand these inequalities into $a^2+2abx_1+\ldots+f^2x_2^4\ge0$
and replace not only the term $x_1^4$ and $x_2^4$ by $y_1$ and $y_2$ as before but also all other nonlinear terms
by new variables $y_3$, $y_4$ and so on. This would lead to an infinite family of linear inequalities parametrized by $a,b,c,\ldots\in\R$.
The hope is that inequalities that were redundant before the linearization, could be valuable after the linearization.

One of the problems is
that one does in general not know how to deal algorithmically with \emph{infinite} systems of linear inequalities. Happily, there is a way of
turning the infinite system parametrized by $a,b,c,\dots$ into one single LMI. In order to get rid of the parameters, one first separates the
parameters $a,b,c,\ldots$ from the monomial expressions $1,x_1,x_2,x_1^2,\ldots$ by the following trick: Rewrite $(***)$ as a (matrix) product of
row and column vectors as follows:
\[
\begin{pmatrix}a&b&c&d&e&f\end{pmatrix}
\begin{pmatrix}1\\x_1\\x_2\\x_1^2\\x_1x_2\\x_2^2\end{pmatrix}
\begin{pmatrix}1&x_1&x_2&x_1^2&x_1x_2&x_2^2\end{pmatrix}
\begin{pmatrix}a\\b\\c\\d\\e\\f\end{pmatrix}\ge0.
\]
Multiplying the two interior vectors gives a symmetric matrix of monomials of a certain structure. More precisely it transform our condition
to
\[\begin{pmatrix}a&b&c&d&e&f\end{pmatrix}
\begin{pmatrix}1&x_1&x_2&x_1^2&x_1x_2&x_2^2\\
x_1&x_1^2&x_1x_2&x_1^3&x_1^2x_2&x_1x_2^2\\
x_2&x_1x_2&x_2^2&x_1^2x_2&x_1x_2^2&x_2^3\\
x_1^2&x_1^3&x_1^2x_2&x_1^4&x_1^3x_2&x_1^2x_2^2\\
x_1x_2&x_1^2x_2&x_1x_2^2&x_1^3x_2&x_1^2x_2^2&x_1x_2^3\\
x_2^2&x_1x_2^2&x_2^3&x_1^2x_2^2&x_1x_2^3&x_2^4\\
\end{pmatrix}
\begin{pmatrix}a\\b\\c\\d\\e\\f\end{pmatrix}\ge0.
\]
If we now linearize again, we get again an infinite family of linear inequalities, namely the same as we would have had got before
by simply expanding and linearizing $(***)$:
\[\begin{pmatrix}a&b&c&d&e&f\end{pmatrix}
\begin{pmatrix}1&x_1&x_2&y_3&y_4&y_5\\
x_1&y_3&y_4&y_6&y_7&y_8\\
x_2&y_4&y_5&y_7&y_8&y_9\\
y_3&y_6&y_7&y_1&y_{11}&y_{12}\\
y_4&y_7&y_8&y_{11}&y_{12}&y_{10}\\
y_5&y_8&y_9&y_{12}&y_{10}&y_2\\
\end{pmatrix}
\begin{pmatrix}a\\b\\c\\d\\e\\f\end{pmatrix}\ge0.
\]
The big advantage is now however that we can get rid of the parameters $a,b,c,\ldots$ by
passing over to a linear matrix inequality. Namely, the above is by \ref{psdpd}(b) valid for all $a,b,c,\ldots\in\R$ if and only if
the LMI
\[\begin{pmatrix}1&x_1&x_2&y_3&y_4&y_5\\
x_1&y_3&y_4&y_6&y_7&y_8\\
x_2&y_4&y_5&y_7&y_8&y_9\\
y_3&y_6&y_7&y_1&y_{11}&y_{12}\\
y_4&y_7&y_8&y_{11}&y_{12}&y_{10}\\
y_5&y_8&y_9&y_{12}&y_{10}&y_2\\
\end{pmatrix}\succeq0\qquad(x_1,x_2,y_1,y_2,\ldots\in\R)
\]
holds.
But there are other redundant inequalities that one can add before the linearization. For example, one could multiply
the inequality $1-x_1+x_2\ge0$ from the original system $(*)$ by a square of a general polynomial of some degree. Because
we do not want to introduce to many new variables $y_i$ in the end (and intuitively we think that each variable $y_i$ should
ideally appear many times in the final LMIs so that there is not too much freedom), we decide to multiply this time just with a general
linear polynomial. So we add for each $a,b,c\in\R$ the inequality
\[(***\,*)\qquad(a+bx_1+cx_2)^2(1-x_1+x_2)\ge0.\]
We could again expand this and linearize (re-using some of the $y_i$ from before of course) to get another infinite family of linear
inequalities. But we again apply our trick that will lead to an LMI by doing this in the following equivalent way:
Rewrite $(***\,*)$ as
\[
\begin{pmatrix}a&b&c\end{pmatrix}
(1-x_1+x_2)
\begin{pmatrix}1\\x_1\\x_2\end{pmatrix}
\begin{pmatrix}1&x_1&x_2\end{pmatrix}
\begin{pmatrix}a\\b\\c\end{pmatrix}\ge0
\]
where the second (invisible) multiplication dot out of the four is a scalar product and the others are matrix products of row and column vectors. Linearizing
\begin{multline*}
(1-x_1+x_2)
\begin{pmatrix}1\\x_1\\x_2\end{pmatrix}
\begin{pmatrix}1&x_1&x_2\end{pmatrix}=(1-x_1+x_2)\begin{pmatrix}1&x_1&x_2\\x_1&x_1^2&x_1x_2\\x_2&x_1x_2&x_2^2\end{pmatrix}
\\
=\begin{pmatrix}1-x_1+x_2&x_1-x_1^2+x_1x_2&x_2-x_1x_2+x_2^2\\x_1-x_1^2+x_1x_2&x_1^2-x_1^3+x_1^2x_2&x_1x_2-x_1^2x_2+x_1x_2^2\\x_2-x_1x_2+x_2^2&x_1x_2-x_1^2x_2+x_1x_2^2&x_2^2-x_1x_2^2+x_2^3\end{pmatrix}
\end{multline*}
we obtain the LMI
\[
\begin{pmatrix}
-x_1+x_2+1&x_1-y_3+y_4&x_2-y_4+y_5\\
x_1-y_3+y_4&y_3-y_6+y_7&y_4-y_7+y_8\\
x_2-y_4+y_5&y_4-y_7+y_8&y_5-y_8+y_9
\end{pmatrix}\succeq0
\qquad
(x_1,x_2,y_1,y_2,\ldots\in\R).
\]
Finally, we could also add redundant inequalities stemming from the inequality $1-x_1^4-x_2^4$. So we could multiply it
by a square of a general polynomial of some degree. Again, to economize additional variables $y_i$, we decide to take degree zero
here which amounts to linearize the inequality $1-x_1^4-x_2^4$ just like this without further ado. This leads to the
the linear inequality
\[1-y_1-y_2\ge0\qquad(y_1,y_2\in\R)\]
that we had already in $(**)$ and which can be seen as an LMI of size $1$.
All in all, we consider now the following system of three
LMIs of sizes $6$, $3$ and $1$ (which could easily be written as a single LMI of size $10$ by forming a block diagonal matrix)
\[(\Box)\qquad
\begin{aligned}
\begin{pmatrix}1&x_1&x_2&y_3&y_4&y_5\\
x_1&y_3&y_4&y_6&y_7&y_8\\
x_2&y_4&y_5&y_7&y_8&y_9\\
y_3&y_6&y_7&y_1&y_{11}&y_{12}\\
y_4&y_7&y_8&y_{11}&y_{12}&y_{10}\\
y_5&y_8&y_9&y_{12}&y_{10}&y_2\\
\end{pmatrix}&\succeq0
\\
\begin{pmatrix}
-x_1+x_2+1&x_1-y_3+y_4&x_2-y_4+y_5\\
x_1-y_3+y_4&y_3-y_6+y_7&y_4-y_7+y_8\\
x_2-y_4+y_5&y_4-y_7+y_8&y_5-y_8+y_9
\end{pmatrix}&\succeq0
\\
1-y_1-y_2&\ge0.
\end{aligned}
\qquad(x_1,x_2,y_1,y_2,\ldots\in\R).
\]
It is clear that the projection
\[\{(x_1,x_2)\in\R^2\mid\exists y_1,\ldots,y_{12}:(x_1,x_2,y_1,\ldots,y_{12})\text { solves }(\Box)\}\]
of the set of solutions of $(\Box)$ onto $x$-space contains the set of solutions of the original system $(*)$.
Each solution of $(*)$ can be made into a solution of $(\Box)$ by assigning to each $y_i$ the value of the nonlinear
monomial expression that has been replaced by $y_i$. In this sense $(\Box)$ is a so-called \emph{relaxation} of $(*)$.

Conversely, it is not clear in what extent a solution of $(\Box)$ might perhaps
give rise to a solution of $(*)$. A first step to investigate this, is to find a more schematic way of describing the solution set of
$(\Box)$. This time we will hide the parameters $a,b,c,...$ not in an LMI but in a truncated quadratic module [$\to$ \ref{dftrunc}].
Note that, before we came up
with the idea of writing it as an LMI, we had produced simply certain infinite systems of linear inequalities in the $x_1,x_2,y_1,y_2,\ldots$.
If we take, this time really redundantly, also sums of these linear inequalities, we observe that we effectively have added
the linearizations of the inequalities
\[p(x)\ge0\qquad(x\in\R^n)\]
for each $p\in M_4(g_1,g_2)\in\R[X_1,X_2]_4$
where $g_1:=1-X_1+X_2\in\R[X_1,X_2]$ and $g_2:=1-X_1^4-X_2^4\in\R[X_1,X_2]$. An elegant way of
expressing this is to replace the variables of the linearized system $x_1,x_2,y_1,y_2,\ldots$ by the unknown
values \[\ph(X_1),\ph(X_2),\ph(X_1^4),\ph(X_2^4),\ldots\] of some linear function $\ph\colon\R[X_1,X_2]_4\to\R$
that satisfies $\ph(1)=1$. Knowing that a linear function $\ph\colon\R[X_1,X_2]_4\to\R$ can be uniquely determined by
arbitrarily prescribed values on the
monomials, we see easily that the solution set of $(\Box)$ can be identified with the state space [$\to$ \ref{defstate}]
\[S(\R[X_1,X_2]_4,M_4(g_1,g_2),1)\subseteq\R[X_1,X_2]^*.\]
\end{ex}

\begin{df}\label{lasserredef}
Let $d\in\N$, $m\in\N_0$ and $\g\in\R[\x]^m$. Then we call
\[L_d(\g):=S(\R[\x]_d,M_d(\g),1)\subseteq\R[\x]_d^*\] the \emph{degree $d$ state space}
associated to $\g$ and
\[S_d(\g):=\{(\ph(X_1),\ldots,\ph(X_n))\mid\ph\in L_d(\g)\}\subseteq\R^n\]
the \emph{degree $d$ Lasserre relaxation} associated to $\g$.
\end{df}

\begin{rem}\label{hierarchy}
Let $m\in\N_0$ and $\g\in\R[\x]^m$.
\begin{enumerate}[(a)]
\item For each $d\in\N$, $L_d(\g)$ is obviously a convex subset of the real vector space $\R[\x]_d^*$.
\item For each $d\in\N$,  $S_d(\g)$ is a convex subset of $\R^n$ due to (a).
\item We have \[S(\g)\subseteq S_d(\g)\] for all $d\in\N$ since the evaluation at $x$
\[\R[\x]_d\to\R,\ p\mapsto p(x)\]
is a degree $d$ state with $(\ph(X_1),\ldots,\ph(X_n))=x$ for each $x\in S(\g)$.
\item We even have \[\conv S(\g)\subseteq S_d(\g)\] for all $d\in\N$ by (b) and (c).
\item We have
\[S_{d+1}(\g)\subseteq S_d(\g)\]
for all $d\in\N$ for if we restrict a degree $d+1$ state associated to $\g$ to $\R[\x]_d$, we get a degree $d$ state associated to $\g$.
\item If there is some $i\in\{1,\ldots,m\}$ such that $g_i=0$ or $\deg g_i>d$, then $M_d(\g)$ and therefore $L_d(\g)$ and $S_d(\g)$ do
not change if we remove $g_i$ from $\g$ [$\to$ \ref{dftrunc}].  In this sense, $g_i$ is simply ignored by $M_d(\g)$, $L_d(\g)$ and $S_d(\g)$ . 
Therefore one often considers only those $d\in\N$ that are greater than or equal to  the degree of each $g_i$.
\item Let $d\in\N$ such that $g_1,\ldots,g_m\in\R[\x]_d\setminus\{0\}$ and set $g_0:=1\in\R[\x]_d$. For $i\in\{0,\dots,m\}$, set
$r_i:=\lfloor\frac{d-\deg g_i}2\rfloor$ [$\to$ \ref{dftrunc}]. Let $1,u_1,\ldots,u_N$ be the monomial basis of $\R[\x]_d$.
Then $\{(\ph(u_1),\ldots,\ph(u_N))\mid\ph\in L_d(\g)\}$ can be defined by m+1 simultaneous LMIs of sizes $\dim(\R[\x]_{r_0}),\ldots,
\dim(\R[\x]_{r_m})$ in
$N$ variables. This can be seen analogously to Example \ref{x4y4}.
\end{enumerate}
\end{rem}

\begin{lem}\label{modclo}
Let $d,m\in\N_0$ and $\g\in\R[\x]^m$. Suppose $S_d(\g)\ne\emptyset$. Then
\[\overline{M_d(\g)}=\{f\in\R[\x]_d\mid\forall\ph\in L_d(\g):\ph(f)\ge0\}.\]
\end{lem}

\begin{proof}
The right hand side of the claimed equation equals
\[\bigcap_{\ph\in L_d(\g)}\ph^{-1}(\R_{\ge0})\]
and is therefore closed since every $\ph\in L_d(\g)$ is linear and therefore continuous [$\to$~\ref{tacitly}].
To prove the inclusion from left to right, it is therefore enough to show that $M_d(\g)$ is contained in the right hand side, which is trivial.

For the converse inclusion, we let $f\in\R[\x]_d\setminus\overline{M_d(\g)}$ and we want to find $\ph\in L_d(\g)$ with $\ph(f)<0$.
By \ref{seplocconvs} there are a linear $\ph_0\colon\R[\x]_d\to\R$ with $\ph_0(g)<\ph_0(f)$ for all $g\in\overline{M_d(\g)}$.
Because $M_d(\g)$ is a cone, it follows easily that $\ph_0(M_d(\g))\subseteq\R_{\le0}$ using a scaling argument. Since
$0\in M_d(\g)$, we have moreover $0<\ph_0(f)$.
Setting $\ph_1:=-\ph_0$, we have $\ph_1(M_d(\g))\subseteq\R_{\ge0}$ and $\ph_1(f)<0$.
If $\ph_1(1)>0$ then we set $\ph:=\frac1{\ph_1(1)}\ph_1\in L_d(\g)$ and we are done. Suppose therefore $\ph_1(1)=0$.
From $S_d(\g)\ne\emptyset$ it follows that $L_d(\g)\ne\emptyset$. Choose $\ps\in L_d(\g)$. Now $\ps+N\ph_1\in L_d(\g)$ for all $N\in\N$.
Choose $N\in\N$ big enough such that $\ps(f)+N\ph_1(f)<0$ and set $\ph:=\ps+N\ph_1$.
\end{proof}

\begin{rem}\label{usedallover}
The following trivial fact will be crucial and we will use it tacitly:
For all $f\in\R[\x]_1$ and linear $\ph\colon\R[\x]_1\to\R$ with $\ph(1)=1$, we have
\[f(\ph(X_1),\ldots,\ph(X_n))=\ph(f).\]
\end{rem}

\begin{lem}\label{exact1}
Let $d\in\N$, $m\in\N_0$ and $\g\in\R[\x]^m$ such that $S_d(\g)\ne\emptyset$. Then the following are equivalent:
\begin{enumerate}[\normalfont(a)]
\item $S_d(\g)\subseteq\overline{\conv S(\g)}$
\item $\forall f\in\R[\x]_1:(f\ge0\text{ on }S(\g)\implies f\in\overline{M_d(\g)})$
\end{enumerate}
\end{lem}

\begin{proof}
Suppose first that (a) holds. In order to show (b), let $f\in\R[\x]_1$ with $f\ge0$ on $S(\g)$. Since $f$ is linear it follows that
$f\ge0$ on $\conv S(\g)$ and by continuity even on its closure.
Then (a) implies that in particular $f\ge0$ on $S_d(\g)$. But then
$\ph(f)=f(\ph(X_1),\ldots,\ph(X_n))\ge0$ for all $\ph\in L_d(\g)$. By Lemma \ref{modclo}, this implies $f\in\overline{M_d(\g)}$ as desired.

Conversely, suppose now that (b) holds and let $x\in\R^n\setminus\overline{\conv S(\g)}$. We show that $x\notin S_d(\g)$.
By \ref{seplocconvs} there is a linear form $f_0\in\R[\x]_1$ and $r\in\R$ with $f_0(y)\le r<f_0(x)$ for all $y\in\overline{\conv S(\g)}$.
Setting $f:=r-f_0\in\R[\x]_1$, it follows that $f(y)\ge0>f(x)$ for all $y\in\overline{\conv S(\g)}$. In particular $f\ge0$ on $S(\g)$ and therefore
$f\in\overline{M_d(\g)}$ by (b). By Lemma \ref{modclo}, we have now $f(\ph(X_1),\ldots,\ph(X_n))=\ph(f)\ge0$ for all $\ph\in L_d(\g)$. Hence
$f\ge0$ on $S_d(\g)$. From $f(x)<0$ we now deduce that $x\notin S_d(\g)$ as desired.
\end{proof}

\begin{dfpro}\emph{[$\to$ \ref{dfconv}, \ref{defcone}]}\quad Let $V$ be a real vector space and $A\subseteq V$. The smallest cone containing $A$ is obviously
\[\cone A:=\left\{\sum_{i=1}^m\la_ix_i\mid m\in\N_0,\la_i\in\R_{\ge0},x_i\in A\right\}.\]
We call it the cone \emph{generated} by $A$ or the \emph{conic hull} of $A$.
\end{dfpro}

\begin{rem}\label{convcone}
Let $V$ be real vector space and $A\subseteq V$. Consider the subset $A\times\{1\}$ of the real vector space $V\times\R$.
One easily sees that \[\conv A=\{x\mid(x,1)\in\cone(A\times\{1\})\}.\]
\end{rem}

\begin{lem}[Carathéodory]\label{cara}
Let $n\in\N_0$, $V$ an $n$-dimensional real vector space and $A\subseteq V$.
\begin{enumerate}[(a)]
\item $\forall x\in\cone A:\exists B:(B\subseteq A\et~\#B\le n\et x\in\cone B)$
\item $\forall x\in\conv A:\exists B:(B\subseteq A\et~\#B\le n+1\et x\in\conv B)$
\end{enumerate}
\end{lem}

\begin{proof}
To prove (a), we may suppose that $A$ generates $V$ as a vector space since otherwise we can pass over from $V$ to the subspace
generated by $A$. Let $x\in\cone A$. Choose a finite set $E\subseteq A$ such that $E$ generates $V$ and
$x\in\cone E$. By \ref{fundthmlin} there is a basis
$B\subseteq E$ of $V$ such that $x\in\cone B$ or there is some $\ell\in V^*$ with $\ell(E)\subseteq\R_{\ge0}$ and $\ell(x)<0$. But the latter
cannot occur since $\ell(E)\subseteq\R_{\ge0}$ would imply $\ell(x)\in\ell(\cone E)\subseteq\R_{\ge0}$. Clearly $\#B=n$ since $B$ is a basis.

To prove (b), let $x\in\conv A$. By Remark \ref{convcone}, we have $(x,1)\in\cone(A\times\{1\})$. Since any subset of $A\times\{1\}$ is of the form $B\times\{1\}$ for some $B\subseteq A$,
(a) yields that there is some $B\subseteq A$ with $\#B\le\dim(V\times\R)=n+1$ and $(x,1)\in\cone(B\times\{1\})$. Again by Remark
\ref{convcone}, we get $x\in\conv B$.
\end{proof}

\begin{lem}\label{convhullcompact}
If $S\subseteq\R^n$ is compact, then $\conv S$ is compact.
\end{lem}

\begin{proof} Set $\De:=\{\la\in\R_{\ge0}^{n+1}\mid\la_1+\ldots+\la_{n+1}=1\}$. The set
$\De\times S^{n+1}\subseteq\R^{n+1+n(n+1)}$ is bounded and closed and therefore compact.
By Carathéodory \ref{cara}(b), we have that $\conv S$ is the image of the map
\[\De\times S^{n+1}\to\R^n,\ (\la_1,\ldots,\la_{n+1},x_1,\ldots,x_{n+1})\mapsto\sum_{i=1}^{n+1}\la_ix_i.\]
But since this map is continuous, its image is compact [$\to$ \ref{quasicompactimage}].
\end{proof}

\begin{lem}\label{homogeneousclosed}
Let $V$ and $W$ finite-dimensional real vector spaces [$\to$ \ref{tacitly}] and \[f\colon V\to W\] continuous with $f^{-1}(\{0\})=\{0\}$. Suppose there exists $r\in\R_{>0}$ such that \[\forall v\in V:\forall\la\in\R_{\ge0}:f(\la x)=\la^rf(x).\]
Then $f(V)$ is closed in $W$.
\end{lem}

\begin{proof}
Make both $V$ and $W$ into normed vector spaces by fixing arbitrary norms. Let $(x_k)_{k\in\N}$ be a sequence in $V$ such that
$\lim_{k\to\infty}f(x_k)=:y$ exists. We show that $y\in f(V)$. Since $f(0)=0$, we suppose WLOG $y\ne0$. Then WLOG
$f(x_k)\ne0$ and in particular $x_k\ne0$ for all $k\in\N$. Due to $f^{-1}(\{0\})=\{0\}$, we have for the continuous function
\[h\colon S:=\{x\in V\mid\|x\|=1\}\to\R,\ x\mapsto\|f(x)\|\]
that $h>0$ on $S$. Because $S$ is compact, there is some $\ep\in\R_{>0}$ such that $h\ge\ep$ on $S$. Now
\[\|f(x_k)\|=\|x_k\|^r\left\|f\left(\frac{x_k}{\|x_k\|}\right)\right\|\ge\|x_k\|^r\ep.\]
The convergent sequence $(f(x_k))_{k\in\N}$ is of course bounded and thus $(x_k)_{k\in\N}$ is now also bounded.
By the Bolzano–Weierstrass theorem, we may WLOG suppose that $(x_k)_{k\in\N}$ is convergent with $x:=\lim_{k\to\infty}x_k\in V$.
Now \[y=\lim_{k\to\infty}f(x_k)\overset{f\text{ continuous}}=f\left(\lim_{k\to\infty}x_k\right)=f(x)\in f(V).\]
\end{proof}

\begin{lem}\label{truncclosed}
Let $d,m\in\N_0$ and $\g\in\R[\x]^m$. Suppose that $S(\g)$ has nonempty interior. Then $M_d(\g)$ is closed.
\end{lem}

\begin{proof}
Obviously $M_d(\g)$ does not change if we discard those $g_i$ which are either the zero polynomial or whose degree exceeds $d$
[$\to$ \ref{hierarchy}(f)].
Moreover, discarding some of the $g_i$ does certainly not take away any points from $S(\g)$.
Therefore we have WLOG $0\le\deg(g_i)\le d$ for all $i\in\{1,\dots,m\}$.
Set $g_0:=1\in\R[\x]$. For $i\in\{0,\dots,m\}$, set $r_i:=\left\lfloor\frac{d-\deg g_i}2\right\rfloor$ and
$C_i:=\cone(\{p^2\mid p\in\R[\x]_{r_i}\})\subseteq\R[\x]_{2r_i}$.
By \ref{dftrunc}, we get easily
\[
M_d(\g)=
\left\{\sum_{i=0}^ms_ig_i\mid s_0\in C_0,\ldots,s_m\in C_m\right\}\subseteq\R[\x]_d.
\]
Setting $k_i:=\dim(\R[\x]_{2r_i})$ for $i\in\{0,\dots,m\}$, we get from this easily by Carathéodory \ref{cara}(a) that
\[
M_d(\g)=
\left\{\sum_{i=0}^m\left(\sum_{j=1}^{k_i}p_{ij}^2\right)g_i~\middle|~p_{ij}\in\R[\x]_{r_i}\right\}\subseteq\R[\x]_d.
\]
In other words, $M_d(\g)$ is the image of the map
\[
f\colon\prod_{i=0}^m\R[\x]_{r_i}^{k_i}\to\R[\x]_d,\ ((p_{ij})_{j\in\{1,\ldots,k_i\}})_{i\in\{0,\ldots,m\}}\mapsto
\sum_{i=0}^m\sum_{j=1}^{k_i}p_{ij}^2g_i.
\]
This map is quadratically homogeneous, that is $f(\la p)=\la^2f(p)$ for all $p$ in its domain and all $\la\in\R$.
If we can prove that $f(p)=0$ implies $p=0$ for all $p$ in the domain of $f$, then we will be done by Lemma \ref{homogeneousclosed}.
Consider the polynomial $h:=g_1\cdots g_m\ne0$. Since $S(\g)$ has non-empty interior, Lemma \ref{pol0} implies that
$h$ does not vanish on the whole of $S(\g)$. Hence
\[S':=\{x\in S(\g)\mid h(x)\ne0\}=\{x\in\R^n\mid g_1(x)>0,\ldots,g_m(x)>0\}\]
is non-empty.
Now if
\[
p:=((p_{ij})_{j\in\{1,\ldots,k_i\}})_{i\in\{0,\ldots,m\}}\in\prod_{i=0}^m\R[\x]_{r_i}^{k_i}
\]
such that $f(p)=0$, then each $p_{ij}$ vanishes on $S'$ so that $p=0$ as desired because $S'$ is open.
\end{proof}

\begin{thm}\label{exact2}\emph{[$\to$ \ref{exact1}]}\quad
Let $d\in\N$, $m\in\N_0$ and $\g\in\R[\x]^m$ such that $S(\g)$ is compact with nonempty interior. Then the following are equivalent:
\begin{enumerate}[\normalfont(a)]
\item $S_d(\g)=\conv S(\g)$
\item $\forall f\in\R[\x]_1:(f\ge0\text{ on }S(\g)\implies f\in M_d(\g))$
\end{enumerate}
\end{thm}

\begin{proof}
We have $\emptyset\ne S(\g)^\circ\subseteq S(\g)\subseteq S_d(\g)$ and therefore $S_d(\g)\ne\emptyset$.
By the Lemmata \ref{convhullcompact} and \ref{truncclosed}, we have that $\conv S(\g)$ and $M_d(\g)$ are closed.
Now the theorem follows from Lemma \ref{exact1}.
\end{proof}

\begin{cor}\label{exact3}
Let $d\in\N$, $m\in\N_0$ and $\g\in\R[\x]^m$ such that $S(\g)$ is compact with nonempty interior. Suppose that we are given
$F\subseteq\R[\x]_1$ such that \[\conv S(\g)=\{x\in\R^n\mid\forall f\in F:f(x)\ge0\}.\] Then the following are equivalent:
 \begin{enumerate}[\normalfont(a)]
\item $S_d(\g)=\conv S(\g)$
\item $F\subseteq M_d(\g)$
\end{enumerate}
\end{cor}

\begin{proof}(a)$\implies$(b) follows directly from the corresponding implication in Theorem \ref{exact2}. For the other implication, suppose
that (b) holds and let $x\in S_d(\g)$. We claim that $x\in\conv S(\g)$. Choose $\ph\in L_d(\g)$ such that $x=(\ph(X_1),\ldots,\ph(X_n))$.
By our hypothesis on $F$, it is enough to show that $f(x)\ge0$ for
all $f\in F$. But if $f\in F$, then $f\in M_d(\g)$ by (b) and hence $\ph(f)\ge0$. Since $f$ is linear, we have
$f(x)=f(\ph(X_1),\ldots,\ph(X_n))=\ph(f)\ge0$ as desired.
\end{proof}

\begin{ex}\label{x4y4rev}
Taking up our Example \ref{x4y4}, we consider again the same system of polynomial inequalities
\[(*)\qquad\quad 1-x_1+x_2\ge0,\quad 1-x_1^4-x_2^4\ge0\qquad(x_1,x_2\in\R).\]
This has given rise to a system $(\Box)$ of three LMIs of sizes $6$, $3$ and $1$
in $14$ unknowns $x_1,x_2,y_1,\ldots y_{12}$. We had asked the question whether
\[\{(x_1,x_2)\in\R^2\mid\exists y_1,\ldots,y_{12}:(x_1,x_2,y_1,\ldots,y_{12})\text { solves }(\Box)\}\]
equals the set of solutions of $(*)$. Translating this into our setting, we declare
\[g_1:=1-X_1-X_2\in\R[X_1,X_2],\qquad g_2:=1-X_1^4-X_2^4\in\R[X_1,X_2]\]
and $\g:=(g_1,g_2)$. Then $S(\g)$ is the set of solutions of $(*)$ and the
degree four state space associated to $\g$ can be identified with the set of solutions of $(\Box)$.
Moreover, the just mentioned question just asks whether $S(\g)=S_4(\g)$. We leave it as an exercise to the reader to show that 
$S(\g)$ is convex and compact with nonempty interior and that
\begin{align*}
F:=\{f\in\R[\x]_1\mid\exists x\in\R^2:(&f(x)=0=g_2(x)~\et\\
&\nabla f(x)=\nabla g_2(x))\}\cup\{g_1\}\subseteq\R[\x_1]
\end{align*}
satisfies $S(\g)=\{x\in\R^n\mid\forall f\in F:f(x)\ge0\}$. By Corollary \ref{exact3}, the question whether $S(\g)=S_4(\g)$
is now equivalent to $F\subseteq M_4(\g)$. Clearly $g_1\in M_4(\g)$. 
Now suppose $f\in F\setminus\{g_1\}$. We will then show
that $f\in M_4(g_2)\subseteq M_4(\g)$. By \ref{dh1888} and \ref{soslongrem}(b) this
means that there is some $\la\in\R_{\ge0}$ such that $f-\la g_2\ge0$ on $\R^2$. We claim that we can take $\la:=1$,
i.e., $f-g_2\ge0$ on $\R^2$. For this purpose, write $f=a_1X_1+a_2X_2+b$ with $a_1,a_2,b\in\R$.
Choose $x\in\R^2$ with $f(x)=0=g_2(x)$ and $\nabla f(x)=\nabla g_2(x)$. We have
\begin{align*}
&a_1x_1+a_2x_2+b=0=1-x_1^4-x_2^4\qquad\text{and}\\
&\begin{pmatrix}a_1\\a_2\end{pmatrix}=\nabla f(x)=\nabla g_2(x)=-4\begin{pmatrix}x_1^3\\x_2^3\end{pmatrix}.
\end{align*}
It follows that $-4x_1^4-4x_2^4+b=0$ and hence $b=4(x_1^4+x_2^4)=4$. Therefore we have to show that
\begin{align*}
h&:=f-g_2=-4x_1^3X_1-4x_2^3X_2+4-1+X_1^4+X_2^4\\
&=X_1^4+X_2^4-4x_1^3X_1-4x_2^3X_2+3
\end{align*}
is nonnegative on the whole of $\R^2$. Since the leading form $X_1^4+X_2^4$ of $h$ [$\to$ \ref{introhom}] is positive definite [$\to$ \ref{introhom}],
it is easy to show that $h$ assumes a minimum on $\R^2$. Choose $y\in\R^2$ such that $h(y)\le h(z)$ for all $z\in\R^2$. Then
\[4\begin{pmatrix}y_1^3\\y_2^3\end{pmatrix}-4\begin{pmatrix}x_1^3\\x_2^3\end{pmatrix}=\nabla h(y)=0.\]
Using that $\R\to\R,\ t\mapsto t^3$ is injective, we deduce $x=y$. Hence \[0=f(x)-g(x)=h(x)=h(y)\le h(z)\] for all $z\in\R^2$ as desired.
\end{ex}

\section{Strict quasiconcavity}

\begin{df}\label{dfqc}
Let $g\in\R[\x]$. If $x\in\R^n$, then we call $g$
\emph{strictly concave} at $x$ if $(\hess g)(x)\prec0$
and \emph{strictly quasiconcave} at $x$ if \[((\nabla g)(x))^Tv=0\implies
v^T(\hess g)(x)v<0\]
for all $v\in\R^n\setminus\{0\}$. If $S\subseteq\R^n$,
we call $g$ \emph{strictly (quasi-)concave} on $S$ is $g$ is
\emph{strictly (quasi-)concave} at every point of $S$.
\end{df}

\begin{rem}
Let $g\in\R[\x]$ and $x\in\R^n$ such that $\nabla g(x)=0$.
\begin{enumerate}[(a)]
\item
$g$ is strictly quasiconcave at $x$ if and only if
$\hess g(x)\prec0$.
\item
If $g$ is strictly quasiconcave at $x$ and $g(x)=0$, then there is a neighborhood $U$ of $x$ such that
$U\cap S(g)=\{x\}$.
\end{enumerate}
\end{rem}

\begin{lem}\label{prepqcc}
Let $n\in\N$, $g\in\R[\x]$ and $x\in\R^n$ such that $g(x)=0$ and $\nabla g(x)\ne0$.
Suppose $v_1,\dots,v_n$ form a basis of $\R^n$, $U$ is an open neighborhood of $0$ in
$\R^{n-1}$, $\ph\colon U\to\R$ is smooth and satisfies $\ph(0)=0$ as well as
\[(*)\qquad g(x+\xi_1v_1+\ldots+\xi_{n-1}v_{n-1}+\ph(\xi)v_n)=0\]
for all $\xi=(\xi_1,\dots,\xi_{n-1})\in U$. Then the following hold:
\begin{enumerate}[(a)]
\item $(\nabla g(x))^Tv_1=\ldots=(\nabla g(x))^Tv_{n-1}=0\iff
\nabla\ph(0)=0$
\item If $\nabla\ph(0)=0$ and $(\nabla g(x))^Tv_n>0$, then
\[\text{$g$ is strictly quasiconcave at $x$}\iff\hess\ph(0)\succ0.
\]
\end{enumerate}
\end{lem}

\begin{proof}
Taking the derivative of $(*)$ with respect to $\xi_i$, we get
\[(**)\qquad
(\nabla g(x+\xi_1v_1+\ldots+\xi_{n-1}v_{n-1}+\ph(\xi)v_n))^T\left(v_i+\frac{\partial\ph(\xi)}{\partial\xi_i}v_n\right)=0\]
for all $i\in\{1,\dots,n-1\}$ and $\xi\in U$. Setting here $\xi$ to $0$,
we get \[(\nabla g(x))^T\left(v_i+\frac{\partial\ph(\xi)}{\partial\xi_i}\middle|_{\xi=0}v_n\right)=0\] 
for each $i\in\{1,\dots,n-1\}$. From this, (a) follows easily (for ``${\implies}$'' use that $(\nabla g(x))^Tv_n\ne0$ since $v_1,\dots,v_n$ is a basis).
Taking the derivative of $(**)$ with respect to $\xi_j$, we get
\begin{multline*}
\left(v_j+\frac{\partial\ph(\xi)}{\partial\xi_j}v_n\right)^T(\hess g(x+\xi_1v_1+\ldots+\xi_{n-1}v_{n-1}+\ph(\xi)v_n))\left(v_i+\frac{\partial\ph(\xi)}{\partial\xi_i}v_n\right)\\
+(\nabla g(x+\xi_1v_1+\ldots+\xi_{n-1}v_{n-1}+\ph(\xi)v_n))^T
\left(\frac{\partial^2\ph(\xi)}{\partial\xi_i\partial\xi_j}v_n\right)=0
\end{multline*}
for all $i,j\in\{1,\dots,n-1\}$ and $\xi\in U$.
To prove (b), suppose now that $\nabla\ph(0)=0$ and $(\nabla g(x))^Tv_n>0$. Then the preceding equation implies
\[\hess\ph(0)=-\frac1{(\nabla g(x))^Tv_n}(v_i^T(\hess g(x))v_j)_{i,j\in\{1,\dots,n-1\}}.\]
Since $v_1,\ldots,v_{n-1}$ now form a basis of the orthogonal complement of $\nabla g(x)$ by (a), the matrix
$(v_i^T(\hess g(x))v_j)_{i,j\in\{1,\dots,n-1\}}$ is negative definite if and only if $g$ is strictly quasiconcave at $x$
(see Definition \ref{dfqc}).
\end{proof}

\noindent
The following proposition is important for understanding the notion of quasiconcavity. It is trivial that
quasiconcavity of a polynomial $g$ at $x$ depends only on $x$ and the function $V\to\R,\ x\mapsto g(x)$
where $V$ is an arbitrarily small neighborhood of $x$. But if $g(x)=0$ and
$\nabla g(x)\ne0$, then it actually
depends only on $x$ and the function
\[V\to\{-1,0,1\},\ x\mapsto\sgn(g(x))\]
as the equivalence of Conditions (a) and (b) of the following proposition show.

\begin{notation}
For $g\in\R[\x]$, we call
\[Z(g):=\{x\in\R^n\mid g(x)=0\}\]
the \emph{real zero set} of $g$.
\end{notation}

\begin{pro}\label{qcc}
Let $n\in\N$, $g\in\R[\x]$ and $x\in\R^n$ such that
\[g(x)=0\text{ and }\nabla g(x)\ne0.\]
Suppose that $V$ is a neighborhood of $x$. Then the following are equivalent:
\begin{enumerate}[\normalfont(a)]
\item $g$ is strictly quasiconcave at $x$.
\item There is a basis $v_1,\dots,v_n$ of $\R^n$, an open neighborhood $U$ of $0$ in $\R^{n-1}$
and a smooth function $\ph\colon U\to\R$ such that $\ph(0)=0$, $\nabla\ph(0)=0$,
$\hess\ph(0)\succ0$,
\[(*)\qquad x+\xi_1v_1+\ldots+\xi_{n-1}v_{n-1}+\ph(\xi)v_n\in Z(g)\cap V\]
for all $\xi\in U$ and
\[(**)\qquad x+\la v_n\in S(g)\cap V\]
for all small enough $\la\in\R_{>0}$.
\item Condition (b) holds with ``basis'' replaced by ``orthogonal basis''.
\end{enumerate}
For any basis $v_1,\dots,v_n$ of $\R^n$ like in (b), one has
\[(***)\qquad(\nabla g(x))^Tv_1=\ldots=(\nabla g(x))^Tv_{n-1}=0\qquad\text{and}\qquad(\nabla g(x))^Tv_n>0.
\]
\end{pro}

\begin{proof}
Using Lemma \ref{prepqcc}(a), it is easy to show that any $v_1,\dots,v_n$ like in (b) satisfy $(***)$ using that $(\nabla g(x))^Tv_n=0$ would
contradict the hypothesis $\nabla g(x)\ne0$ since $v_1,\dots,v_n$ is a basis. Now Part (b) of the same lemma shows that (b) implies (a).
Since it is trivial that (c) implies (b), it only remains to show that (a) implies (c).

To this end,
let (a) be satisfied. In order to show (c), choose an orthogonal basis $v_1,\dots,v_n$ of $\R^n$ satisfying $(***)$.
The implicit function theorem yields an open neighborhood
$U$ of the origin in $\R^{n-1}$ such that for each $\xi=(\xi_1,\dots,\xi_{n-1})\in U$ there is a
unique $\ph(\xi)\in\R$ satisfying $(*)$, in particular $\ph(0)=0$. Moreover, one can choose $U$ such that the resulting function
$\ph\colon U\to\R$ is smooth. From $(\nabla g(x))^Tv_n>0$, we get $(**)$. From Part (a) of Lemma \ref{prepqcc}, we get
$\nabla\ph(0)=0$. From Part (b) of the same lemma and from (a), we obtain $\hess\ph(0)\succ0$.
\end{proof}

\begin{rem}\label{quasiconcavelift}
For $g\in\R[\x]$ and $x\in\R^n$, the following are obviously equivalent:
\begin{enumerate}[(a)]
\item $g$ is strictly quasiconcave at $x$.
\item $\exists\la\in\R:\la(\nabla g(x))(\nabla g(x))^T\succ(\hess g)(x)$
\end{enumerate}
\end{rem}

\begin{pro}\label{psdinterior}
$\R^{n\times n}_{\succeq0}:=\{A\in\R^{n\times n}\mid A\succeq0\}$ is a cone in the vector space
$S\R^{n\times n}$ whose interior \emph{[$\to$ \ref{tacitly}]} is
$\R^{n\times n}_{\succ0}:=\{A\in\R^{n\times n}\mid A\succ0\}$.
\end{pro}

\begin{proof}
Equip $S\R^{n\times n}$ with the norm defined by
\[\|A\|:=\max_{\substack{x\in\R^n\\\|x\|\le1}}|x^TAx|\]
for $A\in S\R^{n\times n}$. By \ref{topvsex}(c), this norm (as any other norm)
induces the unique vector space topology on $S\R^{n\times n}$. 

If $A$ is an interior point of $\R^{n\times n}_{\succeq0}$, then there exists $\ep\in\R_{>0}$ such that
$A-\ep I_n\succeq0$ and thus $A\succeq\ep I_n\succ0$.

Conversely, let $A\in\R^{n\times n}$ satisfy
$A\succ0$. We show that $A$ is an interior point of $\R^{n\times n}_{\succeq0}$.
By \ref{psdeq}, the lowest eigenvalue $\ep$ of $A$ is nonnegative since $A\succeq0$. Actually, we have
even $\ep>0$
since $A$ has trivial kernel due to $A\succ0$. Now $A-\ep I_n$ has only nonnegative
eigenvalues and thus $A-\ep I_n\succeq0$ by \ref{psdeq}.
It suffices to show that a ball around $A$ with radius $\ep$ in
$S\R^{n\times n}$ is contained in $\R^{n\times n}_{\succeq0}$. For this purpose, let
$B\in S\R^{n\times n}$ with $\|B-A\|\le\ep$ and fix $x\in\R^n$ with $\|x\|=1$. We have to show that
$x^TBx\ge0$. But we have $x^TBx=x^TAx+x^T(B-A)x\ge x^TAx-\|B-A\|\ge\ep x^TI_nx-\ep=0$.
\end{proof}

\begin{lem}\label{strictly2neighbor}
Let $g\in\R[\x]$. If $g$ is strictly (quasi-)concave at $x\in\R^n$,
then there is a neighborhood $U$ of $x$ such that $g$ is
strictly (quasi-)concave on $U$.
\end{lem}

\begin{proof}
The first statement follows from the openness of $\R^{n\times n}_{\succ0}$ [$\to$ \ref{psdinterior}]
by the continuity of $\R^n\to S\R^{n\times n},\ x\mapsto(\hess g)(x)$. The second statement follows
similarly by using the equivalence of (a) and (b) in Remark \ref{quasiconcavelift}.
\end{proof}

\begin{rem}\label{derexo}
Let $k\in\N_0$, $g\in\R[\x]$ and $x\in\R^n$ with $g(x)=0$. Then
\begin{align*}
(\nabla(g(1-g)^k))(x)&=(\nabla g)(x)\qquad\text{and}\\
(\hess(g(1-g)^k))(x)&=(\hess g-2k(\nabla g)(\nabla g)^T)(x).
\end{align*}
\end{rem}

\begin{lem}\label{quasi2concave}
Suppose $g\in\R[\x]$, $u\in\R^n$ and $g(u)=0$.
Then the following are equivalent:
\begin{enumerate}[(a)]
\item $g$ is strictly quasiconcave at $u$.
\item There exists $k\in\N$ such that $g(1-g)^k$ is strictly concave at $u$.
\item There exists $k\in\N$ such that for all $\ell\in\N$ with $\ell\ge k$, we have
that $g(1-g)^\ell$ is strictly concave at $u$.
\end{enumerate}
\end{lem}

\begin{proof}
Combine Remarks \ref{derexo} and \ref{quasiconcavelift}.
\end{proof}

\section{Lagrange multipliers from real closed fields}

\begin{rem}\label{gradientpositive}
If $u,x\in\R^n$ with $v:=x-u\ne0$, $g\in\R[\x]$, $g$ is strictly quasiconcave at $u$, $g(u)=0$ and $g\ge0$ on $\conv\{u,x\}$, then
obviously $(\nabla g(u))^Tv>0$.
\end{rem}

\begin{lem}[Existence of real Lagrange multipliers]\label{lagrange}
Suppose $u\in\R^n$, $f\in\R[\x]$, $m\in\N_0$, $\g\in\R[\x]^m$.
Let $U$ be a neighborhood of $u$ in $\R^n$ such that $U\cap S(\g)$
is convex and not a singleton. Moreover, suppose $g_1,\ldots,g_m$ are strictly quasiconcave at $u$, $f\ge0$ on $U\cap S(\g)$ and
\[f(u)=g_1(u)=\ldots=g_m(u)=0.\]
Then there are $\la_1,\ldots,\la_m\in\R_{\ge0}$ such that
\[\nabla f(u)=\sum_{i=1}^m\la_i\nabla g_i(u).\]
\end{lem}

\begin{proof}
Choose $x\in U\cap S(\g)$ with $v:=x-u\ne0$. By Remark \ref{gradientpositive}, we have \[(\nabla g_i(u))^Tv>0\]
for $i\in\{1,\ldots,m\}$ since
$U\cap S(\g)$ is convex. Assume the required Lagrange multipliers do not exist. Then \ref{fundthmlin} yields $w\in\R^n$
such that $((\nabla f)(u))^Tw<0$ and $((\nabla g_i)(u))^Tw\ge0$ for all $i\in\{1,\ldots,m\}$. Replacing $w$ by $w+\ep v$ for some
small $\ep>0$, we get even $((\nabla g_i)(u))^Tw>0$ for all $i\in\{1,\ldots,m\}$. Then for all sufficiently small $\de\in\R_{>0}$,
we have $u+\de v\in U\cap S(\g)$ but $f(u+\de v)<0$ $\lightning$.
\end{proof}

\begin{df}{}[$\to$ \ref{interiorclosure}]\label{dfboundary}
Let $M$ be a topological space and $A\subseteq M$.
We call \[\partial A:=\overline A\setminus A^\circ=\overline A\cap\overline{M\setminus A}\]
the \emph{boundary} of $A$.
\end{df}

\begin{df}\label{nearconvexboundary}
Let $S\subseteq\R^n$.
We call \[\convbd S:=S\cap\partial\conv S\] the \emph{convex boundary} of $S$.
Obviously,
 \[\convbd S=\{x\in S\mid\forall U\in\mathcal U_x:U\not\subseteq\conv S\}.\]
We say that $S$ \emph{has nonempty interior near its convex boundary}
if $\convbd S\subseteq\overline{S^\circ}$.
\end{df}

\begin{pro}\label{convbdchar}
Let $S\subseteq\R^n$. Then
\[\convbd S=\{u\in S\mid\exists\ph\in(\R^n)^*\setminus\{0\}:\forall x\in S:\ph(u)\le\ph(x)\}.\]
\end{pro}

\begin{proof}
``$\supseteq$'' Let $u\in S$ and $\ph\in(\R^n)^*\setminus\{0\}$ such that
$\forall x\in S:\ph(u)\le\ph(x)$. Then even $\forall x\in\conv S:\ph(u)\le\ph(x)$. Choose
$v\in\R^n$ such that $\ph(v)>0$. Then $\ph(u-\ep v)<\ph(u)$ and hence $u-\ep v\notin\conv S$ for
each $\ep\in\R_{>0}$. It follows that no neighborhood of $u$ is contained in the convex hull of
$S$. Hence $u\in\convbd S$.

\smallskip
``$\subseteq$''
If $\dim\conv S<n$ [$\to$ \ref{dfdimconv}], we have $\partial\conv S=\overline{\conv S}$ and hence
$\convbd S=S$ and one easily finds $\ph\in(\R^n)^*\setminus\{0\}$ that is constant on
$\conv S$. So now suppose that $\dim\conv S=n$. Let $u\in\convbd S$. By
Theorem \ref{hasaface}, we get an exposed face $F$ of $\conv S$ with $\dim F<n$ and $u\in F$.
Choose $\ph\colon\R^n\to\R$ linear such that
\[F=\{y\in\conv S\mid\forall x\in\conv S:\ph(y)\le\ph(x)\}.\]
Since $\dim F<n$, we have obviously $\ph\ne0$.
\end{proof}

\begin{lem}\label{prep1}
Let $B\subseteq\R^n$ be a closed ball in $\R^n$, $m\in\N_0$ and $\g\in\R[\x]^m$.
Suppose that $g_1,\ldots,g_m\in\R[\x]$ are strictly quasiconcave on $B$.
Then the following hold:
\begin{enumerate}[(a)]
\item $S:=B\cap S(\g)$ is convex.
\item Every linear form from $\R[\x]\setminus\{0\}$ [$\to$ \ref{longremi}(a)] has at most one minimizer on $S$.
\item Let $u$ be a minimizer of the linear form $f\in\R[\x]\setminus\{0\}$ on $S$ and set
\[I:=\{i\in\{1,\ldots,m\}\mid g_i(u)=0\}.\]
Then $u$
is also minimizer of $f$ on $S':=\{x\in B\mid\forall i\in I:g_i(x)\ge0\}\supseteq S$. 
\end{enumerate}
\end{lem}

\begin{proof}
(a)
Let $x,y\in S$ with $x\ne y$ and $i\in\{1,\ldots,m\}$. The polynomial \[f:=g_i(Tx+(1-T)y)\in\R[T]\]
attains a minimum $a$ on $[0,1]_\R$ [$\to$ \ref{takeson}]. We have to show $a\ge0$.
Because of $f(0)=g_i(y)\ge0$ and $f(1)=g_i(x)\ge0$, it is enough to show that this minimum is not attained
in a point $t\in(0,1)_\R$. Assume it is. Then $f'(t)=0$, i.e., $((\nabla g_i)(z))^Tv=0$ for $z:=tx+(1-t)y$
and $v:=x-y\ne0$. Since $z\in B$ and hence $g_i$ is strictly quasiconcave at $z$, it follows that
$v^T((\hess g_i)(z))v<0$, i.e., $f''(t)<0$. Then $f<a$ on a neighborhood of $t$ [$\to$~\ref{sgnbounds}(b)]
$\lightning$.

\medskip
(b) Suppose $x$ and $y$ are minimizers of the linear form $f\in\R[\x]\setminus\{0\}$ on $S$.
Then $x,y\in\convbd S$ by \ref{convbdchar}. Since $f$ is linear, it is constant on $\aff\{x,y\}$. Hence even
\[\conv\{x,y\}\overset{\text{(a)}}\subseteq\aff\{x,y\}\cap S\overset{\ref{convbdchar}}\subseteq\convbd S\overset{\text{(a)}}=S\cap\partial S
=S\cap(\overline S\setminus S^\circ)\overset{\text{$S$ closed}}=S\setminus S^\circ=\partial S.\]
Since $\conv\{x,y\}\setminus\{x,y\}\subseteq B^\circ$, we have then that
$\conv\{x,y\}\setminus\{x,y\}\subseteq Z(g_1\dotsm g_m)$.
Assume now for a contradiction that $x\ne y$.
Then this implies that at least one of the $g_i$ vanishes on $\aff\{x,y\}$. Fix a corresponding $i$.
Setting $v:=y-x$, we have then $((\nabla g_i)(x))^Tv=0$ and $v^T((\hess g_i)(x))v=0$.
Since $g_i$ is strictly quasiconcave at $x$, this implies $v=0$, i.e., $x=y$ as desired.

\medskip
(c) By definition of $I$, the sets $S$ and $S'$ coincide on a neighborhood of $u$ in $\R^n$. Hence $u$ is a
\emph{local} minimizer of $f$ on $S'$. Since $S'$ is convex by (a) and $f$ is linear, $u$ is also a (\emph{global})
minimizer of $f$ on $S'$.
\end{proof}

\begin{lem}\label{lagrange2}
Suppose $B$ is a closed ball in $\R^n$, $m\in\N_0$ and $\g\in\R[\x]^m$.
Suppose that $g_1,\ldots,g_m\in\R[\x]$ are strictly quasiconcave on $B$ and that $S:=B\cap S(\g)$
has nonempty interior. Then the following hold:
\begin{enumerate}[(a)]
\item For every real closed extension field $R$ of $\R$
and all linear forms $f\in R[\x]\setminus\{0\}$,
$f$ has a unique minimizer on $\transfer_{\R,R}(S)$.
\item
For every real closed extension field $R$ of $\R$,
all linear forms $f\in R[\x]$ with \[\|\nabla f\|_2=1\] [$\to$~\ref{normcont}]
(note that $\nabla f\in R^n$ as $f$ is linear) and every
$u\in\transfer_{\R,R}(B^\circ)$ which minimizes $f$ on $\transfer_{\R,R}(S)$,
there are $\la_1,\ldots,\la_m\in\O_R\cap R_{\ge0}$ with
\[\la_1+\ldots+\la_m\notin\m_R\] such that both $f-f(u)-\sum_{i=1}^m\la_ig_i$ and its gradient
vanish at $u$.
\end{enumerate}
\end{lem}

\begin{proof}
(a) Consider the class of all real closed extension fields $R$ of $\R$ such that
all linear forms from $R[\x]\setminus\{0\}$ have a unique minimizer on $\transfer_{\R,R}(S)$.
By real quantifier elimination [$\to$ \ref{elim}], this is easily seen to be a $0$-ary $\R$-semialgebraic class
[$\to$ \ref{introsemialg}]. By \ref{nothingorall}, this class is either empty or consists of all
real closed extensions fields of $\R$. Hence it suffices to prove the statement in the case $R=\R$
[$\to$ \ref{tprinciple}]. But then the unicity part follows from Lemma \ref{prep1}(b) and the existence part
from \ref{takeson}.

\medskip
(b) Now let $R$ be a real closed field extension of $\R$, $f\in R[\x]$ a linear form with $\|\nabla f\|_2=1$ and
$u$ a minimizer of $f$ on $\transfer_{\R,R}(S\cap B^\circ)$.
Set \[I:=\{i\in\{1,\ldots,m\}\mid g_i(u)=0\}\]
and define the set
\[S':=\{x\in B\mid\forall i\in I:g_i(x)\ge0\}\supseteq S\]
which is convex by  \ref{prep1}(a).
Using the Tarski principle [$\to$ \ref{tprinciple}],
one shows easily that $u$ is a minimizer of $f$ on $\transfer_{\R,R}(S')$ by Lemma \ref{prep1}(c).
Note also that of course $u\in\O_R^n$ and $\st(u)\in S$.

Now Lemma \ref{lagrange} says in particular that for all linear forms $\widetilde f\in\R[\x]$ and
minimizers $\widetilde u$ of $\widetilde f$ on $S'\cap B^\circ$ with $\forall i\in I:g_i(\widetilde u)=0$, there is a family
$(\la_i)_{i\in I}$ in $\R_{\ge0}$ such that \[\nabla\widetilde f=\sum_{i\in I}\la_i\nabla g_i(\widetilde u).\]

Using the Tarski principle [$\to$ \ref{tprinciple}], we see that actually for all real closed extension fields
$\widetilde R$ of $\R$, all linear forms $\widetilde f\in\widetilde R[\x]$ and
all minimizers $\widetilde u$ of $\widetilde f$ on $\transfer_{\R,R}(S'\cap B^\circ)$
with $\forall i\in I:g_i(\widetilde u)=0$, there is a family
$(\la_i)_{i\in I}$ in $R_{\ge0}$ such that
\[\nabla\widetilde f=\sum_{i\in I}\la_i\nabla g_i(\widetilde u).\]

We apply this to $\widetilde R:=R$, $\widetilde u:=u$, $\widetilde f:=f$ and thus obtain a family
$(\la_i)_{i\in I}$ in $R_{\ge0}$ such that
\[(*)\qquad\nabla f=\sum_{i\in I}\la_i\nabla g_i(u).\]

In order to show that $\la_i\in\O_R$ for all $i\in I$, we choose a point $x\in S'^\circ\ne\emptyset$ with $\prod_{i\in I}g_i(x)\ne0$
and thus $g_i(x)>0$ for all $i\in I$. Setting $v:=x-u\in\O_R^n$, we get from $(*)$ that
$(\nabla f)^Tv=\sum_{i\in I}\la_i(\nabla g_i(u))^Tv$. Since $\st(u)\in S'$ and $S'$ is convex,
Remark~\ref{gradientpositive} yields $\st((\nabla g_i(u))^Tv)=(\nabla g_i(\st(u)))^T\st(v)>0$ for all $i\in I$
(use that $\st(u)\ne x$ since $g_i(\st(u))=0$ while $g_i(x)>0$).
Together with $\la_i\ge0$ for all $i\in I$, this shows $\la_i\in\O_R$ for all $i\in I$ as desired.

It now suffices to show that $\sum_{i\in I}\la_i\notin\m_R$. But this is clear since $(*)$ yields in particular
\[1=\|\nabla f\|_2\le\sum_{i\in I}\la_i\|(\nabla g_i)(u)\|_2\le\left(\sum_{i\in I}\la_i\right)\max_{i\in I}\|(\nabla g_i)(u)\|_2\]
(note that $I\ne\emptyset$ by the first inequality) which readily implies $\sum_{i\in I}\la_i\notin\m_R$.
\end{proof}

\begin{lem}\label{finitecontact}
Let $m\in\N_0$ and $\g\in\R[\x]^m$ such that $S:=S(\g)$ is compact and has nonempty interior near its convex boundary.
Suppose that $g_i$ is strictly quasiconcave on \[(\convbd S)\cap Z(g_i)\] for each
$i\in\{1,\dots,m\}$.
Let $R$ be real closed extension field of $\R$ and
$f\in R[\x]$ be a linear form with $\|\nabla f\|_2=1$.
Then the following hold:
\begin{enumerate}[(a)]
\item $F:=\{u\in S\mid\forall x\in S:\st(f(u))\le\st(f(x))\}$
is a finite subset of $\convbd S$.
\item $S':=\transfer_{\R,R}(S)\subseteq\O_R^n$ and
$f$ has a unique minimizer $x_u$ on \[\{x\in S'\mid \st(x)=u\}\] for each $u\in F$.
\item For every $u\in F$, there are
$\la_{u1},\ldots,\la_{um}\in\O_R\cap R_{\ge0}$ with
\[\la_{u1}+\ldots+\la_{um}\notin\m_R\] such that both $f-f(x_u)-\sum_{i=1}^m\la_{ui}g_i$ and its gradient
vanish at $x_u$.
\end{enumerate}
\end{lem}

\begin{proof}
(a) Obviously $\st(f)\ne0$ and hence \[F=\{u\in S\mid\forall x\in S:(\st(f))(u)\le(\st(f))(x)\}\subseteq
\convbd S\] by Proposition \ref{convbdchar}.
We now prove that $F$ is finite.
WLOG $S\ne\emptyset$. Set [$\to$ \ref{takeson}] \[a:=\min\{(\st(f))(x)\mid x\in S\}\] so that
\[F=\{u\in S\mid(\st(f))(u)=a\}.\] By compactness of $S$, it is enough to show that
every $x\in S$ possesses a neighborhood $U$ in $S$ such that $U\cap F\subseteq\{x\}$.
This is trivial for the points of $S\setminus F$. So consider an arbitrary point $x\in F$.
Since $x\in\convbd S$, each $g_i$ is positive or strictly
quasiconcave at $x$. According to \ref{strictly2neighbor}, we can choose a closed ball $B$ of positive radius around
$x$ in $\R^n$ such that each $g_i$ is positive or strictly
quasiconcave even on $B$. By Lemma \ref{prep1}(b), $\st(f)$ has at most one minimizer on $U:=S\cap B$,
namely $x$, i.e., $U\cap F\subseteq\{x\}$.

\medskip
(b) First observe that
$S':=\transfer_{\R,R}(S)\subseteq\O_R^n$ since the
transfer from $\R$ to $R$ is an isomorphism of Boolean algebras [$\to$ \ref{transfer}]:
Choosing $N\in\N$ with $S\subseteq[-N,N]_\R^n$, we have $S'\subseteq
\transfer_{\R,R}([-N,N]^n_\R)=[-N,N]^n_R\subseteq\O_R^n$.

Now we fix $u\in F$ and we show that $f$ has a unique minimizer on
\[A:=\{x\in S'\mid\st(x)=u\}.\]
Choose $\ep\in\R_{>0}$ such that
each $g_i$ is strictly quasiconcave or positive on
the ball \[B:=\{v\in\R^n\mid\|v-u\|\le\ep\}.\]
Since $u\in\convbd S\subseteq\overline{S^\circ}$,
Lemma \ref{lagrange2}(a) says that $f$ has a unique minimizer $x$ on $\transfer_{\R,R}(S\cap B)$.
Because of $A\subseteq\transfer_{\R,R}(S\cap B)$, it is thus enough to show $x\in A$.
Note that $u\in F\cap B\subseteq S\cap B\subseteq\transfer_{\R,R}(S\cap B)$ and thus
$f(x)\le f(u)$. This implies $\st(f(\st(x)))=\st(f(x))\le\st(f(u))$ which yields together with $\st(x)\in S$ that
$\st(x)\in F$ (and $\st(f(\st(x)))=\st(f(u))$). Again by Lemma \ref{lagrange2}(a), $\st(f)$ has a unique minimizer on
$S\cap B$ . But $u$ and $\st(x)$ are both a minimizer of $\st(f)$ on $S\cap B$ (note that $\st(x)\in S\cap B$).
Hence $u=\st(x)$ and thus $x\in A$ as desired.

\medskip
(c) Fix $u\in F$. Choose again $\ep\in\R_{>0}$ such that
each $g_i$ is strictly quasiconcave or positive on
the ball $B:=\{v\in\R^n\mid\|v-u\|\le\ep\}$ and such that $B\cap F=\{u\}$.
Since $x_u\in\transfer_{\R,R}(B^\circ)$ obviously minimizes $f$ on $\transfer_{\R,R}(S\cap B)$, we get the necessary
Lagrange multipliers by Lemma \ref{lagrange2}(b).
\end{proof}

\section{Linear polynomials and truncated quadratic modules}

\begin{exo}\label{calculusexo}
For all $k\in\N$ and $x\in[0,1]_\R$, we have $x(1-x)^k\le\frac1k$.
\end{exo}

The main geometric idea in the proof of the following theorem is as follows: Consider a hyperplane that isolates a basic closed
semialgebraic subset of $\R^n$ and that is defined over a real closed extension field of $\R$. Because we want to apply
Theorem $\ref{mainrep}$ to get a sums of squares ``isolation certificate'', the points where the hyperplane gets infinitesimally
close to the set pose problems unless the hyperplane \emph{exactly} touches the set in the respective point. The idea is
to find a nonlinear infinitesimal deformation of the hyperplane so that all ``infinitesimally near points'' becoming ``touching points''.
This would be easier (although still not obvious) if there is at most one ``infinitesimally near point'' but since we deal in this article
with \emph{not necessarily convex} basic closed semialgebraic sets, it is crucial to cope with several such points.

\begin{thm}\label{linearstability}
Let $m\in\N_0$ and $\g\in\R[\x]^m$ such that $M(\g)$ is Archimedean and suppose that $S:=S(\g)$
has nonempty interior near its convex boundary.
Suppose that $g_i$ is strictly quasiconcave on $(\convbd S)\cap Z(g_i)$ for each
$i\in\{1,\dots,m\}$. Let $R$ be a real closed extension field of $\R$ and
$\ell\in\O_R[\x]_1$ such that $\ell\ge0$ on $\transfer_{\R,R}(S)$.
Then $\ell$ lies in the quadratic module generated by $g_1,\ldots,g_m$ in $\O_R[\x]$.
\end{thm}

\begin{proof}
We will apply Theorem $\ref{mainrep}$. Since $S$ is compact, we can rescale the $g_i$ and suppose WLOG that
\[g_i\le1\text{ on }S\]
for $i\in\{1,\ldots,m\}$.
Let $M$ denote the quadratic module generated by $g_1,\ldots,g_m$ in $\O_R[\x]$.
Since $M(\g)$ is Archimedean, also $M$ is Archimedean by
\ref{archmodulechar}(b) and \ref{archmodulecharrcf}(b). Moreover, $S$ could now alternatively be defined
from $M$ as in Theorem $\ref{mainrep}$. Write
\[\ell=f-c\]
with a linear form $f\in\O_R[\x]$ and $c\in\O_R$.
By a rescaling argument, we can suppose that at least one of the coefficients of $\ell$ lies
in $\O_R^\times$ [$\to$ \ref{ordval}].
If $\st(\ell(x))>0$ for all $x\in S$, then Theorem $\ref{mainrep}$ applied to $\ell$ with $k=0$ yields $\ell\in M$
and we are done. Hence we can from now on suppose that there is some $u\in S$ with $\st(\ell(u))=0$. For such an $u$,
we have $\st(c)=\st(f(u))$ so that at least one coefficient of $f$ must lie $\O_R^\times$.
By another rescaling, we now can suppose WLOG that $\|\nabla f\|_2=1$. Now we are in the situation of
Lemma~\ref{finitecontact} and we define
\[F,\quad (x_u)_{u\in F}\quad\text{and}\quad
(\la_{ui})_{(u,i)\in F\times\{1,\ldots,m\}}\]
accordingly. Note that
\[F=\{u\in S\mid\st(\ell(u))=0\}\ne\emptyset\]
since $\st(\ell(x))\ge0$ for all $x\in S$.
We have $f(x_u)-c=\ell(x_u)\ge0$ and \[\st(f(x_u)-c)=\st(\ell(u))=0\] for all $u\in F$.
Hence $f(x_u)-c\in\m_R\cap R_{\ge0}$ for all $u\in F$. We thus have
\[\ell-\underbrace{(f(x_u)-c)}_{=:\la_{u0}\in\m_R\cap R_{\ge0}}-\sum_{i=1}^m\underbrace{\la_{ui}}_{\in\O_R\cap R_{\ge0}}g_i\in I_{x_u}^2\]
for all $u\in F$ by \ref{finitecontact}(c) and \ref{membershipix}. Evaluating this in $x_u$ (and using $g_i(x_u)\ge0$) yields
\begin{gather}
\tag{$*$}g_i(x_u)\ne0\implies\la_{ui}=0\qquad\text{and thus}\\
\tag{$**$}\la_{ui}g_i\equiv_{I_{x_u}^2}\la_{ui}g_i(1-g_i)^k
\end{gather}
for all $u\in F$, $i\in\{1,\ldots,m\}$ and $k\in\N$.
By the Chinese remainder theorem, we find polynomials $s_0,\ldots,s_m\in\O_R[\x]$
such that $s_i\equiv_{I_{x_u}^3}\sqrt{\la_{ui}}\in\O_R$ for all $u\in F$ and $i\in\{0,\ldots,m\}$
because the ideals $I_{x_u}^3$ ($u\in F$) are pairwise coprime [$\to$ \ref{coprime}] (use that
$\st(x_u)=u\ne v=\st(x_v)$ for all $u,v\in F$ with $u\ne v$).
By an easy scaling argument, we can even guarantee that the coefficients
of $s_0$ lie in $\m_R$ since $\sqrt{\la_{u0}}\in\m_R$. Then we have
\begin{equation}
\tag{$***$}s_i^2\equiv_{I_{x_u}^3}\la_{ui}
\end{equation}
which means in other words
\[s_i^2(x_u)=\la_{ui},\qquad(\nabla(s_i^2))(x_u)=0\qquad\text{and}\qquad(\hess(s_i^2))(x_u)=0\]
for all $i\in\{0,\ldots,m\}$  and $k\in\N$.
It suffices to show that there is $k\in\N$ such that the polynomial
\[\ell-s_0^2-\sum_{i=1}^ms_i^2(1-g_i)^{2k}g_i\overset{(***)}{\underset{(**)}\in}\bigcap_{u\in F}I_{x_u}^2\]
lies in $M$ since this implies immediately $\ell\in M$. By Theorem $\ref{mainrep}$, this task reduces to find $k\in\N$ such that
$f_k>0$ on $S\setminus F$ and $(\hess(f_k))(u)\succ0$ for all $u\in F$ where
\[f_k:=\st(\ell)-\sum_{i=1}^m\st(s_i^2)(1-g_i)^{2k}g_i\in\R[\x]
\]
is the standard part of this polynomial. Note for later use that $f_k$ and $\nabla f_k$ vanish on $F$ for all $k\in\N$.
In order to find such a $k$, we calculate
\begin{align*}
(\hess f_k)(u)&\overset{(***)}=-\sum_{i=1}^m\st(\la_{ui})\hess((1-g_i)^{2k}g_i)(u)\\
&\overset{\ref{derexo}}{\underset{(*)}=}\sum_{i=1}^m\st(\la_{ui})(4k(\nabla g_i)(\nabla g_i)^T-\hess g_i)(u)
\end{align*}
for $u\in F$ and $k\in\N$.
By Lemma \ref{quasi2concave} we can choose $k\in\N$ such that
$g_i(1-g_i)^{2k}$ is strictly concave on $\{x\in F\mid g_i(x)=0\}$ for $i\in\{1,\ldots,m\}$. Since $\st(\la_1)+\ldots+\st(\la_m)>0$
[$\to$ \ref{finitecontact}(c)], we get together with $(*)$ and \ref{derexo} that for all sufficiently large $k$, we
have $(\hess f_k)(u)\succ0$ for all $u\in F$. In particular, we can choose $k_0\in\N$ such that
$\hess(f_{k_0})(u)\succ0$ for all $u\in F$. Since $f_{k_0}$ and $\nabla f_{k_0}$ vanish on $F$,
we have by elementary analysis that there is an open subset $U$ of $\R^n$ containing $F$ such that $f_{k_0}>0$ on $U\setminus F$.
Now $S\setminus U$ is compact so that we can choose $N\in\N$ with $\st(\ell)\ge\frac1N$ and
$\st(s_i^2)\le N$ on $S\setminus U$. Then $f_k\ge\frac1N-m\frac N{2k}$ on $S\setminus U$ by
Exercise \ref{calculusexo} since $0\le g_i\le1$ on $S$ for all $i\in\{1,\ldots,m\}$.
For all sufficiently large $k\in\N$ with $k\ge k_0$, we now have $f_k>0$ on $S\setminus U$ and
because of $f_k\ge f_{k_0}>0$ on $(S\cap U)\setminus F$ (use again that $0\le g_i\le1$ on $S$) even $f_k>0$ on $S\setminus F$.
\end{proof}

\begin{cor}\label{linearstabilitycor}
Let $m\in\N_0$ and $\g\in\R[\x]^m$ such that $M(\g)$ is Archimedean and suppose that $S:=S(\g)$
has nonempty interior near its convex boundary.
Suppose that $g_i$ is strictly quasiconcave on $(\convbd S)\cap Z(g_i)$ for each
$i\in\{1,\dots,m\}$. Let $R$ be a real closed extension field of $\R$ and
$\ell\in R[\x]_1$ such that $\ell\ge0$ on $\transfer_{\R,R}(S)$.
Then $\ell$ lies in the quadratic module generated by $g_1,\ldots,g_m$ in $R[\x]$.
\end{cor}

\begin{cor}\label{strqc}
Let $m\in\N_0$ and $\g\in\R[\x]^m$ such that $M(\g)$ is Archimedean and suppose that $S(\g)$
has nonempty interior near its convex boundary.
Suppose that $g_i$ is strictly quasiconcave on $(\convbd S(\g)\cap Z(g_i)$ for each
$i\in\{1,\dots,m\}$.
Then there exists \[d\in\N\] such that for all $\ell\in\R[\x]_1$ with $\ell\ge0$ on $S(\g)$, we have \[\ell\in M_d(\g).\]
\end{cor}

\begin{proof}{}(cf. the proofs of Theorems \ref{h17bound} and \ref{putinarzerosdegreebound})
For each $d\in\N$, consider the class $S_d$
of all pairs $(R,a_0,a_1,\ldots,a_n)$ where $R$ is a real closed extension field of $\R$ and
$a_0,a_1,\ldots,a_n\in R$
such that whenever \[\forall x\in\transfer_{\R,R}(S):a_1x_1+\ldots+a_nx_n+a_0\ge0\]
holds, the polynomial $a_1X_1+\ldots+a_nX_n+a_0$ is a sum
of $d$ elements from $R[\x]$ where each term in the sum is of degree at most $d$
and is of the form $p^2g_i$ with $p\in R[\x]$ and $i\in\{0,\dots,m\}$ where $g_0:=1\in R[\x]$ [$\to$ \ref{commentonbounds}(a)].
By real quantifier elimination \ref{elim}, it
is easy to see that this is an $(n+1)$-ary $\R$-semialgebraic class.
Set $\mathcal E:=\{S_d\mid d\in\N\}$ and
observe that $\forall d_1,d_2\in\N:\exists d_3\in\N:S_{d_1}\cup S_{d_2}\subseteq
S_{d_3}$ (take $d_3:=\max\{d_1,d_2\}$). By \ref{linearstabilitycor},
we have $\bigcup\mathcal E=\mathcal R_{n+1}$. Now
\ref{finitenesscor} yields $\set_\R(S_d)=\R^{n+1}$ for some $d\in\N$.
\end{proof}

Our lecture notes culminate in the following result which is a contribution to the theory of solving systems of polynomial inequalities.

\begin{cor}[Kriel, Schweighofer \cite{ks'}]
Let $m\in\N_0$ and $\g\in\R[\x]^m$ such that $M(\g)$ is Archimedean and suppose that $S(\g)$
has nonempty interior near its convex boundary.
Suppose that $g_i$ is strictly quasiconcave on $(\convbd S(\g))\cap Z(g_i)$ for each
$i\in\{1,\dots,m\}$. Then \[S_d(\g)=\conv S(\g)\] for all sufficiently large $d\in\N$.
\end{cor}

\begin{proof} First consider the special case $S(\g)=\emptyset$. In this case, $-1\in M(\g)$ for example by \ref{putinar} or by \ref{strqc}. By Definition \ref{lasserredef}, this
entails $L_d(\g)=\emptyset$ and thus $S_d(\g)=\emptyset=S(\g)$ for all sufficiently large $d\in\N$.

Now suppose that $S(\g)\ne\emptyset$. Since $M(\g)$ is Archimedean, $S(\g)$ is then compact. By \ref{convbdchar} and \ref{takeson}, it follows that $\convbd S(\g)\ne\emptyset$.
In particular, $S^\circ\ne\emptyset$ by Definition \ref{nearconvexboundary}. Hence the conditions of Theorem \ref{exact2} are met and what we have to show is therefore
exactly that there is $d\in\N$ such that
\[\forall f\in\R[\x]_1:(f\ge0\text{ on }S(\g)\implies f\in M_d(\g)).\]
But this is exactly what Corollary \ref{strqc} says.
\end{proof}









\backmatter

\printindex

\end{document}